\documentclass[aap]{imsart}

\RequirePackage{amsthm,amsmath,amsfonts,amssymb}

\RequirePackage[OT1]{fontenc}
\RequirePackage{graphicx,color,amssymb,amsmath,amsthm,mathrsfs}
\RequirePackage[authoryear]{natbib}
\RequirePackage[colorlinks,citecolor=blue,urlcolor=blue, hypertexnames=false]{hyperref}
\RequirePackage{tikz}
\RequirePackage{bbm}
\usepackage{caption} 
\usepackage{xr}
\usepackage{multibib}

\newcites{SM}{References}

\startlocaldefs
\numberwithin{equation}{section}
\numberwithin{figure}{section}
\numberwithin{table}{section}
\theoremstyle{plain}
\newtheorem{thm}{Theorem}[section]
\newtheorem{lemma}[thm]{Lemma}
\newtheorem{proposition}[thm]{Proposition}
\newtheorem{cor}[thm]{Corollary}
\newtheorem{prop}[thm]{Proposition}
\theoremstyle{remark}

\def\R{\mathbb{R}}

\def\N{\mathbb{N}}
\def\Q{\mathbb{Q}}
\def\Z{\mathbb{Z}}

\def\1{\mathbbm{1}}

\def\E{\mathbb{E}}
\def\Pr{\mathbb{P}}
\def\cF{\mathcal{F}}

\endlocaldefs

\makeatletter
\let\c@table\c@figure
\makeatother

\begin{document}

\begin{frontmatter}
\title{The level of self-organized criticality in oscillating Brownian motion: $\MakeLowercase{n}$-consistency and stable Poisson-type convergence of the MLE}
\runtitle{$\MakeLowercase{n}$-consistency and stable Poisson-type convergence of the MLE}

\begin{aug}
\author[A]{\fnms{Johannes}~\snm{Brutsche}\ead[label=e1]{johannes.brutsche@stochastik.uni-freiburg.de}}
\and
\author[A]{\fnms{Angelika}~\snm{Rohde}\ead[label=e3]{angelika.rohde@stochastik.uni-freiburg.de}}
\address[A]{Mathematical Institute, University of Freiburg\printead[presep={,\ }]{e1}\printead[presep={,\ }]{e3}}
\end{aug}

\begin{abstract}
For some discretely observed path of oscillating Brownian motion with level of self-organized criticality $\rho_0$, we prove in the infill asymptotics that the MLE is $n$-consistent, where $n$ denotes the sample size, and derive its limit distribution with respect to stable convergence. As the transition density of this homogeneous Markov process is not even continuous in $\rho_0$, the analysis is highly non-standard. Therefore,  interesting and somewhat unexpected phenomena occur: The likelihood function splits into several components, each of them contributing very differently depending on how close the argument $\rho$ is to $\rho_0$. Correspondingly, the MLE is successively excluded to lay outside a compact set, a $1/\sqrt{n}$-neighborhood and finally a $1/n$-neighborhood of $\rho_0$ asymptotically.
The crucial argument to derive the stable convergence is to exploit the semimartingale structure of the sequential suitably rescaled local log-likelihood function (as a process in time). Both sequentially and as a process in $\rho$, it exhibits a bivariate Poissonian behavior in the stable limit with its intensity being a multiple of the local time at $\rho_0$.
\end{abstract}



\end{frontmatter}


\numberwithin{equation}{section}
\addtocontents{toc}{\protect\setcounter{tocdepth}{0}}

\section{Introduction}

Let $X=(X_t)_{t\in [0,1]}$ be the Markov process solving the homogeneous stochastic differential equation (SDE)
\begin{align}\label{eq:SDE_OBM}
dX_t = \sigma_{\rho}(X_t) dW_t, \qquad X_0 =x_0\in\R,
\end{align} 
where $W$ is a standard Brownian motion and for different (but arbitrary) numbers $\alpha,\beta>0$, the diffusion coefficient is given by
\begin{align}\label{eq:coefficient_OBM}
\sigma_{\rho}(x) := \begin{cases} \alpha, &\textrm{ if } x< \rho, \\ \beta, &\textrm{ if } x \geq\rho.\end{cases}
\end{align}
As shown in~\cite{LeGall}, the SDE \eqref{eq:SDE_OBM} possesses a unique strong solution. We call the parameter $\rho\in\R$ the {\it level of self-organized criticality}. In physics and biology, such type of process naturally arises when describing diffusive motion in porous or highly inhomogeneous media. In contrast to traditional change-point models where the structural break in volatility occurs in time, its occurrence here depends  purely  on the state of the process $X$ itself. For $\rho=0$, the process described by~\eqref{eq:SDE_OBM} has been named oscillating Brownian motion (OBM), introduced and  first studied in~\cite{Keilson/Wellner}, where  its transition semigroup is provided in particular. By Proposition~$4.2$ in \cite{Blanchard/Roeckner/Russo}, the transition density $p_t^\rho(x,y)$ from state $x$ to state $y$ within time $t$ of the homogeneous Markov process $X$ for general $\rho$ is given by the unique solution of the Kolmogorov forward (or Fokker--Planck) equation
\begin{equation}\label{eq: FP} 
\frac12 \frac{\partial^2}{\partial y^2}\left( \sigma_{\rho}(y)^2 p_t^{\rho}(x,y)\right) = \frac{\partial}{\partial t} p_t^{\rho}(x,y)
\end{equation}
in a distributional sense. 
Here, we are in one of the rare situations where the partial differential equation has an explicit solution that is given in~\eqref{eq:transition_density}. It is important to anticipate that this transition density  is not continuous in the parameter $\rho$.

In this article, we develop the theory of the maximum likelihood estimator (MLE) for $\rho$ on the basis of discrete  observations 
\begin{align}\label{eq:observations}
X_{i/n}, \ \ i=1,\dots, n, 
\end{align} 
as  $n\rightarrow\infty$. In this infill asymptotics or so-called high-frequency observation scheme, averaged squared estimators for $\alpha^2$ and $\beta^2$ have been studied in~\cite{Lejay/Pigato} when $\rho=0$ is known. They especially establish stable convergence of their estimators towards Gaussian mixtures at $\sqrt{n}$-convergence rate. Likewise,  when the level of self-organized criticality is known to the statistician, \cite{Mazzonetto} proved very recently the convergence of suitable statistics to the local time. For so-called drifted OBM, \cite{Lejay/PigatoII} study maximum likelihood estimators for drift parameters in case $\rho=0$. 
Within the high-frequency literature, our study is complementary in the sense that we treat $\alpha$ and $\beta$ as given while inferring about $\rho$. We also quote the contributions~\cite{Kutoyants}, \cite{Su/Chan_1}, \cite{Su/Chan_2}, and~\cite{Mazzonetto/Pigato} for inference in ergodic threshold diffusions as the time horizon of the observed trajectories goes to infinity, where estimation of the threshold is already an issue for a continuous record of observations if the threshold is solely present in the drift.
By the Markovian structure, the likelihood function for~\eqref{eq:observations} factorizes into a product of transition densities such that the log-likelihood function has the form
\[ L_n(x_0,\dots, x_n;\rho) := \sum_{k=1}^n \log\left(p_{1/n}^\rho(x_{k-1},x_k)\right).\]
This function is not continuous in $\rho$ (indeed it is not even upper semicontinuous), but càdlàg for the version~\eqref{eq:transition_density} of the transition density. We define the MLE $\hat{\rho}_n$ as some arbitrary but measurably selected representative of
\begin{align*}
\underset{\rho\in\R}{\mathrm{Argsup}\ } L_n\left(X_0, X_{1/n}, \dots, X_1; \rho\right),
\end{align*}  
where for any càdlàg function $f:\R\rightarrow\R$ 
\[ \underset{x\in\R}{\mathrm{Argsup}\ } f(x) := \bigg\{x\in\R:\ \max\Big\{\lim_{u\nearrow x, u<x} f(u), f(x)\Big\} = \sup_{y\in\R} f(y)\bigg\}. \] 
Note that $\mathrm{Argsup}_\rho L_n(\rho)$ is always a closed subset of $\R$ but may be empty in which case we define $\hat{\rho}_n=0$. If it is not empty, a measurable selection always exists bei Proposition~$2.8$ in~\cite{Ferger}. We write $\arg\sup$ if $\mathrm{Argsup}$ is a singleton.
As the MLE hinges on the likelihood function, the discontinuity of the latter is precisely the reason why the MLE theory is getting highly non-standard and exhibits fascinating phenomena. For the mathematical analysis it turns out to be purposeful to rewrite $\hat{\rho}_n = \rho_0+\hat{\theta}_n$ for the true parameter $\rho_0$, where
\[ \hat{\theta}_n \in \underset{\theta\in\R}{\mathrm{Argsup}\ } \sum_{k=1}^n \log\left( p_{1/n}^{\rho_0+\theta}(X_{(k-1)/n},X_{k/n})\right).\]
Thus, $\hat{\rho}_n\approx \rho_0$ corresponds to $\hat{\theta}_n\approx 0$. Observe also that $\hat{\theta}_n\in\mathrm{Argsup}_{\theta\in\R} \ell_n(\theta)$, where
\begin{align}\label{def:log_likelihood_normalized}
\ell_n(\theta) := \ell_n(X_0,X_{1/n},\dots, X_1;\theta) := \sum_{k=1}^n \log\Bigg(\frac{p_{1/n}^{\rho_0+\theta}(X_{(k-1)/n},X_{k/n})}{p_{1/n}^{\rho_0}(X_{(k-1)/n},X_{k/n)}}\Bigg).
\end{align}
Due to the four regimes $\{x<\rho, y\leq\rho\}, \{x\geq\rho, y>\rho\}, \{x<\rho<y\}$, and $\{y\leq \rho\leq x\}$ appearing in the transition density~\eqref{eq:transition_density}, every summand in $\ell_n(\theta)$ splits into nine disjoint regimes which are listed in~\eqref{eq:cases_I_jk}. Correspondingly, $\ell_n(\theta)$ itself admits a decomposition into nine sums and, impressively, the nature of the MLE is characterized by their subtle interplay. Which of these sums dominantly impact the shape and 'randomness' of $\ell_n(\theta)$ depends crucially on whether $\theta \gg 1/\sqrt{n}$, $\theta\asymp 1/\sqrt{n}$ or $\theta\ll 1/\sqrt{n}$. Realizations of the log-likelihood function are depicted in Figure~\ref{figure:Likelihood_example}. Two characteristic features immediately catch the eye: 
\begin{itemize}
\item the triangular shape in a neighborhood of $\rho_0$ (corresponding to $\theta=0$),
\item the jumps within the $1/n$-environment of $\rho_0$ with piecewise linear behavior in between.
\end{itemize}


\begin{figure}[h]
\centering
\includegraphics[scale=0.285]{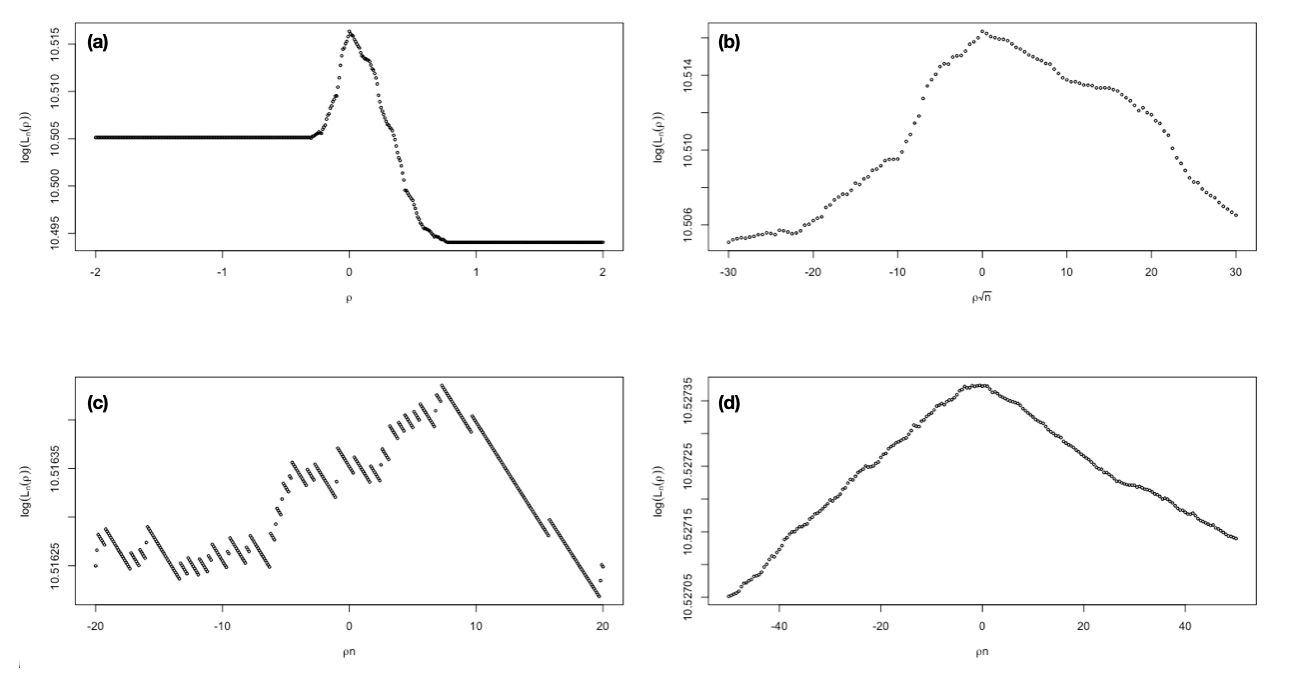}
\caption{\small Different realizations of the log-likelihood function for $n=10,000$, based on a path of $X$ simulated for the parameters $\alpha=0.5$, $\beta=0.7$ and $\rho_0=0$ on a grid of the interval $[0,1]$ with gridsize $10^{-6}$. \textbf{(a)} A path in the 'global' environment. \textbf{(b)} A path in the $1/\sqrt{n}$-neighborhood. \textbf{(c)} A path in the $1/n$-neighborhood. \textbf{(d)} An average over $50$ realizations of the likelihood in the $1/n$-neighborhood.} \label{figure:Likelihood_example}
\end{figure}

\noindent
To formally state the result, define two independent standard Poisson processes $N$ and $N'$ on $\R_{\geq 0}$ on a suitable so-called very good extension of the original probability space $(\Omega,\cF,\Pr)$ that are independent of $W$, see Section~\ref{Section:proof_thm} for details. Here, we restrict attention to the case of $\cF$ being the completed $\sigma$-field generated by $W$. Then we define a process $\ell = (\ell(z))_{z\in\R}$ via
\begin{align}\label{eq:def_ell}
\begin{split}
\ell(z) &:= \left( \1_{\{z\geq 0\}} \left( b_{\alpha,\beta} - \frac{1}{\beta^2}\log\left(\frac{\beta^2}{\alpha^2}\right)\right)+  \1_{\{z< 0\}} \left( b_{\alpha,\beta}' - \frac{1}{\alpha^2}\log\left(\frac{\alpha^2}{\beta^2}\right)\right) \right) |z| \\
&\hspace{1cm} + \left( \1_{\{z\geq 0\}}\log(\beta^2/\alpha^2)N(z/\beta^2) + \1_{\{z<0\}}\log(\alpha^2/\beta^2) N'((-z/\alpha^2)-) \right),
\end{split}
\end{align} 
where the constants $b_{\alpha,\beta},b_{\alpha,\beta}' <0$ are given explicitly in~\eqref{def:constant_b_alpha_beta}. Note that $\ell(z)$ is the sum of a two-sided compensated Poisson process and a negative drift, where we have used the left-continuous version $N'(\bullet -)$ to ensure $\ell$ being càdlàg. For statistical purposes, we even prove $\cF$-stable convergence (denoted as $\cF$-$st$) of the MLE instead of mere convergence in law. This notion of convergence was first introduced in~\cite{Renyi} and further studied in~\cite{Aldous/Eagleson}, see Section~\ref{Section:proof_thm} for a detailed description. Denoting by $L_t^{\rho}(X)$ the local time of $X$ in $\rho$ at time $t$, we are now in a position to state the main result of this article.

\begin{thm}\label{thm:MLE_limiting_distribution}
On the extended probability space, 
\[ n\left(\hat{\rho}_n-\rho_0\right) \1_{\{L_1^{\rho_0}(X)>0 \}}\ \stackrel{\cF-st}{\longrightarrow}\  \underset{z\in\R}{\arg\sup\ } \ell(z L_1^{\rho_0}(X))\1_{\{L_1^{\rho_0}(X)>0 \}},\]
where the right-hand side is a well-defined $\R$-valued random variable almost surely.
\end{thm}

\paragraph*{Statistical consequence}
Note that both the intensity of the two-sided Poisson process and the drift of the process on the right-hand side in Theorem~\ref{thm:MLE_limiting_distribution} are given as multiples of the local time $L_1^{\rho_0}(X)$. As concerns statistical applications, however, $L_1^{\rho_0}(X)$ is not observable and here the $\cF$-stable convergence is the way out: Due to this stronger mode of convergence as compared to convergence in distribution, both sides in Theorem~\ref{thm:MLE_limiting_distribution} can be multiplied by the $\cF$-measurable local time $L_1^{\rho_0}(X)$ and the convergence is still in place. By further replacing the local time by a local time estimator, we are even able to derive a limit that does not depend on any unknown quantities any longer. To be more precise, let $(\hat{L}_n^\rho)_{\rho\in\R}$ be a family of estimators satisfying
\begin{align}\label{eq_proof:uniform_local_time_estimator}
\hat{L}_n^{\hat{\rho}_n} \longrightarrow_{\Pr_{\rho_0}} L_1^{\rho_0}(X).
\end{align} 
Here, the subscript $\rho_0$ in the probability $\Pr_{\rho_0}$ indicates that $X$ solves~\eqref{eq:SDE_OBM} with $\rho=\rho_0$. Based on~\cite{Mazzonetto}, an example of local time estimators satisfying~\eqref{eq_proof:uniform_local_time_estimator} is given in Section~\ref{Appendix:Statistical_Application} of the supplement. The special structure of the limit $\ell$ reveals the distributional identity
\[ \mathcal{L}_{\rho_0}\left(\left. L_1^{\rho_0}(X)\ \underset{z\in\R}{\arg\sup\ }\ell\big(z L_1^{\rho_0}(X)\big) \right| L_1^{\rho_0}(X)>0\right)  =\mathcal{L}\left( \underset{z\in\R}{\arg\sup\ } \ell(z)\right),\]
and the stable convergence in Theorem~\ref{thm:MLE_limiting_distribution} together with the stochastic convergence in~\eqref{eq_proof:uniform_local_time_estimator} and Theorem~$3.18$(b) in~\cite{Haeusler/Luschgy} yield the joint stable convergence
\[ \Big(\hat{L}_n^{\hat{\rho}_n}, n\left(\hat{\rho}_n-\rho_0\right)\1_{\{L_1^{\rho_0}(X)>0\}} \Big) \stackrel{\cF-st}{\longrightarrow} \Big( L_1^{\rho_0}(X), \underset{z\in\R}{\arg\sup\ } \ell(z L_1^{\rho_0}(X))\1_{\{L_1^{\rho_0}(X)>0\}}\Big). \]
Together with the (stable) continuous mapping theorem and Theorem~$3.17$(iv) in~\cite{Haeusler/Luschgy}, this implies weak convergence of the conditional laws
\[ \mathcal{L}_{\rho_0}\left(\left. n\hat{L}_n^{\hat{\rho}_n} \big(\hat{\rho}_n-\rho_0\big)\right| L_1^{\rho_0}(X)>0\right) \Longrightarrow \mathcal{L}\left(\underset{z\in\R}{\arg\sup\ } \ell(z)\right).\]
Note that the right-hand side is independent of $X$ and $\rho_0$. Let $\kappa\in (0,1)$ be given and denote by $q_{\kappa/2}$ and $q_{1-\kappa/2}$ the $\kappa/2$- and $(1-\kappa/2)$-quantile of $\arg\sup_{z\in\R} \ell(z)$, respectively. Then 
\begin{align}\label{eq:confidence_interval}
\left[  \hat{\rho}_n -\frac{1}{n\hat{L}_n^{\hat{\rho}_n}}   q_{1-\kappa/2}, \hat{\rho}_n -\frac{1}{n\hat{L}_n^{\hat{\rho}_n}}  q_{\kappa/2} \right]
\end{align} 
is an asymptotic $(1-\kappa)$-confidence interval for $\rho_0$ conditional on $\{L_1^{\rho_0}(X)>0\}$ by means of Theorem~\ref{thm:MLE_limiting_distribution} (which should be read as $\R$ in case $\hat{L}_n^{\hat{\rho}_n}=0$). Note that $\{L_1^{\rho_0}(X) =0\}$ coincides with the event where the process $(X_t)_{0\leq t\leq 1}$ has not crossed the level $\rho_0$ according to Corollary~$29.18$ in~\cite{Kallenberg}. An unconditional confidence set can be obtained by first testing the null hypothesis $L_1^{\rho_0}(X)=0$ and adjusting the confidence level accordingly.\\

\paragraph*{Overview of the proof}
Central for the strategy of proof of Theorem~\ref{thm:MLE_limiting_distribution} is the characteristic property
\[ n(\hat{\rho}_n-\rho_0)\in\underset{z\in\R}{\mathrm{Argsup}\ } \ell_n(z/n). \]
Correspondingly, the proof contains two major steps, which are stated in the following two propositions.

\begin{prop}\label{prop:n-consistency}
The MLE $\hat{\rho}_n$ is $n$-consistent on $\{L_1^{\rho_0}(X)>0\}$, i.e.
\[ \lim_{K\to\infty} \limsup_{n\to\infty} \Pr_{\rho_0}\big( n\left|\hat{\rho}_n - \rho_0\right| >K, L_1^{\rho_0}(X) >0\big) =0.\]
\end{prop}

\begin{prop}\label{prop:convergence_ell_n}
For any $K>0$, we have
\begin{align*}
\left( \ell_n(z/n)\right)_{z\in [-K,K]} \stackrel{\cF-st}{\longrightarrow}\ \left( \ell(z L_1^{\rho_0}(X))\right)_{z\in [-K,K]}
\end{align*} 
in the Skorohod space $\mathcal{D}([-K,K])$.
\end{prop}

\noindent
Based on those results, the proof of Theorem~\ref{thm:MLE_limiting_distribution} is then completed with argsup-continuous mapping type arguments.

The proof of Proposition~\ref{prop:n-consistency} adopts the M-estimation approach, c.f. Section~$3.2$ in~\cite{Vaart/Wellner}. To this aim, the log-likelihood function $\ell_n$ is decomposed into a drift $B_n$ and a sum of martingale differences $M_n$ by adding and subtracting the sum of conditional expectations of each of its increments with respect to $\cF_{(k-1)/n}$, see~\eqref{eq:decomposition_ell_n}. While the martingale part $M_n$ consists of martingale increments which are convenient to work with, the problem of this natural decomposition is the randomness of the drift term $B_n$. This prevents us from directly applying Markov's inequality in the slicing device, where expressions of the form $\Pr_{\rho_0}\big( \sup_{\theta \in S_{n,j}} M_n(\theta) > -\sup_{\theta\in S_{n,j}} B_n(\theta)|L_1^{\rho_0}(X)>0\big)$ need to be bounded for suitably chosen partitions $S_{n,j}, j\in\N$, of the parameter space. Instead, we first deduce deterministic bounds for $\sup_{\theta\in S_{n,j}} B_n(\theta)\1_{A(n)}$ on certain sets $A(n)$ with high probability that require a precise understanding of the drift term. Here, the particular difficulty is caused by the decomposition of $\ell_n(\theta)$ into nine disjoint regimes which contribute to the drift rather differently depending on how close $\theta$ is to zero, see Figure~\ref{figure:route_of_proof}. Correspondingly, the MLE is successively excluded to lay outside a compact set, a $1/\sqrt{n}$-neighborhood and finally a $1/n$-neighborhood of $\rho_0$ asymptotically. For $\theta$ small enough, Proposition~\ref{prop:expansion_of_drift_t} provides an expansion that mimics the triangular shape of $\ell_n$ close to the true parameter seen in Figure~\ref{figure:Likelihood_example}, provided that $L_1^{\rho_0}(X)>0$. Already for larger values of $\theta$ in the $1/\sqrt{n}$-environment we are facing the problem that due the $n$-dependence of the step size in the transition density, the remainder terms in Taylor expansions within regimes (that Proposition~\ref{prop:expansion_of_drift_t} is built on) are of the same order as the leading terms and are therefore not useful any longer. As Figure~\ref{figure:Likelihood_example} shows, the triangular shape is indeed globally not valid.

The proof of Proposition~\ref{prop:convergence_ell_n} requires to prove $\cF$-stable convergence of finite dimensional distributions (fidis) and tightness. Unlike in many situations, the much more delicate part is to establish stable convergence of fidis, whereas tightness is a consequence of the standard moment criterion for tightness in the Skorohod space, combined with a thorough understanding of $\ell_n$, in particular the moments bounds provided in Proposition~\ref{prop:moment_product}. By means of the Cramér--Wold device, convergence of fidis is a consequence of convergence of arbitrary linear combinations of $\ell_n(z/n)$ for different $z$. As the process $(\ell_n(z/n))_{z\in\R}$ has no specific structure in its parameter $z$, for example being Markovian or a martingale, the crucial idea is to artificially consider the sequential process
\begin{align}\label{eq:def_ell_nt}
\ell_{n,t}(z/n) = \sum_{k=1}^{\lfloor nt\rfloor} \log\left(\frac{p_{1/n}^{\rho_0+z/n}(X_{(k-1)/n},X_{k/n})}{p_{1/n}^{\rho_0}(X_{(k-1)/n},X_{k/n)}}\right), \qquad t\in [0,1], 
\end{align} 
as a process in time and use its specific semimartingale decomposition given in~\eqref{eq:decomposition_ell_n} with respect to the discretized filtration $(\cF_{\lfloor nt\rfloor/n})_{t\in [0,1]}$. \cite{Jacod_stableGaussian} provides a stable convergence result in exactly this setting of infill asymptotics, but only for conditional Gaussian limits. \cite{Jacod_stablePII} describes stable convergence towards a more general class of processes that covers Poisson limits, but in its current formulation, this result does not apply to processes defined on discretized filtrations. We therefore adapt this result by combining it with the first one to bridge the gap between those two, see Proposition~\ref{prop:version_Jacod}. Based on this, proving stable convergence of fidis is traced back to proving (uniform) stochastic convergence of the semimartingale characteristics of the sequential version of any linear combination $\kappa_1\ell_n(z_1/n)+\dots + \kappa_N\ell_n(z_N/n)$. Fascinatingly, only two of the nine regimes in~\eqref{eq:cases_I_jk} contribute to the stochastic fluctuation of $\ell_n(z/n)$ asymptotically. The resulting limit can then be constructed based on a bivariate Poisson process as we observe Poissonian behavior both in time $t$ and location $z$. Remarkably, it turns out that $(\ell(z))_{z\geq 0}$ indeed is given as a martingale plus drift, whereas we did no have a martingale structure of $\ell_n(z/n)$ in $z$.

The article is organized as follows: On the basis of Propositions~\ref{prop:n-consistency} and~\ref{prop:convergence_ell_n}, the proof of Theorem~\ref{thm:MLE_limiting_distribution} is conducted in Section~\ref{Section:proof_thm}. Section~\ref{Section:likelihood} contains important features of the log-likelihood function, including the explicit form of the transition density, the above mentioned nine regimes, the martingale and drift decomposition of $\ell_n$ together with the central expansion of the drift and moment bounds for increments of $\ell_n$. The proof of Proposition~\ref{prop:n-consistency} is given in Section~\ref{Section:Consistency}. Section~\ref{Section:Limiting_distribution} contains the proof of Proposition~\ref{prop:convergence_ell_n}, including a heuristic reasoning for the Poissonian structure of the limit and an explicit construction of $(\ell(z))_{z\in\R}$. Remaining auxiliary results, contained in the Appendix Sections~\ref{App:notation}-\ref{App:conditions_prop_stable}, are deferred to the supplement.

\section{Proof of Theorem~\ref{thm:MLE_limiting_distribution}}\label{Section:proof_thm}

In this section the proof of Theorem~\ref{thm:MLE_limiting_distribution} is conducted on the basis of Proposition~\ref{prop:n-consistency} and~\ref{prop:convergence_ell_n}. The property of the limit is outsourced to Lemma~\ref{lemma:uniqueness_argsup_ell}.

\paragraph*{Preliminaries on stable convergence}
Let $(\Omega,\cF,\Pr)$ be a probability space supporting a standard Brownian motion $W$ with $\mathbb{F}=(\cF_t)_{t\in [0,1]}$ being the augmented filtration induced by $W$ and restrict attention to the case $\cF=\cF_1$. Let $(\Omega',\cF',(\cF_t')_{t\geq 0})$ be the canonical space of all $\R^2$-valued càdlàg functions on $\R_{\geq 0}$ with the canonical process $(N(z), N(z)')(\omega') = \omega'(z)$ and the right-continuous filtration generated by $(N,N')$. Furthermore, let $\Pr'$ be the unique probability measure on $(\Omega',\cF')$ under which the processes $N$ and $N'$ are independent standard Poisson processes. 
We then define a stochastic basis $\tilde{\mathcal{B}}=(\tilde{\Omega}, \tilde{\cF}, (\tilde{\cF}_t)_{t\geq 0}, \tilde{\Pr})$ via
\begin{align}\label{eq:extended_probability_space}
\tilde{\Omega} := \Omega\times \Omega', \quad \tilde{\cF} := \cF\otimes \cF', \quad \tilde{\cF}_t := \bigcap_{s>t} \cF_s\otimes\cF_s', \quad \tilde{\Pr} := \Pr\otimes\Pr'.
\end{align} 
It is clear in this case that the extension is \textit{very good} (in the sense of Definition~$II.7.1$ in~\cite{Jacod/Shiryaev}), meaning that every martingale on $\mathcal{B}$ is also a martingale on $\tilde{\mathcal{B}}$. Finally, we can define $(\ell(z))_{z\in\R}$ as given in~\eqref{eq:def_ell} on the stochastic basis $\tilde{\mathcal{B}}$. \\
Let $(Z_n)_{n\in\N}$ be a sequence of random variables with values in a metric space $E$ and defined on $(\Omega,\cF,\Pr)$ and $Z$ an $E$-valued random variable on $(\tilde{\Omega},\tilde{\cF},\tilde{\Pr})$. Then we say that $(Z_n)_{n\in\N}$ converges \textit{$\cF$-stably in law} to $Z$ if
\[ \E\left[ V f(Z_n)\right] \longrightarrow \tilde{\E}\left[ V f(Z)\right],\]
for all $f:E\longrightarrow\R$ bounded and continuous and all bounded random variables $V$ on $(\Omega,\cF)$, where $\tilde{\E}$ denotes the expectation with respect to $\tilde{\Pr}$. This property is (slightly) stronger than mere convergence in law. It applies in particular for $E$ being the Skorohod space $\mathcal{D}([0,1])$.

\begin{lemma}\label{lemma:uniqueness_argsup_ell}
Provided that $L_1^{\rho_0}(X)>0$, the set $\mathrm{Argsup}_{z\in\R} \ell(z L_1^{\rho_0}(X))$ is $\tilde{\Pr}$-almost surely non-empty and a singleton.
\end{lemma}
\begin{proof}
The equality $\ell(0)=0$ reveals the upper bound
\begin{align*}
&\tilde{\Pr}\Big(\mathrm{Argsup}_{z\in\R} \ell\big(z L_1^{\rho_0}(X)\big) = \emptyset, L_1^{\rho_0}(X)>0 \Big)\\
&\hspace{1cm}\leq \tilde{\Pr} \left( \bigcap_{M\in\N} \bigg\{ \sup_{|z|>2^M} \ell(zL_1^{\rho_0}(X)) \geq 0\bigg\}\cap\big\{ L_1^{\rho_0}(X)>0\big\}\right) \\
&\hspace{1cm}  = \lim_{M\to\infty} \tilde{\Pr}\bigg( \sup_{|z|>2^M} \ell(zL_1^{\rho_0}(X)) \geq 0, L_1^{\rho_0}(X)>0\bigg).
\end{align*} 
By disintegration, construction of $\tilde{\Pr}$ in~\eqref{eq:extended_probability_space}, independence of $N,N'$ and $X$ and monotone convergence,
\begin{align*}
& \lim_{M\to\infty}\tilde{\Pr}\bigg( \sup_{|z|>2^M} \ell(zL_1^{\rho_0}(X)) \geq 0, L_1^{\rho_0}(X)>0\bigg)\\
&\hspace{0.5cm}= \lim_{M\to\infty}\int_{(0,\infty)} \tilde{\Pr}\bigg(\sup_{|z|>2^M} \ell(zL_1^{\rho_0}(X)) \geq 0 \bigg| L_1^{\rho_0}(X)=c\bigg) \tilde{\Pr}^{L_1^{\rho_0}(X)}(dc) \\
&\hspace{0.5cm} = \int_{(0,\infty)}\lim_{M\to\infty} \tilde{\Pr}\bigg(\sup_{|z|>2^M} \ell(cz) \geq 0\bigg) \Pr^{L_1^{\rho_0}(X)}(dc).
\end{align*}  
Recall that under $\Pr'$, both $N$ and $N'$ are independent standard Poisson processes. Then we obtain by the union bound in the first step and  Doob's maximal inequality in the second one,
\begin{align*}
\Pr'\left( \sup_{z>2^M} \ell(cz) \geq 0\right) & \leq \sum_{j\geq M} \Pr'\left( \sup_{z\in (2^j,2^{j+1}]} \log\left(\frac{\beta^2}{\alpha^2}\right)\left( N(cz/\beta^2)- \frac{cz}{\beta^2}\right) \geq -b_{\alpha,\beta}c 2^j \right) \\
& \leq \log\left(\frac{\beta^2}{\alpha^2}\right)\frac{2}{(-b_{\alpha,\beta})c} \sum_{j\geq M} 2^{-j}\E'\left[\left( N(c2^{j+1}/\beta^2)- \frac{c2^{j+1}}{\beta^2}\right)^2\right]^\frac12 \\
& = \log\left(\frac{\beta^2}{\alpha^2}\right)\frac{2}{(-b_{\alpha,\beta})\beta\sqrt{c}} \sum_{j\geq M} 2^{-j} 2^{(j+1)/2} \stackrel{M\to\infty}{\longrightarrow} 0
\end{align*}
for any $c>0$. The case $z\leq -2^M$ can be dealt with analogously and it follows that the random set $\mathrm{Argsup}_{z\in\R} \ell(z L_1^{\rho_0}(X)) $ is non-empty almost surely.
For the uniqueness we obtain again by disintegration as above
\begin{align*}
\tilde{\Pr}\bigg( \#\underset{z\in\R}{\mathrm{Argsup}\ } \ell(zL_1^{\rho_0}(X)) >1, L_1^{\rho_0}(X)>0 \bigg) \hspace{-0.2mm} =\hspace{-0.2mm} \int_{(0,\infty)} \hspace{-2mm}\Pr'\bigg( \#\underset{z\in\R}{\mathrm{Argsup}\ } \ell(cz) >1 \bigg) \Pr^{L_1^{\rho_0}(X)}(dc)
\end{align*} 
and distinguish cases to prove that the integrand is equal to zero:
\begin{itemize}
\item[$\bullet\ \boldsymbol{\#([0,\infty)\cap \mathrm{Argsup}_z \ell(cz))\geq 2}$.] Let $z_1,z_2\in \mathrm{Argsup}_z \ell(cz)$ with $z_1,z_2\geq 0$. By monotonicity of the (deterministic) first summand in the definition of $\ell(cz)$, it follows that $N(cz_i/\beta^2) - N(cz_i/\beta^2-)\neq 0$ for $i=1,2$. Moreover, as the Poisson process $N$ takes only values in $\N_0$ we obtain that $\ell(cz_1)=\ell(cz_2)$ is only possible for $z_1-z_2\in \kappa_{\alpha,\beta}\Z$ for suitable $\kappa_{\alpha,\beta}$.
Let $T_n$ be the $n$-th jump time of $N(c/\beta^2)$. Then $T_n-T_m$, $m<n$, is gamma distributed with parameters $n-m$ and $c/\beta^2$, and we conclude
\begin{align*}
\Pr'\left( \# \Big( [0,\infty)\cap \underset{z\in\R}{\mathrm{Argsup}\ } \ell(cz) \Big) >1 \right) &\leq \Pr'\left( \bigcup_{m,n\in\N:\ m<n} \left\{ T_n - T_m \in\kappa_{\alpha,\beta}\Z\right\} \right) \\
&\leq \sum_{m,n\in\N:\ m<n} \Pr'\left( T_n - T_m \in\kappa_{\alpha,\beta}\Z\right) =0,
\end{align*}
as the probability of a gamma distributed random variable being in some countable set is zero.\smallskip

\item[$\bullet\ \boldsymbol{\#((-\infty,0]\cap \mathrm{Argsup}_z \ell(cz))\geq 2}$.] This follows as in the previous case.\smallskip

\item[$\bullet\ \boldsymbol{\#(-\infty,0]\cap\mathrm{Argsup}_z \ell(cz))\geq 1, \#([0,\infty)\cap\mathrm{Argsup}_z \ell(cz))\geq 1 }$.] Let $z_1\geq 0$ and $z_2<0$ with $z_1,z_2\in \mathrm{Argsup}_z \ell(cz)$. As in the first case, $N$ jumps at $z_1$ and $N'$ at $z_2$. As both $N$ and $N'$ only take values in $\N_0$, for each $z_1$ there exists a countable set $A(z_1)$ such that $\ell(cz_1)=\ell(cz_2)$ is only valid for $z_2\in A(z_1)$. Denoting by $T_n^1, T_n'$ the $n$-th jump time of $N(c\bullet/\beta^2)$ and $N'(c\bullet/\alpha^2)$, respectively, we obtain from independence of $N$ and $N'$,
\begin{align*}
&\Pr'\left( \#\Big( (-\infty,0]\cap \underset{z\in\R}{\mathrm{Argsup}\ } \ell(cz)\Big) \geq 1, \#\Big( [0,\infty)\cap \underset{z\in\R}{\mathrm{Argsup}\ } \ell(cz)\Big) \geq 1  \right) \\
&\hspace{1cm}\leq \sum_{n=1}^\infty \int_0^\infty \Pr'\left(\left.\bigcup_{m\in\N} \{T_m' \in A(z_1)\} \right| T_n = z_1\right) (\Pr')^{T_n}(dz_1) \\
&\hspace{1cm} \leq \sum_{m,n=1}^\infty \int_0^\infty \Pr'\left( T_m' \in A(z_1)\right) (\Pr')^{T_n}(dz_1) =0.
\end{align*}
\end{itemize}
\end{proof}


\begin{proof}[Proof of Theorem~\ref{thm:MLE_limiting_distribution}]
Recall $n(\hat{\rho}_n-\rho_0)\in\mathrm{Argsup}_{z\in\R}\ell_n(z/n)$. By Proposition~\ref{prop:n-consistency}, the sequence $(n(\hat{\rho}_n-\rho_0)\1_{\{L_1^{\rho_0}(X)>0\}})_{n\in\N}$ is tight and by Proposition~\ref{prop:convergence_ell_n} and Lemma~$3$ in Section~$16$ of~\cite{Billingsley_2} we have 
\begin{align}\label{eq_proof:main_thm_1}
\left( \ell_n(z/n)\right)_{z\in\R} \stackrel{\cF-st}{\longrightarrow}\ \left( \ell(zL_1^{\rho_0}(X))\right)_{z\in\R}. 
\end{align} 
Here, the Skorohod space $\mathcal{D}(\R)$ is endowed with the topology of Skorohod convergence on compact sets. Let $F\in\cF$ be a set with $\Pr(F)>0$. Recalling that~\eqref{eq_proof:main_thm_1} is given on the extended probability space $(\tilde{\Omega},\tilde{\cF},\tilde{\Pr})$ defined in~\eqref{eq:extended_probability_space} and denoting $A:=\{L_1^{\rho_0}(X)>0\}$, we define $\tilde{\Pr}_{A\cap F}$ via 
\[ \tilde{\Pr}_{A\cap F}(B) := \frac{\tilde{\Pr}((A\cap F)\cap B)}{ \tilde{\Pr}(A\cap F)} \]
for any $B\in\tilde{\cF}$ and the identification of $A\cap F$ with the set $(A\cap F)\times\Omega'\in\tilde{\cF}$. As the convergence~\eqref{eq_proof:main_thm_1} is $\cF$-stable, by Theorem~$3.17$(iv) in~\cite{Haeusler/Luschgy}, we find
\[ \left( \ell_n(z/n)\right)_{z\in\R}\ \stackrel{\tilde{\Pr}_{A\cap F}}{\Longrightarrow}\ \left( \ell(zL_1^{\rho_0}(X))\right)_{z\in\R},\]
where $\stackrel{\tilde{\Pr}_{A\cap F}}{\Longrightarrow}$ denotes weak convergence of random variables with respect to the measure $\tilde{\Pr}_{A\cap F}$. Clearly, $(n(\hat{\rho}_n-\rho_0)\1_{\{L_1^{\rho_0}(X)>0\}})_{n\in\N}$ is also tight with respect to $\tilde{\Pr}_{A\cap F}$ and by Lemma~\ref{lemma:uniqueness_argsup_ell} the Argsup of $(\ell(zL_1^{\rho_0}(X)))_{z\in\R}$ is a singleton in $\R$ almost surely with respect to this measure $\tilde{\Pr}_{A\cap F}$. Hence, Theorem~$3.12$ in~\cite{Ferger} yields for any bounded and continuous function $f:\R\rightarrow\R$,
\[ \tilde{\E}_{\rho_0}\left[ \1_{A\cap F} f\left( \underset{z\in\R}{\arg\sup\ }\ell_n(z/n)\right)\right] \longrightarrow \tilde{\E}_{\rho_0}\left[ \1_{A\cap F} f\left( \underset{z\in\R}{\arg\sup\ }\ell(zL_1^{\rho_0}(X))\right)\right].\]
From this, we obtain
\begin{align*}
\tilde{\E}_{\rho_0}\left[ \1_{F} f\left( \underset{z\in\R}{\arg\sup\ }\ell_n(z/n)\1_{A} \right)\right] \longrightarrow \tilde{\E}_{\rho_0}\left[ \1_{F} f\left( \underset{z\in\R}{\arg\sup\ }\ell(zL_1^{\rho_0}(X))\1_{A}\right)\right].
\end{align*} 
By Theorem~$3.17$(iv) in~\cite{Haeusler/Luschgy}, this implies
\[ \underset{z\in\R}{\arg\sup\ }\ell_n(z/n)\1_{\{L_1^{\rho_0}(X)>0\}} \stackrel{\cF-st}{\longrightarrow}\ \underset{z\in\R}{\arg\sup\ }\ell(zL_1^{\rho_0}(X))\1_{\{L_1^{\rho_0}(X)>0\}} \]
and the claim of Theorem~\ref{thm:MLE_limiting_distribution} follows.
\end{proof}

\section{The log-likelihood function and its fundamental properties}\label{Section:likelihood}

The transition density $p_t^{\rho}(x,y)$ of the Markov process solving the SDE~\eqref{eq:SDE_OBM} is given as 

\begin{align}\label{eq:transition_density}
p_t^\rho(x,y) = \begin{cases}
\frac{1}{\sqrt{2\pi t}\alpha}\left[ \exp\left( -\frac{(y-x)^2}{2t\alpha^2}\right) - \frac{\alpha-\beta}{\alpha+\beta}\exp\left(-\frac{(y-2\rho+x)^2}{2t\alpha^2}\right)\right] &\textrm{ for } x<\rho, y\leq \rho,\vspace{0.2cm} \\
\frac{1}{\sqrt{2\pi t}\beta}\left[ \exp\left( -\frac{(y-x)^2}{2t\beta^2}\right) + \frac{\alpha-\beta}{\alpha+\beta}\exp\left(-\frac{(y-2\rho+x)^2}{2t\beta^2}\right)\right] &\textrm{ for } x\geq\rho, y>\rho, \vspace{0.2cm} \\
\frac{2}{\alpha+\beta}\frac{\alpha}{\beta}\frac{1}{\sqrt{2\pi t}}\exp\left( -\frac{1}{2t}\left(\frac{y-\rho}{\beta} - \frac{x-\rho}{\alpha}\right)^2\right) & \textrm{ for } x<\rho< y, \vspace{0.2cm} \\
\frac{2}{\alpha+\beta}\frac{\beta}{\alpha}\frac{1}{\sqrt{2\pi t}}\exp\left( -\frac{1}{2t}\left(\frac{y-\rho}{\alpha} - \frac{x-\rho}{\beta}\right)^2\right) & \textrm{ for } y\leq\rho\leq x. \vspace{0.2cm} \\
\end{cases}
\end{align}
Note that $p_t^\rho(x,y) = p_t^0(x-\rho, y-\rho)$. It is easily verified that $p_t^{\rho}(x,y)$ solves~\eqref{eq: FP}, i.e. 
\begin{align*}\label{eq:Fokker-Planck_equation}
\int_\R p_t^\rho(x,y)\varphi(y) dy = \varphi(x) + \int_0^t\int_\R \frac{1}{2} \sigma_{\rho}(y)^2 p_s^\rho(x,y)\varphi''(y) dy ds
\end{align*}
for every $t>0$ and any $\varphi\in\mathcal{S}(\R)$, where $\mathcal{S}(\R)$ denotes the Schwartz space of rapidly decreasing functions on $\R$. 
Corresponding to the four regimes in the transition density~\eqref{eq:transition_density},
\[ p_{1/n}^{\rho_0+\theta}(X_{(k-1)/n},X_{k/n}) \Big/ p_{1/n}^{\rho_0+\theta'}(X_{(k-1)/n},X_{k/n}) \]
is decomposed according to nine different regimes. For $\theta'\leq \theta$, those are given explicitly as
\begin{align}\label{eq:cases_I_jk}
\begin{split}
&I_{1,k}^{\theta',\theta} := \{X_{(k-1)/n}< \rho_0+\theta', X_{k/n} \leq \rho_0+\theta'\},\\
&I_{2,k}^{\theta',\theta} := \{X_{(k-1)/n}<\rho_0+\theta'< X_{k/n}\leq \rho_0+\theta\},\\
&I_{3,k}^{\theta',\theta} := \{X_{(k-1)/n}<\rho_0+\theta'\leq\rho_0+\theta <X_{k/n}\}, \\
&I_{4,k}^{\theta',\theta} := \{X_{k/n}\leq\rho_0+\theta'\leq X_{(k-1)/n}<\rho_0+\theta \}, \\
&I_{5,k}^{\theta',\theta} :=  \{\rho_0+\theta'\leq X_{(k-1)/n}<\rho_0+\theta, \rho_0+\theta'< X_{k/n}\leq\rho_0+\theta \}, \\
&I_{6,k}^{\theta',\theta} :=  \{ \rho_0+\theta'\leq X_{(k-1)/n} < \rho_0+\theta < X_{k/n} \}, \\
&I_{7,k}^{\theta',\theta} :=  \{ X_{k/n}\leq\rho_0+\theta'\leq \rho_0+\theta\leq X_{(k-1)/n}\}, \\
&I_{8,k}^{\theta',\theta} :=  \{ \rho_0+\theta'< X_{k/n}\leq \rho_0+\theta\leq X_{(k-1)/n}\}, \\
&I_{9,k}^{\theta',\theta} :=  \{ \rho_0+\theta \leq X_{(k-1)/n}, \rho_0+\theta <X_{k/n}\}.
\end{split}
\end{align}
Note that all those regimes are (pairwise) disjoint. If $\theta'=0$, we write $I_{j,k}^{\theta} = I_{j,k}^{0,\theta}$. In particular, 
\begin{align}\label{eq:def_I}
\ell_n(\theta) = \sum_{j=1}^9\sum_{k=1}^n  \log\left(\frac{p_{1/n}^{\rho_0+\theta}(X_{(k-1)/n},X_{k/n})}{p_{1/n}^{\rho_0}(X_{(k-1)/n},X_{k/n)}}\right)\1_{I_{j,k}^{\theta}} =: \sum_{j=1}^9 \mathcal{I}_j(\theta).
\end{align} 
This is the decomposition of $\ell_n(\theta)$ into nine sums as already mentioned in the introduction and constitues the origin of the fascinating phenomena within the analysis of the MLE. Regimes and their correspondences are depicted in Figure~\ref{figure:regimes}. Interestingly, there are five different probabilistic types of contributions.

\begin{figure}[h]
\centering
\includegraphics[scale=0.32]{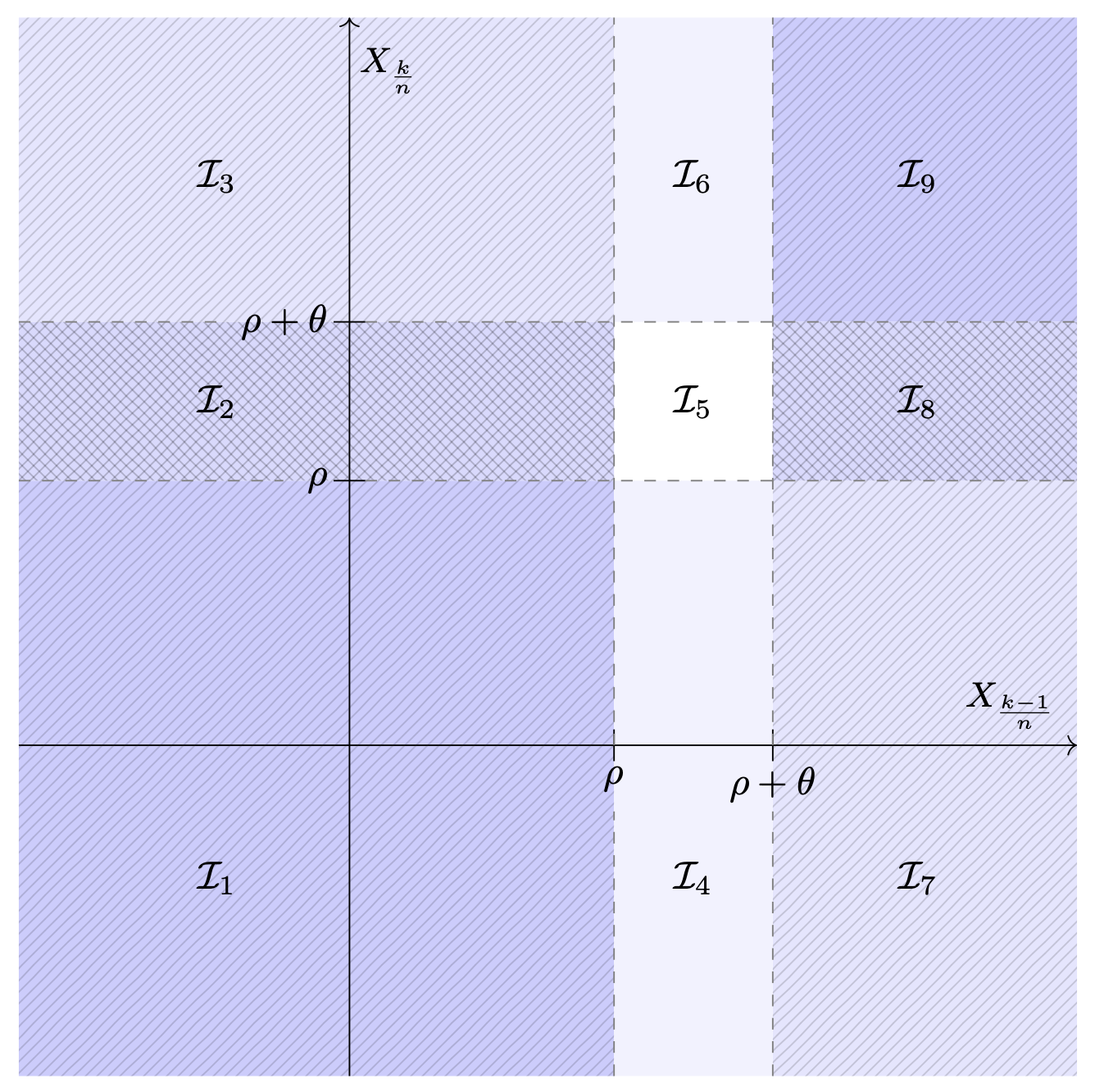}
\caption{\small Visualization of regimes~\eqref{eq:cases_I_jk} corresponding to the respective indicators in the nine summands in~\eqref{eq:def_I}. Regimes which will turn out to be treated analogously are colored the same. Moreover, regimes contributing to the drift and variance in the $1/n$-neighborhood are marked with north-east and north-west lines, respectively. }\label{figure:regimes}
\end{figure}

\noindent
Recall the sequential log-likelihood process $\ell_{n,t}(\theta)$ given in~\eqref{eq:def_ell_nt}, for which we have $\ell_{n,1}(\theta)=\ell_n(\theta)$, where $\ell_n(\theta)$ is given in~\eqref{def:log_likelihood_normalized}. It turns out to be purposeful to decompose 
\begin{align}\label{eq:decomposition_ell_n}
\begin{split}
\ell_{n,t}(\theta) = M_{n,t}(\theta) + B_{n,t}(\theta)
\end{split}
\end{align}
into a martingale term and a drift term by adding and subtracting the sum of the $\cF_{(k-1)/n}$-conditional expectation of each increment of $\ell_{n,t}(\theta)$, i.e. 
\[B_{n,t}(\theta) = \sum_{k=1}^{\lfloor nt\rfloor}\E_{\rho_0}\left[\left.\log\left(\frac{p_{1/n}^{\rho_0+\theta}(X_{(k-1)/n},X_{k/n})}{p_{1/n}^{\rho_0}(X_{(k-1)/n},X_{k/n)}}\right)\right| X_{(k-1)/n}\right]. \]
Note that $(M_{n,t})_{t\in [0,1]}$ is actually a martingale with respect to the discretized filtration $(\cF_{\lfloor nt\rfloor/n})_{t\in [0,1]}$.
The $t$-dependent version of $\ell_n(\theta)$ plays a central role in the study of the limiting distribution, whereas the $n$-consistency only uses $\ell_n(\theta) = M_n(\theta)+B_n(\theta)$, where $M_n(\theta):=M_{n,1}(\theta)$ and $B_n(\theta) :=B_{n,1}(\theta)$, correspondingly. Nevertheless, we will exploit the fact that $M_n(\theta)$ is a sum of martingale differences for the $n$-consistency proof, where the (negative) drift $B_n(\theta)$ is shown to dominate the stochastic fluctuations for $\theta$ being not too close to $0$. Note that $B_n(\theta)\leq 0$ follows direcly by Jensen's inequality for conditional expectation, whereas it will turn out to be challenging to bound it away from zero, and especially sufficiently far away from zero, for general $\theta$. For $\theta$ with $|\theta|\leq K/\sqrt{n}$ with $K>0$ arbitrary, we deduce an expansion for $B_{n,t}(\theta)$. Defining two numerical constants
\begin{align}\label{eq:F_alpha_beta}
F_{\alpha,\beta} := -\left(\frac{2(\alpha-\beta)}{\alpha\beta} + \frac{\alpha}{\beta}\frac{2}{\alpha+\beta}\log\left(\frac{\beta^2}{\alpha^2}\right)\right) 
\end{align}
and 
\begin{align}\label{eq:tildeF_alpha_beta}
\tilde{F}_{\alpha,\beta} := -\left(\frac{2(\beta-\alpha)}{\alpha\beta} + \frac{\beta}{\alpha}\frac{2}{\alpha+\beta}\log\left(\frac{\alpha^2}{\beta^2}\right)\right)
\end{align}
this expansion reads as follows.

\begin{prop}\label{prop:expansion_of_drift_t}
Let $K>0$ and $|\theta|\leq K/\sqrt{n}$. Then we have
\begin{align*}
B_{n,t}(\theta) = -|\theta| \left(\1_{\{\theta\geq 0\}}F_{\alpha,\beta}+\1_{\{\theta< 0\}} \tilde{F}_{\alpha,\beta}\right)\Lambda_{\alpha,\beta}^n\left((X_{(k-1)/n})_{1\leq k\leq \lfloor nt\rfloor}\right) + r_n(t,\theta),
\end{align*} 
where
\begin{align}\label{eq:Lambda_nt(X)}
\begin{split}
\Lambda_{\alpha,\beta}^n\left((X_{(k-1)/n})_{1\leq k\leq \lfloor nt\rfloor}\right) &= \frac{1}{\sqrt{2\pi/n}} \sum_{k=1}^{\lfloor nt\rfloor} \1_{\{X_{(k-1)/n}<\rho_0\}}\exp\left( -\frac{(X_{(k-1)/n}-\rho_0)^2}{2\alpha^2/n}\right)\\
&\hspace{1cm} + \frac{1}{\sqrt{2\pi/n}} \sum_{k=1}^{\lfloor nt\rfloor} \1_{\{X_{(k-1)/n}\geq\rho_0\}}\exp\left( -\frac{(X_{(k-1)/n}-\rho_0)^2}{2 \beta^2/n}\right)
\end{split}
\end{align} 
and 
\[ \E_{\rho_0}\left[ \sup_{|\theta'|\leq |\theta|}\sup_{s\leq t} \frac{|r_n(s,\theta')|}{|\theta'|}\right] \leq C_{\alpha,\beta}(K) |\theta| n^{3/2} \sqrt{t},\]
for some constant $C_{\alpha,\beta}(K)>0$ that is independent of $n,t$ and $\theta$.
\end{prop}

\noindent
The proof, which is deferred to Section~\ref{Appendix_auxiliary} in the supplementary material, reveals that the constants $F_{\alpha,\beta}, \tilde{F}_{\alpha,\beta}$ are composed of the terms belonging to $I_{1,k}^\theta,I_{2,k}^\theta, I_{3,k}^\theta$ and $I_{7,k}^\theta,I_{8,k}^\theta,I_{9,k}^\theta$. Note that these constants are strictly greater than zero for $\alpha\neq\beta$. Moreover, as they are different, the triangular shape described by
\[ -|\theta| \left(\1_{\{\theta\geq 0\}}F_{\alpha,\beta}+\1_{\{\theta< 0\}} \tilde{F}_{\alpha,\beta}\right)\Lambda_{\alpha,\beta}^n\left((X_{(k-1)/n})_{1\leq k\leq \lfloor nt\rfloor}\right) \]
is not symmetric around $\theta=0$, which matches Figure~\ref{figure:Likelihood_example}. The steepness of the triangle depends on the magnitude of $\Lambda_{\alpha,\beta}^n$. The next lemma shows that $\Lambda_{\alpha,\beta}^n/n$ converges in probability to a multiple of $L_t^{\rho_0}(X)$, uniformly in time.

\begin{lemma}\label{lemma:Riemann_approximation}
Let $X=(X_t)_{t\in [0,1]}$ be a solution to~\eqref{eq:SDE_OBM} with diffusion parameter~\eqref{eq:coefficient_OBM} and  let $f:\R\longrightarrow\R$ be a measurable bounded function such that $\int_\R |x|^k f(x)dx<\infty$ for $k=0,1,2$. Then for every $\epsilon>0$ we have
\[ \Pr_{\rho_0}\left( \sup_{t\in [0,1]} \left| \frac{1}{\sqrt{n}}\sum_{k=1}^{\lfloor nt\rfloor} f\left( \sqrt{n}(X_{(k-1)/n}-\rho_0)\right) - \lambda_{\alpha,\beta,\rho}(f)L_t^{\rho_0}(X) \right| >\epsilon\right) \stackrel{n\to\infty}{\longrightarrow}\ 0,\]
where 
\[ \lambda_{\alpha,\beta}(f) := \frac{1}{\alpha^2} \int_{-\infty}^0 f(x)dx + \frac{1}{\beta^2}\int_0^\infty f(x)dx.\]
\end{lemma}
\begin{proof}
The proof for $\rho_0=0$ follows the proof of Lemma~$4.3$ in~\cite{Lejay/Pigato} (see also~\cite{Jacod_Riemann_approx}). For general $\rho_0$ define $\tilde{X}_t := X_t - \rho_0$.
Then $\tilde{X}_0=x_0-\rho_0$ and 
\[ d\tilde{X}_t = dX_t = \sigma_{\rho_0}(X_t)dW_t = \sigma_{\rho_0}(\tilde{X}_t+\rho_0) dW_t.\]
The statement of the proposition now follows from the case $\rho_0=0$ by noting that $L_t^{\rho_0}(X) = L_t^0(\tilde{X})$, $f(\sqrt{n}(X_t-\rho_0)) = f(\sqrt{n} \tilde{X}_t)$ and
\[ \sigma_{\rho_0}(\tilde{X}_t+\rho_0) = \alpha\1_{\{\tilde{X}_t+\rho_0 <\rho_0\}} + \beta\1_{\{\tilde{X}_t+\rho_0\geq \rho_0 \}} = \alpha\1_{\{\tilde{X}_t <0\}} + \beta\1_{\{\tilde{X}_t\geq 0\}} = \sigma_0(\tilde{X}_t).\]
\end{proof}

\noindent
Proposition~\ref{prop:expansion_of_drift_t} and Lemma~\ref{lemma:Riemann_approximation} applied to $f=f_{\alpha,\beta}$ with 
\[ f_{\alpha,\beta}(x) = \1_{(-\infty,0)}(x) \exp(-x^2/(2\alpha^2)) + \1_{[0,\infty)}(x)\exp(-x^2/(2\beta^2))\] yield the following limiting drift of $(\ell_{n,t}(z/n))_{t\in [0,1]}$:

\begin{cor}\label{cor:limiting_drift}
With
\begin{align}\label{def:constant_b_alpha_beta}
\begin{split}
b_{\alpha,\beta} &:= F_{\alpha,\beta}\lambda_{\alpha,\beta}\big(f_{\alpha,\beta}\big) = \frac{\alpha^2-\beta^2}{\alpha^2\beta^2} + \frac{1}{\beta^2}\log\left(\frac{\beta^2}{\alpha^2}\right),  \\
b_{\alpha,\beta}' &:= \tilde{F}_{\alpha,\beta}\lambda_{\alpha,\beta}\big(f_{\alpha,\beta}\big) = \frac{\beta^2-\alpha^2}{\alpha^2\beta^2} + \frac{1}{\alpha^2}\log\left(\frac{\alpha^2}{\beta^2}\right),
\end{split}
\end{align}
we have for any $K>0$,
\[ \sup_{z\in [-K,K]}\sup_{s\leq t} \left| B_{n,s}(z/n) - |z| \left(\1_{\{z\geq 0\}}b_{\alpha,\beta} +\1_{\{z < 0\}}b_{\alpha,\beta}'\right) L_s^{\rho_0}(X)\right| \longrightarrow_{\Pr_{\rho_0}} 0. \]
\end{cor}

\noindent
Note that $b_{\alpha,\beta},b_{\alpha,\beta}'<0$ for $\alpha\neq\beta$. Finally, we prove a moment bound on the summands that appear in the definition of $\ell_n(\theta)$ and $B_n(\theta)$. We already incorporate in the statement that parts of $\mathcal{I}_{2}(\theta)$ and $\mathcal{I}_{8}(\theta)$ equal $\pm\log(\beta^2/\alpha^2)$ on their respective interval which are the only contributing expressions to the variance (for $\theta \ll 1/\sqrt{n}$, see Figure~\ref{figure:regimes}) and have to be dealt with seperately. As will be seen later, these are exactly the terms which drive the martingale part in the limiting distribution. Hence, we exclude them in the following result on the moment bound and define for $\theta'<\theta$
\begin{align}\label{eq:def_Z_j}
Z_k^j(\theta',\theta) := \log\left(\frac{p_{1/n}^{\rho_0+\theta}(X_{(k-1)/n},X_{k/n})}{p_{1/n}^{\rho_0+\theta'}(X_{(k-1)/n},X_{k/n})}\right) \1_{I_{j,k}^{\theta',\theta}}
\end{align} 
for $j\in\{1,3,4,5,6,7,9\}$, whereas
\begin{align}\label{eq:def_Z_j28}
\begin{split}
Z_k^2(\theta',\theta) &:= \left[\log\left(\frac{p_{1/n}^{\rho_0+\theta}(X_{(k-1)/n},X_{k/n})}{p_{1/n}^{\rho_0+\theta'}(X_{(k-1)/n},X_{k/n})}\right) - \log\left(\frac{\beta^2}{\alpha^2}\right)\right] \1_{I_{2,k}^{\theta',\theta}}, \\
Z_k^8(\theta',\theta) &:= \left[ \log\left(\frac{p_{1/n}^{\rho_0+\theta}(X_{(k-1)/n},X_{k/n})}{p_{1/n}^{\rho_0+\theta'}(X_{(k-1)/n},X_{k/n})}\right)- \log\left(\frac{\beta^2}{\alpha^2}\right)\right] \1_{I_{8,k}^{\theta',\theta}}.
\end{split}
\end{align} 
Furthermore, for each $j=1,\dots, 9$, we set
\[ \overline{Z}_k^j(\theta',\theta) := \E_{\rho_0}\left[ Z_k^j(\theta',\theta) \mid X_{(k-1)/n}\right]. \]

\begin{prop}\label{prop:moment_product}
Let $m\in\N$, $0<k_1<k_2<\dots < k_m$, $d_i\in\N$ with $d_i\leq D$ and $j_i\in\{1,\dots,9\}$ for all $i=1,\dots, m$. Moreover, let $K>0$, $-K/\sqrt{n}\leq\theta'<\theta\leq K/\sqrt{n}$ and $Y_k^j(\theta',\theta)\in \{Z_k^j(\theta',\theta), \overline{Z}_k^j(\theta',\theta) \}$. Then there exists a constant $C=C(\alpha,\beta,m,D)>0$ such that (denote $k_0=0$)
\[ \E_{\rho_0}\left[ \prod_{i=1}^m \left|Y_{k_i}^{j_i}(\theta',\theta)\right|^{d_i}\right] \leq C n^{\frac12\sum_{i=1}^m d_i} |\theta-\theta'|^{\sum_{i=1}^m d_i} \prod_{i=1}^m \frac{1}{\sqrt{k_i-k_{i-1}}}.\]
\end{prop}

The proof is deferred to Section~\ref{Appendix_auxiliary}.

\section{Proof of Proposition~\ref{prop:n-consistency}}\label{Section:Consistency}

The Markov property of the process compels the decomposition~\eqref{eq:decomposition_ell_n} into the martingale part $M_n(\theta)$ and the drift term $B_n(\theta)$. As in classical M-estimation, the overall idea of the proof is to show that the (negative) drift $B_n(\theta)$ dominates the stochastic fluctuation $M_n(\theta)$ outside a $1/n$-neighborhood of the true parameter. While the probabilistic properties of the martingale part are reasonably compatible with this general framework, the randomness of the drift together with its decomposition into the disjoint regimes $\{X_{(k-1)/n}\leq \rho_0\}$, $\{\rho_0<X_{(k-1)/n}\leq \rho_0+\theta\}$, and $\{\rho_0+\theta > X_{(k-1)/n}\}$ make the derivation of sufficiently tight (deterministic) bounds with sufficiently high probability rather involved. Therefore, the proof of $n$-consistency splits into three parts where the reasoning for each of them is quite different. Consecutively, we show as summarized in Figure~\ref{figure:route_of_proof}:

\begin{itemize}
\item $\boldsymbol{\hat{\rho}_n = \mathcal{O}_{\Pr_{\rho_0}}(1)}$. For $\theta\notin [\rho_0-K,\rho_0+K]$ for some $K>0$, the term $\mathcal{I}_5(\theta)$ (see~\eqref{eq:def_I}) dominates the other parts of the log-likelihood in the following (informal) sense: For every $\epsilon>0$, there exists $K=K(\epsilon)$ such that
\[ \inf_{\theta\notin [\rho_0-K,\rho_0+K]} |\mathcal{I}_5(\theta)|\quad \textrm{ is large as compared to }\quad \sup_{\theta\notin [\rho_0-K,\rho_0+K]} \left| \ell_n(\theta)-\mathcal{I}_5(\theta)\right| \] 
with probability $>1-\epsilon$ eventually. Roughly speaking, the reason is that the stochastic order of $\mathcal{I}_5(\theta)$ scales proportionately in $\theta$, whereas the stochastic order of all remaining terms is driven by the number of observations falling into different, but small intervals, respectively. The uniformity in this argument follows from the H\"{o}lder continuity of the sample path.

\item $\boldsymbol{\sqrt{n}|\hat{\rho}_n - \rho_0| = \mathcal{O}_{\Pr_{\rho_0}}(1)}$. In order to prove that the MLE is not outside a $1/\sqrt{n}$-neighborhood of $\rho_0$, we employ a slicing argument and decompose the still dominant term $\mathcal{I}_5(\theta)$ into its drift and martingale part. As compared to the first step, arguing with the sample path regularity is not sufficiently tight any longer to attain a smaller surrounding than an $n^{-1/2+\epsilon}$-neighborhood of $\rho_0$. Instead, we employ a bracketing argument together with a discrete local time approximation (Lemma~\ref{lemma:L2-norm_occupation_approximation}).

\item $\boldsymbol{n|\hat{\rho}_n -\rho_0| = \mathcal{O}_{\Pr_{\rho_0}}(1)}$. To finally prove $n$-consistency of the MLE, we again employ a slicing argument in combination with a chaining relying on second moment bounds. At this time, it turns out that the thorough control of the drift $B_n(\theta)$ on each slice is analytically highly challenging: 
As soon as $|\theta-\rho_0|\asymp n^{-1/2}$, \textit{all} terms $\mathcal{I}_1(\theta), \dots, \mathcal{I}_9(\theta)$ are of the same stochastic order of magnitude. At the same time, the likelihood function $\ell_n(\theta)$ cannot be expanded into a Taylor series due to its discontinuities in the parameter $\theta$. Although for each $j=1,\dots, 9$ and any $k=1,\dots, n$,
\[ \theta\mapsto \log\left( p_{1/n}^{\rho_0+\theta}(X_{(k-1)/n},X_{k/n}) \Big/ p_{1/n}^{\rho_0+\theta'}(X_{(k-1)/n},X_{k/n}) \right) \1_{I_{j,k}^{0,\theta}}\]\noindent
possesses a Taylor expansion in principle, we are facing the problem that the remainder terms in each of these expansions are of the same order as their leading terms.
\begin{itemize}
\item In order to bridge the regime until the Taylor expansion is helpful, we observe that for each $k=1,\dots, n$,
\[ \E_{\rho_0}\left[\left.\log\left(\frac{p_{1/n}^{\rho_0+\theta}(X_{(k-1)/n},X_{k/n})}{p_{1/n}^{\rho_0}(X_{(k-1)/n},X_{k/n})}\right) \right| X_{(k-1)/n}\right]\]
equals the negative Kullback--Leibler divergence of the conditional distributions $\Pr_{\rho_0+\theta}(X_{k/n}\in\cdot\mid X_{(k-1)/n})$ and $\Pr_{\rho_0}(X_{k/n}\in\cdot\mid X_{(k-1)/n})$. Therefore, Pinsker's inequality provides an upper bound of the negative drift $B_n(\theta)$ in terms of total variation which is further estimated by
\[ B_n(\theta)\1_A \leq -n^{3/2} |\theta|^2 \zeta\1_A,\]
where $\zeta>0$ is a constant and $A$ an event with high probability. However, this estimate is by far too weak to reach the $1/n$-environment of $\rho_0$.

\item Once we enter the neighborhood of $\rho_0$ where the Taylor expansion is meaningful, we get the sharper bound
\[ B_n(\theta)\1_{A'} \leq -n|\theta|\zeta'\1_{A'},\]
where $\zeta'>0$ is a constant and $A'$ an event with high probability. This mimics the triangular shape observed in Figure~\ref{figure:Likelihood_example}. 
\end{itemize}
\end{itemize}

\begin{figure}[h]
\centering
\includegraphics[scale=0.5]{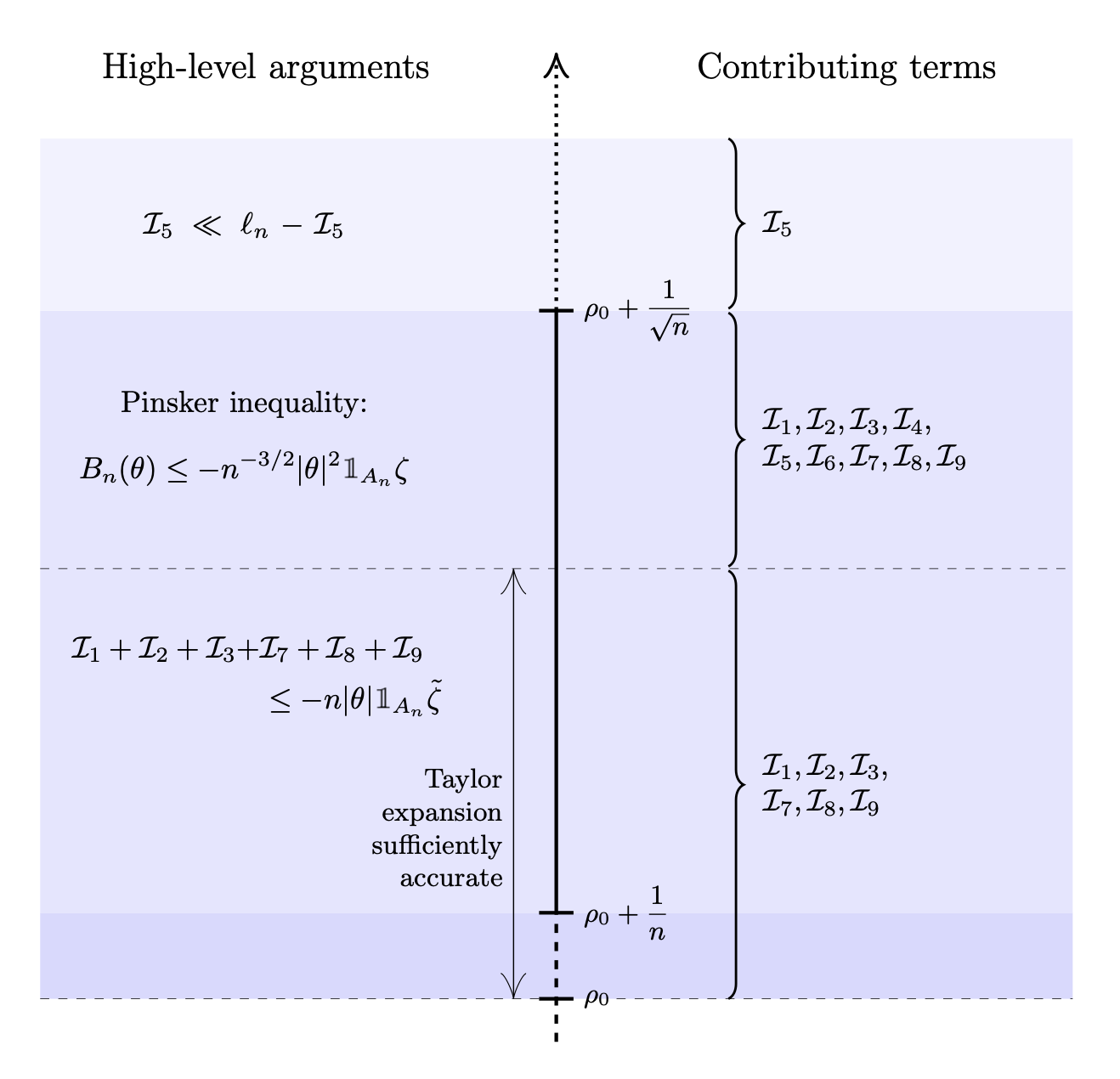}
\caption{Route of proof of the $n$-consistency.}\label{figure:route_of_proof}
\end{figure}

\noindent
Before giving the complete proof, we define parameters $\gamma,\Gamma,\xi$ and specify certain sets. All of those are given for fixed $\epsilon>0$ which is suppressed in the notation. For this, we define $\overline{X} := \sup_{s\leq 1}X_s$, $\underline{X} := \inf_{s\leq 1}X_s$.
\begin{itemize}
\item By the Burkholder-Davis-Gundy inequality, it is easily seen that both $\overline{X}$ and $\underline{X}$ are stochastically bounded. Thus, there exists a constant $\Gamma$ such that $\Pr_{\rho_0}(A_1) >1-\epsilon$ for the event
\begin{align}\label{eq:set_A1}
A_1 := \left\{ \rho_0-\Gamma < \underline{X},\overline{X}<\rho_0+\Gamma \right\}.
\end{align} 
\item Define $B_k = \{ L_1^y(X) >0 \textrm{ for all } y\in [\rho_0-1/k, \rho_0+1/k]\}$. Then $B_{k}\subset B_{k+1}$ for all $k\in\N$ and by continuity of measures from below, we have $\lim_{K\to\infty} \Pr_{\rho_0}(\bigcup_{k=1}^K B_k) = \Pr_{\rho_0}(L_1^{\rho_0}(X)>0)$. Consequently, for every $\epsilon>0$ there exists $K_\epsilon$ such that
\[ \Pr_{\rho_0}\left(\left. \bigcup_{k=1}^{K_\epsilon} B_k\right| L_1^{\rho_0}(X)>0\right) = \Pr_{\rho_0}\left( B_{K_\epsilon}|L_1^{\rho_0}(X)>0\right) >1-\epsilon,\]
and on $B_{K_\epsilon}$ we have $L_1^y(X)>0$ for all $y\in [\rho_0-1/K_\epsilon, \rho_0+1/K_\epsilon]$ which by Corollary~$29.18$ in~\cite{Kallenberg} gives $\overline{X}\geq\rho_0+1/K_\epsilon$ and $\underline{X}\leq\rho_0-1/K_\epsilon$ on $B_{K_\epsilon}$. Thus, defining $\gamma:=1/K_\epsilon$, the probability of 
\begin{align}\label{eq:set_A2}
A_2 := \left\{ \overline{X}\geq\rho_0+\gamma\textrm{ and } \underline{X}\leq\rho_0-\gamma\right\}
\end{align}
satisfies $\Pr_{\rho_0}(A_2|L_1^{\rho_0}(X)>0)>1-\epsilon$.
\item Because $L_1^\cdot(X)$ has a continuous version (as $X$ is a continuous martingale) and every continuous function on a compact set attains its minimum and maximum, Corollary~$29.18$ in~\cite{Kallenberg} about range and support of continuous local martingales reveals
\[ \Pr_{\rho_0}\left( \left.\inf_{y\in [\rho_0-\gamma/2, \rho_0+\gamma/2]} L_1^y(X) >0 \ \right| A_2\right) =1.\]
Again, by continuity of measures from below (analogously to the previous argument), we conclude the existence of $\xi>0$ such that $\Pr_{\rho_0}(A_3\mid A_2)>1-\epsilon$, where
\begin{align}\label{eq:set_A3}
A_3 = \left\{\inf_{y\in [\rho_0-\gamma/2, \rho_0+\gamma/2]} L_1^y(X) >\xi\right\}. 
\end{align} 
\item We define
\begin{align}\label{eq:set_A4}
A_4(n) := \left\{ \sup_{|t-s|<1/n} |X_t - X_s| \leq n^{-4/9} \right\}.
\end{align}
Then by Markov's inequality and Theorem~$1$ in~\cite{Fischer/Nappo} there exists a constant $\tilde{C}>0$ such that
\begin{align*}
\Pr_{\rho_0}\left(\sup_{|t-s|<1/n} |X_t - X_s| > n^{-4/9} \right) &\leq n^{4/9} \E_{\rho_0}\left[ \sup_{|t-s|<1/n} |X_t - X_s|\right] \\
& \leq \tilde{C}n^{-1/18}\sqrt{\log(2n)} \longrightarrow 0.
\end{align*} 
Consequently, there exists $n_0\in\N$ such that for all $n\geq n_0$ we have $\Pr_{\rho_0}(A_4(n))>1-\epsilon$.
\end{itemize}

\subsection{The MLE is not outside a $1/\sqrt{n}$-neighborhood of $\rho_0$}\label{subsection:sqrt(n)_consistency}
In this subsection, we will prove that
\begin{align}\label{eq_proof:sqrt(n)_consistency}
\left| \hat{\rho}_n - \rho_0\right|\1_{\{L_1^{\rho_0}(X)>0\}} = \mathcal{O}_{\Pr_{\rho_0}}\left(\frac{1}{\sqrt{n}}\right).
\end{align}
For the proof of this $\sqrt{n}$-consistency, recall the normalized log-likelihood $\ell_n(\theta)$ given in~\eqref{def:log_likelihood_normalized} and fix an arbitrary $\epsilon>0$. As $\ell_n(0)=0$ and $n(\hat{\rho}_n-\rho_0)\in\mathrm{Argsup}_{z\in\R}\ell_n(z/n)$, we know that $|\hat{\rho}_n -\rho_0|>K/\sqrt{n}$ implies that $\sup_{|\theta|>K/\sqrt{n}}\ell_n(\theta)\geq 0$. Consequently, we are going to show that
\begin{align*}
&\limsup_{K\to\infty} \limsup_{n\to\infty} \Pr_{\rho_0}\left( \sqrt{n}|\hat{\rho}_n - \rho_0|> K, L_1^{\rho_0}(X)>0 \right) \\
&\hspace{2cm} \leq \limsup_{K\to\infty} \limsup_{n\to\infty}\Pr_{\rho_0}\left( \sup_{\sqrt{n}|\theta|>K} \ell_n(\theta) \geq 0, L_1^{\rho_0}(X)>0 \right) < \epsilon.
\end{align*} 
To this aim, we further split
\begin{align*}
&\Pr_{\rho_0}\left( \sup_{\sqrt{n}|\theta|>K} \ell_n(\theta) \geq 0, L_1^{\rho_0}(X)>0 \right)\\
&\hspace{1cm} = \Pr_{\rho_0}\left( \sup_{\sqrt{n}\theta >K} \ell_n(\theta) \geq 0, L_1^{\rho_0}(X)>0 \right) + \Pr_{\rho_0}\left( \sup_{-\sqrt{n}\theta >K} \ell_n(\theta) \geq 0, L_1^{\rho_0}(X)>0 \right)
\end{align*} 
and only discuss the first summand as the second one can be dealt with analogously. The key idea is now to decompose $\ell_n(\theta)$ into a dominant term that is part of $\mathcal{I}_5(\theta)$ and a remainder that includes all other $\mathcal{I}_j(\theta)$, $j\neq 5$, and can be bounded independently of $\theta$ (see Lemma~\ref{lemma:bound_L_sqrt(n)}). To this aim, we define the two quantities
\[ N_n^L(\theta) := \sum_{k=1}^n \left[ \log\left(\frac{\beta}{\alpha}\right)-\frac{(X_{k/n}-X_{(k-1)/n})^2}{2/n}\left(\frac{1}{\alpha^2}-\frac{1}{\beta^2}\right)\right] \1_{\{\rho_0+L/\sqrt{n}\leq X_{(k-1)/n}< \rho_0+\theta\}}\]
and its compensator
\begin{align*}
\overline{N}_n^L(\theta) &= \sum_{k=1}^n \left[ \log\left(\frac{\beta}{\alpha}\right)- \E_{\rho_0}\left[\left. \frac{(X_{k/n}-X_{(k-1)/n})^2}{2/n} \right| X_{(k-1)/n}\right]\left(\frac{1}{\alpha^2}-\frac{1}{\beta^2}\right)\right] \\
&\hspace{1.5cm} \cdot\1_{\{\rho_0+L/\sqrt{n}\leq X_{(k-1)/n}< \rho_0+\theta\}}, 
\end{align*} 
where the parameter $L$ will be specified later. We then have the following result (its proof can be found in Section~\ref{App:Section:proofs_sqrt(n)_consistency} in the Appendix):

\begin{lemma}\label{lemma:bound_L_sqrt(n)}
Let $K,L>1$, $\epsilon>0$ and $\Theta_n^1 := [K/\sqrt{n},n^{-1/4}]$, $\Theta_n^2:=(n^{-1/4},\infty)$. Then there exists a sequence of sets $(A_n)_{n\in\N}$ with $\Pr_{\rho_0}(A_n^c|L_1^{\rho_0}(X)>0)\leq\epsilon$ for $n\geq n_0$, such that for $i=1,2$,
\[ \sup_{\theta\in\Theta_n^{i}} \ell_n(\theta)\1_{A_n} \leq \sup_{\theta \in\Theta_n^{i}} N_n^L(\theta)\1_{A_n} + F_n^{i}(K,L),\]
where $F_n^{i}(K,L)\geq 0$ and $\E_{\rho_0}[F_n^1(K,L)n^{-1/2}] \leq C_{\alpha,\beta}L<\infty$, $\E_{\rho_0}[F_n^2(K,L)n^{-2/3}] \leq C_{\alpha,\beta}L<\infty$ with a constant $C_{\alpha,\beta}$ independent of both $K,L$ and $n$.
\end{lemma}

\noindent
Note that $\Theta_n^1$ and $\Theta_n^2$ in Lemma~\ref{lemma:bound_L_sqrt(n)} are disjoint and $\Theta_n^1\cup\Theta_n^2= \{\theta:\ \theta\geq K/\sqrt{n}\}$. Furthermore, it should be noted that the moment bound on $F_n^1$ of order $n^{1/2}$ is optimal and the technically most involved part of the proof. It is based on a bound in $L^2(\Pr_{\rho_0})$ on the error of discretely approximating an occupation times integral (see Lemma~\ref{lemma:L2-norm_occupation_approximation}) and a bracketing argument (see Lemma~\ref{lemma:uniform_stochastic_occupation_approximation}). The order $n^{2/3}$ for $F_n^2$ is not optimal, but sufficient for our purpose and derived using only knowledge about the expected modulus of continuity of the path $X$. To proceed, we recall the sets $A_1,A_2, A_3$ and $A_4(n)$ given in \eqref{eq:set_A1}-\eqref{eq:set_A4} and let the constants be specified in such a way that the bounds of their respective probabilities are given for $\epsilon/10$ (instead of $\epsilon$). Furthermore, we introduce two additional events:
\begin{itemize}
\item Recall $F_n^1(K,L)$ and $F_n^2(K,L)$ from Lemma~\ref{lemma:bound_L_sqrt(n)} and define
\[ A_5(n) := \left\{ |F_n^1(K,L)| \leq C_F L\sqrt{n},\ |F_n^2(K,L)| \leq C_F Ln^{2/3} \right\},\]
where $C_F>0$ is chosen large enough such that $\Pr_{\rho_0}(A_5(n))>1-\epsilon/10$ for all $n\geq n_1$ and appropriate $n_1\in\N$. This is possible by the tightness of $(F_n^1(K,L)/(Ln^{1/2}))_{n\in\N}$ and $(F_n^2(K,L)/(Ln^{2/3}))_{n\in\N}$ implied by the moment bound in Lemma~\ref{lemma:bound_L_sqrt(n)}.
\item Let $K>L$. Then, by Lemma~\ref{lemma:Riemann_approximation}, the random variable
\[ \frac{1}{\sqrt{n}}\sum_{k=1}^n \1_{\{\rho_0+L/\sqrt{n}\leq X_{(k-1)/n}<\rho_0+K/\sqrt{n}\}} = \frac{1}{\sqrt{n}}\sum_{k=1}^n g\left(\sqrt{n}(X_{(k-1)/n}-\rho_0)\right)\]
with $g(x) = \1_{[L,K)}(x)$ converges in probability to $(K-L)\beta^{-2}L_1^{\rho_0}(X)$. In particular, this random variable is bounded from below by $(K-L)\beta^{-2}\xi$ on the set $A_3$. Consequently, there exists $n_2\in\N$ such that $\Pr_{\rho_0}(A_6(n)|A_3)>1-\epsilon/10$ for $n\geq n_2$, where
\[ A_6(n) := \left\{ \frac{1}{\sqrt{n}}\sum_{k=1}^n \1_{\{\rho_0+L/\sqrt{n}< X_{(k-1)/n}<\rho_0+K/\sqrt{n}\}} > \left(K-L\right)\beta^{-2}\xi/2 \right\}.\]
\end{itemize}
Finally, we define
\[ A(n) = A_1 \cap A_2 \cap A_3\cap A_4(n) \cap A_5(n) \cap A_6(n)\cap A_7(n),\]
where $A_7(n)$ is the set $A_n$ in Lemma~\ref{lemma:bound_L_sqrt(n)} for $\epsilon/10$. Then by the findings above we have for $n\geq \max\{n_0,n_1,n_2\}$ large enough that (suppressing the index $n$ in the events)
\begin{align*}
\Pr_{\rho_0}( \{L_1^{\rho_0}(X)>0\}\cap A(n)^c) &= \Pr_{\rho_0} \left( A(n)^c| L_1^{\rho_0}(X)>0\right) \Pr_{\rho_0}(L_1^{\rho_0}(X)>0) \\
& \leq \Pr_{\rho_0}(L_1^{\rho_0}(X)>0) \sum_{j=1}^7 \Pr(A_j^c|L_1^{\rho_0}(X)>0) \leq \epsilon,
\end{align*}
where the last inequality uses the probability bounds on $\Pr_{\rho_0}(A_j^c)$ for $j=1,4,5$, $\Pr_{\rho_0}(A_3^c|A_2)$, $\Pr_{\rho_0}(A_6^c|A_3)$ and $\Pr(A_k^c|L_1^{\rho_0}(X)>0)$ for $k=2,7$ together with 
\[ \Pr_{\rho_0}\left( A_6(n)^c|L_1^{\rho_0}(X)>0\right) \leq \Pr_{\rho_0}\left( A_6(n)^c|A_3\right) + \Pr_{\rho_0}\left( A_3^c|L_1^{\rho_0}(X)>0\right)\]
and
\[ \Pr_{\rho_0}\left( A_3^c|L_1^{\rho_0}(X)>0\right) \leq \Pr_{\rho_0}\left( A_3^c|A_2\right) + \Pr_{\rho_0}\left( A_2^c|L_1^{\rho_0}(X)>0\right).\] 
Both of the last estimates follow by the law of total probability applied to $\Pr(\cdot\mid L_1^{\rho_0}(X)>0)$ and noting that $A_2,A_3\subset \{L_1^{\rho_0}(X)>0\}$.
After those preliminaries, we now start with the main part of the proof of \eqref{eq_proof:sqrt(n)_consistency}, which is split into three parts.

\section*{Slicing argument}
Subsequently, comma-separated lists within probabilities should be read as the intersection of the corresponding subsets of $\Omega$. For $j\in\Z$, we define sets $S_{n,j}$ via
\[ S_{n,j} := \left\{ \theta: 2^{j}<\sqrt{n}\theta \leq 2^{j+1} \right\}.\]
Then for $n\geq \max\{n_0,n_1\}$ and $K=2^M$,
\begin{align}\label{eq_proof:sqrt(n)_decomposition}
\begin{split}
&\Pr_{\rho_0}\left( \sup_{\sqrt{n}\theta >2^M} \ell_n(\theta) \geq 0, L_1^{\rho_0}(X)>0 \right) \\
&\hspace{0.5cm}\leq \sum_{\substack{j\geq M \\ 2^j \leq \Gamma \sqrt{n}}} \Pr_{\rho_0}\left( \sup_{\theta\in S_{n,j}} \ell_n(\theta) \geq 0, A(n) \right) +\Pr_{\rho_0}\left( \sup_{\theta> \Gamma} \ell_n(\theta) \geq 0, A(n)\right)\\
&\hspace{1.5cm}   + \Pr_{\rho_0}\left( A(n)^c \cap \{L_1^{\rho_0}(X)>0\}\right) \\
&\hspace{0.5cm}\leq \sum_{\substack{j\geq M \\ 2^j \leq \Gamma \sqrt{n}}} \Pr_{\rho_0}\left( \sup_{\theta\in S_{n,j}} \ell_n(\theta) \geq 0, A(n) \right) +\Pr_{\rho_0}\left( N_n^L(\Gamma)+F_n^2(K,L)\geq 0, A(n)\right) +\epsilon,
\end{split}
\end{align}
where the last step uses Lemma~\ref{lemma:bound_L_sqrt(n)} and $N_n(\theta) = N_n(\Gamma)$ for all $\theta\geq \overline{X}-\rho_0$, together with $\overline{X}-\rho_0\leq \Gamma$ on $A(n)$. In what follows, we are going to analyze the first two of these summands and show that they vanish as $M\to\infty$, uniformly in $n$. With Lemma~\ref{lemma:bound_L_sqrt(n)} and
\[ F_{n,j}(K,L):=F_n^1(K,L)\1_{\{2^j\leq n^{-1/4}\}}+ F_n^2(K,L)\1_{\{2^j> n^{-1/4}\}}, \]
we find for the first summand in~\eqref{eq_proof:sqrt(n)_decomposition},
\begin{align}\label{eq_proof:sqrt(n)_decomposition_probability}
\begin{split}
&\Pr_{\rho_0}\left( \sup_{\theta\in S_{n,j}} \ell_n(\theta) \geq 0, A(n) \right)\\
&\hspace{0.5cm} \leq \Pr_{\rho_0}\left( \sup_{\theta\in S_{n,j}} \left(N_n^L(\theta) + F_{n,j}(K,L)\right) \geq 0, A(n) \right) \\
&\hspace{0.5cm}\leq \Pr_{\rho_0}\left( \sup_{\theta\in S_{n,j}} \left(N_n^L(\theta) - \overline{N}_n^L(\theta) \right)  \geq -\sup_{\theta\in S_{n,j}}\overline{N}_n^L(\theta) - |F_{n,j}(K,L)|, A(n) \right).
\end{split}
\end{align}
Next, we will find a lower bound for the right-hand side in this probability, in particular for $\sup_\theta \overline{N}_n^L(\theta)$, on the set $A(n)$. Here, the problem is that $\overline{N}_n^L(\theta)$ is still random and not necessarily negative. However, we will subsequently prove the following:
\begin{itemize}
\item[] Claim I:\ $\overline{N}_n^L(\theta)<0$ for large values of $L$. 
\item[] Claim II:\ The subsequent inequalities~\eqref{eq_proof:sqrt(n)_bound_small_j} and~\eqref{eq_proof:sqrt(n)_bound_large_j}, meaning heuristically that $\overline{N}_n^L(\theta)$ scales (almost) at least proportionately in $\theta$.
\end{itemize}
The heuristic reason for Claim I is that $\E_{\rho_0}[n(X_{k/n}-X_{(k-1)/n})^2\mid X_{(k-1)/n}] \approx \beta^2$ for large values of $L$ (meaning $X_{(k-1)/n}$ being not too close to the change point $\rho_0$).
In this case, each summand appearing in the definition of $\overline{N}_n^L(\theta)$ is negative, in particular
\[ \overline{N}_n^L(\theta) \approx c_{\alpha,\beta} \sum_{k=1}^n \1_{[\rho_0+L/\sqrt{n},\rho_0+\theta)}(X_{(k-1)/n}) \]
for some negative constant $c_{\alpha,\beta}<0$. According to this, $\overline{N}_n^L(\theta)$ is given approximately as a multiple of the number of observations falling into the interval $[\rho_0+L/\sqrt{n},\rho_0+\theta)$, which heuristically explains Claim II.\\
Now we give the details and start with Claim I. Its heuristic is made precise in the following preliminary lemma. The proof is deferred to Section~\ref{App:Section:proofs_sqrt(n)_consistency} in the Appendix.

\begin{lemma}\label{lemma:E_squared_difference}
For every $\epsilon>0$ there exists $L=L(\epsilon)>0$ such that for all $1\leq k\leq n$ and every $\theta>L/\sqrt{n}$,
\[ \left| \E_{\rho_0}\left[\left. n(X_{k/n}-X_{(k-1)/n})^2  \right| X_{(k-1)/n}\right] -\beta^2\right| \1_{\{\rho_0+L/\sqrt{n}\leq X_{(k-1)/n}< \rho_0+\theta\}} \leq \epsilon.\]
\end{lemma}

\noindent
First, take $\epsilon_0>0$ to be small enough such that
\[ d_{\alpha,\beta} := \sup_{|\delta|<\epsilon_0}\left( \log\left(\frac{\beta}{\alpha}\right)- \frac{1+\delta}{2}\left(\frac{\beta^2}{\alpha^2}-1\right)\right) < 0,\]
which is possible since $\log(\beta/\alpha)-(\beta^2/\alpha^2-1)/2<0$ for all $\alpha\neq\beta$. This follows from the fact that $f(x):=\log(x)-x^2/2 + 1/2$ has a unique maximum at $x=1$ with $f(1)=0$. Using Lemma~\ref{lemma:E_squared_difference} we choose $L=L(\epsilon_0)$ large enough such that 
\begin{align*}
&\left( \log\left(\frac{\beta}{\alpha}\right)- \E_{\rho_0}\left[\left. \frac{(X_{k/n}-X_{(k-1)/n})^2}{2/n} \right| X_{(k-1)/n}\right]\frac{\beta^2-\alpha^2}{\alpha^2\beta^2}\right)\1_{\{\rho_0+L/\sqrt{n}\leq X_{(k-1)/n}< \rho_0+\theta\}} \\
&\hspace{1.5cm}  < d_{\alpha,\beta}\1_{\{\rho_0+L/\sqrt{n}\leq X_{(k-1)/n}< \rho_0+\theta\}}
\end{align*} 
for $k=1,\dots, n$. Note that this choice of $L$ is independent of $K$, which is important as in the definition of $A_6(n)$ we require $K>L$. Claim I now follows by summation over $k$.\\

\noindent
Building on these results, we now establish Claim II. Here, we have (recall $d_{\alpha,\beta}<0$)
\begin{align*}
-\sup_{\theta\in S_{n,j}} \overline{N}_n(\theta) &\geq - \sup_{\theta\in S_{n,j}}\sum_{k=1}^n d_{\alpha,\beta} \1_{\{\rho_0+L/\sqrt{n}\leq X_{(k-1)/n}<\rho_0+\theta\}} \\
&= -d_{\alpha,\beta} \sum_{k=1}^n \1_{\{\rho_0+L/\sqrt{n}\leq X_{(k-1)/n}< \rho_0+2^j/\sqrt{n}\}}
\end{align*}
To continue bounding $\sup_\theta\overline{N}_n^L(\theta)$ on $A(n)$ and establish Claim II, we distinguish the two cases $2^j\leq n^{1/4}$ and $2^j>n^{1/4}$ when counting the observations in the interval $[\rho_0+L/\sqrt{n}, \rho_0+2^j/\sqrt{n})$. The reason for treating these two cases separately is that for small values of $2^j$, it is enough to simply consider the smaller interval $[\rho_0+L/\sqrt{n}, \rho_0+K/\sqrt{n})$, see~\eqref{eq_proof:sqrt(n)_bound_small_j}. This, however, is not sharp enough for larger values of $2^j$, where we need to consider that the number of observations within a certain interval scales with its length with high probability. Once the order (in $n$) of the length of the interval $[\rho_0+L/\sqrt{n}, \rho_0+2^j/\sqrt{n})$ is larger than the modulus of continuity of the path $X$, we can switch from the sum of indicators (that count observations in the interval) to an occupation integral over a slighly smaller domain (but with length of the same order), see~\eqref{eq_proof:sqrt(n)_bound_large_j_new}. Now we make these heuristics precise: For $2^j\leq n^{1/4}$, we simply use that $j\geq M$, meaning that $ \1_{\{\rho_0+L/\sqrt{n}\leq X_{(k-1)/n}< \rho_0+2^j/\sqrt{n}\}}\geq  \1_{\{\rho_0+L/\sqrt{n}\leq X_{(k-1)/n}< \rho_0+2^M/\sqrt{n}\}}$ and consequently on $A_6(n)$,
\begin{align}\label{eq_proof:sqrt(n)_bound_small_j}
\begin{split}
-\sup_{\theta\in S_{n,j}} \overline{N}_n(\theta) &\geq -d_{\alpha,\beta}\sum_{k=1}^n \1_{\{\rho_0+L/\sqrt{n} \leq X_{(k-1)/n}< \rho_0+2^M/\sqrt{n}\}} \\
&\geq -d_{\alpha,\beta} \sqrt{n} \left( 2^M-L\right)\beta^{-2}\xi/2.
\end{split}
\end{align} 
For $2^j>n^{1/4}$, we make use of $A_4(n)$ being part of $A(n)$. We first observe that on $A(n)$, 
\[ \int_{(k-1)/n}^{k/n} \1_{[\rho_0+L/\sqrt{n} + n^{-4/9}, \rho_0 + 2^j/\sqrt{n} -n^{-4/9}]}(X_s)ds \neq 0 \]
implies $\1_{\{\rho_0+L/\sqrt{n}\leq X_{(k-1)/n}<\rho_0+2^j/\sqrt{n}\}}=1$. Together with the occupation times formula, 
\begin{align}\label{eq_proof:sqrt(n)_bound_large_j_new}
\begin{split}
-\sup_{\theta\in S_{n,j}} \overline{N}_n(\theta) & \geq (-d_{\alpha,\beta}) n \int_0^1 \1_{[\rho_0+L/\sqrt{n} + n^{-4/9}, \rho_0 + 2^j/\sqrt{n} -n^{-4/9}]}(X_s)ds \\
&\geq \frac{(-d_{\alpha,\beta})n}{\max\{\alpha^2,\beta^2\}} \int_{\rho_0+L/\sqrt{n} + n^{-4/9}}^{\rho_0 + (2^j/\sqrt{n} -n^{-4/9})\wedge \gamma/2} L_1^y(X) dy \\
&\geq \frac{(-d_{\alpha,\beta})\xi n}{\max\{\alpha^2,\beta^2\}}\left(\big(2^j/\sqrt{n}-2n^{-4/9}-L/\sqrt{n}\big)\wedge \big(\gamma/2 - L/\sqrt{n}-n^{-4/9}\big)\right),
\end{split}
\end{align}
where the the parameter $\gamma$ appears in the definition of the set $A_2$ and the second-last line follows from the fact that we work on $A_3$. For $2^j>n^{1/4}$ and $n$ large enough,
\[ \frac{2^j}{\sqrt{n}}\left(1- 2^{-(j-1)} n^{1/18}-2^{-j}L\right) \geq  \frac{2^j}{\sqrt{n}}\left(1- 2n^{1/18-1/4}-Ln^{-1/4} \right) \geq 2^{j-1}/\sqrt{n}, \]
such that~\eqref{eq_proof:sqrt(n)_bound_large_j_new} gives
\begin{align}\label{eq_proof:sqrt(n)_bound_large_j}
-\sup_{\theta\in S_{n,j}} \overline{N}_n(\theta) \geq \frac{(-d_{\alpha,\beta})\xi}{\max\{\alpha^2,\beta^2\}´}n\left(2^{j-1}/\sqrt{n}\wedge\gamma/4 \right).
\end{align}
Together with~\eqref{eq_proof:sqrt(n)_bound_small_j}, this establishes Claim II.\\

\noindent
Now we combine everything for the first summand in~\eqref{eq_proof:sqrt(n)_decomposition}. With~\eqref{eq_proof:sqrt(n)_decomposition_probability}, the lower bound~\eqref{eq_proof:sqrt(n)_bound_small_j} for $2^j\leq n^{1/4}$, the bound~\eqref{eq_proof:sqrt(n)_bound_large_j} for $2^j>n^{1/4}$ and the bound on $F_n^i(K,L)$ on $A_5$ we find for $M,n$ large enough with Markov's inequality
\begin{align}\label{eq_proof:sqrt(n)_last_line_first_step}
\begin{split}
&\sum_{\substack{j\geq M \\ 2^j \leq \Gamma \sqrt{n}}} \Pr_{\rho_0}\left( \sup_{\theta\in S_{n,j}} \ell_n(\theta) \geq 0, A(n) \right) \\
&\hspace{1cm} \leq  \sum_{\substack{j\geq M \\ 2^j \leq n^{1/4}}} \E_{\rho_0}\left[ \sup_{\theta\in S_{n,j}} \left|N_n^L(\theta) - \overline{N}_n^L(\theta) \right|\right] \frac{1}{C_1\sqrt{n}(2^M-L)-C_F\sqrt{n}} \\
&\hspace{1.5cm}   + \sum_{\substack{j\geq M \\ n^{1/4} <2^j \leq \Gamma \sqrt{n}}} \E_{\rho_0}\left[ \sup_{\theta\in S_{n,j}} \left| N_n^L(\theta) - \overline{N}_n^L(\theta) \right| \right] \frac{1}{C_1 n (1\wedge 2^j/\sqrt{n}) - C_F Ln^{2/3}}, 
\end{split}
\end{align}
where $C_1=C_1(\alpha,\beta, \gamma,\xi,\Gamma)>0$ is some universal constant.

\section*{Evaluating the expectation: Modified chaining}
It remains to evaluate the expectation of the supremum in the last display which is done with a modified chaining procedure. It should be noted that the chaining will be interrupted when $N_n^L(\theta)-\overline{N}_n^L(\theta)$ is compared for values $\theta,\theta'$ with difference of order $1/n$ and the resulting remainder is treated in essence using variance bounds. This strategy is tailor-made to the special structure of the martingale term $N_n^L(\theta)-\overline{N}_n^L(\theta)$ that allows to deduce tight bounds for the remainder that are not available in general. The details are given in Appendix~\ref{App:Section:proofs_sqrt(n)_consistency} and yield the bound
\begin{align}\label{eq_proof:sqrt(n)_chaining_bound_new}
&\E_{\rho_0}\left[ \sup_{\theta\in S_{n,j}} \left( N_n^L(\theta)-\overline{N}_n^L(\theta)\right)\right] \leq C_2 2^{(j+1)/2}n^{1/4} \left( j\log(2) + \log(n)\right),
\end{align}
where $C_2=C_2(\alpha,\beta)>0$ is a suitable constant.

\section*{Finalizing the proof of~\eqref{eq_proof:sqrt(n)_consistency}}
From~\eqref{eq_proof:sqrt(n)_last_line_first_step}, we obtain the following bound for the first summand in~\eqref{eq_proof:sqrt(n)_decomposition}:
\begin{align*}
&\sum_{\substack{j\geq M \\ 2^j \leq \Gamma \sqrt{n}}} \Pr_{\rho_0}\left( \sup_{\theta\in S_{n,j}} \ell_n(\theta) \geq 0, A(n) \right) \\
&\hspace{1cm} \leq \sum_{\substack{j\geq M \\ 2^j \leq n^{1/4}}} \frac{C_2 2^{(j+1)/2}n^{1/4} \left( j\log(2) + \log(n)\right)}{C_1\sqrt{n}(2^M-L)-C_F\sqrt{n}} \\
&\hspace{2cm} + \sum_{\substack{j\geq M \\ n^{1/4} <2^j \leq \Gamma \sqrt{n}}} \frac{C_2 2^{(j+1)/2}n^{1/4} \left( j\log(2) + \log(n)\right)}{C_1 n (1\wedge 2^j/(2\sqrt{n})) - C_F Ln^{2/3}} \\
&\hspace{1cm} \leq  C_3 n^{-1/8}(\log(n))^2 + C_3\sum_{j\geq M } j2^{-j/2},
\end{align*}
for a constant $C_3=C_3(\alpha,\beta,\gamma,\xi,\Gamma)>0$ and $n$ large enough. To complete the discussion for the first summand in~\eqref{eq_proof:sqrt(n)_decomposition}, we now make the choice of $M$ by taking it large enough to satisfiy $C_4\sum_{j\geq M}j 2^{-j/2} <\epsilon$ and then to take $n$ large enough to have the first summand also $<\epsilon$.\\

\noindent
It remains to deal with the second summand in~\eqref{eq_proof:sqrt(n)_decomposition}.  By the same argument as used for~\eqref{eq_proof:sqrt(n)_bound_large_j} we find
\[ -\overline{N}_n^L(\Gamma) \geq - nd_{\alpha,\beta}\xi\gamma/2\]
on $A(n)$ and consequently,
\begin{align*}
&\Pr_{\rho_0}\left( N_n^L(\Gamma) + F_n^2(K,L) \geq 0, A(n)\right)\\
&\hspace{1cm} \leq\Pr_{\rho_0}\left( N_n^L(\Gamma) -\overline{N}_n^L(\Gamma) \geq -\overline{N}_n^L(\Gamma)- |F_n^2(L,K)| , A(n)\right) \\
&\hspace{1cm} \leq \frac{\E_{\rho_0}\left[ (N_n^L(\Gamma)-\overline{N}_n^L(\Gamma))^2\right]}{\left[C_{\alpha,\beta}(\xi,\gamma)n - C_{\alpha,\beta}(L,K)n^{2/3}\right]^2}  \leq \frac{Cn\Gamma}{\left[C_{\alpha,\beta}(\xi,\gamma)n - C_{\alpha,\beta}(L,K)n^{2/3}\right]^2},
\end{align*}
which converges to zero for $n\to\infty$. Thus, the first summands in~\eqref{eq_proof:sqrt(n)_decomposition} are shown to converge to zero and~\eqref{eq_proof:sqrt(n)_consistency} follows.

\subsection{The MLE is not outside a $1/n$-neighborhood of $\rho_0$}\label{subsection:n_consistency}
In this subsection, we will complete the proof of Proposition~\ref{prop:n-consistency} by proving $n$-consistency of the MLE $\hat{\rho}_n$. We will show tightness of $n|\hat{\rho}_n-\rho_0|$ by proving that
\begin{align}\label{eq_proof:n_consistency_1}
\lim_{M\to\infty} \limsup_{n\to\infty} \Pr_{\rho_0} \left( n|\hat{\rho}_n-\rho_0| > 2^M, L_1^{\rho_0}(X)>0\right) = 0.
\end{align}
The proof again combines slicing and chaining techniques and will exploit the fact that we already established $\sqrt{n}$-consistency of $\hat{\rho}_n$ which allows to consider only $\theta\leq K/\sqrt{n}$ for suitable $K>0$. At this point it is crucial that (almost) sharp upper bounds on the drift term $B_n(\theta)$ are needed and the next lemma provides those for $|\theta|<K/\sqrt{n}$. Note that the $n|\theta|$ estimate for small $\theta$ is needed to describe the triangular shape observed in Figure~\ref{figure:Likelihood_example} sufficiently accurate to archive the $1/n$ rate of convergence. However, as described in the beginning of Section~\ref{Section:Consistency}, a (Taylor) expansion of the log-likelihood function in each regime is only helpful for all $|\theta| \leq \kappa_0/\sqrt{n}$ with $\kappa_0>0$ small enough. For all other $\kappa_0/\sqrt{n}\leq\theta\leq K/\sqrt{n}$, we rely on the observation that the negative drift $B_n$ is given by the sum over certain Kullback--Leibler divergences which are then bounded using Pinsker's inequality. 

\begin{lemma}\label{lemma:bound_drift}
Let $K>0$. Then for every $\epsilon>0$ there exist constants $\kappa_0,\zeta>0$, $n_0\in\N$ and a sequence of sets $(A_n)_{n\in\N}$ with $\Pr_{\rho_0}(A_n^c|L_1^{\rho_0}(X)>0)\leq\epsilon$ for $n\geq n_0$ such that 
\[ B_n(\theta)\1_{A_n} \leq \begin{cases}
-n|\theta|\zeta \1_{A_n}, & \textrm{ if } |\theta|\leq \frac{\kappa_0}{\sqrt{n}}, \\
-n^{3/2} |\theta|^2 \zeta \1_{A_n}, & \textrm{ if } \frac{\kappa_0}{\sqrt{n}}<|\theta|\leq K/\sqrt{n}.
\end{cases} \]
\end{lemma}
\begin{proof}
We will only give a brief sketch of the proof, the details can be found in Appendix~\ref{App:Section:proofs_n_consistency}.
First, we specify $\kappa_0$ and $\zeta_1$ such that by the first inequality is valid for $|\theta|\leq \kappa_0/\sqrt{n}$ with $\zeta=\zeta_1$ by using Proposition~\ref{prop:expansion_of_drift_t}. For $\kappa_0/\sqrt{n}\leq |\theta|\leq K/\sqrt{n}$, this expansion is not helpful as the remainder term (which is of the same order as the leading term) is too large. Denoting by $\Pr_{\rho,1/n}^{X_{(k-1)/n}}(\cdot) = \Pr_{\rho}(X_{k/n}\in\cdot\mid X_{(k-1)/n})$ the distribution of $X_{k/n}$ given $X_{(k-1)/n}$ in our model with parameter $\rho$ and by $KL(\Pr_1,\Pr_2)$ the Kullback--Leibler divergence of two probability measures $\Pr_1$, $\Pr_2$, we observe
\[ \E_{\rho_0}\left[\left. \log\left(\frac{p_{1/n}^{\rho_0+\theta}(X_{(k-1)/n},X_{k/n})}{p_{1/n}^{\rho_0}(X_{(k-1)/n},X_{k/n})}\right)\ \right|\ X_{(k-1)/n}\right] = - KL\left( \Pr_{\rho_0,1/n}^{X_{(k-1)/n}}, \Pr_{\rho_0+\theta,1/n}^{X_{(k-1)/n}}\right). \]
An explicit evaluation of this Kullback--Leibler divergence is analytically highly challenging, in particular because expressions built on the first two regimes of the transition density~\eqref{eq:transition_density} contain logarithms of sums. In order to find a form that is more convenient to work with in our setting, we apply the first Pinsker inequality (\cite{Tsybakov}, Lemma~$2.5$) and find an upper bound in terms of total variation. Due to Scheffé's theorem (\cite{Tsybakov}, Lemma~$2.1$) the second inequality in the statement can then be deduced from the inequality
\begin{align}\label{eq:bound_TV_distance}
\begin{split}
&\int_{\R} \left| p_{1/n}^{\rho_0+\theta}(X_{(k-1)/n},y) - p_{1/n}^{\rho_0}(X_{(k-1)/n},y) \right| dy \\
&\hspace{2cm} \geq C_{\alpha,\beta}^1 \theta \sqrt{n} \exp\left( -\frac{(X_{(k-1)/n}-\rho_0)^2}{\min\{\alpha^2,\beta^2\}/n} - C_{\alpha,\beta}^2 K\right),
\end{split}
\end{align}
where $C_{\alpha,\beta}^1, C_{\alpha,\beta}^2>0$ are suitable constants independent of $n$ and $\theta$.
\end{proof}

\noindent
We return to the proof of~\eqref{eq_proof:n_consistency_1}. Let $\epsilon>0$ be arbitrary and $K>0$ large enough such that $\Pr_{\rho_0}(\sqrt{n}|\hat{\rho}_n - \rho_0|>K, L_1^{\rho_0}(X)>0)<\epsilon$. This choice of $K$ is possible by the $\sqrt{n}$-consistency established in Subsection~\ref{subsection:sqrt(n)_consistency} (see~\eqref{eq_proof:sqrt(n)_consistency}). Moreover, let $\kappa_0,\zeta>0$ be the constants and $(A_n)_{n\in\N}$ the sequence of set introduced in Lemma~\ref{lemma:bound_drift}. By the definition of $\hat{\rho}_n$ satisfying $n(\hat{\rho}_n-\rho_0)\in\mathrm{Argsup}_{z\in\R}\ell_n(z/n)$ and using $\ell_n(0)=0$,
\begin{align}\label{eq_proof:n_consistency_9}
\begin{split}
&\Pr_{\rho_0}\left(n|\hat{\rho}_n-\rho_0| >2^M, L_1^{\rho_0}(X)>0\right) \\
&\hspace{1cm}\leq \Pr_{\rho_0}\left(\{n|\hat{\rho}_n-\rho_0| >2^M, |\hat{\rho}_n-\rho_0|\leq K/\sqrt{n}\} \cap A_n\right) + \Pr_{\rho_0}(A_n^c \cap \{L_1^{\rho_0}(X)>0\}) \\
&\hspace{1.5cm} + \Pr_{\rho_0}\left( |\hat{\rho}_n-\rho_0| >K/\sqrt{n},L_1^{\rho_0}(X)>0\right) \\
&\hspace{1cm} \leq \Pr_{\rho_0}\left( \sup_{2^M< n\theta\leq K\sqrt{n}} \ell_n(\theta)\geq 0, A_n\right) + \Pr_{\rho_0}\left( \sup_{2^M< -n\theta\leq K\sqrt{n}} \ell_n(\theta)\geq 0, A_n\right) +2\epsilon.
\end{split}
\end{align}
In what follows, we treat the first probability on the right-hand side, the second one can be dealt with analogously. With the shells
\[ \overline{S}_{n,j} := \left\{ \theta\in\Theta:\ 2^{j} <n\theta\leq 2^{j+1}\right\} \]
for $n\in\N$ and $j\in\Z$, we then bound by subadditivity of $\Pr_{\rho_0}$,
\begin{align}\label{eq_proof:n_consistency_2}
\Pr_{\rho_0}\left( \sup_{2^M< n\theta\leq K\sqrt{n}} \ell_n(\theta)\geq 0, A_n\right) \leq \sum_{\substack{j\geq M \\ 2^j\leq K\sqrt{n}}} \Pr_{\rho_0}\left( \sup_{\theta\in \overline{S}_{n,j}} \ell_n(\theta)\geq 0, A_n\right).
\end{align} 
To proceed, we note that
\begin{align*}
&\left\{\sup_{\theta\in \overline{S}_{n,j}} \ell_n(\theta)\geq 0\right\} = \left\{\sup_{\theta\in \overline{S}_{n,j}} \ell_n(\theta) - \sup_{\theta\in \overline{S}_{n,j}} B_n(\theta)\geq -\sup_{\theta\in \overline{S}_{n,j}} B_n(\theta)\right\} \\
&\hspace{0.5cm} \subset  \left\{\sup_{\theta\in \overline{S}_{n,j}} \left(\ell_n(\theta) - B_n(\theta)\right)\geq -\sup_{\theta\in \overline{S}_{n,j}} B_n(\theta)\right\} =  \left\{\sup_{\theta\in \overline{S}_{n,j}} M_n(\theta)\geq -\sup_{\theta\in \overline{S}_{n,j}} B_n(\theta)\right\}.
\end{align*}
Then~\eqref{eq_proof:n_consistency_2} can be extended to
\[ \Pr_{\rho_0}\left( \sup_{2^M< n\theta\leq K\sqrt{n}} \ell_n(\theta)\geq 0, A_n\right) \leq \sum_{\substack{j\geq M \\ 2^j\leq K\sqrt{n}}} \Pr_{\rho_0}\left( \sup_{\theta\in \overline{S}_{n,j}} M_n(\theta)\geq -\sup_{\theta\in \overline{S}_{n,j}} B_n(\theta), A_n\right).\]
Next, we bound the supremum of $B_n(\theta)$ over $\overline{S}_{n,j}$ with the bounds provided in Lemma~\ref{lemma:bound_drift} where it is necessary to distinguish the cases where $2^j\leq \kappa_0\sqrt{n}$ and $2^j>\kappa_0\sqrt{n}$. This gives on $A_n$ the bounds
\[ -\sup_{\theta\in \overline{S}_{n,j}} B_n(\theta)\1_{A_n} = \inf_{\theta\in \overline{S}_{n,j}} \left(-B_n(\theta)\1_{A_n}\right) \geq \begin{cases}
2^j \zeta & \textrm{ for } 2^j \leq \kappa_0\sqrt{n}, \\ n^{-1/2} 2^{2j}\zeta & \textrm{ for } 2^j\geq \kappa_0\sqrt{n},
\end{cases}\]
and consequently with Markov's inequality,
\begin{align}\label{eq_proof:n_consistency_3}
\begin{split}
\Pr_{\rho_0}\left( \sup_{2^M< n\theta\leq K\sqrt{n}} \ell_n(\theta)\geq 0, A_n\right) & \leq  \sum_{\substack{j\geq M \\ 2^j\leq \kappa_0\sqrt{n}}} \E_{\rho_0}\left[ \sup_{\theta\in \overline{S}_{n,j}} |M_n(\theta)|\1_{A_n}\right] \frac{1}{\zeta 2^j} \\
&\hspace{0.5cm} + \sum_{\substack{j\geq M \\ \kappa_0\sqrt{n} < 2^j\leq K\sqrt{n}}} \E_{\rho_0}\left[\sup_{\theta\in \overline{S}_{n,j}} |M_n(\theta)|\1_{A_n}\right]  \frac{\sqrt{n}}{\zeta 2^{2j}}.
\end{split}
\end{align}
In order to bound the expectations of the suprema appearing in the right-hand side of~\eqref{eq_proof:n_consistency_3}, we apply the modified chaining techique already mentioned in Subsection~\ref{subsection:sqrt(n)_consistency}. To this aim, we split $M_n(\theta) = M_n^1(\theta) + M_n^2(\theta)$ into two parts, where 
\begin{align*}
M_n^1(\theta) &= \sum_{k=1}^n \log\left(\frac{\beta^2}{\alpha^2}\right)\1_{\{X_{(k-1)/n}<\rho_0< X_{k/n}\leq \rho_0+\theta\}} + \log\left(\frac{\beta^2}{\alpha^2}\right)\1_{\{\rho_0< X_{k/n}\leq \rho_0+\theta\leq X_{(k-1)/n}\}} \\
&\hspace{0.5cm} - \sum_{k=1}^n  \E_{\rho_0}\left[\log\left(\frac{\beta^2}{\alpha^2}\right)\1_{\{X_{(k-1)/n}<\rho_0< X_{k/n}\leq \rho_0+\theta\}} \right. \\
&\hspace{3cm} \left. \left. + \log\left(\frac{\beta^2}{\alpha^2}\right)\1_{\{\rho_0< X_{k/n}\leq\rho_0+\theta\leq X_{(k-1)/n}\}} \right| X_{(k-1)/n}\right]
\end{align*}
and $M_n^2(\theta) = M_n(\theta)-M_n^1(\theta)$. The martingale $M_n^1(\theta)$ consists of parts of $\mathcal{I}_{2}(\theta)$ and $\mathcal{I}_{8}(\theta)$ that determine the order of the variance of $M_n(\theta)$, whereas all remaining terms are summarized in $M_n^2(\theta)$, which has a variance of smaller order for $|\theta|\ll n^{-1/2}$. The reason for this is that each summand in $M_n^1(\theta)$ is a multiple of an indicator with a factor independent of $X_{(k-1)/n}, X_{k/n}$ and $\theta$, whereas the summands in $M_n^2(\theta)$ are multiples of indicators where the factor scales (almost) linearly in $\theta$. It should be noted that this is the deeper reason behind the Poisson limit established in Theorem~\ref{thm:MLE_limiting_distribution} and the non-validity of the Lindeberg condition of the classical martingale CLT (see Section~\ref{Section:Limiting_distribution}). The decomposition of $M_n(\theta)$ now yields
\begin{align}\label{eq_proof:n_consistency_4}
\E_{\rho_0}\left[\sup_{\theta\in \overline{S}_{n,j}} |M_n(\theta)|\1_{A_n}\right]  \leq \E_{\rho_0}\left[\sup_{\theta\in \overline{S}_{n,j}} |M_n^1(\theta)|\right]  + \E_{\rho_0}\left[\sup_{\theta\in \overline{S}_{n,j}} |M_n^2(\theta)|\right] 
\end{align} 
and we bound both terms on the right-hand side seperately (and with different techniques). 
\begin{itemize}
\item[$\bullet\ \boldsymbol{M_n^1}$.] By direct evaluation (Appendix~\ref{App:Section:proofs_n_consistency}), we have for some constant $C_1=C_1(\alpha,\beta)>0$,
\begin{align}\label{eq_proof:n_consistency_5}
\E_{\rho_0}\left[ \left( M_n^1(\theta) - M_n^1(\theta')\right)^2\right] \leq C_1 n|\theta-\theta'|.
\end{align}
Note that this bound is the same as in~\eqref{eq_proof:sqrt(n)_consistency_1}. This is somewhat remarkable as $N_n^L$ in~\eqref{eq_proof:sqrt(n)_consistency_1} was given as a part of $\mathcal{I}_5(\theta)$ which was the dominant one outside the $n^{-1/2}$-environment of $\rho_0$, but is no longer part of the dominant term $M_n^1$ within the $n^{-1/2}$-environment. Defining the metric $\rho_n(\theta,\theta')=\sqrt{C_1 n|\theta-\theta'|}$, by the same steps that were used to derive~\eqref{eq_proof:sqrt(n)_chaining_bound} and~\eqref{eq_proof:sqrt(n)_Pisier} we then have 
\begin{align}\label{eq_proof:n_consistency_6}
\begin{split}
\E_{\rho_0}\left[\sup_{\theta\in \overline{S}_{n,j}} |M_n^1(\theta)|\right] &\leq \int_1^{\sqrt{C_1}2^{(j+1)/2}} \sqrt{C_1} 2^{(j+1)/2} u^{-1} du  \\
&\hspace{0.5cm} + \left( \sum_{\theta\in \overline{S}_{n,j}\cap\mathcal{T}_{k_0}^n}  \E_{\rho_0}\left[ \sup_{\theta'\in U_{1/(C_1n)}(\theta)} \left( M_n^1(\theta) - M_n^1(\theta') \right)^2 \right] \right)^\frac12,
\end{split}
\end{align}
where $\mathcal{T}_{k_0}^n$ is the subset of the interval $[0,2^{j+1}/n]$ that remains when stopping the procedure~\eqref{eq_proof:sqrt(n)_chaining_integral} for $\rho_n(\theta,\theta')\leq 1$. From the construction of the chaining sets it follows that $\mathcal{T}_{k_0}^n$ contains at most $4C_1 2^{j+1}$ elements. The crucial observation making this chaining with a remainder term work is that we find a random variable $\overline{M}_n(\theta)$ such that
\begin{align}\label{eq_proof:n_consistency_7}
\sup_{\theta'\in U_{1/(C_1n)}(\theta)} \left( M_n^1(\theta)-M_n^1(\theta')\right)^2 \leq \overline{M}_n^1(\theta) \qquad \textrm{ and }\qquad \E_{\rho_0}\left[ \overline{M}_n^1(\theta)^2\right] \leq C_2
\end{align}
for some $C_2=C_2(\alpha,\beta)>0$ independent of $\theta$ (see Appendix~\ref{App:Section:proofs_n_consistency} for a detailed derivation). As the set $\overline{S}_{n,j}\cap\mathcal{T}_n$ consists of less then $4C_2 2^{j+1}$ elements, we then conclude from~\eqref{eq_proof:n_consistency_6} that
\begin{align*}
\E_{\rho_0}\left[\sup_{\theta\in \overline{S}_{n,j}} |M_n^1(\theta)| \right] \leq \sqrt{C_1}2^{(j+1)/2}\left(\log(\sqrt{C_1}) + \frac{j+1}{2} \log(2)\right) + 4\sqrt{C_1C_2} 2^{(j+1)/2},
\end{align*}
which is bounded from above by $C_3 j 2^{j/2}$ for some constant $C_3=C_3(\alpha,\beta)>0$. \smallskip

\item[$\bullet\ \boldsymbol{M_n^2}$.] Because $M_n^2(\theta) - M_n^2(\theta') = \sum_{k=1}^n d_k$ with martingale increments $d_k$ (meaning that $\E[d_k|X_{(k-1)/n}] = 0$), Proposition~\ref{prop:moment_product} reveals
\begin{align}\label{eq_proof:n_consistency_8}
\E_{\rho_0}\left[ \left( M_n^2(\theta) - M_n^2(\theta')\right)^2\right] & \leq C_4 |\theta-\theta'|^2 n^{3/2} =: \tilde{\rho}_n(\theta,\theta')^2
\end{align} 
for some constant $C_4=C_4(\alpha,\beta)>0$. Denoting 
\[D\left(u, \mathcal{T},\rho_n\right) := \max\{ \#\mathcal{T}_0: \mathcal{T}_0\subset \mathcal{T}, \rho_n(\theta,\theta')\geq u \textrm{ for different } \theta,\theta'\in\mathcal{T}_0\}\]
the chaining step~\eqref{eq_proof:sqrt(n)_chaining_integral} for the function $\psi(x)=x^2$ and metric $\tilde{\rho}_n(\theta,\theta')$ given in~\eqref{eq_proof:n_consistency_8} then reveals (analogously to~\eqref{eq_proof:sqrt(n)_chaining_bound}) 
\begin{align*}
\E_{\rho_0}\left[\sup_{\theta\in \overline{S}_{n,j}} |M_n^2(\theta)|\right] &\leq \int_0^{C_4^{1/2}n^{-1/4} 2^{j+1}} \left( D(u,\overline{S}_{n,j},\rho_n)\right)^{1/2} du \\
&\leq \int_0^{C_4^{1/2}n^{-1/4} 2^{j+1}} C_4^{1/4} 2^{(j+1)/2} n^{-1/8} u^{-1/2} du  =: C_5 2^j n^{-1/4}.
\end{align*} 
\end{itemize}
\noindent
In conclusion, \eqref{eq_proof:n_consistency_3} reveals
\begin{align*}
\Pr_{\rho_0}\left( \sup_{2^M< n\theta\leq K\sqrt{n}} \ell_n(\theta)\geq 0, A_n\right) &\leq  \sum_{\substack{j\geq M \\ 2^j\leq \kappa_0\sqrt{n}}} \left(\frac{C_3 2^{j/2}j}{\zeta 2^j} + \frac{C_5 2^{j}n^{-1/4}}{\zeta 2^j} \right) \\
&\hspace{1cm} +  \sum_{\substack{j\geq M \\ \kappa_0\sqrt{n}<2^j\leq K\sqrt{n}}}  \left( \frac{C_3 2^{j/2}j\sqrt{n}}{\zeta 2^{2j}} + \frac{C_5 2^j n^{1/4}}{\zeta 2^{2j}}\right) \\
&\hspace{-0.5cm}\leq \frac{C_3(1+\kappa_0)}{\zeta\kappa_0} \sum_{j>M} j2^{-j/2} + \frac{C_5(1+\kappa_0)}{\zeta\kappa_0} n^{-1/4}\log_2(K^2 n).
\end{align*}
As the series $\sum_{j\geq 0} j 2^{-j/2}$ is summable, the last term (and thus the right-hand side of~\eqref{eq_proof:n_consistency_2}) tends to zero for $M\to\infty$ and $n\to\infty$. The proof is then finished by~\eqref{eq_proof:n_consistency_9}.

\section{Proof of Proposition~\ref{prop:convergence_ell_n}}\label{Section:Limiting_distribution}


Before starting the proof, we want to draw attention to a particular feature of $M_n(z/n)$ that suggests the Poissonian structure in the limit. It turns out that the only terms of $\ell_n(z/n)$ contibuting to the variance of $M_n(z/n)$ are parts of $\mathcal{I}_{2}(z/n)$ and $\mathcal{I}_8(z/n)$ (see Figure~\ref{figure:regimes}), more precisely, they are given by the expressions
\begin{align}\label{eq:contributing_variance_1}
\log\left(\frac{\beta^2}{\alpha^2}\right) \sum_{k=1}^n \left(\1_{I_{2,k}^{z/n}}-\E_{\rho_0}\left[\left. \1_{I_{2,k}^{z/n}}\right|  X_{(k-1)/n}\right]\right)
\end{align} 
and
\begin{align}\label{eq:contributing_variance_2}
\log\left(\frac{\beta^2}{\alpha^2}\right)\sum_{k=1}^n \left( \1_{I_{8,k}^{z/n}} - \E_{\rho_0}\left[\left. \1_{I_{8,k}^{z/n}}\right| X_{(k-1)/n}\right]\right),
\end{align} 
with $I_{2,k}^{z/n}$ and $I_{8,k}^{z/n}$ given in~\eqref{eq:cases_I_jk}. Note that both of them are rare events. By orthogonality of martingale increments, one can show by means of Lemma~\ref{lemma:upper_bound_transition_density} and Corollary~\ref{cor:bound_exp_X^2} that the second moment of these expressions scales linearly in $z$ with some bound independent of $n$. By Proposition~\ref{prop:moment_product}, all other terms in $M_n(z/n)$ have second moments scaling as $z^2\mathcal{O}(1/\sqrt{n})$. Being sums of indicators over rare events, \eqref{eq:contributing_variance_1} and~\eqref{eq:contributing_variance_2} indeed violate the Lindeberg condition of the classical martingale CLT if $L_1^{\rho_0}(X)>0$, which can be seen as follows:
Setting 
\[ V_{nk}:=\log(\beta^2/\alpha^2)\1_{I_{2,k}^{z/n}}, \]
we obtain with Lemma~\ref{lemma:upper_bound_transition_density} and Corollary~\ref{cor:bound_exp_X^2} that $\sum_{k=1}^n \E_{\rho_0}[V_{nk}\mid X_{(k-1)/n}]^2 \rightarrow 0$ in probability. Using this twice,
\begin{align*}
&\sum_{k=1}^n \E_{\rho_0}\left[\left. (V_{nk}-\E_{\rho_0}[V_{nk}|X_{(k-1)/n}])^2 \1_{\{|V_{nk}-\E_{\rho_0}[V_{nk}|X_{(k-1)/n}]|>\epsilon\}}\right| X_{(k-1)/n}\right]\\
&\hspace{1cm} \geq \sum_{k=1}^n \E_{\rho_0}\left[ \left. V_{nk}^2 \1_{\{|V_{nk}-\E_{\rho_0}[V_{nk}|X_{(k-1)/n}]|>\epsilon\}}\right| X_{(k-1)/n}\right] + o_{\Pr_{\rho_0}}(1)\\
&\hspace{1cm} \geq \sum_{k=1}^n \E_{\rho_0}\left[ \left. V_{nk}^2 \1_{\{|V_{nk}|>\epsilon/2\}}\right| X_{(k-1)/n}\right]+ o_{\Pr_{\rho_0}}(1)\\
&\hspace{1cm}\geq c_{\alpha,\beta} \frac{1}{\sqrt{n}} \sum_{k=1}^n \1_{\{X_{(k-1)/n}<\rho_0\}} \exp\left(-\frac{(X_{(k-1)/n}-\rho_0)^2}{2\min\{\alpha^2,\beta^2\}/n}\right) + o_{\Pr_{\rho_0}}(1)
\end{align*}
where the lower bound on the transition density in~Lemma~\ref{lemma:upper_bound_transition_density} has been used in the last inequality with the corresponding constant $c_{\alpha,\beta}$. Finally, this expression converges to some multiple of the local time $L_T^{\rho_0}(X)$ by Lemma~\ref{lemma:Riemann_approximation}.

\begin{proof}[Proof of Proposition~\ref{prop:convergence_ell_n}]
By the Remark following Proposition~$3.9$ in~\cite{Haeusler/Luschgy}, the statement of the proposition follows if we verify
\begin{itemize}
\item[(i)] \textit{Stable convergence of fidis}: For any $N\in\N$ and $z_1,\dots, z_N\in\R$ we have $\cF$-stable convergence of the finite dimensional distributions, i.e.
\begin{align}\label{eq_proof:fidis}
\big( \ell_n(z_1/n), \dots, \ell_n(z_N/n)\big) \stackrel{\cF-st}{\longrightarrow}\ \big( \ell(z_1L_1^{\rho_0}(X)), \dots, \ell(z_N L_1^{\rho_0}(X))\big).
\end{align} 
\item[(ii)] \textit{Tightness}: The process $(\ell_n(z/n))_{z\in [-K,K]}$ is tight as an element of $\mathcal{D}([-K,K])$.
\end{itemize}
Part (ii) is proven using a moment criterion for tightness in the Skorohod space (\cite{Billingsley_2}, Remark $(13.14)$ after Theorem~$13.5$) and fully given in Lemma~\ref{lemma:tightness}. Compared to many other situations, where the convergence of fidis is easier to derive, the harder part in our setting is to establish (i). By a stable version of the Cramér--Wold device (\cite{Haeusler/Luschgy}, Corollary~$3.19$(iii)), establishing \eqref{eq_proof:fidis} is equivalent to proving
\begin{align}\label{eq_proof:stable_fidi_Cramer-Wold}
\kappa_1\ell_n(z_1/n)+ \dots + \kappa_N \ell_n(z_N/n)\stackrel{\cF-st}{\longrightarrow}\ \kappa_1\ell(z_1 L_1^{\rho_0}(X)) + \dots + \kappa_N\ell(z_N L_1^{\rho_0}(X))
\end{align}  
for any choice of $\kappa_1,\dots, \kappa_N\in\R$.\\

\noindent
\textbf{Stable convergence of linear combinations in~\eqref{eq_proof:stable_fidi_Cramer-Wold}.} Proving~\eqref{eq_proof:stable_fidi_Cramer-Wold} is the most involved part of the proof. In fact, we are going to derive a stable limit of 
\begin{align}\label{eq_proof:linear_combination}
\kappa_1\ell_{n,\bullet}(z_1/n)+ \dots + \kappa_N \ell_{n,\bullet} (z_N/n)
\end{align} 
as stable convergence in the Skorohod space $\mathcal{D}([0,1])$. The reason for this is that arguments used for processes, in particular martingale arguments, can be applied. To the best of our knowledge, Theorem~$2.1$ in~\cite{Jacod_stableGaussian} is the only result that deals with stable convergence in the setting of infill asymptotics and does not require a certain nestedness condition on the filtration (that is not valid in our setup). As the described result in~\cite{Jacod_stableGaussian} only covers a continuous (in time) limit, we have to do some modification of Theorem~$4.1$ in~\cite{Jacod_stablePII} that covers limit processes with jumps but does not allow in its current formulation to treat convergence of processes $X^n$, where each $X^n$ is defined on a different stochastic basis $\mathcal{B}^n$. The resulting Proposition~\ref{prop:version_Jacod} is presented in Section~\ref{App:Jacod_version}. 
In order to properly present its application, we assume without loss of generality that $z_m<\dots <z_1<0< z_{m+1}< \dots< z_N$ and define with $M=N-m$
\[ z_0=0, \qquad z_j^- := z_j,\ j=1,\dots, m\quad\textrm{ and } z_j^+:= z_{m+j},\ j=1,\dots, M.\]
Moreover, we denote the corresponding $\kappa_i$ by $\kappa_i^{\pm}$ such that
\begin{align*}
\sum_{i=1}^n \kappa_i \ell_{n,t}(z_i/n) = \sum_{i=1}^m \kappa_i^- \ell_{n,t}(z_i^-/n) + \sum_{i=1}^M \kappa_i^+ \ell_{n,t}(z_i^+/n).
\end{align*}
With this notation, we are prepared to state the stable convergence result for the process~\eqref{eq_proof:linear_combination} as follows:

\begin{proposition}\label{prop:stable_fidi_process}
Let $(\Omega,\cF, (\cF_t)_{t\in [0,1]}, \Pr)$ be the standard Wiener space. Then there exists a very good extension $(\overline{\Omega},\overline{\cF},(\overline{\cF}_t)_{t\in [0,1]}, \overline{\Pr})$ and a process $\overline{\ell}$ on this extension such that the process
\[ \kappa_1\ell_{n,\bullet}(z_1/n)+ \dots + \kappa_N \ell_{n,\bullet} (z_N/n)\]
converges $\cF$-stably to $\overline{\ell}$ and the characteristics $(B,C,\nu)$ of $(\overline{\ell},W)$ are given by 
\[ B=\begin{pmatrix} b_{\alpha,\beta}' L_\bullet^{\rho_0}(X)\sum_{i=1}^m \kappa_i^- |z_i^-| + b_{\alpha,\beta}L_\bullet^{\rho_0}(X) \sum_{i=1}^M \kappa_i^+ |z_i^+|  \\ 0\end{pmatrix}, \qquad C=\begin{pmatrix} 0 & 0 \\ 0 & \mathrm{id}_{[0,1]}\end{pmatrix},\]
where $\mathrm{id}_{[0,1]}:[0,1]\rightarrow [0,1]$ denotes the identity, and 
\begin{align*}
\nu(dt, dx, dy) &= \frac{1}{\alpha^2} dL_t^{\rho_0}(X) \otimes \sum_{i=1}^m\delta_{\left(\log\left(\frac{\alpha^2}{\beta^2}\right) ( \kappa_i^- + \dots + \kappa_m^-), 0\right)}(dx,dy) \left|z_i^- - z_{i-1}^-\right|  \\
&\hspace{0.5cm} + \frac{1}{\beta^2}dL_t^{\rho_0}(X) \otimes \sum_{i=1}^M \delta_{\left(\log\left(\frac{\beta^2}{\alpha^2}\right) ( \kappa_i^+ + \dots + \kappa_M^+), 0\right)}(dx,dy) \left|z_i^+ - z_{i-1}^+\right|.
\end{align*} 
\end{proposition}

\noindent
The proof of this result is deferred to the next Subsection~\ref{Subsection:proof_stable_fidi_process}.\\

\noindent
\textbf{Construction of the limit $\overline{\ell}$.}
To explicitly construct the limiting process $\overline{\ell}$ in Proposition~\ref{prop:stable_fidi_process}, it is necessary to capture its nature in both the variables $t$ and $z$. To this aim, it is constructed based on two independent bivariate Poisson processes (instead of an univariate Poisson process that is sufficient to describe the limit of the MLE).
For the technical details, let 
\[ \Omega'' =\big\{\N_0\textrm{-valued measure on } ([0,1]\times\R, \mathcal{B}([0,1]\times\R))\big\}^2, \]
\[\cF'' = \sigma\Big((\mu_1,\mu_2)\in\Omega'': \mu_1(B_1)=k_1, \mu_2(B_2)=k_2:  B_1,B_2\in\mathcal{B}([0,1]\times\R), k_1,k_2\in\N_0\Big)\]
with the process 
\[ N(t,z) = \begin{cases} \omega_1''([0,t]\times [0,-z)), &\textrm{ if } z< 0,\\
\omega_2''([0,t]\times [0,z]), &\textrm{ if } z\geq 0. \end{cases} \]
Furthermore, set
\[ \cF_s'' = \bigcap_{t>s} \sigma(N(u,z): u\leq t, z\in\R)\]
and define a measure $\hat{\Pr}(d\omega,d\omega'') = \Pr(d\omega)Q_\omega(d\omega'')$ on $(\Omega\times \Omega'', \cF\otimes\cF'')$ such that for the identity process on $\Omega''$ the first component is a Poisson point process with intensity measure $\nu_1^\omega$ given by
\[ \nu_1^\omega([s,t], [z_1,z_2]) := \big( L_t^{\rho_0}(X)(\omega) - L_s^{\rho_0}(X)(\omega)\big) |z_1-z_2| \alpha^{-2}\]
and the second component of the identity process is an independent Poisson point process with intensity measure $\nu_2^\omega$ given by
\[ \nu_2^\omega([s,t], [z_1,z_2]) := \big( L_t^{\rho_0}(X)(\omega) - L_s^{\rho_0}(X)(\omega)\big) |z_1-z_2| \beta^{-2}.\]
Endowing $(\Omega\times \Omega'', \cF\otimes\cF'',\hat{\Pr})$ with the filtration $(\bigcap_{t>s}\cF_t\otimes\cF_t'')_{s\in [0,1]}$, the resulting stochastic basis is a very good extension of $(\Omega,\cF, (\cF_s)_{s\in [0,1]},\Pr)$, because the mapping $\omega\mapsto Q_\omega(N(u,z)=k) = |z|L_u^{\rho_0}(X)(\omega)$ is $\cF_u$-measurable, $L_\bullet^{\rho_0}(X)$ is continuous and $(\cF_s'')_{s\in [0,1]}$ is right-continuous.
Building on these definitions, we set
\begin{align*}
\overline{\ell}(t,z) &:= \left( \1_{\{z\geq 0\}} \left( b_{\alpha,\beta} - \frac{1}{\beta^2}\log\left(\frac{\beta^2}{\alpha^2}\right)\right)+  \1_{\{z< 0\}} \left( b_{\alpha,\beta}' - \frac{1}{\alpha^2}\log\left(\frac{\alpha^2}{\beta^2}\right)\right)\right) |z|L_t^{\rho_0}(X) \\
&\hspace{1cm} + \left( \1_{\{z\geq 0\}}\log(\beta^2/\alpha^2)N(t,z) + \1_{\{z<0\}}\log(\alpha^2/\beta^2) N(t,z) \right).
\end{align*}
In what follows, we prove the characteristics of 
\begin{align}\label{eq:linear_combination_limit}
\left( \sum_{j=1}^m \kappa_j^- \overline{\ell}(\bullet, z_j^-) + \sum_{j=1}^M \kappa_j^+ \overline{\ell}(\bullet, z_j^+), W\right)
\end{align} 
to be equal to the characteristics given in Proposition~\ref{prop:stable_fidi_process}. As they are clearly $(\cF_t)_{t\in [0,1]}$-predictable, Theorem~$3.2$ in~\cite{Jacod_stablePII} reveals that the $\cF$-conditional law of the processes $\ell$ in Proposition~\ref{prop:stable_fidi_process} and the first coordinate of~\eqref{eq:linear_combination_limit} are the same.\\

In order to derive the characteristics of~\eqref{eq:linear_combination_limit}, we need a preliminary result. Here, we slightly abuse notation and passagewise explicitly highlight the dependence of random variables on their respective argument $\omega$ and $\omega'$ within conditional expectation.

\begin{lemma}\label{lemma:conditional_exp_extension}
Let
\[ \hat{\Omega}=\Omega\times\Omega',\quad \hat{\cF}=\cF\otimes\cF',\quad \hat{\cF}_t = \bigcap_{s>t} \cF_s\otimes\cF_s',\quad \hat{\Pr}(d\omega,d\omega') = \Pr(d\omega)Q_\omega(d\omega')\]
be a very good extension of $(\Omega, \cF, (\cF_t)_{t\in [0,1]}, \Pr)$ in the sense of Definition~$II.7.1$ in~\cite{Jacod/Shiryaev}. In particular, for any $A'\in\cF_t'$, the map $\omega\mapsto Q_\omega(A')$ is $\cF_t$-measurable. Let $X:\Omega\longrightarrow\R$ be a random variables in $L^1(\Pr)$. Then for any $s\in [0,1]$,
\[ \hat{\E}\left[ X | \cF_s\otimes\cF_s'\right] = \E\left[ X|\cF_s\right]\quad a.s. \]
and for $Y_s:\Omega'\longrightarrow\R$ being in $L^1(\Pr')$ and independent of $\cF_s'$,
\[ \hat{\E}\left[ Y_s|\cF_s\otimes\cF_s'\right] = \E\left[\left. \E_{Q_\omega}\left[ Y_s \right] \right| \cF_s\right] \quad a.s.\]
\end{lemma}
\begin{proof}
Let $A\in\cF_s$ and $A'\in\cF_s'$. Then
\begin{align*}
&\hat{\E}\left[ \1_{A\times A'} \E[X|\cF_s] \right] = \hat{\E}\left[ \1_{A'}\E[\1_A X|\cF_s]\right] = \E\left[ Q_\omega(A') \E[\1_AX|\cF_s]\right] \\
&\hspace{0.5cm} =\E\left[ \E[\1_A Q_\omega(A')X|\cF_s]\right] = \E\left[\1_A Q_\omega(A')X\right] = \E\left[ \E_{Q_\omega}[\1_A\1_{A'} X]\right] = \hat{\E}\left[ \1_{A\times A'} X\right].
\end{align*}
Here, the third step follows as $Q_\omega(A')$ is $\cF_s$ measurable. Now the first claim follows from the definition of conditional expectation and $\cF_s\otimes\cF_s'=\sigma(A\times A': A\in\cF_s, A'\in\cF_s')$.\\
For the second claim, let $A,A'$ be as above. Then,
\begin{align*}
\hat{\E}\left[ \1_{A\times A'} \E[\E_{Q_\omega}[Y_s]|\cF_s]\right] &= \E\big[ \1_A \E_{Q_\omega}\left[ \1_{A'}\E[\E_{Q_\omega}[Y_s]|\cF_s]\right]\big]\\
& =\E\big[ \1_A Q_\omega(A')\E[\E_{Q_\omega}[Y_s]|\cF_s]\big]  \\
& = \E\big[ \E[\1_A Q_\omega(A')\E_{Q_\omega}[Y_s]|\cF_s]\big]\\
& =  \E\big[ \1_A Q_\omega(A')\E_{Q_\omega}[Y_s]\big] \\
&= \E\big[ \1_A \E_{Q_\omega}[\1_{A'}Y_s]\big]\\
& = \hat{\E}\left[\1_{A\times A'} Y_s\right]
\end{align*}
and the second claim follows.
\end{proof}

We now derive the characteristics of the process in~\eqref{eq:linear_combination_limit} and decompose
\begin{align*}
&\begin{pmatrix} \sum_{i=1}^m \kappa_i^- \overline{\ell}(\bullet,z_i^-) + \sum_{i=1}^M \kappa_i^+ \overline{\ell}(\bullet,z_i^+) \\ W\end{pmatrix} = B + M, 
\end{align*}
where $M=(M_t)_{t\in [0,1]}$ is given by
\begin{align*}
M &= \begin{pmatrix}0\\W\end{pmatrix} + \begin{pmatrix} \log\left(\frac{\alpha^2}{\beta^2}\right)\sum_{j=1}^m \kappa_j^- \left( N(\bullet,z_i^-) - |z_i^-|L_\bullet^{\rho_0}(X)/\alpha^2 \right) \\ 0 \end{pmatrix} \\
&\hspace{1cm} +  \begin{pmatrix} \log\left(\frac{\beta^2}{\alpha^2}\right)\sum_{j=1}^M \kappa_j^+ \left( N(\bullet,z_i^+) - |z_i^+|L_\bullet^{\rho_0}(X)/\beta^2 \right) \\ 0 \end{pmatrix}.
\end{align*}
We establish the statement for each characteristic seperately:\smallskip
\begin{itemize}
\item[$\boldsymbol{B}$:] First, note that as $L_\bullet^{\rho_0}(X)$ is continuous, $B$ is predictable. Then, the statement about the first characteristic $B$ follows if we can show that $M$ is a $\hat{\Pr}$-martingale with respect to the filtration $(\hat{\cF}_t)_{t\in [0,1]}$. We prove this for $m=0$, $M=1$, the general case follows easily by linearity of conditional expectation. We have for $0\leq s\leq t\leq 1$, using Lemma~\ref{lemma:conditional_exp_extension}
\[ \hat{\E}\left[\left.  \begin{pmatrix} 0 \\ W_t\end{pmatrix}\right| \cF_s\right]= \E\left[\left.  \begin{pmatrix} 0 \\ W_t\end{pmatrix}\right| \cF_s\right]=\begin{pmatrix} 0 \\ W_s\end{pmatrix}\]
and consequently
\begin{align*}
&\hat{\E}\left[ M_t|\cF_s\otimes \cF_s''\right] \\
&\hspace{1cm} =\begin{pmatrix} 0 \\ W_s\end{pmatrix} +  \log\left(\frac{\beta^2}{\alpha^2}\right)\begin{pmatrix} \hat{\E}\left[\left. \kappa_1^+\left( N(t,z_1^+) - z_1^+L_t^{\rho_0}(X)/\beta^2\right)\right| \cF_s\otimes \cF_s''\right] \\ 0\end{pmatrix}.
\end{align*}
To proceed, we use that by Lemma~\ref{lemma:conditional_exp_extension},
\begin{align*}
&\hat{\E}\left[\left. z_1^+L_t^{\rho_0}(X)/\beta^2\right| \cF_s\otimes\cF_s''\right] = \E\Big[\left. z_1^+L_t^{\rho_0}(X)/\beta^2 \right| \cF_s\Big].
\end{align*}
As $N(s,z_1^+)$ is $\cF_s''$-measurable and $N(t,z_1^+)-N(s,z_1^+) = \omega_2''((s,t]\times [0,z_1^+])$ is independent of $\cF_s''$, Lemma~\ref{lemma:conditional_exp_extension} reveals
\begin{align*}
\hat{\E}\left[ N(t,z_1^+)| \cF_s\otimes \cF_s''\right] &= \hat{\E}\left[ N(t,z_1^+) - N(s,z_1^+) | \cF_s\otimes \cF_s''\right] +\hat{\E}\left[ N(s,z_1^+)| \cF_s\otimes \cF_s''\right] \\
&=\E\left[ \E_{Q_\omega}[ N(t,z_1^+) - N(s,z_1^+)]\big| \cF_s\right] + N(s,z_1^+).
\end{align*}
Here and in what follows, by slight abuse of notation, we explicitly indicate the dependence of $Q_\omega$ on $\omega$ to highlight that the conditional expectation affects the corresponding expression. Recalling that the second component $\omega_2''$ of the identity on $\Omega''$ is a Poisson point process with intensity measure $\nu_2^\omega([s,t]\times[z_1,z_2]) = \left( L_t^{\rho_0}(X)(\omega)-L_s^{\rho_0}(X)(\omega)\right) |z_1-z_2|/\beta^2$ under $Q_\omega$,
\begin{align*}
\E_{Q_\omega}[ N(t,z_1^+) - N(s,z_1^+)] = \E_{Q_\omega}[ \omega_2''((s,t]\times [0,z_1^+])] = \left(L_t^{\rho_0}(X)-L_s^{\rho_0}(X)\right)z_1^+/\beta^2.
\end{align*} 
Consequently,
\begin{align*}
&\hat{\E}\left[\left. \kappa_1^+\left( N(t,z_1^+) - z_1^+L_t^{\rho_0}(X)/\beta^2\right)\right| \cF_s\otimes\cF_s''\right] \\
&\hspace{1cm} = \E\left[\left. \kappa_1^+  N(s,z_1^+) + \kappa_1^+\left(L_t^{\rho_0}(X)-L_s^{\rho_0}(X)\right)z_1^+/\beta^2 -\kappa_1^+L_t^{\rho_0}(X)z_1^+/\beta^2 \right|\cF_s\right] \\
&\hspace{1cm} =  \kappa_1^+ N(s,z_1^+)-\kappa_1^+L_s^{\rho_0}(X)z_1^+/\beta^2
\end{align*}
which yields
\[ \hat{\E}\left[ M_t|\cF_s\otimes\cF_s''\right] = M_s\quad a.s.\]
Let $(s_n)_{n\in\N}$ be a decreasing sequence with $s_n\searrow s$. Then by right-continuity of $M$ and the tower property of conditional expectation, together with dominated convergece that is applicable by the Burkholder-Davis-Gundy inequality as $M$ is square-integrable,
\[ \hat{\E}[M_t|\hat{\cF}_s]=\hat{\E}\big[\hat{\E}[M_t|\cF_{s_n}\otimes\cF_{s_n}'']\big|\hat{\cF}_s\big] = \hat{\E}\big[ M_{s_n}|\hat{\cF}_s\big] \longrightarrow  \hat{\E}[M_s|\hat{\cF}_s] = M_s.  \]

\item[$\boldsymbol{C}$:] The continuous martingale part of $M$ is $M^c=(0,W)^t$. Hence, it is clear that
\[ C=\begin{pmatrix} 0 & 0 \\ 0 & \mathrm{id}_{[0,1]}\end{pmatrix}.\]

\item[$\boldsymbol{\nu}$:] The third characteristic $\nu$ is given as the predictable compensator of the jump measure associated with the process $(Z,W)$, where $Z_s=\sum_{i=1}^m \kappa_i^- \overline{\ell}(s,z_i^-) + \sum_{i=1}^M \kappa_i^+ \overline{\ell}(s,z_i^+)$, i.e. of
\begin{align*}
\hat{\mu} &= \sum_{s\geq 0} \1_{\{\Delta((Z,W)_s)\neq (0,0)\}} \delta_{(s,\Delta Z_s, \Delta W_s)} = \sum_{s\geq 0} \1_{\{\Delta Z_s\neq 0\}}\delta_{(s,\Delta Z_s,0)}.
\end{align*} 
To proceed, we first note that $\Delta Z_s = \log\left(\frac{\alpha^2}{\beta^2}\right)\sum_{i=1}^m \Delta N(s,z_i^-) + \log\left(\frac{\beta^2}{\alpha^2}\right) \sum_{i=1}^M \Delta N(s,z_i^+)$. Morover, we observe:\smallskip
\begin{itemize}
\item[(a)] For $i<j$, we have $z_i^->z_j^-$ and thus
\[ N(t,z_j^-) = \omega_1''([0,t]\times [0,-z_j^-)) = \omega_1''([0,t]\times [0,-z_i^-)) + \omega_1''([0,t]\times [-z_i^-,-z_j^-)). \]
From this, it follows that $N(\bullet,-z_j^-)$ jumps every time the process $N(\bullet,-z_i^-)$ jumps.
\item[(b)] Analogously, for $i<j$, we have $z_i^+<z_j^+$ and thus
\[ N(t,z_j^+) = \omega_2''([0,t]\times [0,z_j^+]) = \omega_2''([0,t]\times [0,z_i^+]) + \omega_2''([0,t]\times [z_i^+,z_j^+]). \]
From this, it follows that $N(\bullet,z_j^+)$ jumps every time the process $N(\bullet,z_i^+)$ jumps.
\end{itemize}\smallskip
From these considerations and the fact that $\omega_1''([0,\bullet]\times [0,-z])$ and $\omega_2''([0,\bullet]\times [0,z'])$ jump at the same time with probability zero (due to their independence), it follows that
\begin{align*}
(g\star \hat{\mu})_t &= \sum_{u\leq t}\sum_{i=1}^m \1_{\{ \Delta Z_u = \log\left(\frac{\alpha^2}{\beta^2}\right)\left(\kappa_i^- + \dots + \kappa_m^-\right)\}} g\left(\log\left(\frac{\alpha^2}{\beta^2}\right)\left( \kappa_i^- + \dots + \kappa_m^-\right),0\right) \\
&\hspace{1cm} + \sum_{u\leq t}\sum_{i=1}^M \1_{\{ \Delta Z_u =\log\left(\frac{\beta^2}{\alpha^2}\right)\left( \kappa_i^+ + \dots + \kappa_M^+\right)\}} g\left(\log\left(\frac{\beta^2}{\alpha^2}\right)\left( \kappa_i^+ + \dots + \kappa_M^+\right) ,0\right).
\end{align*}
By continuity of $L_\bullet^{\rho_0}(X)$, $(g\star\nu)$ is predictable with respect to $(\hat{\cF}_s)_{s\in [0,1]}$. It remains to show that $(g\star\hat{\mu})_\bullet - (g\star\nu)_\bullet$ is an $(\hat{\cF}_s)_{s\in [0,1]}$-martingale. From (a) and (b) above, we obtain
\begin{align*}
\sum_{s<u\leq t} \1_{\{ \Delta Z_u = \kappa_i^- + \dots + \kappa_m^-\}} &= \sum_{s<u\leq t} \1_{\{ \Delta\left( \kappa_1^- N(\bullet, z_1^-) + \dots +\kappa_m^- N(\bullet, z_m^-)\right)_u  = \kappa_i^- + \dots + \kappa_m^-\}} \\
&=  \omega_1''((s,t]\times [0,-z_i^-)) - \omega_1''((s,t]\times [0, -z_{i-1}^-))\\
& = \omega_1''((s,t]\times [-z_{i-1}^-,-z_i^-)).
\end{align*}
Analogously, 
\[ \sum_{s<u\leq t} \1_{\{ \Delta Z_u = \kappa_i^+ + \dots + \kappa_M^+\}} = \omega_2''((s,t]\times (z_{i-1}^+,z_i^+]). \]
Then,
\begin{align*}
&\hat{\E}\left[ (g\star\hat{\mu})_t-(g\star\nu)_t -(g\star\hat{\mu})_s+(g\star\nu)_s \Big| \hat{\cF}_s\right] \\
&\hspace{0.5cm} = \sum_{i=1}^m g\left( \log\left(\frac{\alpha^2}{\beta^2}\right)\left(\kappa_i^- + \dots + \kappa_m^-\right) ,0\right) \\
&\hspace{1.5cm} \cdot \hat{\E}\left[\left. \omega_1''((s,t]\times[-z_{i-1}^-,-z_i^-)) - \frac{L_t^{\rho_0}(X) - L_s^{\rho_0}(X)}{\alpha^2} |z_i^- - z_{i-1}^-| \right| \hat{\cF}_s \right] \\
&\hspace{1cm} + \sum_{i=1}^M g\left(\log\left(\frac{\beta^2}{\alpha^2}\right) \left( \kappa_i^+ + \dots + \kappa_M^+\right), 0\right) \\
&\hspace{1.5cm} \cdot\hat{\E}\left[\left. \omega_2''((s,t]\times(z_{i-1}^+,z_i^+]) - \frac{L_t^{\rho_0}(X) - L_s^{\rho_0}(X)}{\beta^2} |z_i^+ - z_{i-1}^+| \right| \hat{\cF}_s \right]
\end{align*}
and by the same arguments used for $B$, both conditional expectations vanish.
\end{itemize}

\noindent
\textbf{Finalizing the proof of~\eqref{eq_proof:stable_fidi_Cramer-Wold}.}
From Proposition~\ref{prop:stable_fidi_process} and the explicit construction of $\overline{\ell}$ in the preceding step, we obtain
\begin{align}\label{eq_proof:stable_fidi_process}
\kappa_1\ell_{n,\bullet}(z_1/n)+ \dots + \kappa_N \ell_{n,\bullet} (z_N/n)\stackrel{\cF-st}{\longrightarrow}\ \kappa_1\overline{\ell} (\bullet, z_1) + \dots + \kappa_N \overline{\ell}(\bullet, z_N).
\end{align}
as stable convergence in the Skorohod space $\mathcal{D}([0,1])$. Then, \eqref{eq_proof:stable_fidi_Cramer-Wold} follows from this by using that the projection onto the endpoint $1$ is continuous in $\mathcal{D}([0,1])$, combined with a stable version of the continuous mapping theorem (\cite{Haeusler/Luschgy}, Theorem~$3.18$(c)) and the fact that 
\[  \left(L_1^{\rho_0}(X), \left(\overline{\ell}(1,z)\right)_{z\in\R}\right) \stackrel{\mathcal{L}}{=} \left( L_1^{\rho_0}(X), \left(\ell(z L_1^{\rho_0}(X))\right)_{z\in\R}\right).  \]
\end{proof}

\subsection{Proof of Proposition~\ref{prop:stable_fidi_process}}\label{Subsection:proof_stable_fidi_process}
Recall that
\[ \kappa_1\ell_{n,\bullet}(z_1/n) + \dots + \kappa_N\ell_{n,\bullet}(z_N/n) = \sum_{j=1}^N \kappa_j B_{n,\bullet}(z_j/n) + \sum_{j=1}^N \kappa_j M_{n,\bullet}(z_j/n), \]
which is the canonical semimartingale decomposition of the process on the left-hand side with respect to the filtration $(\cF_t^n)_{t\in [0,1]}$ with $\cF_t^n := \cF_{\lfloor nt\rfloor/n}$ on $(\Omega,\cF,\Pr)$. Furthermore, note that 
\[\kappa_1\ell_{n,t}(z_1/n) + \dots + \kappa_N\ell_{n,t}(z_N/n) = \sum_{k=1}^{\lfloor nt\rfloor} y_{nk}\]
for the random variables
\[  y_{nk} = \sum_{j=1}^N \kappa_j \log\left(\frac{p_{1/n}^{\rho_0+z_j/n}(X_{(k-1)/n},X_{k/n})}{p_{1/n}^{\rho_0}(X_{(k-1)/n},X_{k/n})}\right).\]
In order to apply Proposition~\ref{prop:version_Jacod} we have to establish the following for every $t\in [0,1]$:
\begin{align}\label{condition_1char}
\sup_{s\leq t}\left| \sum_{j=1}^N \kappa_j B_{n,s}(z_j/n) - \sum_{j=1}^N \kappa_j\left(\1_{\{z_j\geq 0\}}b_{\alpha,\beta} + \1_{\{z_j<0\}}b_{\alpha,\beta}'\right)|z_j| L_s^{\rho_0}(X)\right| \longrightarrow_{\Pr_{\rho_0}} 0,
\end{align}
\begin{align}\label{condition_2char_1}
\begin{split}
&\sum_{k=1}^{\lfloor nt\rfloor} \E_{\rho_0}\left[ (y_{nk} - \E_{\rho_0}[y_{nk}|\cF_{(k-1)/n}])^2 |\cF_{(k-1)/n}\right] \\
&\hspace{0.5cm} \longrightarrow_{\Pr_{\rho_0}}  \log\left(\frac{\alpha^2}{\beta^2}\right)^2 \frac{1}{\alpha^2} L_t^{\rho_0}(X) \sum_{i=1}^m \left( \kappa_i^- + \dots + \kappa_m^-\right)^2 \left|z_i^- - z_{i-1}^-\right|  \\
&\hspace{2cm} + \log\left(\frac{\beta^2}{\alpha^2}\right)^2 \frac{1}{\beta^2}L_t^{\rho_0}(X) \sum_{i=1}^M \left( \kappa_i^+ + \dots + \kappa_M^+\right)^2 \left|z_i^+ - z_{i-1}^+\right|,
\end{split}
\end{align}
\begin{align}\label{condition_2char_2}
&\sum_{k=1}^{\lfloor nt\rfloor} \E_{\rho_0}\left[ (y_{nk} - \E_{\rho_0}[y_{nk}|\cF_{(k-1)/n}])(W_{k/n}-W_{(k-1)/n}) |\cF_{(k-1)/n}\right] \longrightarrow_{\Pr_0}\ 0,
\end{align}
and for each $g:\R^2\longrightarrow\R$ measurable, non-negative, bounded, Lipschitz-continuous and vanishing in a neighborhood of zero,
\begin{align}\label{condition_3char}
\sum_{k=1}^{\lfloor nt\rfloor} \E_{\rho_0}\left[\left. g\left(y_{nk},W_{k/n}-W_{(k-1)/n}\right)\right| \cF_{(k-1)/n}\right] \longrightarrow_{\Pr_{\rho_0}}\ (g\star \nu)_t.
\end{align}
Finally, we also show for suitable $\kappa>0$ that
\begin{align}\label{condition_second_moment}
\E_{\rho_0}\left[ \sum_{k=1}^n y_{nk}^2 \1_{\{|y_{nk}|>\kappa\}}\right] \longrightarrow\ 0.
\end{align}
Condition~\eqref{condition_1char} follows from Corollary~\ref{cor:limiting_drift} and conditions~\eqref{condition_2char_1}, \eqref{condition_2char_2} and~\eqref{condition_second_moment} are proven in Appendix~\ref{App:conditions_prop_stable}. We are going to show the most interesting one, namely~\eqref{condition_3char} on the jump characteristic that brings out the bivariate Poissonion nature of the limit $\overline{\ell}$. To this aim, it is necessary to separate the terms that contribute to the jumps of the process by rewriting
\begin{align*}
y_{nk} &=   \sum_{l=1}^m  \kappa_l^- \left(-\sum_{j=1}^9 Z_k^j(z_l^-/n,0) + \log\left(\frac{\alpha^2}{\beta^2}\right)\left( \1_{I_{2,k}^{z_l^-/n,0}} + \1_{I_{8,k}^{z_l^-/n,0}}\right)\right) \\
&\hspace{1cm} +  \sum_{l=1}^M  \kappa_l^+ \left(\sum_{j=1}^9Z_k^j(0,z_l^+/n) + \log\left(\frac{\beta^2}{\alpha^2}\right)\left( \1_{I_{2,k}^{0,z_{l}^+/n}} + \1_{I_{8,k}^{0,z_l^+/n}}\right)\right).
\end{align*}
Let $\delta>0$ be chosen such that $g(x)=0$ for $\|x\|_2<\delta$ which is possible since $g$ vanishes in a neighborhood of zero. Furthermore, we define
\[ J_k :=  \bigcup_{l=1}^m \left( I_{2,k}^{z_l^-/n,0}\cup I_{8,k}^{z_l^-/n,0}\right) \cup\bigcup_{l=1}^M \left(I_{2,k}^{0,z_{l}^+/n} \cup I_{8,k}^{0,z_l^+/n} \right)\]
which corresponds to those cases for $X_{(k-1)/n}, X_{k/n}$ that contribute to the variance of at least one of $\ell_n(z_m^-/n),\dots, \ell_n(z_M^+/n)$. All other cases then are summarized in $J_k^c$ and we rewrite with $w_{nk}=W_{k/n}-W_{(k-1)/n}$,
\begin{align}\label{eq_proof:3char_1}
\begin{split}
&\sum_{k=1}^{\lfloor nt\rfloor} \E_{\rho_0}\left[ g(y_{nk}, w_{nk})|\cF_{(k-1)/n}\right] \\
&\hspace{0.7cm}= \sum_{k=1}^{\lfloor nt\rfloor} \E_{\rho_0}\left[ g(y_{nk}, w_{nk})\1_{J_k}|\cF_{(k-1)/n}\right] + \sum_{k=1}^{\lfloor nt\rfloor} \E_{\rho_0}\left[ g(y_{nk}, w_{nk})\1_{J_k^c}|\cF_{(k-1)/n}\right].
\end{split}
\end{align}
It will turn out that the first summand is the contributing term whereas the second one is of smaller order which is discussed in the following: First, by our choice of $\delta$ and the boundedness assumption on $g$,
\begin{align*}
\sum_{k=1}^{\lfloor nt\rfloor} \E_{\rho_0}\left[ g(y_{nk}, w_{nk})\1_{J_k^c}|\cF_{(k-1)/n}\right]  &= \sum_{k=1}^{\lfloor nt\rfloor} \E_{\rho_0}\left[ g(y_{nk},w_{nk})\1_{J_k^c}\1_{\{\|(y_{nk},w_{nk})\|_2\leq \delta\}}|\cF_{(k-1)/n}\right]\\
&\hspace{0.2cm} +\sum_{k=1}^{\lfloor nt\rfloor} \E_{\rho_0}\left[ g(y_{nk},w_{nk})\1_{J_k^c} \1_{\{\|(y_{nk},w_{nk})\|_2>\delta\}}|\cF_{(k-1)/n}\right] \\
&\leq \|g\|_\textrm{sup} \sum_{k=1}^{\lfloor nt\rfloor} \E_{\rho_0}\left[\1_{J_k^c} \1_{\{\|(y_{nk},w_{nk})\|_2>\delta\}}|\cF_{(k-1)/n}\right].
\end{align*}
For the rest of the proof, we will denote $C=C(\alpha,\beta,m,M,z_m^-,\dots, z_M^+, \kappa_m^+,\dots, \kappa_M^+)$ for a constant that depends on both $\alpha,\beta$ and the parameters of the linear combination $\kappa_1\ell_{n,t}(z_1/n)+\dots, \kappa_N\ell_{n,t}(z_N/n)$. By Proposition~\ref{prop:moment_product} we then find the bound
\begin{align*}
&\E_{\rho_0}\left[ \sum_{k=1}^{\lfloor nt\rfloor} \E_{\rho_0}\left[\1_{J_k^c} \1_{\{\|(y_{nk},w_{nk})\|_2>\delta\}}|\cF_{(k-1)/n}\right] \right]  \leq \frac{1}{\delta^4} \sum_{k=1}^{\lfloor nt\rfloor}  \E_{\rho_0}\left[ \1_{J_k^c}\|(y_{nk},w_{nk})\|_2^4 \right] \\
&\hspace{0.1cm} \leq \frac{C}{\delta^4} \sum_{k=1}^{\lfloor nt\rfloor}  \E_{\rho_0}\left[  \left(|w_{nk}|^2+ \sum_{j=1,3,4,5,6,7,9}\left( \sum_{l=1}^m  \kappa_l^- Z_k^j(z_l^-/n,0) +  \sum_{l=1}^M  \kappa_l^+ Z_k^j(0,z_l^+/n)\right)^2 \right)^2 \right] \\
&\hspace{0.1cm} \leq \frac{C}{\delta^4}\sum_{k=1}^{\lfloor nt\rfloor}\sum_{j=1,3,4,5,6,7,9} \left( \sum_{l=1}^m |\kappa_l^-|^4 \E_{\rho_0}\left[ |Z_k^j(z_l^-/n,0)|^4\right] + \sum_{l=1}^M |\kappa_l^+|^4 \E_{\rho_0}\left[ |Z_k^j(0,z_l^+/n)|^4\right]\right) \\
&\hspace{1cm} + \frac{C}{\delta^4}\sum_{k=1}^{\lfloor nt\rfloor} \E_{\rho_0}\left[ |w_{nk}|^4\right] \leq \frac{C}{\delta^4}\sum_{k=1}^{\lfloor nt\rfloor} \left( \frac{1}{n\sqrt{k}} + \frac{1}{n^2}\right) \leq \frac{C}{\delta^4} \left( \frac{1}{\sqrt{n}} + \frac1n\right)\longrightarrow 0.
\end{align*}\smallskip

\noindent
It remains to study the first summand on the right-hand side of~\eqref{eq_proof:3char_1} that is further split as
\begin{align*}
&\sum_{k=1}^{\lfloor nt\rfloor} \E_{\rho_0}\left[ g(y_{nk}, w_{nk})\1_{J_k}|\cF_{(k-1)/n}\right] \\
&\hspace{0.5cm} = \sum_{k=1}^{\lfloor nt\rfloor} \E_{\rho_0}\left[ g(y_{nk},0)\1_{J_k}|\cF_{(k-1)/n}\right] +\sum_{k=1}^{\lfloor nt\rfloor}\E_{\rho_0}\left[ \left(g(y_{nk}, w_{nk})-g(y_{nk},0)\right)\1_{J_k}|\cF_{(k-1)/n}\right].
\end{align*}
By Lemma~\ref{lemma:upper_bound_transition_density} and Corollary~\ref{cor:bound_exp_X^2}, we exemplarily obtain 
\begin{align*}
\begin{split}
&\E_{\rho_0}\left[\1_{I_{2,k}^{0,z_l^-/n}}\right] \leq C\sqrt{n}\E_{\rho_0}\left[\1_{\{X_{(k-1)/n}<\rho_0+z_{l-1}^-/n \}}\int_{\rho_0+z_l^-/n}^{\rho_0+z_{l-1}^-/n} \exp\left(-\frac{(y-X_{(k-1)/n})^2}{2\max\{\alpha^2,\beta^2\}/n}\right)dy \right]  \\
&\hspace{0.5cm}\leq C\frac{|z_l^- - z_{l-1}^-|}{\sqrt{n}}\E_{\rho_0}\left[ \exp\left(-\frac{(X_{(k-1)/n}-\rho_0-z_{l-1}^-/n)^2}{2\max\{\alpha^2,\beta^2\}/n}\right) \right] \leq  C\frac{|z_l^- - z_{l-1}^-|}{\sqrt{nk}},
\end{split}
\end{align*}
and by similar estimations, we get
\begin{align}\label{eq_proof:3char_4}
\E_{\rho_0}\left[ \1_{J_k}\right] \leq C_{\alpha,\beta}(z_m^-,\dots, z_M^+) \frac{1}{\sqrt{nk}}.
\end{align}
From this, the boundedness and Lipschitz assumption on $g$ (where $L_g$ denotes the Lipschitz constant of $g$) and Markov's inequality, it follows for any $\epsilon>0$ that
\begin{align*}
&\E_{\rho_0}\left[ \left| \sum_{k=1}^{\lfloor nt\rfloor} \E_{\rho_0}\left[ \left(g(y_{nk}, w_{nk})-g(y_{nk},0)\right)\1_{J_k}|\cF_{(k-1)/n}\right]\right| \right] \\
&\hspace{1cm} \leq \sum_{k=1}^{\lfloor nt\rfloor} \epsilon\E_{\rho_0}\left[ \1_{J_k}\right] + \|g\|_{\sup} \Pr_{\rho_0}\left( |w_{nk}|>\epsilon/L_g\right) \leq C\epsilon + L_g^4 \epsilon^{-4} \frac{nt}{n^2}.
\end{align*}
Thus, 
\[ \sum_{k=1}^{\lfloor nt\rfloor} \E_{\rho_0}\left[ \left(g(y_{nk}, w_{nk})-g(y_{nk},0)\right)\1_{J_k}|\cF_{(k-1)/n}\right] = o_{\Pr_{\rho_0}}(1)\]
and it suffices to investigate the limit in probability of
\[  \sum_{k=1}^{\lfloor nt\rfloor} \E_{\rho_0}\left[ g(y_{nk},0)\1_{J_k}|\cF_{(k-1)/n}\right].\]
For this study, we abbreviate $\tilde{g}(x) := g(x,0)$ and note that $\tilde{g}$ is bounded, Lipschitz-continuous with Lipschitz constant $L_g$ and vanishes in a neighborhood of zero (by our assumption on $g$). Suitably rearranging the indicators in the definition of $J_k$ into indicators of disjoint sets yields
\begin{align*}
\1_{J_k} &= \sum_{l=1}^m \left( \1_{\{X_{(k-1)/n}< \rho_0+z_l^-/n<X_{k/n}\leq \rho_0+z_{l-1}^-/n \}} + \1_{\{\rho_0+z_l^-/n<X_{k/n}\leq \rho_0+z_{l-1}^-/n \leq \rho_0\leq X_{(k-1)/n} \}} \right) \\
&\hspace{0.1cm} + \sum_{l=1}^M \left(\1_{\{ X_{(k-1)/n} < \rho_0 \leq \rho_0 + z_i^+/n< X_{k/n}\leq \rho_0+z_{i+1}^+/n\}} + \1_{\{\rho_0 + z_i^+/n< X_{k/n}\leq \rho_0+z_{i+1}^+/n \leq X_{(k-1)/n}\}}\right),
\end{align*}
and we obtain a corresponding decomposition
\begin{align}\label{eq_proof:T_1-T4}
\sum_{k=1}^{\lfloor nt\rfloor} \E_{\rho_0}\left[ \tilde{g}(y_{nk}) \1_{J_k}| \cF_{(k-1)/n}\right] = T_1 + T_2 + T_3 + T_4, 
\end{align}
where
\begin{align*}
T_1 &:= \sum_{k=1}^{\lfloor nt\rfloor}\sum_{l=1}^m \E_{\rho_0}\left[\left. \tilde{g}\left( y_{nk}\right)\1_{\{X_{(k-1)/n}< \rho_0+z_l^-/n<X_{k/n}\leq \rho_0+z_{l-1}^-/n \}} \right| \cF_{(k-1)/n}\right]\\
T_2 &:= \sum_{k=1}^{\lfloor nt\rfloor}\sum_{l=1}^m \E_{\rho_0}\left[\left. \tilde{g}\left( y_{nk}\right)\1_{\{\rho_0+z_l^-/n<X_{k/n}\leq \rho_0+z_{l-1}^-/n \leq \rho_0\leq X_{(k-1)/n} \}}  \right| \cF_{(k-1)/n}\right]\\
T_3 &:=\sum_{k=1}^{\lfloor nt\rfloor}\sum_{l=1}^M \E_{\rho_0}\left[\left. \tilde{g}\left( y_{nk}\right)\1_{\{ X_{(k-1)/n} < \rho_0 \leq \rho_0 + z_i^+/n< X_{k/n}\leq \rho_0+z_{i+1}^+/n\}} \right| \cF_{(k-1)/n}\right]\\
T_4 &:= \sum_{k=1}^{\lfloor nt\rfloor}\sum_{l=1}^M \E_{\rho_0}\left[\left. \tilde{g}\left( y_{nk}\right) \1_{\{\rho_0 + z_i^+/n< X_{k/n}\leq \rho_0+z_{i+1}^+/n \leq X_{(k-1)/n}\}} \right| \cF_{(k-1)/n}\right].
\end{align*}
In the following, we are going to evaluate these four summands seperately.
\begin{itemize}
\item[$\bullet\ \boldsymbol{T_1}.$] With $A_{kl}=\{X_{(k-1)/n}< \rho_0+z_l^-/n<X_{k/n}\leq \rho_0+z_{l-1}^-/n \}$ we have
\[ A_{kl}\cap I_{j,k}^{0,z_{u}^+/n} =\emptyset\qquad \textrm{ and }\qquad A_{kl}\cap I_{8,k}^{z_v^-/n,0}=\emptyset\]
for all $j=2,8$, $1\leq l,v\leq m$, $1\leq u\leq M$ and $1\leq k\leq n$, and obtain
\begin{align*}
&\tilde{g}\left( y_{nk}\right)\1_{\{X_{(k-1)/n}< \rho_0+z_l^-/n<X_{k/n}\leq \rho_0+z_{l-1}^-/n \}} \\
&\hspace{0.5cm} = \tilde{g}\left( \sum_{u=1}^m  \kappa_u^- \left(-\sum_{j=1}^9 Z_k^j(z_u^-/n,0)\1_{A_{kl}} + \log\left(\frac{\alpha^2}{\beta^2}\right) \left(\1_{I_{2,k}^{z_u^-/n,0}}+\1_{I_{8,k}^{z_u^-/n,0}}\right)\1_{A_{kl}}\right) \right. \\
&\hspace{1cm} \left. +\sum_{u=1}^M  \kappa_u^+ \left(\sum_{j=1}^9 Z_k^j(0,z_u^+/n)\1_{A_{kl}} + \log\left(\frac{\beta^2}{\alpha^2}\right)\left( \1_{I_{2,k}^{0,z_u^+/n}}+\1_{I_{8,k}^{0,z_u^+/n}}\right)\1_{A_{kl}} \right)\right)\\
&\hspace{0.5cm} = \tilde{g}\left( \sum_{u=1}^m\sum_{j=1}^9 \kappa_u^- Z_k^j(z_u^-/n,0)\1_{A_{kl}} + \sum_{u=1}^M\sum_{j=1}^9 \kappa_u^+ Z_k^j(0,z_u^+/n)\1_{A_{kl}} \right. \\
&\hspace{1cm}  + \log\left(\frac{\alpha^2}{\beta^2}\right)\left( \sum_{u=l}^m\kappa_u^-\1_{A_{kl}} -\sum_{u=l}^m\sum_{v=u}^m \kappa_v^- \1_{\{\rho_0+z_v^-/n \leq X_{(k-1)/n}< \rho_0\}} \right. \\
&\hspace{6.5cm} \cdot\1_{\{\rho_0+z_l^-/n<X_{k/n}\leq \rho_0+z_{l-1}^-/n\}}\1_{A_{kl}} \Big)\Bigg)\1_{A_{kl}}. 
\end{align*}
Consequently,
\begin{align*}
T_1  = \sum_{k=1}^{\lfloor nt\rfloor}\sum_{l=1}^m \E_{\rho_0}\left[\left.  \tilde{g}\left( \left(\kappa_l^-+\dots + \kappa_m^-\right) \log\left(\frac{\alpha^2}{\beta^2}\right) \right) \1_{A_{kl}} \right| \cF_{(k-1)/n}\right] + r_{1,n},
\end{align*}
where for any $\epsilon>0$,
\begin{align*}
&\E_{\rho_0}\left[|r_{1,n}|\right] \leq \epsilon \sum_{k=1}^{\lfloor nt\rfloor}\sum_{l=1}^m \E_{\rho_0}\left[\1_{A_{kl}}\right] + 2\|g\|_{\sup}\sum_{k=1}^{\lfloor nt\rfloor}\sum_{l=1}^m \Pr_{\rho_0}\left( \left|\sum_{u=1}^m\sum_{j=1}^9 \kappa_u^- Z_k^j(z_u^-/n,0)\1_{A_{kl}} \right. \right. \\
&\hspace{0.5cm} \left.\left.  + \sum_{u=1}^M\sum_{j=1}^9 \kappa_u^+ Z_k^j(0,z_u^+/n)\1_{A_{kl}}\right| + \sum_{u=l}^m\sum_{v=u}^m |\kappa_v^-|\1_{\{\rho_0+z_m^-/n\leq X_{(k-1)/n},X_{k/n}\leq \rho_0\}}  >\frac{\epsilon}{L_g}\right).
\end{align*}
By the same estimation as used for~\eqref{eq_proof:3char_4}, we obtain by Lemma~\ref{lemma:upper_bound_transition_density} and Corollary~\ref{cor:bound_exp_X^2}
\[ \E_{\rho_0}\left[ \1_{\{\rho_0+z_m^-/n\leq X_{(k-1)/n},X_{k/n}\leq \rho_0\}}\right] \leq \frac{C}{n\sqrt{k}}.\]
Then, by Markov's inequality together with Proposition~\ref{prop:moment_product},
\begin{align*}
&\Pr_{\rho_0}\left( \left|\sum_{u=1}^m\sum_{j=1}^9 \kappa_u^- Z_k^j(z_u^-/n,0)\1_{A_{kl}} + \sum_{u=1}^M\sum_{j=1}^9 \kappa_u^+ Z_k^j(0,z_u^+/n)\1_{A_{kl}}\right|\right. \\
&\hspace{2cm} \left.   + \sum_{u=l}^m\sum_{v=u}^m |\kappa_v^-|\1_{\{\rho_0+z_m^-/n\leq X_{(k-1)/n},X_{k/n}\leq \rho_0\}}  >\frac{\epsilon}{L_g} \right)\\
&\hspace{1cm} \leq \frac{2L_g^2}{\epsilon^2}\E_{\rho_0}\left[\left|\sum_{u=1}^m\sum_{j=1}^9 \kappa_u^- Z_k^j(z_u^-/n,0)\1_{A_{kl}} + \sum_{u=1}^M\sum_{j=1}^9 \kappa_u^+ Z_k^j(0,z_u^+/n)\1_{A_{kl}}\right|^2\right]\\
&\hspace{2cm} + \frac{2L_g^2}{\epsilon^2}\E_{\rho_0}\left[\left(\sum_{u=l}^m\sum_{v=u}^m |\kappa_v^-|\1_{\{\rho_0+z_m^-/n\leq X_{(k-1)/n},X_{k/n}\leq \rho_0\}}\right)^2\right]\\
&\hspace{1cm} \leq \frac{CL_g^2}{\epsilon^2} \left(\sum_{u=1}^m \E_{\rho_0}\left[|Z_k^j(z_u^-/n,0)|^2\right] + \sum_{u=1}^M \E_{\rho_0}\left[ |Z_k^j(0,z_u^+/n)|^2\right] \right)\\
&\hspace{2cm} + \frac{CL_g^2}{\epsilon^2}\sum_{u=l}^m\sum_{v=u}^m \E_{\rho_0}\left[\1_{\{\rho_0+z_m^-/n\leq X_{(k-1)/n},X_{k/n}\leq \rho_0\}}\right]\\
&\hspace{1cm} \leq \frac{CL_g^2}{\epsilon^2} \frac{1}{n\sqrt{k}}
\end{align*}
such that with~\eqref{eq_proof:3char_4},
\begin{align}\label{eq_proof:3char_5}
\begin{split}
\E_{\rho_0}\left[|r_{1,n}|\right] &\leq \sum_{k=1}^{\lfloor nt\rfloor}\sum_{l=1}^m  \left( \frac{C\epsilon}{\sqrt{nk}} + \frac{CL_g^2}{\epsilon^2}\frac{1}{n\sqrt{k}}\right) \leq C\epsilon + \frac{CL_g^2}{\epsilon^2}\frac{\sqrt{nt}}{n}.
\end{split}
\end{align}
Note that $\limsup_{n\to\infty} \E_{\rho_0}[|r_{1,n}|]\leq C\epsilon$ and as $\epsilon>0$ was chosen arbitrary, this implies $r_{1,n}=o_{\Pr_{\rho_0}}(1)$. Next, by directly inserting the transition density for the corresponding regime,
\begin{align*}
&\sum_{k=1}^{\lfloor nt\rfloor} \E_{\rho_0}\left[\1_{A_{kl}}|\cF_{(k-1)/n}\right]\\
&\hspace{0.1cm}= \frac{2}{\alpha+\beta}\frac{\beta}{\alpha}\frac{|z_l^- - z_{l-1}^-|}{n \sqrt{2\pi/n}}\sum_{k=1}^{\lfloor nt\rfloor} \1_{\{X_{(k-1)/n}<\rho_0+z_m^-/n\}}\exp\left(-\frac{(X_{(k-1)/n}-\rho_0)^2}{2\alpha^2/n}\right) + o_{\Pr_{\rho_0}}(1)\\
 &\hspace{0.1cm}= \frac{2}{\alpha+\beta}\frac{\beta}{\alpha}\frac{|z_l^- - z_{l-1}^-|}{n \sqrt{2\pi/n}}\sum_{k=1}^{\lfloor nt\rfloor} \1_{\{X_{(k-1)/n}<\rho_0\}}\exp\left(-\frac{(X_{(k-1)/n}-\rho_0)^2}{2\alpha^2/n}\right) + o_{\Pr_{\rho_0}}(1),
\end{align*} 
where the first step follows by a Taylor expansion and is given in detail for the treatment of $S_{31}(k)$ within the verification of~\eqref{condition_2char_1} in Appendix~\ref{App:conditions_prop_stable} and the second one follows as
\[ \sum_{k=1}^{\lfloor nt\rfloor}  \E_{\rho_0}\left[ \left|\1_{\{X_{(k-1)/n}<\rho_0+z_m^-/n\}}-\1_{\{X_{(k-1)/n}<\rho_0\}}\right|\exp\left(-\frac{(X_{(k-1)/n}-\rho_0)^2}{2\alpha^2/n}\right)\right] \longrightarrow 0, \]
which is shown analogously to $r_n^{7,2}$ in the proof of Proposition~\ref{prop:expansion_of_drift_t}. Then, with~\eqref{eq_proof:3char_5},
\begin{align*}
T_1 &= \frac{2}{\alpha+\beta}\frac{\beta}{\alpha}\frac{1}{\sqrt{2\pi}}\frac{1}{\sqrt{n}}\sum_{k=1}^{\lfloor nt\rfloor} \1_{\{X_{(k-1)/n}<\rho_0\}}\exp\left(-\frac{(X_{(k-1)/n}-\rho_0)^2}{2\alpha^2/n}\right) \\
&\hspace{1cm} \cdot \sum_{l=1}^m  \tilde{g}\left( \left(\kappa_l^-+\dots + \kappa_m^-\right) \log\left(\frac{\alpha^2}{\beta^2}\right)\right)|z_l^- - z_{l-1}^-| +  o_{\Pr_{\rho_0}}(1).
\end{align*}\smallskip

\item[$\bullet\ \boldsymbol{T_2}.$] With $A_{kl}=\{\rho_0+z_l^-/n<X_{k/n}\leq \rho_0+z_{l-1}^-/n \leq \rho_0<X_{(k-1)/n} \}$ we obtain
\[ A_{kl}\cap I_{j,k}^{0,z_{u}^+/n} =\emptyset,\qquad A_{kl}\cap I_{2,k}^{z_v^-/n,0}=\emptyset,\qquad \textrm{ and } \qquad A_{kl} \cap I_{8,k}^{z_w^-/n,0} = \emptyset\]
for all $j=2,8$, $1\leq l,v\leq m$, $1\leq u\leq M$, $w\leq l-1$ and $1\leq k\leq n$. Then,
\begin{align}\label{eq_proof:3char_2}
\begin{split}
T_2 &= \sum_{k=1}^{\lfloor nt\rfloor}\sum_{l=1}^m \E_{\rho_0}\left[\left. \tilde{g}\left( y_{nk}\right)\1_{\{\rho_0+z_l^-/n<X_{k/n}\leq \rho_0+z_{l-1}^-/n \leq \rho_0\leq X_{(k-1)/n} \}} \right| \cF_{(k-1)/n}\right] \\
&= \sum_{k=1}^{\lfloor nt\rfloor}\sum_{l=1}^m \E_{\rho_0}\left[ \tilde{g}\left( \sum_{u=1}^m\sum_{j=1}^9 \kappa_u^- Z_k^j(z_u^-/n,0)\1_{A_{kl}} + \sum_{u=l}^m\kappa_u^-\log\left(\frac{\alpha^2}{\beta^2}\right)\1_{A_{kl}} \right.\right.\\
&\hspace{3.5cm} \left.\left.\left. + \sum_{u=1}^M\sum_{j=1}^9 \kappa_u^+ Z_k^j(0,z_u^+/n)\1_{A_{kl}}\right)\1_{A_{kl}}\right|\cF_{(k-1)/n} \right] \\
& = \sum_{k=1}^{\lfloor nt\rfloor}\sum_{l=1}^m \E_{\rho_0}\left[\left.  \tilde{g}\left( \left(\kappa_l^-+\dots + \kappa_m^-\right) \log\left(\frac{\alpha^2}{\beta^2}\right) \right) \1_{A_{kl}} \right| \cF_{(k-1)/n}\right] + r_{2,n},
\end{split}
\end{align} 
where for any $\epsilon>0$,
\begin{align*}
&\E_{\rho_0}\left[|r_{2,n}|\right] \leq \epsilon \sum_{k=1}^{\lfloor nt\rfloor}\sum_{l=1}^m \E_{\rho_0}\left[\1_{A_{kl}}\right] \\
&\hspace{0.5cm}  + \sum_{k=1}^{\lfloor nt\rfloor}\sum_{l=1}^m \Pr_{\rho_0}\left( \left|\sum_{u=1}^m\sum_{j=1}^9 \kappa_u^- Z_k^j(z_u^-/n,0)\1_{A_{kl}} + \sum_{u=1}^M\sum_{j=1}^9 \kappa_u^+ Z_k^j(0,z_u^+/n)\1_{A_{kl}}\right| >\frac{\epsilon}{L_g}\right).
\end{align*} 
By the same arguments used for $r_{1,n}$ in the treatment of $T_1$ before, we obtain
\begin{align}\label{eq_proof:3char_3}
\begin{split}
\E_{\rho_0}\left[|r_{2,n}|\right] &\leq C\epsilon + \frac{C L_g^2}{\epsilon^2}\frac{\sqrt{nt}}{n}
\end{split}
\end{align} 
and consequently $r_{2,n}=o_{\Pr_{\rho_0}}(1)$. Next, 
\begin{align*}
&\sum_{k=1}^{\lfloor nt\rfloor} \E_{\rho_0}\left[\1_{A_{kl}}|\cF_{(k-1)/n}\right]\\
&\hspace{0.5cm}= \frac{2}{\alpha+\beta}\frac{\beta}{\alpha}\frac{|z_l^- - z_{l-1}^-|}{n \sqrt{2\pi/n}}\sum_{k=1}^{\lfloor nt\rfloor} \1_{\{X_{(k-1)/n}\geq\rho_0\}}\exp\left(-\frac{(X_{(k-1)/n}-\rho_0)^2}{2\beta^2/n}\right) + o_{\Pr_{\rho_0}}(1), 
\end{align*} 
as was already proven in the treatment of $S_{32}$ in the proof of~\eqref{condition_2char_1}. Combining this with~\eqref{eq_proof:3char_2} and~\eqref{eq_proof:3char_3} then gives
\begin{align*}
T_2 &= \frac{2}{\alpha+\beta}\frac{\beta}{\alpha}\frac{1}{\sqrt{2\pi}}\frac{1}{\sqrt{n}}\sum_{k=1}^{\lfloor nt\rfloor} \1_{\{X_{(k-1)/n}\geq\rho_0\}}\exp\left(-\frac{(X_{(k-1)/n}-\rho_0)^2}{2\beta^2/n}\right)\\
&\hspace{2cm} \cdot  \sum_{l=1}^m \tilde{g}\left( \left(\kappa_l^-+\dots + \kappa_m^-\right) \log\left(\frac{\alpha^2}{\beta^2}\right) \right) |z_l^- - z_{l-1}^-|  + o_{\Pr_{\rho_0}}(1).
\end{align*}

\item[$\bullet\ \boldsymbol{T_3}.$] By similar arguments as used for $T_2$ we obtain
\begin{align*}
T_3 &= \frac{2}{\alpha+\beta}\frac{\alpha}{\beta}\frac{1}{\sqrt{2\pi}}\frac{1}{\sqrt{n}}\sum_{k=1}^{\lfloor nt\rfloor} \1_{\{X_{(k-1)/n}<\rho_0\}}\exp\left(-\frac{(X_{(k-1)/n}-\rho_0)^2}{2\alpha^2/n}\right) \\
&\hspace{1cm} \cdot \sum_{l=1}^M  \tilde{g}\left( \left(\kappa_l^+ +\dots + \kappa_M^+\right) \log\left(\frac{\beta^2}{\alpha^2}\right)\right)|z_l^+ - z_{l-1}^+| + o_{\Pr_{\rho_0}}(1).
\end{align*}

\item[$\bullet\ \boldsymbol{T_4}.$] By similar arguments as used for $T_1$ we obtain
\begin{align*}
T_4 &= \frac{2}{\alpha+\beta}\frac{\alpha}{\beta}\frac{1}{\sqrt{2\pi}}\frac{1}{\sqrt{n}}\sum_{k=1}^{\lfloor nt\rfloor} \1_{\{X_{(k-1)/n}\geq\rho_0\}}\exp\left(-\frac{(X_{(k-1)/n}-\rho_0)^2}{2\beta^2/n}\right) \\
&\hspace{1cm} \cdot \sum_{l=1}^M  \tilde{g}\left( \left(\kappa_l^+ +\dots + \kappa_M^+\right) \log\left(\frac{\beta^2}{\alpha^2}\right)\right)|z_l^+ - z_{l-1}^+| + o_{\Pr_{\rho_0}}(1).
\end{align*}
\end{itemize}
Recalling the definition~\eqref{eq:Lambda_nt(X)} for $\Lambda_{\alpha,\beta}^n$, summing up our expansions and then applying Lemma~\ref{lemma:Riemann_approximation} finally gives with~\eqref{eq_proof:T_1-T4}
\begin{align*}
&\sum_{k=1}^{\lfloor nt\rfloor} \E_{\rho_0}\left[ \tilde{g}(y_{nk}) \1_{J_k}| \cF_{(k-1)/n}\right]  \\
&\hspace{0.2cm} = \frac{2}{\alpha+\beta}\frac{\beta}{\alpha}\sum_{l=1}^m \tilde{g}\left( \left(\kappa_l^-+\dots + \kappa_m^-\right) \log\left(\frac{\alpha^2}{\beta^2}\right) \right) |z_l^- - z_{l-1}^-| \frac1n \Lambda_{\alpha,\beta}^n\left((X_{(k-1)/n})_{1\leq k\leq \lfloor nt\rfloor}\right)\\
&\hspace{0.5cm} +  \frac{2}{\alpha+\beta}\frac{\alpha}{\beta}\sum_{l=1}^M \tilde{g}\left( \left(\kappa_l^+ +\dots + \kappa_M^+\right) \log\left(\frac{\beta^2}{\alpha^2}\right) \right) |z_l^+ - z_{l-1}^+| \frac1n \Lambda_{\alpha,\beta}^n\left((X_{(k-1)/n})_{1\leq k\leq \lfloor nt\rfloor}\right) \\
&\hspace{0.5cm}+ o_{\Pr_{\rho_0}}(1) \\
&\hspace{0.2cm} \longrightarrow_{\Pr_{\rho_0}} L_t^{\rho_0}(X) \frac{1}{\alpha^2 }\sum_{l=1}^m \tilde{g}\left( \left(\kappa_l^-+\dots + \kappa_m^-\right) \log\left(\frac{\alpha^2}{\beta^2}\right) \right) |z_l^- - z_{l-1}^-| \\
&\hspace{0.5cm} +  L_t^{\rho_0}(X) \frac{1}{\beta^2}\sum_{l=1}^M \tilde{g}\left( \left(\kappa_l^+ +\dots + \kappa_M^+\right)\log\left(\frac{\beta^2}{\alpha^2}\right) \right) |z_l^+ - z_{l-1}^+|. 
\end{align*}


%
%

\begin{acks}[Acknowledgments]
The authors would like to thank two anonymous referees for their valuable comments and suggestions improving the representation of the article and Nils Kober for his help producing Figure~\ref{figure:route_of_proof}.
\end{acks}
\begin{funding}
This work was supported by the DFG Research Unit $5381$, RO $3766/8$-$1$ and CRC $1597$, Project-ID $499552394$.
\end{funding}


\bibliographystyle{imsart-nameyear} 
\bibliography{Bibliography}       

\newpage
\setcounter{section}{0}
\setcounter{page}{1}
\renewcommand*\thesection{\Alph{section}}

\begin{frontmatter}

\title{\vspace{-0.3cm} Supplement to "The level of self-organized criticality in oscillating Brownian motion: $\MakeLowercase{n}$-consistency and stable Poisson-type convergence of the MLE"}
\runtitle{Supplement to "Stable limit theory for the MLE of OBM"}


\begin{aug}
\author[A]{\fnms{Johannes}~\snm{Brutsche}}
\and
\author[A]{\fnms{Angelika}~\snm{Rohde}}
\address[A]{Mathematical Institute, University of Freiburg\printead[presep={,\ }]{e1}\printead[presep={,\ }]{e3}}
\end{aug}

\runauthor{J. Brutsche and A. Rohde}


\vspace{-0.5cm}
\begin{keyword}[class=MSC]
\kwd[Primary ]{60F05}\kwd{62M05}
\kwd[; secondary ]{62F12}\kwd{62E20}
\end{keyword}

\begin{keyword}
\kwd{Stable Poisson convergence}
\kwd{infill asymptotics}
\kwd{$n$-consistency}
\kwd{MLE}
\end{keyword}

\end{frontmatter}

\vspace{-0.1cm}
This supplementary material is organized as follows:\\
\vspace{-0.2cm}

{\small
\tableofcontents
\addtocontents{toc}{\protect\setcounter{tocdepth}{3}} 
}

\section{Notation}\label{App:notation}
Throughout the whole technical supplement, $C_{\alpha,\beta}$ denotes a real and positive constant that only depends on $\alpha$ and $\beta$ but may change from line to line. Any other dependencies are highlighed explicitly, for example by writing $C_{\alpha,\beta}(K)$ in case the constant additionally depends on $K$.
We frequently use the notation
\[ \Pr( A,B ) := \Pr(A\cap B).\]

We denote the four regimes of the transition density~\eqref{eq:transition_density} with $P_i^\rho(x,y;t)$, $i=1,2,3,4$, i.e. 
\begin{align*}
P_1^\rho(x,y;t) &:= \frac{1}{\sqrt{2\pi t}\alpha}\left[ \exp\left( -\frac{(y-x)^2}{2t\alpha^2}\right) - \frac{\alpha-\beta}{\alpha+\beta}\exp\left(-\frac{(y-2\rho+x)^2}{2t\alpha^2}\right)\right], \\
P_2^\rho(x,y;t) &:= \frac{1}{\sqrt{2\pi t}\beta}\left[ \exp\left( -\frac{(y-x)^2}{2t\beta^2}\right) + \frac{\alpha-\beta}{\alpha+\beta}\exp\left(-\frac{(y-2\rho+x)^2}{2t\beta^2}\right)\right], \\
P_3^\rho(x,y;t) &:= \frac{2}{\alpha+\beta}\frac{\alpha}{\beta}\frac{1}{\sqrt{2\pi t}}\exp\left( -\frac{1}{2t}\left(\frac{y-\rho}{\beta} - \frac{x-\rho}{\alpha}\right)^2\right), \\
P_4^\rho(x,y;t) &:= \frac{2}{\alpha+\beta}\frac{\beta}{\alpha}\frac{1}{\sqrt{2\pi t}}\exp\left( -\frac{1}{2t}\left(\frac{y-\rho}{\alpha} - \frac{x-\rho}{\beta}\right)^2\right).
\end{align*}
Note that in constrast to $p_t^\rho(x,y)$, each $P_i^\rho(x,y;t)$ is continuous (and differentiable) in the parameter $\rho$.

\section{Preliminary results}

\begin{lemma}\label{lemma:upper_bound_transition_density}
The transition density \eqref{eq:transition_density} is dominated by a Gaussian density, i.e.
\[ p_t^\rho(x,y)\leq \frac{2}{\alpha+\beta}\frac{1}{\sqrt{2\pi t}}\frac{\max\{\alpha,\beta\}}{\min\{\alpha,\beta\}} \exp\left( - \frac{(y-x)^2}{2t\max\{\alpha^2,\beta^2\}}\right). \]
Moreover, 
\[p_t^\rho(x,y) \geq \frac{2}{\alpha+\beta}\frac{1}{\sqrt{2\pi t}}\frac{\min\{\alpha,\beta\}}{\max\{\alpha,\beta\}} \exp\left( - \frac{(y-x)^2}{2t\min\{\alpha^2,\beta^2\}}\right). \]
\end{lemma}
\begin{proof}
This is straightforward and done by a case-by-case study for the four regimes.
\end{proof}

\begin{cor}\label{cor:bound_exp_X^2}
Let $a\in\R$, $b\in\R_{>0}$ and $l<k$. Then there exists a constant $c_{\alpha,\beta}>0$ such that
\begin{align*}
&\E_{\rho_0}\left[\left. \exp\left( - \frac{(X_{k/n} - a)^2}{2b}\right)\right| X_{l/n}\right] \\
&\hspace{1cm} \leq c_{\alpha,\beta} \sqrt{\frac{b}{b+(k-l)\max\{\alpha^2,\beta^2\}/n}} \exp\left(-\frac12 \frac{(a-x_0)^2}{b+(k-l)\max\{\alpha^2,\beta^2\}/n}\right).
\end{align*}
In particular, we have
\[ \E_{\rho_0}\left[\left. \exp\left( - \frac{(X_{k/n} - a)^2}{2b/n}\right)\right| X_{l/n}\right] \leq c_{\alpha,\beta}(b)\frac{1}{\sqrt{k-l}}\]
for some constant $c_{\alpha,\beta}(b)>0$ not depending on $n$ and $a$.
\end{cor}
\begin{proof}
By Lemma~\ref{lemma:upper_bound_transition_density}, the left-hand side of the statement is bounded from above by
\begin{align*}
\frac{2}{\alpha+\beta}\frac{1}{\sqrt{2\pi (k-l)/n}}\frac{\max\{\alpha,\beta\}}{\min\{\alpha,\beta\}} \int_\R  \exp\left( - \frac{(z - a)^2}{2b}\right)\exp\left( - \frac{(z-x_0)^2}{2(k-l)\max\{\alpha^2,\beta^2\}/n}\right) dz.
\end{align*}
Set $d:=(k-l)\max\{\alpha^2,\beta^2\}/n$. For the integral, we find by completing the square,
\begin{align*}
&\int_\R  \exp\left( - \frac{(z - a)^2}{2b}\right)\exp\left( - \frac{(z-x_0)^2}{2d}\right) dz\\
&\hspace{1cm} = \exp\left( -\frac{a^2d +x_0^2 b}{2bd} + \left( \frac{ad+x_0 b}{b+d}\right)^2\frac{b+d}{2bd}\right) \\
&\hspace{3cm} \cdot \int_\R \exp\left( - \frac{b+d}{2bd}\left( z- \frac{ad+x_0 b}{b+d}\right)^2 \right) dz \\
&\hspace{1cm} = \exp\left( -\frac12 \frac{(a-x_0)^2}{b+d}\right) \sqrt{2\pi \frac{bd}{b+d}}.
\end{align*}
Then, the desired upper bound follows for the constant
\[ c_{\alpha,\beta} = \frac{2}{\alpha+\beta}\frac{\max\{\alpha^2,\beta^2\}}{\min\{\alpha,\beta\}}.\]
\end{proof}

Corollary~\ref{cor:bound_exp_X^2} is often used in combination with the inequalities
\[ \sum_{k=1}^n \frac{1}{\sqrt{k}} \leq 2\sqrt{n}, \quad \textrm{ and }\quad \sum_{k=1}^n \sum_{l=1}^{k-1} \frac{1}{\sqrt{k(k-l)}}\leq 4n.\]
Both of them are easily obtained by comparison with corresponding integral and applied without further notice.\\


Within the next lemma and its proof, we use the notation $a\wedge b := \min\{a,b\}$ for real numbers $a,b\in\R$.

\begin{lemma}\label{lemma:second_moment_sum_intervals}
Let $f_n:\R^2\longrightarrow\R$ be functions such that for every $c>0$ we have the bound $|f_n(x,y)|^m \exp(-n(x-y)^2/c)\leq C(m,c)$ for some constant $C=C(m,c)>0$, where $m\in\{1,2\}$ and $C$ being independent of $n$. Moreover, let $J_1,J_2\subset\R$ be two intervals (possibly infinite). Then for some constant $C(\alpha,\beta,c)>0$ that depends only on $\alpha,\beta$ and $c$,
\begin{align*}
&\E_{\rho_0}\left[ \left(\sum_{k=1}^n f_n(X_{(k-1)/n},X_{k/n})\1_{J_1}(X_{(k-1)/n})\1_{J_2}(X_{k/n}) \right)^2\right]\\
&\hspace{1.5cm} \leq C(\alpha,\beta,c)n\left( 1\wedge \lambda(J_1)\wedge\lambda(J_2)\right) + C(\alpha,\beta,c)n^2\left(1\wedge \lambda(J_1)^2\wedge \lambda(J_2)^2\right)
\end{align*} 
and
\begin{align*}
&\E_{\rho_0}\left[ \left(\sum_{k=1}^n \E_{\rho_0}\left[f_n(X_{(k-1)/n},X_{k/n})\1_{J_1}(X_{(k-1)/n})\1_{J_2}(X_{k/n})\mid X_{(k-1)/n}\right] \right)^2\right] \\
&\hspace{1.5cm}\leq C(\alpha,\beta,c)n\left( 1\wedge \lambda(J_1)\wedge\lambda(J_2)\right) + C(\alpha,\beta,c) n^2\left(1\wedge \lambda(J_1)^2\wedge \lambda(J_2)^2\right).
\end{align*} 
\end{lemma}
\begin{proof}
We will repeatedly use that for $l<k$ and $i=1,2$ by Lemma~\ref{lemma:upper_bound_transition_density}
\begin{align}\label{eq_proof:expectation_indicator}
\begin{split}
\E_{\rho_0}\left[ \1_{J_i}(X_{k/n})\mid X_{l/n}\right] &\leq C_{\alpha,\beta}\int_{J_i}\frac{1}{\sqrt{(k-l)/n}} \exp\left(-\frac{(y-X_{l/n})^2}{2\max\{\alpha^2,\beta\}(k-l)/n}\right) dy \\
&\leq C_{\alpha,\beta}\left( 1\wedge \frac{\lambda(J_i)}{\sqrt{(k-l)/n}}\right).
\end{split}
\end{align}
Moreover, for every interval $J_i$, there exist constants $c_1(J_i), c_2(J_i)$ (the endpoints of the interval, possibly $\pm\infty$) such that for any $d>0$ and with the shorthand notation
\[ \nu(J_i, X_{(k-1)/n}) := \min\{(X_{(k-1)/n}-c_1(J_i))^2,(X_{(k-1)/n}-c_2(J_i))^2\}\]
we have
\begin{align}\label{eq_proof:bound_integral_J_i}
\begin{split}
&\int_{J_i} \sqrt{n} \exp\left(-\frac{(y-X_{(k-1)/n})^2}{d/n}\right) dy  \\
&\hspace{0.2cm}\leq \sqrt{\pi d}\wedge\left(\1_{J_i}(X_{(k-1)/n})\sqrt{\pi d}+\1_{J_i^c}(X_{(k-1)/n})\lambda(J_i)\sqrt{n}\exp\left(-\frac{\nu(J_i,X_{(k-1)/n})}{d/n}\right)\right).
\end{split}
\end{align}
By Corollary~\ref{cor:bound_exp_X^2}, we find for $l<k$
\begin{align}\label{eq_proof:bound_exp_nu}
\begin{split}
&\E_{\rho_0}\left[\left. \exp\left(-\frac{\nu(J_i,X_{(k-1)/n}}{d/n}\right)\right| X_{l/n}\right] \\
&\hspace{0.2cm}\leq \E_{\rho_0}\left[\left. \exp\left(-\frac{(X_{(k-1)/n}-c_1(J_i))^2}{d/n}\right)\right| X_{l/n}\right] + \E_{\rho_0}\left[\left.\exp\left(-\frac{(X_{(k-1)/n}-c_2(J_i))^2}{d/n}\right)\right| X_{l/n}\right] \\
&\hspace{0.2cm} \leq C_{\alpha,\beta} \sqrt{\frac{d/n}{d/n+(k-l)/n}} \left( \exp\left(-\frac{(c_1(J_i)-X_{l/n})^2}{d/n}\right) + \exp\left(-\frac{(c_2(J_i)-X_{l/n})^2}{d/n}\right) \right) \\
&\hspace{0.2cm} \leq C_{\alpha,\beta}(d) \frac{1}{\sqrt{k-l}},
\end{split}
\end{align}
where it is important to note that the final constant does not depend on $J_i$. We now start with the first assertion of the lemma. Here, the squared terms are bounded with Lemma~\ref{lemma:upper_bound_transition_density} and the inequalities~\eqref{eq_proof:expectation_indicator}, \eqref{eq_proof:bound_integral_J_i} and~\eqref{eq_proof:bound_exp_nu} by
\begin{align}\label{eq_proof:estimation_squared_term}
\begin{split}
&\E_{\rho_0}\left[ f_n(X_{(k-1)/n},X_{k/n})^2\1_{J_1}(X_{(k-1)/n})\1_{J_2}(X_{k/n}) \right] \\
&\hspace{0.2cm} \leq C_{\alpha,\beta}\E_{\rho_0}\left[\1_{J_1}(X_{(k-1)/n}) \int_{J_2} f_n(X_{(k-1)/n},y)^2 \sqrt{n} \exp\left(-\frac{(y-X_{(k-1)/n})^2}{2\max\{\alpha^2,\beta^2\}/n}\right) dy \right] \\
&\hspace{0.2cm} \leq C_{\alpha,\beta}(c)\E_{\rho_0}\left[\1_{J_1}(X_{(k-1)/n}) \int_{J_2} \sqrt{n} \exp\left(-\frac{(y-X_{(k-1)/n})^2}{4\max\{\alpha^2,\beta^2\}/n}\right) dy \right] \\
&\hspace{0.2cm} \leq C_{\alpha,\beta}(c)\E_{\rho_0}\Bigg[\1_{J_1}(X_{(k-1)/n}) \Big( 1\wedge \Big(\1_{J_2}(X_{(k-1)/n}) \\
&\hspace{3cm} \left.\left.\left.+ \1_{J_2^c}(X_{(k-1)/n}) \lambda(J_2) \sqrt{n}\exp\left(-\frac{\nu(J_2,X_{(k-1)/n})}{4\max\{\alpha^2,\beta^2\}/n}\right)\right)\right) \right] \\
&\hspace{0.2cm} \leq C_{\alpha,\beta}(c)\Bigg( \E_{\rho_0}\left[\1_{J_1}(X_{(k-1)/n})\right] \wedge \E_{\rho_0}\Big[\1_{J_1\cap J_2}(X_{(k-1)/n}) \\
&\hspace{3cm} \left.\left. + \1_{J_1\cap J_2^c}(X_{(k-1)/n})\lambda(J_2)\sqrt{n}\exp\left(-\frac{\nu(J_2,X_{(k-1)/n})}{4\max\{\alpha^2,\beta^2\}/n}\right)\right] \right) \\
&\hspace{0.2cm} \leq C_{\alpha,\beta}(c)\Bigg( \E_{\rho_0}\left[\1_{J_1}(X_{(k-1)/n})\right] \wedge \Big(\E_{\rho_0}\left[\1_{J_2}(X_{(k-1)/n})\right] \\
&\hspace{3cm} \left.\left. + \lambda(J_2)\sqrt{n}\E_{\rho_0}\left[\exp\left(-\frac{\nu(J_2,X_{(k-1)/n})}{4\max\{\alpha^2,\beta^2\}/n}\right)\right] \right)\right) \\
&\hspace{0.2cm} \leq C_{\alpha,\beta}(c)\left( 1\wedge \frac{\lambda(J_1)}{\sqrt{k/n}} \wedge \frac{\lambda(J_2)}{\sqrt{k/n}}\right).
\end{split}
\end{align}
Using the Cauchy-Schwarz inequality, 
\begin{align*}
&\E_{\rho_0}\left[ f_n(X_{(k-2)/n},X_{(k-1)/n})\1_{J_1}(X_{(k-2)/n})\1_{J_2}(X_{(k-1)/n})f_n(X_{(k-1)/n},X_{k/n})\1_{J_1}(X_{(k-1)/n})\1_{J_2}(X_{k/n}) \right]\\
&\hspace{1cm} \leq C_{\alpha,\beta}(c)\left( 1\wedge \frac{\lambda(J_1)}{\sqrt{k/n}} \wedge \frac{\lambda(J_2)}{\sqrt{k/n}}\right).
\end{align*}
For the cross terms with $l<k-1$, we find by iterative conditioning, again Lemma~\ref{lemma:upper_bound_transition_density} as well as the inequalities~\eqref{eq_proof:expectation_indicator}, \eqref{eq_proof:bound_integral_J_i} and~\eqref{eq_proof:bound_exp_nu},
\begin{align*}
&\E_{\rho_0}\left[ f_n(X_{(l-1)/n},X_{l/n})\1_{J_1}(X_{(l-1)/n})\1_{J_2}(X_{l/n})f_n(X_{(k-1)/n},X_{k/n})\1_{J_1}(X_{(k-1)/n})\1_{J_2}(X_{k/n}) \right]\\
&\hspace{0.2cm} = \E_{\rho_0}\Big[ f_n(X_{(l-1)/n},X_{l/n})\1_{J_1}(X_{(l-1)/n})\1_{J_2}(X_{l/n})\1_{J_1}(X_{(k-1)/n})\\
&\hspace{3.2cm} \cdot \E_{\rho_0}\left[f_n(X_{(k-1)/n},X_{k/n})\1_{J_2}(X_{k/n})\mid \cF_{(k-1)/n}\right] \Big]\\
&\hspace{0.2cm} \leq C_{\alpha,\beta}\E_{\rho_0}\left[ f_n(X_{(l-1)/n},X_{l/n})\1_{J_1}(X_{(l-1)/n})\1_{J_2}(X_{l/n})\1_{J_1}(X_{(k-1)/n}) \right. \\
&\hspace{3cm} \left.\cdot\int_{J_2} f_n(X_{(k-1)/n},y) \sqrt{n} \exp\left(-\frac{(y-X_{(k-1)/n})^2}{2\max\{\alpha^2,\beta^2\}/n}\right) dy \right]\\
&\hspace{0.2cm} \leq C_{\alpha,\beta}(c)\E_{\rho_0}\Big[ f_n(X_{(l-1)/n},X_{l/n})\1_{J_1}(X_{(l-1)/n})\1_{J_2}(X_{l/n})\1_{J_1}(X_{(k-1)/n})  \\
&\hspace{1cm} \cdot\left.  \left( 1\wedge \left(\1_{J_2}(X_{(k-1)/n}) + \1_{J_2^c}(X_{(k-1)/n})\lambda(J_2)\sqrt{n} \exp\left(-\frac{\nu(J_2,X_{(k-1)/n})}{4\max\{\alpha^2,\beta^2\}/n}\right)\right)\right)\right] \\
&\hspace{0.2cm} \leq C_{\alpha,\beta}(c)\E_{\rho_0}\Big[ f_n(X_{(l-1)/n},X_{l/n})\1_{J_1}(X_{(l-1)/n})\1_{J_2}(X_{l/n}) \Big( \E_{\rho_0}\left[\1_{J_1}(X_{(k-1)/n})\mid\cF_{l/n}\right]  \\
&\hspace{0.5cm} \left.\left. \wedge \left( \E_{\rho_0}\left[\1_{J_2}(X_{(k-1)/n})\mid \cF_{l/n}\right] + \lambda(J_2)\sqrt{n} \E_{\rho_0}\left[\left. \exp\left(-\frac{\nu(J_2,X_{(k-1)/n})}{4\max\{\alpha^2,\beta^2\}/n}\right)\right| \cF_{l/n}\right]\right)\right)\right] \\
&\hspace{0.2cm} \leq  C_{\alpha,\beta}(c)\left( 1\wedge \frac{\lambda(J_1)}{\sqrt{(k-l)/n}} \wedge \frac{\lambda(J_2)}{\sqrt{(k-l)/n}}\right)\E_{\rho_0}\Big[ f_n(X_{(l-1)/n},X_{l/n})\1_{J_1}(X_{(l-1)/n})\1_{J_2}(X_{l/n}) \Big] \\
&\hspace{0.2cm} \leq  C_{\alpha,\beta}(c)\left( 1\wedge \frac{\lambda(J_1)}{\sqrt{(k-l)/n}} \wedge \frac{\lambda(J_2)}{\sqrt{(k-l)/n}}\right)\left( 1\wedge \frac{\lambda(J_1)}{\sqrt{l/n}} \wedge \frac{\lambda(J_2)}{\sqrt{l/n}}\right) \\
&\hspace{0.2cm} = C_{\alpha,\beta}(c)\left( 1 \wedge \frac{\lambda(J_1)\wedge\lambda(J_2)}{\sqrt{(k-l)/n}}\wedge \frac{\lambda(J_1)\wedge\lambda(J_2)}{\sqrt{l/n}} \wedge\frac{\lambda(J_1)^2\wedge\lambda(J_2)^2\wedge\lambda(J_1)(\lambda(J_2)}{\sqrt{(k-l)l}/n}\right),
\end{align*}
where the last inequality follows as in~\eqref{eq_proof:estimation_squared_term}. In particular, we obtain
\begin{align*}
&\E_{\rho_0}\left[ f_n(X_{(l-1)/n},X_{l/n})\1_{J_1}(X_{(l-1)/n})\1_{J_2}(X_{l/n})f_n(X_{(k-1)/n},X_{k/n})\1_{J_1}(X_{(k-1)/n})\1_{J_2}(X_{k/n}) \right]\\
&\hspace{1cm} \leq C_{\alpha,\beta}(c)\left( 1 \wedge\frac{\lambda(J_1)^2\wedge\lambda(J_2)^2}{\sqrt{(k-l)l}/n}\right) 
\end{align*}
and thus,
\begin{align*}
&\E_{\rho_0}\left[ \left(\sum_{k=1}^n f_n(X_{(k-1)/n},X_{k/n})\1_{J_1}(X_{(k-1)/n})\1_{J_2}(X_{k/n}) \right)^2\right] \\
&\hspace{0.2cm} \leq 3C_{\alpha,\beta}(c) \sum_{k=1}^n \left( 1\wedge \frac{\lambda(J_1)}{\sqrt{k/n}} \wedge \frac{\lambda(J_2)}{\sqrt{k/n}}\right) + 2C_{\alpha,\beta}(c)\sum_{k=1}^n \sum_{l=1}^{k-2} \left( 1 \wedge\frac{\lambda(J_1)^2\wedge\lambda(J_2)^2}{\sqrt{(k-l)l}/n}\right) \\
&\hspace{0.2cm} \leq C_{\alpha,\beta}(c) \left( n\wedge \left( \lambda(J_1)\wedge \lambda(J_2)\right)\sqrt{n}\sum_{k=1}^n \frac{1}{\sqrt{k}}\right) \\
&\hspace{2cm} + C_{\alpha,\beta}(c) \left( n^2\wedge \left( \lambda(J_1)^2\wedge \lambda(J_2)^2\right)n \sum_{k=1}^n \sum_{l=1}^{k-2} \frac{1}{\sqrt{(k-l)l}}\right) \\
&\hspace{0.2cm} \leq C_{\alpha,\beta}(c)n \left( 1\wedge \lambda(J_1)\wedge \lambda(J_2)\right) + C_{\alpha,\beta}(c)n^2 \left(1\wedge \lambda(J_1)^2\wedge \lambda(J_2)^2\right).
\end{align*}
This gives the first assertion of the lemma. The second one is proven similarly. Here, we have for the squared terms using Jensen's inequality and~\eqref{eq_proof:estimation_squared_term}
\begin{align*}
&\E_{\rho_0}\left[ \E_{\rho_0}\left[f_n(X_{(k-1)/n},X_{k/n})\1_{J_1}(X_{(k-1)/n})\1_{J_2}(X_{k/n})\mid X_{(k-1)/n}\right]^2 \right] \\
&\hspace{0.2cm} \leq \E_{\rho_0}\left[ f_n(X_{(k-1)/n},X_{k/n})^2\1_{J_1}(X_{(k-1)/n})\1_{J_2}(X_{k/n}) \right] \leq  C_{\alpha,\beta}(c)\left( 1\wedge \frac{\lambda(J_1)}{\sqrt{k/n}} \wedge \frac{\lambda(J_2)}{\sqrt{k/n}}\right).
\end{align*}
Again, the cross terms for $l=k-1$ give the same bound by Cauchy-Schwarz' inequality. Up to slight modification, the case $l<k-1$ works analogously as above using Lemma~\ref{lemma:upper_bound_transition_density} and the inequalities~\eqref{eq_proof:expectation_indicator}, \eqref{eq_proof:bound_integral_J_i} and~\eqref{eq_proof:bound_exp_nu}.
\end{proof}

\begin{lemma}\label{lemma:bound_log(C+exp)}
Let $C>0$ and $(a_n)_{n\in\N}$ be a sequence with $a_n\geq 0$ for all $n\in\N$. Then we have
\[\left| \log\left(C+ (1-C)\exp(-a_n)\right)\right|\leq a_n\max\left\{ |1-C|, \frac{|1-C|}{1\wedge C}\right\}.   \]
\end{lemma}
\begin{proof}
By the basic inequality $x/(1+x)\leq \log(1+x)\leq x$ for all $x>-1$ we find
\[ \left|\log\left(C+ (1-C)\exp(-a_n)\right)\right| \leq \max\left\{|1-C|\left|\exp(-a_n)-1\right|, \frac{|1-C|\left|\exp(-a_n)-1\right|}{|C+(1-C)\exp(-a_n)|}\right\}.  \]
Furthermore, $|C+(1-C)\exp(-a_n)|\geq 1\wedge C$ and because $1-x\leq e^{-x}$ for $x\geq 0$, we have $|\exp(-a_n)-1|\leq a_n$. This gives the desired result.
\end{proof}

\section{Local time estimator for OBM}\label{Appendix:Statistical_Application}

In Theorem~$2$ in~\citeSM{App:Mazzonetto} it is shown that the local time estimator $\hat{L}_n^\rho$ given by
\begin{align}\label{eq:hat_L}
\hat{L}_n^{\rho} := \frac{\alpha+\beta}{2}\sqrt{\frac{\pi}{2n}} \sum_{k=1}^n \1_{\{(X_{(k-1)/n}-\rho)(X_{k/n}-\rho) <0\}}
\end{align} 
is a consistent estimator of the local time $L_1^\rho(X)$. From this, it follows (details are given below) that for any $K>0$,
\begin{align}\label{eq:uniform_local_time_estimator}
\sup_{\rho\in U_{K/n}(\rho_0)} \left| \hat{L}_n^\rho - L_1^{\rho_0}(X)\right| \longrightarrow_{\Pr_{\rho_0}} 0.
\end{align} 
Then, Proposition~\ref{prop:n-consistency} reveals that the local time estimator evaluated in the MLE $\hat{\rho}_n$ is again a consistent estimator for $L_1^{\rho_0}(X)$, i.e.
\begin{align*}
\hat{L}_n^{\hat{\rho}_n} \longrightarrow_{\Pr_{\rho_0}} L_1^{\rho_0}(X).
\end{align*}

\begin{proof}[Proof of~\eqref{eq:uniform_local_time_estimator}]
First, we estimate
\[ \sup_{\rho\in U_{K/n}(\rho_0)} \left| \hat{L}_n^\rho - L_1^{\rho_0}(X)\right| \leq \sup_{\rho\in U_{K/n}(\rho_0)} \left| \hat{L}_n^\rho - \hat{L}_n^{\rho_0} \right| + \left| \hat{L}_n^{\rho_0} - L_1^{\rho_0}(X)\right|,\]
where the second term converges to zero in probability by Theorem~$2$ in~\citeSM{App:Mazzonetto}. For the other one, we have by definition of $\hat{L}_n^{\rho}$ that for $\rho\leq \rho_0$,
\begin{align*}
\left| \hat{L}_n^\rho - \hat{L}_n^{\rho_0} \right| &= \frac{1}{\sqrt{n}} \left| \sum_{k=1}^n \1_{\{(X_{(k-1)/n}-\rho)(X_{k/n}-\rho) <0\}}-\1_{\{(X_{(k-1)/n}-\rho_0)(X_{k/n}-\rho_0) <0\}} \right| \\
&= \frac{1}{\sqrt{n}}\sum_{k=1}^n \left( \1_{\{X_{(k-1)/n}< \rho, \rho<X_{k/n}\leq \rho_0\}} + \1_{\{\rho<X_{(k-1)/n}<\rho_0, X_{k/n}<\rho\}} \right. \\
&\hspace{2.5cm} \left. +\1_{\{\rho<X_{(k-1)/n}<\rho_0, X_{k/n}>\rho_0\}}+\1_{\{X_{(k-1)/n}>\rho_0, \rho<X_{k/n}<\rho_0\}} \right)
\end{align*}
Hence,
\begin{align*}
\sup_{\rho_0-K/n\leq \rho\leq \rho_0} \left| \hat{L}_n^\rho - \hat{L}_n^{\rho_0} \right| &\leq \frac{2}{\sqrt{n}}\sum_{k=1}^n \left( \1_{\{\rho_0-K/n<X_{k/n}\leq \rho_0\}} + \1_{\{\rho_0-K/n<X_{(k-1)/n}<\rho_0\}} \right)
\end{align*}
By Lemma~\ref{lemma:second_moment_sum_intervals},
\[ \E_{\rho_0}\left[\left(\frac{2}{\sqrt{n}}\sum_{k=1}^n \left( \1_{\{\rho_0-K/n<X_{k/n}\leq \rho_0\}} + \1_{\{\rho_0-K/n<X_{(k-1)/n}<\rho_0\}} \right)\right)^2\right] \leq C_{\alpha,\beta}(K)\frac1n \]
and thus,
\[ \sup_{\rho_0-K/n\leq \rho\leq \rho_0} \left| \hat{L}_n^\rho - \hat{L}_n^{\rho_0} \right| \longrightarrow_{\Pr_{\rho_0}} 0.\]
The same argument works for $\sup_{\rho_0\leq \rho\leq \rho_0+K/n} | \hat{L}_n^\rho - \hat{L}_n^{\rho_0} |$ and~\eqref{eq:uniform_local_time_estimator} follows.
\end{proof}

\section{Simulation study}

This section contains a simulation of the distribution of the limit $\mathrm{argsup}_{z\in\R}\ell(z)$ together with a study on coverage of the asymptotic confidence set~\eqref{eq:confidence_interval}. In all simulations, we used the smallest member of $\mathrm{Argsup}_{z\in\R}f(z)$ in case of a non-unique maximizer. Moreover, we use the family 
 \[ \big(\hat{L}_n^\rho\big)_{\rho\in\R}\] given in~\eqref{eq:hat_L} and introduced in~\citeSM{App:Mazzonetto} to estimate the local time.

In Figure~\ref{figure:histogram}, the Lebesgue density of $\mathrm{argsup}_{z\in\R}\ell(z)$ is visualized together with the histogram density estimator of 
\[ n\hat{L}_n^{\hat{\rho}_n}(\hat{\rho}_n-\rho). \]
One can observe that this distribution is not symmetric around zero. It is skewed to the left in case $\alpha<\beta$ and skewed ot the right if $\alpha>\beta$. This is reasonable in view of the definition of $\ell(z)$ given in~\eqref{eq:def_ell}.
Moreover, we can observe that the variance of $\mathrm{argsup}_{z\in\R}\ell(z)$ decreases the larger $|\alpha-\beta|$. This is reasonable as in this case $|b_{\alpha,\beta}|$ and $|b_{\alpha,\beta}'|$ increase, corresponding to a steeper (negative) drift for the function $\ell(z)$. This confirms the heuristics that estimation of $\rho_0$ is easier the more $\alpha$ and $\beta$ differ from each other.

\begin{figure}[h]
\centering
\includegraphics[scale=0.285]{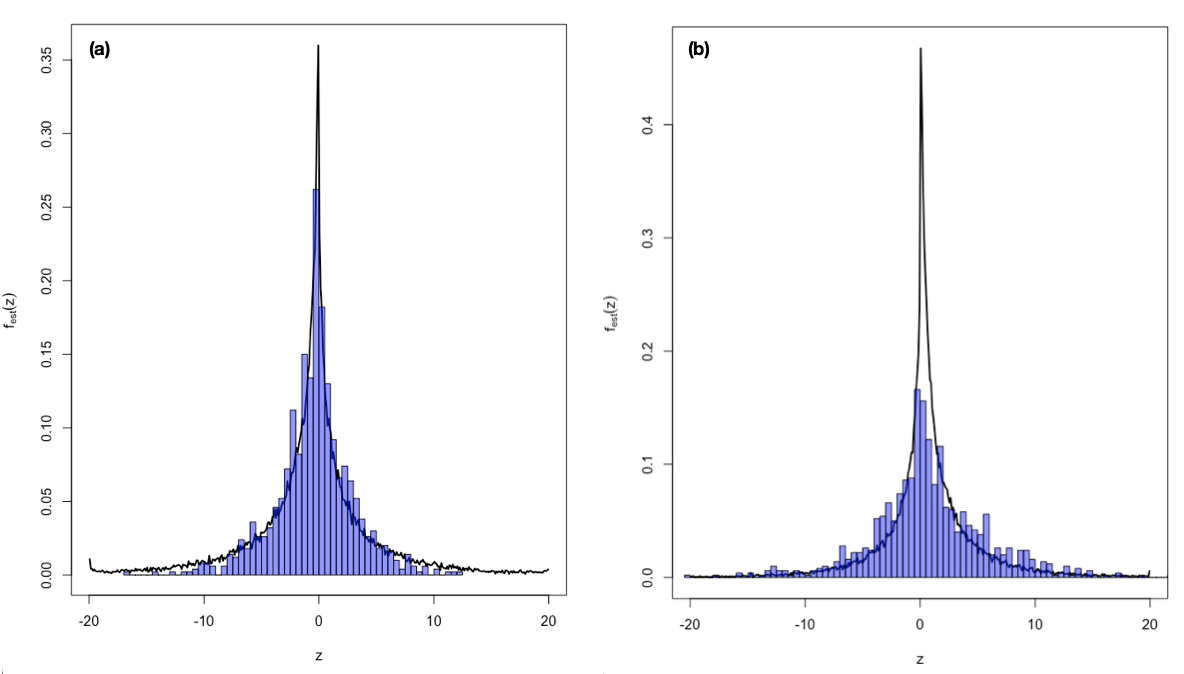}
\caption{\small The black solid line visualizes the Lebesgue density $f_{\mathrm{est}}$ of $\mathrm{argsup}_{z\in\R}\ell(z)$ for $n=1,000$ based on $M=50,000$ samples. The histogram density estimator is based on $N=1,000$ samples of {\footnotesize $n\hat{L}_n^{\hat{\rho}_n}(\hat{\rho}_n-\rho)$}, where the underlying processes were generated on a grid of size $10^{-5}$. In (a), we chose $\alpha=0.5$, $\beta=0.65$ and $\rho_0=0$. In (b) we set $\alpha=0.65$, $\beta=1$ and $\rho_0=0$.}\label{figure:histogram}
\end{figure}

In Table~\ref{Table_prop_rejection} we give the proportion of samples lying in the $(1-\kappa)$-confidence interval~\eqref{eq:confidence_interval}. One can observe that the coverage is already good for small sample sizes $n$. For this simulation, all sample paths were started in the true parameter $x_0=\rho_0$ to ensure that the condition $\{L_1^{\rho_0}(X)>0\}$ is satisfied with probability one (which allows to take into account every sample path, keeping the computational cost low).

\begin{center}
\begin{table}[h]
\begin{tabular}{c|cccc}
\textbf{$n$} & \textbf{$\kappa=0.2$} & \textbf{$\kappa=0.1$} & \textbf{$\kappa=0.05$} & \textbf{$\kappa = 0.01$} \\
\hline
$100$  & $0.745$ & $0.842$ & $0.908$ & $0.965$ \\
$500$ & $0.789$ & $0.882$ & $0.934$ & $0.979$ \\
$1000$ & $0.802$ & $0.899$ & $0.949$ & $0.986$ \\
\end{tabular}\vspace{0.2cm}
\quad 
\begin{tabular}{c|cccc}
\textbf{$n$} & \textbf{$\kappa=0.2$} & \textbf{$\kappa=0.1$} & \textbf{$\kappa=0.05$} & \textbf{$\kappa = 0.01$} \\
\hline
$100$  & $0.780$ & $0.862$ & $0.917$ & $0.963$ \\
$500$ & $0.774$ & $0.876$ & $0.927$ & $0.979$ \\
$1000$ & $0.801$ & $0.897$ & $0.944$ & $0.983$ \\
\end{tabular}\smallskip
\caption{\small{Proportions of samples lying in the $(1-\kappa)$-confidence interval~\eqref{eq:confidence_interval} for different values of $n$ when sampling $N=2,000$ path with parameter $\rho_0=x_0=0$ on a grid of size $1/(100n)$. In the left table we set $(\alpha,\beta)=(0.5, 0.7)$ and in the right one $(\alpha,\beta)=(1,0.65)$. }}\label{Table_prop_rejection}
\end{table}
\end{center}

\section{Expansion of the normalized log-likelihood function}

In this section, we give some expansions of the normalized log-likelhood expressions
\[ \log\left(\frac{p_{1/n}^{\rho_0+\theta}(X_{(k-1)/n},X_{k/n})}{p_{1/n}^{\rho_0+\theta'}(X_{(k-1)/n},X_{k/n})}\right) = \sum_{j=1}^9 \log\left(\frac{p_{1/n}^{\rho_0+\theta}(X_{(k-1)/n},X_{k/n})}{p_{1/n}^{\rho_0+\theta'}(X_{(k-1)/n},X_{k/n})}\right) \1_{I_{j,k}^{\theta',\theta}}\]
depending on the regime $I_{j,k}^{\theta',\theta}$. Those are used in the proof of Propositions~\ref{prop:expansion_of_drift_t} and~\ref{prop:moment_product} and thus contribute to both the $n$-consistency and the limiting distribution. Some expansions are given as a leading term and a remainder, for which the following observation is crucial:
\begin{itemize}
\item The order of the remainders of the case $I_{j,k}$ for different $j$ may be different as expressions of $n$ and $\theta$. Remarkably, however, their conditional expectations given $X_{(k-1)/n}$ are of the same order for all $j=1,\dots, 9$ (as used in Propositions~\ref{prop:expansion_of_drift_t} and~\ref{prop:moment_product}).
\item The remainder terms are of the same order as the leading terms for $|\theta-\theta'|\asymp n^{-1/2}$, whereas they are of smaller order for $|\theta-\theta'|\ll n^{-1/2}$.
\end{itemize}
For the presentation, we always assume $\theta'\leq\theta$, use the notation $I_{j,k}^{\theta',\theta}$ as given in~\eqref{eq:cases_I_jk} and denote by $\xi_k$ an intermediate value of the corresponding Lagrange form in the Taylor expansion. We write $\xi_k$ to highlight its (possible) dependence on $X_{(k-1)/n}$ and $X_{k/n}$. In particular, $\theta'\leq\xi_k\leq\theta$ and $\xi_k$ may be different in different expressions.

\begin{itemize}
\item[$\bullet\ \boldsymbol{I_{1,k}^{\theta',\theta}}$.] Here, we have $X_{(k-1)/n}<\rho_0+\theta',X_{k/n}\leq\rho_0+\theta'$ and present both a first and second order Taylor expansion. The first order expansion is given by
\begin{align}\label{eq:expression_case1_first_order}
\begin{split}
&\log\left(\frac{p_{1/n}^{\rho_0+\theta}(X_{(k-1)/n},X_{k/n})}{p_{1/n}^{\rho_0+\theta'}(X_{(k-1)/n},X_{k/n})}\right) \1_{\{X_{(k-1)/n}<\rho_0+\theta',X_{k/n}\leq\rho_0+\theta'\}} \\
&\hspace{0.2cm} = - \left(\theta-\theta'\right)\frac{\alpha-\beta}{\alpha+\beta} \frac{1}{P_1^{\rho_0+\xi_k}(X_{(k-1)/n},X_{k/n};1/n)} \frac{2(X_{k/n}-2\rho_0-2\xi_k+X_{(k-1)/n})}{(\alpha^2/n)\sqrt{2\pi\alpha^2/n}}\\
&\hspace{1.2cm}\cdot\exp\left(-\frac{(X_{k/n}-2\rho_0-2\xi_k+X_{(k-1)/n})^2}{2\alpha^2/n}\right)\1_{\{X_{(k-1)/n}<\rho_0+\theta',X_{k/n}\leq\rho_0+\theta'\}}
\end{split}
\end{align}
and the second order expansion by
\begin{align}\label{eq:expression_case1_second_order}
\begin{split}
&\log\left(\frac{p_{1/n}^{\rho_0+\theta}(X_{(k-1)/n},X_{k/n})}{p_{1/n}^{\rho_0+\theta'}(X_{(k-1)/n},X_{k/n})}\right) \1_{\{X_{(k-1)/n}<\rho_0+\theta',X_{k/n}\leq\rho_0+\theta'\}} \\
&\hspace{0.2cm} = -\left(\theta-\theta'\right)\frac{\alpha-\beta}{\alpha+\beta}\frac{1}{P_1^{\rho_0+\theta'}(X_{(k-1)/n},X_{k/n};1/n)} \frac{2(X_{k/n}-2\rho_0-2\theta'+X_{(k-1)/n})}{(\alpha^2/n)\sqrt{2\pi\alpha^2/n}} \\
&\hspace{2cm}\cdot\exp\left(-\frac{(X_{k/n}-2\rho_0-2\theta'+X_{(k-1)/n})^2}{2\alpha^2/n}\right)\1_{\{X_{(k-1)/n}<\rho_0+\theta',X_{k/n}\leq\rho_0+\theta'\}} \\
&\hspace{1cm} + \left(\theta-\theta'\right)^2\left[ \frac{\alpha-\beta}{\alpha+\beta}\frac{4}{\alpha^2/n} - \frac{\alpha-\beta}{\alpha+\beta}\frac{4(X_{k/n}-2\rho_0-2\xi_k+X_{(k-1)/n})^2}{\alpha^4/n^2}\right] \\
&\hspace{1.5cm} \cdot \frac{1}{P_1^{\rho_0+\xi_k}(X_{(k-1)/n},X_{k/n};1/n)}\exp\left(-\frac{(X_{k/n}-2\rho_0-2\xi_k+X_{(k-1)/n})^2}{2\alpha^2/n}\right)\\
&\hspace{1.5cm}\cdot\frac{1}{\sqrt{2\pi\alpha^2/n}}\1_{\{X_{(k-1)/n}<\rho_0+\theta',X_{k/n}\leq\rho_0+\theta'\}} \\
&\hspace{1cm} - \left(\theta-\theta'\right)^2 \left(\frac{\alpha-\beta}{\alpha+\beta}\right)^2 \frac{4(X_{k/n}-2\rho_0-2\xi_k+X_{(k-1)/n})^2}{[P_1^{\rho_0+\xi_k}(X_{(k-1)/n},X_{k/n};1/n)]^2} \frac{1}{2\pi\alpha^6/n^3}\\
&\hspace{2cm} \cdot \exp\left(-\frac{(X_{k/n}-2\rho_0-2\xi_k+X_{(k-1)/n})^2}{\alpha^2/n}\right)\1_{\{X_{(k-1)/n}<\rho_0+\theta',X_{k/n}\leq\rho_0+\theta'\}} \\
&\hspace{0.2cm} =: -\left(\theta-\theta'\right)\frac{\alpha-\beta}{\alpha+\beta}\frac{1}{P_1^{\rho_0+\theta'}(X_{(k-1)/n},X_{k/n};1/n)} \frac{2(X_{k/n}-2\rho_0-2\theta'+X_{(k-1)/n})}{(\alpha^2/n)\sqrt{2\pi\alpha^2/n}} \\
&\hspace{1cm}\cdot\exp\left(-\frac{(X_{k/n}-2\rho_0-2\theta'+X_{(k-1)/n})^2}{2\alpha^2/n}\right)\1_{\{X_{(k-1)/n}<\rho_0+\theta',X_{k/n}\leq\rho_0+\theta'\}} + R_1(k,\theta',\theta).
\end{split}
\end{align}
Using boundedness of $x\mapsto (1+x^2)\exp(-x^2/4)$ and that fact that
\begin{align}\label{eq_bound:exp/P1}
\begin{split}
&\left|\frac{\exp\left(-\frac{(X_{(k/n}-2\rho_0-2\xi_k+X_{(k-1)/n})^2}{2\alpha^2/n}\right)}{P_1^{\rho_0+\xi_k}(X_{(k-1)/n},X_{k/n};1/n)n^{-1/2}}\right|\\
&\hspace{1cm} =\sqrt{2\pi}\alpha \left|\frac{1}{\exp\left(\frac{2}{\alpha^2/n}(X_{(k/n}-\rho_0-\xi_k)(X_{(k-1)/n}-\rho_0-\xi_k)\right) - \frac{\alpha-\beta}{\alpha+\beta}}\right| \leq C_{\alpha,\beta}
\end{split}
\end{align}
for $X_{(k-1)/n},X_{k/n}\leq \rho_0+\xi_k$, we obtain
\begin{align}\label{eq:expression_case1_second_order_bound}
\begin{split}
\left| R_1(k,\theta',\theta)\right| &\leq C_{\alpha,\beta} \frac{|\theta-\theta'|^2 n^{3/2}}{P_1^{\rho_0+\xi_k}(X_{(k-1)/n},X_{k/n};1/n)}\1_{\{X_{(k-1)/n}<\rho_0+\theta',X_{k/n}\leq\rho_0+\theta'\}}\\
&\hspace{1cm}\cdot\exp\left(-\frac{(X_{k/n}-2\rho_0-2\xi_k +X_{(k-1)/n})^2}{4\alpha^2/n}\right).
\end{split}
\end{align}

\item[$\bullet\ \boldsymbol{I_{2,k}^{\theta',\theta}}$.] Here, $X_{(k-1)/n}<\rho_0+\theta'<X_{k/n}\leq\rho_0+\theta$ and 
\begin{align}\label{eq:expression_case2}
\begin{split}
&\log\left(\frac{p_{1/n}^{\rho_0+\theta}(X_{(k-1)/n},X_{k/n})}{p_{1/n}^{\rho_0+\theta'}(X_{(k-1)/n},X_{k/n})}\right) \1_{\{X_{(k-1)/n}<\rho_0+\theta'<X_{k/n}\leq\rho_0+\theta\}} \\
&\hspace{0.2cm} = \log\left(\frac{\beta^2}{\alpha^2}\right) \1_{\{X_{(k-1)/n}<\rho_0+\theta'<X_{k/n}\leq\rho_0+\theta\}} \\
&\hspace{0.5cm} + \left[\frac{1}{2/n}\left(\frac{X_{k/n}-\rho_0-\theta'}{\beta}-\frac{X_{(k-1)/n}-\rho_0-\theta'}{\alpha}\right)^2 - \frac{(X_{k/n}-X_{(k-1)/n})^2}{2\alpha^2/n}\right] \1_{I_{2,k}^{\theta',\theta}} \\
&\hspace{0.5cm} + \log\left( \frac{\alpha+\beta}{2\beta}-\frac{\alpha-\beta}{2\beta}\exp\left(-\frac{2}{\alpha^2/n}\left(X_{k/n}-\rho_0-\theta\right)\left(X_{(k-1)/n}-\rho_0-\theta\right)\right)\right) \1_{I_{2,k}^{\theta',\theta}} \\
&\hspace{0.2cm} =: \log\left(\frac{\beta^2}{\alpha^2}\right) \1_{\{X_{(k-1)/n}<\rho_0+\theta'<X_{k/n}\leq\rho_0+\theta\}} + R_2(k,\theta',\theta),
\end{split}
\end{align}
where by a direct evaluation and Lemma~\ref{lemma:bound_log(C+exp)},
\begin{align}\label{eq:expression_case2_bound}
\begin{split}
\left| R_2(k,\theta',\theta)\right| \leq C_{\alpha,\beta} n |\theta-\theta'| \left[ |\theta-\theta'| + |X_{(k-1)/n}-\rho_0-\theta|\right]\1_{\{X_{(k-1)/n}<\rho_0+\theta'<X_{k/n}\leq\rho_0+\theta\}}.
\end{split}
\end{align}

\item[$\bullet\ \boldsymbol{I_{3,k}^{\theta',\theta}}$.] Now we deal with $X_{(k-1)/n}<\rho_0+\theta' \leq\rho_0+\theta<X_{k/n}$. Here, 
\begin{align}\label{eq:expression_case3}
\begin{split}
&\log\left(\frac{p_{1/n}^{\rho_0+\theta}(X_{(k-1)/n},X_{k/n})}{p_{1/n}^{\rho_0+\theta'}(X_{(k-1)/n},X_{k/n})}\right) \1_{\{X_{(k-1)/n}<\rho_0+\theta' \leq\rho_0+\theta<X_{k/n}\}} \\
&\hspace{1cm} = \left[ -\left(\frac{1}{\alpha}-\frac{1}{\beta}\right)n(\theta-\theta')\left(\frac{X_{k/n}-\rho_0-\theta'}{\beta}-\frac{X_{(k-1)/n}-\rho_0-\theta'}{\alpha}\right)\right. \\
&\hspace{2cm}\left. - \frac12\left(\frac{1}{\alpha}-\frac{1}{\beta}\right)^2 n(\theta-\theta')^2\right] \1_{\{X_{(k-1)/n}<\rho_0+\theta' \leq\rho_0+\theta<X_{k/n}\}}.
\end{split}
\end{align}

\item[$\bullet\ \boldsymbol{I_{4,k}^{\theta',\theta}}$.] For $X_{k/n}\leq\rho_0+\theta'\leq X_{(k-1)/n}<\rho_0+\theta$, we find
\begin{align}\label{eq:expression_case4}
\begin{split}
&\log\left(\frac{p_{1/n}^{\rho_0+\theta}(X_{(k-1)/n},X_{k/n})}{p_{1/n}^{\rho_0+\theta'}(X_{(k-1)/n},X_{k/n})}\right) \1_{\{X_{k/n}\leq\rho_0+\theta'\leq X_{(k-1)/n}<\rho_0+\theta\}} \\
&\hspace{0.2cm} = \left[\frac{1}{2/n}\left(\frac{X_{k/n}-\rho_0-\theta'}{\alpha} - \frac{X_{(k-1)/n}-\rho_0-\theta'}{\beta}\right)^2 - \frac{(X_{k/n}-X_{(k-1)/n})^2}{2\alpha^2/n}\right] \1_{I_{4,k}^{\theta',\theta}} \\
&\hspace{0.5cm} + \log\left(\frac{\alpha+\beta}{2\beta} - \frac{\alpha-\beta}{2\beta}\exp\left(-\frac{2}{\alpha^2/n}\left(X_{k/n}-\rho_0-\theta\right)\left(X_{(k-1)/n}-\rho_0-\theta\right)\right)\right) \1_{I_{4,k}^{\theta',\theta}}.
\end{split}
\end{align}
Consequently, by a direct evaluation and Lemma~\ref{lemma:bound_log(C+exp)},
\begin{align}\label{eq:expression_case4_bound}
\begin{split}
& \left|\log\left(\frac{p_{1/n}^{\rho_0+\theta}(X_{(k-1)/n},X_{k/n})}{p_{1/n}^{\rho_0+\theta'}(X_{(k-1)/n},X_{k/n})}\right)\right| \1_{\{X_{k/n}\leq\rho_0+\theta'\leq X_{(k-1)/n}<\rho_0+\theta\}} \\
&\hspace{1cm} \leq C_{\alpha,\beta} n |\theta-\theta'| \left[ |\theta-\theta'| + |X_{k/n}-\rho_0-\theta|\right] \1_{\{X_{k/n}\leq\rho_0+\theta'\leq X_{(k-1)/n}<\rho_0+\theta\}}.
\end{split}
\end{align}

\item[$\bullet\ \boldsymbol{I_{5,k}^{\theta',\theta}}$.] We treat $\rho_0+\theta' \leq X_{(k-1)/n}<\rho_0+\theta, \rho_0+\theta'< X_{k/n}\leq\rho_0+\theta$ and find
\begin{align}\label{eq:expression_case5}
\begin{split}
&\log\left(\frac{p_{1/n}^{\rho_0+\theta}(X_{(k-1)/n},X_{k/n})}{p_{1/n}^{\rho_0+\theta'}(X_{(k-1)/n},X_{k/n})}\right) \1_{\{\rho_0+\theta' \leq X_{(k-1)/n}<\rho_0+\theta, \rho_0+\theta'< X_{k/n}\leq\rho_0+\theta\}} \\
&\hspace{0.2cm} = \left[ \log\left(\frac{\beta}{\alpha}\right) + \frac{(X_{k/n}-X_{(k-1)/n})^2}{2/n}\left(\frac{1}{\beta^2}-\frac{1}{\alpha^2}\right) \right]\\
&\hspace{1cm} \cdot\1_{\{\rho_0+\theta' \leq X_{(k-1)/n}<\rho_0+\theta, \rho_0+\theta'< X_{k/n}\leq\rho_0+\theta\}} \\
&\hspace{1.2cm} + \log\left(\frac{1-\frac{\alpha-\beta}{\alpha+\beta}\exp\left(-\frac{2}{\alpha^2/n}(X_{k/n}-\rho_0-\theta)(X_{(k-1)/n}-\rho_0-\theta)\right)}{1+\frac{\alpha-\beta}{\alpha+\beta}\exp\left(-\frac{2}{\beta^2/n}(X_{k/n}-\rho_0-\theta')(X_{(k-1)/n}-\rho_0-\theta')\right)}\right)\\
&\hspace{2.2cm}\cdot\1_{\{\rho_0+\theta' \leq X_{(k-1)/n}<\rho_0+\theta, \rho_0+\theta'< X_{k/n}\leq\rho_0+\theta\}}.
\end{split}
\end{align}
Using a Taylor expansion, we get the alternative expression
\begin{align}\label{eq:expression_case5_Taylor}
\begin{split}
&\log\left(\frac{p_{1/n}^{\rho_0+\theta}(X_{(k-1)/n},X_{k/n})}{p_{1/n}^{\rho_0+\theta'}(X_{(k-1)/n},X_{k/n})}\right) \1_{\{\rho_0+\theta' \leq X_{(k-1)/n}<\rho_0+\theta, \rho_0+\theta'< X_{k/n}\leq\rho_0+\theta\}} \\
&\hspace{0.2cm} = - \left(\theta-\theta'\right)\frac{\alpha-\beta}{\alpha+\beta} \frac{1}{P_1^{\rho_0+\xi_k}(X_{(k-1)/n},X_{k/n};1/n)} \frac{2(X_{k/n}-2\rho_0-2\xi_k+X_{(k-1)/n})}{(\alpha^2/n)\sqrt{2\pi\alpha^2/n}}\\
&\hspace{0.5cm}\cdot\exp\left(-\frac{(X_{k/n}-2\rho_0-2\xi_k+X_{(k-1)/n})^2}{2\alpha^2/n}\right)\1_{\{\rho_0+\theta' \leq X_{(k-1)/n}<\rho_0+\theta, \rho_0+\theta'< X_{k/n}\leq\rho_0+\theta\}}.
\end{split}
\end{align}

\item[$\bullet\ \boldsymbol{I_{6,k}^{\theta',\theta}}$.] Here, $\rho_0+\theta'\leq X_{(k-1)/n}< \rho_0+\theta<X_{k/n}$ and
\begin{align}\label{eq:expression_case6}
\begin{split}
&\log\left(\frac{p_{1/n}^{\rho_0+\theta}(X_{(k-1)/n},X_{k/n})}{p_{1/n}^{\rho_0+\theta'}(X_{(k-1)/n},X_{k/n})}\right) \1_{\{\rho_0+\theta'\leq X_{(k-1)/n}< \rho_0+\theta<X_{k/n}\}} \\
&\hspace{0.2cm} = \left[ \frac{(X_{k/n}-X_{(k-1)/n})^2}{2\beta^2/n}-\frac{1}{2/n}\left(\frac{X_{k/n}-\rho_0-\theta}{\beta} - \frac{X_{(k-1)/n}-\rho_0-\theta}{\alpha}\right)^2 \right] \1_{I_{6,k}^{\theta',\theta}} \\
&\hspace{0.4cm} - \log\left(\frac{\alpha+\beta}{2\alpha} + \frac{\alpha-\beta}{2\alpha}\exp\left(-\frac{2}{\alpha^2/n}\left(X_{k/n}-\rho_0-\theta'\right)\left(X_{(k-1)/n}-\rho_0-\theta'\right)\right)\right)  \1_{I_{6,k}^{\theta',\theta}}.
\end{split}
\end{align}
Consequently, by a direct evaluation and Lemma~\ref{lemma:bound_log(C+exp)},
\begin{align}\label{eq:expression_case6_bound}
\begin{split}
& \left|\log\left(\frac{p_{1/n}^{\rho_0+\theta}(X_{(k-1)/n},X_{k/n})}{p_{1/n}^{\rho_0+\theta'}(X_{(k-1)/n},X_{k/n})}\right)\right| \1_{\{\rho_0+\theta'\leq X_{(k-1)/n}< \rho_0+\theta<X_{k/n}\}} \\
&\hspace{1cm} \leq C_{\alpha,\beta} n |\theta-\theta'| \left[ |\theta-\theta'| + |X_{k/n}-\rho_0-\theta'|\right] \1_{\{\rho_0+\theta'\leq X_{(k-1)/n}< \rho_0+\theta<X_{k/n}\}}.
\end{split}
\end{align}

\item[$\bullet\ \boldsymbol{I_{7,k}^{\theta',\theta}}$.] In this case, $X_{k/n}\leq\rho_0+\theta' \leq\rho_0+\theta\leq X_{(k-1)/n}$ and a direct evaluation gives
\begin{align}\label{eq:expression_case7}
\begin{split}
&\log\left(\frac{p_{1/n}^{\rho_0+\theta}(X_{(k-1)/n},X_{k/n})}{p_{1/n}^{\rho_0+\theta'}(X_{(k-1)/n},X_{k/n})}\right) \1_{\{X_{k/n}\leq\rho_0+\theta' \leq\rho_0+\theta\leq X_{(k-1)/n}\}} \\
&\hspace{0.5cm} = \left[ \left(\frac{1}{\alpha}-\frac{1}{\beta}\right)n(\theta-\theta')\left(\frac{X_{k/n}-\rho_0-\theta'}{\alpha}-\frac{X_{(k-1)/n}-\rho_0-\theta'}{\beta}\right) \right.  \\
&\hspace{1.5cm} \left. - \frac12\left(\frac{1}{\alpha}-\frac{1}{\beta}\right)^2 n(\theta-\theta')^2\right] \1_{\{X_{k/n}\leq\rho_0+\theta' \leq\rho_0+\theta\leq X_{(k-1)/n}\}}.
\end{split}
\end{align}

\item[$\bullet\ \boldsymbol{I_{8,k}^{\theta',\theta}}$.] Here, we deal with $\rho_0+\theta'<X_{k/n} \leq\rho_0+\theta\leq X_{(k-1)/n}$ and get the expansion
\begin{align}\label{eq:expression_case8}
\begin{split}
&\log\left(\frac{p_{1/n}^{\rho_0+\theta}(X_{(k-1)/n},X_{k/n})}{p_{1/n}^{\rho_0+\theta'}(X_{(k-1)/n},X_{k/n})}\right) \1_{\{\rho_0+\theta'<X_{k/n} \leq\rho_0+\theta\leq X_{(k-1)/n}\}} \\
&\hspace{0.1cm} = \log\left(\frac{\beta^2}{\alpha^2}\right) \1_{\{\rho_0 +\theta'<X_{k/n} <\rho_0 +\theta<X_{(k-1)/n}\}} \\
&\hspace{0.3cm} + \left[\frac{(X_{k/n}-X_{(k-1)/n})^2}{2\beta^2/n}-\frac{1}{2/n}\left(\frac{X_{k/n}-\rho_0-\theta}{\alpha}-\frac{X_{(k-1)/n}-\rho_0-\theta}{\beta}\right)^2\right] \1_{I_{8,k}^{\theta',\theta}} \\
&\hspace{0.3cm} - \log\left( \frac{\alpha+\beta}{2\alpha}+\frac{\alpha-\beta}{2\alpha}\exp\left(-\frac{2}{\beta^2/n}\left(X_{k/n}-\rho_0-\theta'\right)\left(X_{(k-1)/n}-\rho_0-\theta'\right)\right)\right) \1_{I_{8,k}^{\theta',\theta}} \\
&\hspace{0.2cm} =: \log\left(\frac{\beta^2}{\alpha^2}\right) \1_{\{\rho_0+\theta'<X_{k/n} \leq\rho_0+\theta\leq X_{(k-1)/n}\}} + R_8(k,\theta',\theta),
\end{split}
\end{align}
where by a direct evaluation and Lemma~\ref{lemma:bound_log(C+exp)},
\begin{align}\label{eq:expression_case8_bound}
\begin{split}
\left| R_8(k,\theta',\theta)\right| \leq C_{\alpha,\beta} n|\theta-\theta'| \left[ |\theta-\theta'| + |X_{(k-1)/n}-\rho_0-\theta'|\right]\1_{\{\rho_0+\theta'<X_{k/n} \leq\rho_0+\theta\leq X_{(k-1)/n}\}}.
\end{split}
\end{align}

\item[$\bullet\ \boldsymbol{I_{9,k}^{\theta',\theta}}$.] In this case, $\rho_0+\theta\leq X_{(k-1)/n},\rho_0+\theta< X_{k/n}$. As for the first case, we present both first and second order Taylor expansion. The first order expansion is given by
\begin{align}\label{eq:expression_case9_first_order}
\begin{split}
&\log\left(\frac{p_{1/n}^{\rho_0+\theta}(X_{(k-1)/n},X_{k/n})}{p_{1/n}^{\rho_0+\theta'}(X_{(k-1)/n},X_{k/n})}\right) \1_{\{\rho_0+\theta\leq X_{(k-1)/n},\rho_0+\theta< X_{k/n}\}} \\
&\hspace{0.2cm} = \left(\theta-\theta'\right)\frac{\alpha-\beta}{\alpha+\beta} \frac{1}{P_2^{\rho_0+\xi_k}(X_{(k-1)/n},X_{k/n};1/n)}\frac{2(X_{k/n}-2\rho_0-2\xi_k+X_{(k-1)/n})}{(\beta^2/n)\sqrt{2\pi\beta^2/n}}\\
&\hspace{1.2cm}\cdot\exp\left(-\frac{(X_{k/n}-2\rho_0-2\xi_k +X_{(k-1)/n})^2}{2\beta^2/n}\right)\1_{\{\rho_0+\theta\leq X_{(k-1)/n},\rho_0+\theta< X_{k/n}\}}
\end{split}
\end{align}
and the second order expansion by
\begin{align}\label{eq:expression_case9_second_order}
\begin{split}
&\log\left(\frac{p_{1/n}^{\rho_0+\theta}(X_{(k-1)/n},X_{k/n})}{p_{1/n}^{\rho_0+\theta'}(X_{(k-1)/n},X_{k/n})}\right) \1_{\{\rho_0+\theta\leq X_{(k-1)/n},\rho_0+\theta< X_{k/n}\}} \\
&\hspace{0.2cm} = \left(\theta-\theta'\right)\frac{\alpha-\beta}{\alpha+\beta}\frac{1}{P_2^{\rho_0+\theta'}(X_{(k-1)/n},X_{k/n};1/n)} \frac{2(X_{k/n}-2\rho_0-2\theta'+X_{(k-1)/n})}{(\beta^2/n)\sqrt{2\pi\beta^2/n}} \\
&\hspace{2.2cm}\cdot\exp\left(-\frac{(X_{k/n}-2\rho_0-2\theta'+X_{(k-1)/n})^2}{2\beta^2/n}\right)\1_{\{\rho_0+\theta\leq X_{(k-1)/n},\rho_0+\theta< X_{k/n}\}} \\
&\hspace{1.2cm} - \left(\theta-\theta'\right)^2\left[ \frac{\alpha-\beta}{\alpha+\beta}\frac{4}{\beta^2/n} - \frac{\alpha-\beta}{\alpha+\beta}\frac{4(X_{k/n}-2\rho_0-2\xi_k+X_{(k-1)/n})^2}{\beta^4/n^2}\right] \\
&\hspace{1.5cm} \cdot\frac{1}{P_2^{\rho_0+\xi_k}(X_{(k-1)/n},X_{k/n};1/n)}\exp\left(-\frac{(X_{k/n}-2\rho_0-2\xi_k+X_{(k-1)/n})^2}{2\beta^2/n}\right)\\
&\hspace{1.5cm}\cdot  \frac{1}{\sqrt{2\pi\beta^2/n}}\1_{\{\rho_0+\theta\leq X_{(k-1)/n},\rho_0+\theta< X_{k/n}\}} \\
&\hspace{1cm} - \left(\theta-\theta'\right)^2 \left(\frac{\alpha-\beta}{\alpha+\beta}\right)^2 \frac{4(X_{k/n}-2\rho_0-2\xi_k+X_{(k-1)/n})^2}{[P_2^{\rho_0+\xi_k}(X_{(k-1)/n},X_{k/n};1/n)]^2} \frac{1}{2\pi\beta^6/n^3}\\
&\hspace{1.5cm} \cdot \exp\left(-\frac{(X_{k/n}-2\rho_0-2\xi_k+X_{(k-1)/n})^2}{\beta^2/n}\right)\1_{\{\rho_0+\theta\leq X_{(k-1)/n},\rho_0+\theta< X_{k/n}\}} \\
&\hspace{0.2cm} =: \left(\theta-\theta'\right)\frac{\alpha-\beta}{\alpha+\beta}\frac{1}{P_2^{\rho_0+\theta'}(X_{(k-1)/n},X_{k/n};1/n)}\frac{2(X_{k/n}-2\rho_0-2\theta'+X_{(k-1)/n})}{(\beta^2/n)\sqrt{2\pi\beta^2/n}} \\
&\hspace{0.5cm}\cdot\exp\left(-\frac{(X_{k/n}-2\rho_0-2\theta'+X_{(k-1)/n})^2}{2\beta^2/n}\right)\1_{\{\rho_0+\theta\leq X_{(k-1)/n},\rho_0+\theta< X_{k/n}\}} + R_9(k,\theta',\theta).
\end{split}
\end{align}
Using boundedness of $x\mapsto (1+x^2)\exp(-x^2/4)$ and that
\begin{align*}
&\left|\frac{\exp\left(-\frac{(X_{(k/n}-2\rho_0-2\xi_k+X_{(k-1)/n})^2}{2\beta^2/n}\right)}{P_2^{\rho_0+\xi_k}(X_{(k-1)/n},X_{k/n};1/n)\sqrt{1/n}}\right| \\
&\hspace{1cm}= \sqrt{2\pi} \left|\frac{\exp\left(-\frac{(X_{(k/n}-2\rho_0-2\xi_k+X_{(k-1)/n})^2}{2\beta^2/n}\right)}{\exp\left(-\frac{(X_{(k/n}-X_{(k-1)/n})^2}{2\beta^2/n}\right) + \frac{\alpha-\beta}{\alpha+\beta}\exp\left(-\frac{(X_{(k/n}-2\rho_0-2\xi_k+X_{(k-1)/n})^2}{2\beta^2/n}\right)}\right| \\
&\hspace{1cm} =\sqrt{2\pi} \left|\frac{1}{\exp\left(\frac{2}{\beta^2/n}(X_{(k/n}-\rho_0-\xi_k)(X_{(k-1)/n}-\rho_0-\xi_k)\right) + \frac{\alpha-\beta}{\alpha+\beta}}\right|
\end{align*}
can be bounded independently of $n$ for $X_{(k-1)/n},X_{k/n}\geq \rho_0+\xi_k$, we obtain
\begin{align}\label{eq:expression_case9_second_order_bound}
\begin{split}
\left| R_9(k,\theta',\theta)\right| &\leq C_{\alpha,\beta}  \frac{|\theta-\theta'|^2 n^{3/2}}{P_2^{\rho_0+\xi_k}(X_{(k-1)/n},X_{k/n};1/n)}\1_{\{\rho_0+\theta\leq X_{(k-1)/n},\rho_0+\theta< X_{k/n}\}} \\
&\hspace{1cm} \cdot \exp\left(-\frac{(X_{k/n}-2\rho_0-2\xi_k+X_{(k-1)/n})^2}{4\beta^2/n}\right).
\end{split}
\end{align}
\end{itemize}

\section{Proofs of the results of Section~\ref{Section:likelihood}}\label{Appendix_auxiliary}

In this section, we give the proofs of Proposition~\ref{prop:expansion_of_drift_t} and~\ref{prop:moment_product} in the respective order.

\begin{proof}[Proof of Proposition~\ref{prop:expansion_of_drift_t}]
We give a complete description for the case $\theta\geq 0$. The case $\theta<0$ can then be dealt with analogously. By linearity of conditional expectation,
\begin{align*}
&\sum_{k=1}^{\lfloor nt\rfloor} \E_{\rho_0}\left[\left. \log\left( \frac{p_{1/n}^{\rho_0+\theta}(X_{(k-1)/n},X_{k/n})}{p_{1/n}^{\rho_0}(X_{(k-1)/n},X_{k/n})}\right) \right| X_{(k-1)/n}\right]\\
&\hspace{1cm} = \sum_{j=1}^9\sum_{k=1}^{\lfloor nt\rfloor} \E_{\rho_0}\left[\left. \log\left( \frac{p_{1/n}^{\rho_0+\theta}(X_{(k-1)/n},X_{k/n})}{p_{1/n}^{\rho_0}(X_{(k-1)/n},X_{k/n})}\right)\1_{I_{j,k}^\theta} \right| X_{(k-1)/n}\right].
\end{align*} 
In what follows, we evaluate each of the cases $j=1,\dots, 9$ on the right-hand side seperately. It will turn out that the first-order terms in the expansions \eqref{eq:expression_case1_first_order}-\eqref{eq:expression_case9_second_order_bound} for $j\in\{1,2,3,7,8,9\}$ contribute to $-|\theta|F_{\alpha,\beta}\Lambda_{\alpha,\beta}^n((X_{(k-1)/n})_{1\leq k\leq\lfloor nt\rfloor})$, whereas $j=4,5,6$ are absorbed in the remainder term $r_n(t,\theta)$. In what follows, we denote with $\xi_k$ an intermediate point $0\leq \xi_k\leq\theta$ that occurs in the Lagrange form of Taylor's theorem and may vary from line to line.\smallskip

\begin{itemize}
\item[$\bullet\ \boldsymbol{j=1}$.] By the identity~\eqref{eq:expression_case1_second_order} (with $\theta'=0$),
\begin{align*}
&\E_{\rho_0}\left[\left. \log\left( \frac{p_{1/n}^{\rho_0+\theta}(X_{(k-1)/n},X_{k/n})}{p_{1/n}^{\rho_0}(X_{(k-1)/n},X_{k/n})}\right)\1_{I_{1,k}^\theta} \right| X_{(k-1)/n}\right] \\
&\hspace{0.2cm} = -\1_{\{X_{(k-1)/n}<\rho_0\}}\theta \int_{-\infty}^{\rho_0} \frac{\alpha-\beta}{\alpha+\beta} \frac{4(y-2\rho_0+X_{(k-1)/n})}{(2\alpha^2/n)\sqrt{2\pi\alpha^2/n}}\exp\left( - \frac{(y-2\rho_0+X_{(k-1)/n})^2}{2\alpha^2/n}\right) dy\\
&\hspace{1.2cm} + \E_{\rho_0}\left[\left. R_1(k,0,\theta)\right| X_{(k-1)/n}\right].
\end{align*}
Evaluation of the integral reveals
\begin{align*}
&-\theta \int_{-\infty}^{\rho_0} \frac{\alpha-\beta}{\alpha+\beta} \frac{4(y-2\rho_0+X_{(k-1)/n})}{2\alpha^2/n} \frac{1}{\sqrt{2\pi\alpha^2/n}}\exp\left( - \frac{(y-2\rho_0+X_{(k-1)/n})^2}{2\alpha^2/n}\right) dy \\
&\hspace{1cm} = \theta\frac{\alpha-\beta}{\alpha+\beta} \frac{2}{\sqrt{2\pi\alpha^2/n}}\exp\left( - \frac{(X_{(k-1)/n}-\rho_0)^2}{2\alpha^2/n}\right).
\end{align*} 
Thus,
\begin{align*}
&\sum_{k=1}^{\lfloor nt\rfloor} \E_{\rho_0}\left[\left. \log\left( \frac{p_{1/n}^{\rho_0+\theta}(X_{(k-1)/n},X_{k/n})}{p_{1/n}^{\rho_0}(X_{(k-1)/n},X_{k/n})}\right)\1_{I_{1,k}^\theta} \right| X_{(k-1)/n}\right] \\
&\hspace{0.5cm} =\theta\frac{\alpha-\beta}{\alpha+\beta} \frac{2}{\sqrt{2\pi\alpha^2/n}}\sum_{k=1}^{\lfloor nt\rfloor}  \1_{\{X_{(k-1)/n}<\rho_0\}} \exp\left( - \frac{(X_{(k-1)/n}-\rho_0)^2}{2\alpha^2/n}\right) + r_n^1(t,\theta),
\end{align*}
where
\[ r_n^1(t,\theta) := \sum_{k=1}^{\lfloor nt\rfloor} \E_{\rho_0}\left[\left. R_1(k,0,\theta)\right| X_{(k-1)/n}\right]. \]
It remains to estimate the $L^1(\Pr_{\rho_0})$-norm of $\sup_{0\leq \theta'\leq\theta}\sup_{s\leq t} |r_n^1(s,\theta')|$. With~\eqref{eq:expression_case1_second_order_bound} and
\begin{align}\label{eq_bound:P1/P1}
\begin{split}
&\frac{P_1^{\rho_0}(X_{(k-1)/n},X_{k/n};1/n)}{P_1^{\rho_0+\xi_k}(X_{(k-1)/n},X_{k/n};1/n)}\\
&\hspace{1cm} = \frac{1-\frac{\alpha-\beta}{\alpha+\beta}\exp\left(-\frac{2}{\alpha^2/n}(X_{k/n}-\rho_0)(X_{(k-1)/n}-\rho_0)\right)}{1-\frac{\alpha-\beta}{\alpha+\beta}\exp\left(-\frac{2}{\alpha^2/n}(X_{k/n}-\rho_0-\xi_k)(X_{(k-1)/n}-\rho_0-\xi_k)\right)}\leq C_{\alpha,\beta}
\end{split}
\end{align} 
for $X_{(k-1)/n},X_{k/n}\leq\rho_0$, we find with the Gaussian tail inequality
\begin{align*}
&\E_{\rho_0}\left[\left. |R_1(k,0,\theta)|\right| X_{(k-1)/n}\right] \\
&\hspace{0.2cm} \leq C_{\alpha,\beta} n^{3/2}\theta^2 \1_{\{X_{(k-1)/n}< \rho_0\}} \int_{-\infty}^{\rho_0} \frac{P_1^{\rho_0}(X_{(k-1)/n},y;1/n)}{P_1^{\rho_0+\xi_k}(X_{(k-1)/n},y;1/n)} \\
&\hspace{5cm}\cdot \exp\left(-\frac{(y-2\rho_0-2\xi_k+X_{(k-1)/n})^2}{4\alpha^2/n}\right) dy \\
&\hspace{0.2cm} \leq C_{\alpha,\beta}n\theta^2 \1_{\{X_{(k-1)/n}< \rho_0\}} \int_{-\infty}^{(X_{(k-1)/n}-\rho_0-2\xi_k)/\sqrt{2\alpha^2/n}}  \exp\left(-\frac{y^2}{2}\right) dy \\
&\hspace{0.2cm} \leq C_{\alpha,\beta} n\theta^2\1_{\{X_{(k-1)/n}< \rho_0\}} \exp\left(-\frac{(X_{(k-1)/n}-\rho_0)^2}{4\alpha^2/n}\right).
\end{align*}
Consequently, with Corollary~\ref{cor:bound_exp_X^2},
\begin{align*}
\E_{\rho_0}\left[ \sup_{0\leq \theta'\leq\theta}\sup_{s\leq t} \frac{|r_n^1(s,\theta')|}{\theta'}\right] &\leq C_{\alpha,\beta}n\theta \E_{\rho_0}\left[ \sup_{s\leq t}\sum_{k=1}^{\lfloor ns \rfloor} \exp\left(-\frac{(X_{(k-1)/n}-\rho_0)^2}{4\alpha^2/n}\right)\right] \\
&= C_{\alpha,\beta} n\theta  \sum_{k=1}^{\lfloor nt\rfloor} \E_{\rho_0}\left[\exp\left(-\frac{(X_{(k-1)/n}-\rho_0)^2}{4\alpha^2/n}\right)\right] \\
& \leq C_{\alpha,\beta} n\theta \sum_{k=1}^{\lfloor nt\rfloor} \frac{1}{\sqrt{k}} \leq C_{\alpha,\beta} n^{3/2}\theta^2 \sqrt{t}.
\end{align*} \smallskip

\item[$\bullet\ \boldsymbol{j=2}$.] Using the decomposition~\eqref{eq:expression_case2} we find
\begin{align*}
&\sum_{k=1}^{\lfloor nt\rfloor} \E_{\rho_0}\left[\left. \log\left( \frac{p_{1/n}^{\rho_0+\theta}(X_{(k-1)/n},X_{k/n})}{p_{1/n}^{\rho_0}(X_{(k-1)/n},X_{k/n})}\right)\1_{I_{2,k}^\theta} \right| X_{(k-1)/n}\right] \\ 
&\hspace{0.2cm} = \log\left(\frac{\beta^2}{\alpha^2}\right)\sum_{k=1}^{\lfloor nt\rfloor} \1_{\{X_{(k-1)/n}<\rho_0\}}\int_{\rho_0}^{\rho_0+\theta} P_3^{\rho_0}(X_{(k-1)/n},y;1/n)dy + \sum_{k=1}^{\lfloor nt\rfloor} \E_{\rho_0}\left[\left. R_2(k,0,\theta)\right| X_{(k-1)/n}\right] \\
&\hspace{0.2cm} = \theta \frac{2}{\alpha+\beta}\frac{\alpha}{\beta}\frac{1}{\sqrt{2\pi/n}}\log\left(\frac{\beta^2}{\alpha^2}\right)\sum_{k=1}^{\lfloor nt\rfloor} \1_{\{X_{(k-1)/n}<\rho_0\}}\exp\left(-\frac{(X_{(k-1)/n}-\rho_0)^2}{2\alpha^2/n}\right) + r_n^2(t,\theta),
\end{align*}
where $r_n^2(t,\theta)=r_n^{2,1}(t,\theta) + r_n^{2,2}(t,\theta)$ with
\begin{align*}
r_n^{2,1}(t,\theta) &:= \sum_{k=1}^{\lfloor nt\rfloor} \E_{\rho_0}\left[\left. R_2(k,0,\theta)\right| X_{(k-1)/n}\right], \\
 r_n^{2,2}(t,\theta) &:=\sum_{k=1}^{\lfloor nt\rfloor} \1_{\{X_{(k-1)/n}<\rho_0\}}\int_{\rho_0}^{\rho_0+\theta} \left( P_3^{\rho_0}(X_{(k-1)/n},y;1/n)- P_3^{\rho_0}(X_{(k-1)/n},\rho_0;1/n)\right)dy.
\end{align*}
In what follows, we are going to prove the uniform $L^1(\Pr_{\rho_0})$-bound on both $r_n^{2,1}(t,\theta), r_n^{2,2}(t,\theta)$ seperately.\bigskip
\begin{itemize}
\item[$\boldsymbol{- r_n^{2,1}}$.] From the bound~\eqref{eq:expression_case2_bound} for $R_2(k,0,\theta)$, Lemma~\ref{lemma:upper_bound_transition_density} and boundedness of the function $x\mapsto (1+x^2)\exp(-x^2/2)$, we find
\begin{align*}
&\E_{\rho_0}\left[\left. |R_2(k,0,\theta)|\right| X_{(k-1)/n}\right] \\
&\hspace{0.5cm} \leq C_{\alpha,\beta}\theta \E_{\rho_0}\left[\left. \left( \sqrt{n}K + n|X_{(k-1)/n}-\rho_0-\theta|\right) \1_{\{X_{(k-1)/n}<\rho_0<X_{k/n}\leq\rho_0+\theta\}}\right| X_{(k-1)/n}\right] \\
&\hspace{0.5cm} \leq C_{\alpha,\beta}\theta \1_{\{X_{(k-1)/n}<\rho_0\}} \left( \sqrt{n}K + n|X_{(k-1)/n}-\rho_0-\theta|\right) \\
&\hspace{5cm} \cdot\int_{\rho_0}^{\rho_0+\theta} \sqrt{n}\exp\left(-\frac{(y-X_{(k-1)/n})^2}{2\max\{\alpha^2,\beta^2\}/n}\right) dy \\
&\hspace{0.5cm} \leq C_{\alpha,\beta}\theta \sqrt{n}\theta \left( \sqrt{n}K + n|X_{(k-1)/n}-\rho_0-\theta|\right) \exp\left(-\frac{(X_{(k-1)/n}-\rho_0)^2}{2\max\{\alpha^2,\beta^2\}/n}\right) \\
&\hspace{0.5cm} \leq C_{\alpha,\beta} n\theta^2 \left( K + \sqrt{n}|X_{(k-1)/n}-\rho_0-\theta|\right) \exp\left(-\frac{(X_{(k-1)/n}-\rho_0)^2}{2\max\{\alpha^2,\beta^2\}/n}\right) \\
&\hspace{0.5cm} \leq C_{\alpha,\beta} n\theta^2 \left( 2K + \sqrt{n}|X_{(k-1)/n}-\rho_0|\right) \exp\left(-\frac{(X_{(k-1)/n}-\rho_0)^2}{2\max\{\alpha^2,\beta^2\}/n}\right) \\
&\hspace{0.5cm} \leq C_{\alpha,\beta}(K) n\theta^2  \exp\left(-\frac{(X_{(k-1)/n}-\rho_0)^2}{4\max\{\alpha^2,\beta^2\}/n}\right).
\end{align*}
Consequently, by Corollary~\ref{cor:bound_exp_X^2},
\begin{align*}
\E_{\rho_0}\left[ \sup_{0\leq \theta'\leq\theta}\sup_{s\leq t} \frac{|r_n^{2,1}(s,\theta')|}{\theta'}\right] &\leq C_{\alpha,\beta}(K) n\theta \E_{\rho_0}\left[\sup_{s\leq t} \sum_{k=1}^{\lfloor ns\rfloor} \exp\left(-\frac{(X_{(k-1)/n}-\rho_0)^2}{4\max\{\alpha^2,\beta^2\}/n}\right) \right] \\
&= C_{\alpha,\beta}(K) n\theta \sum_{k=1}^{\lfloor nt\rfloor} \E_{\rho_0}\left[\exp\left(-\frac{(X_{(k-1)/n}-\rho_0)^2}{4\max\{\alpha^2,\beta^2\}/n}\right) \right] \\
&\leq C_{\alpha,\beta}(K) n^{3/2}\theta \sqrt{t}.
\end{align*} \smallskip

\item[$\boldsymbol{- r_n^{2,2}}$.] By a first order Taylor expansion of $P_3^{\rho_0}(X_{(k-1)/n},y;1/n)$ in the variable $y$ around $\rho_0$, we get
\begin{align*}
&P_3^{\rho_0}(X_{(k-1)/n},y;1/n)-P_3^{\rho_0}(X_{(k-1)/n},\rho_0;1/n)\\
&\hspace{0.5cm} = -\frac{2}{\alpha+\beta}\frac{\alpha}{\beta} \frac{1}{\sqrt{2\pi/n}}\frac{1}{\beta/n}\left(\frac{\xi_y-\rho_0}{\beta} - \frac{X_{(k-1)/n}-\rho_0}{\alpha}\right) \\
&\hspace{2.5cm} \cdot\exp\left( -\frac{1}{2/n}\left(\frac{\xi_y-\rho_0}{\beta} - \frac{X_{(k-1)/n}-\rho_0}{\alpha}\right)^2\right) (y-\rho_0),
\end{align*}
where for every $y$ we denote with $\xi_y$ the intermediate value appearing in the Lagrange version of the remainder in the Taylor expansion. In particular, $\rho_0\leq\xi_y\leq\rho_0+\theta$. Using boundedness of $x\mapsto x\exp(-x^2/4)$ then yields
\begin{align*}
&\1_{\{X_{(k-1)/n}<\rho_0\}} \left|\int_{\rho_0}^{\rho_0+\theta} \left( P_3^{\rho_0}(X_{(k-1)/n},y;1/n)- P_3^{\rho_0}(X_{(k-1)/n},\rho_0;1/n)\right)dy\right| \\
&\hspace{0.5cm} \leq C_{\alpha,\beta} \1_{\{X_{(k-1)/n}<\rho_0\}} n\theta\int_{\rho_0}^{\rho_0+\theta} \exp\left( -\frac{1}{4/n}\left(\frac{\xi_y-\rho_0}{\beta} - \frac{X_{(k-1)/n}-\rho_0}{\alpha}\right)^2\right) dy \\
&\hspace{0.5cm} \leq C_{\alpha,\beta} \1_{\{X_{(k-1)/n}<\rho_0\}} n\theta^2\exp\left( -\frac{(X_{(k-1)/n}-\rho_0)^2}{4\alpha^2/n}\right).
\end{align*}
Then, with Corollary~\ref{cor:bound_exp_X^2},
\begin{align*}
\E_{\rho_0}\left[\sup_{0\leq\theta'\leq\theta}\sup_{s\leq t} \frac{| r_n^{2,2}(s,\theta')|}{\theta'}\right] \leq C_{\alpha,\beta} n^{3/2} \theta \sqrt{t}.
\end{align*}
\end{itemize}\smallskip

\item[$\bullet\ \boldsymbol{j=3}$.] By~\eqref{eq:expression_case3}, 
\begin{align*}
&\E_{\rho_0}\left[\left. \log\left( \frac{p_{1/n}^{\rho_0+\theta}(X_{(k-1)/n},X_{k/n})}{p_{1/n}^{\rho_0}(X_{(k-1)/n},X_{k/n})}\right)\1_{I_{3,k}^\theta} \right| X_{(k-1)/n}\right] \\
&\hspace{0.5cm} = -\1_{\{X_{(k-1)/n}<\rho_0\}} n\theta \left(\frac{1}{\alpha}-\frac{1}{\beta}\right) \int_{\rho_0+\theta}^{\infty} \left(\frac{(y-\rho_0}{\beta} - \frac{X_{(k-1)/n}-\rho_0}{\alpha}\right)\\
&\hspace{8cm}\cdot P_3^{\rho_0}(X_{(k-1)/n},y;1/n) dy \\
&\hspace{1.5cm}  -\1_{\{X_{(k-1)/n}<\rho_0\}} \frac{n\theta^2}{2}\left(\frac{1}{\alpha}-\frac{1}{\beta}\right)^2 \E_{\rho_0}\left[\1_{\{X_{k/n}> \rho_0+\theta\}}\mid X_{(k-1)/n}\right].
\end{align*}
Evaluation of the integral gives
\begin{align*}
& n\int_{\rho_0+\theta}^{\infty} \left(\frac{(y-\rho_0}{\beta} - \frac{X_{(k-1)/n}-\rho_0}{\alpha}\right) P_3^{\rho_0}(X_{(k-1)/n};1/n) dy \\
&\hspace{2cm} = \frac{2\alpha}{\alpha+\beta}\frac{1}{\sqrt{2\pi/n}}\exp\left( - \frac{(\theta-\frac{\beta}{\alpha}(X_{(k-1)/n}-\rho_0))^2}{2\beta^2/n}\right),
\end{align*}
such that
\begin{align*}
&\sum_{k=1}^{\lfloor nt\rfloor} \E_{\rho_0}\left[\left. \log\left( \frac{p_{1/n}^{\rho_0+\theta}(X_{(k-1)/n},X_{k/n})}{p_{1/n}^{\rho_0}(X_{(k-1)/n},X_{k/n})}\right)\1_{I_{3,k}^\theta} \right| X_{(k-1)/n}\right] \\
&\hspace{0.1cm} =- \theta\left(\frac{1}{\alpha}-\frac{1}{\beta}\right)\frac{2\alpha}{\alpha+\beta}\frac{1}{\sqrt{2\pi/n}} \sum_{k=1}^{\lfloor nt\rfloor} \1_{\{X_{(k-1)/n}<\rho_0\}}\exp\left( - \frac{(X_{(k-1)/n}-\rho_0)^2}{2\alpha^2/n}\right) + r_n^3(t,\theta),
\end{align*}
with $r_n^3(t,\theta) := r_n^{3,1}(t,\theta) + r_n^{3,2}(t,\theta)$, where
\begin{align*}
r_n^{3,1}(t,\theta) &:= -\1_{\{X_{(k-1)/n}<\rho_0\}} \theta\left(\frac{1}{\alpha}-\frac{1}{\beta}\right)\frac{2\alpha}{\alpha+\beta}\frac{1}{\sqrt{2\pi/n}} \\
&\hspace{0.5cm} \cdot\sum_{k=1}^{\lfloor nt\rfloor} \left[\exp\left( - \frac{(\theta-\frac{\beta}{\alpha}(X_{(k-1)/n}-\rho_0))^2}{2\beta^2/n}\right)-\exp\left( - \frac{(X_{(k-1)/n}-\rho_0)^2}{2\alpha^2/n}\right)\right], \\
r_n^{3,2}(t,\theta) &:=  - \frac{n\theta^2}{2}\left(\frac{1}{\alpha}-\frac{1}{\beta}\right)^2\sum_{k=1}^{\lfloor nt\rfloor}\1_{\{X_{(k-1)/n}<\rho_0\}}\E_{\rho_0}\left[\1_{\{X_{k/n}> \rho_0+\theta\}}\mid X_{(k-1)/n}\right].
\end{align*}
In what follows, we are going to prove the moment condition on both $r_n^{3,1}(t,\theta), r_n^{3,2}(t,\theta)$ seperately.\bigskip
\begin{itemize}
\item[$- \boldsymbol{r_n^{3,1}}$.] Using a first order Taylor expansion of the $\exp$-term, we find with boundedness of $x\mapsto |x|\exp(-x^2/4)$ that
\begin{align*}
&\left|\exp\left( - \frac{(\theta-\frac{\beta}{\alpha}(X_{(k-1)/n}-\rho_0))^2}{2\beta^2/n}\right)-\exp\left( - \frac{(X_{(k-1)/n}-\rho_0)^2}{2\alpha^2/n}\right)\right| \1_{\{X_{(k-1)/n}< \rho_0\}} \\
&\hspace{0.2cm} = \theta \frac{|\xi_k-\frac{\beta}{\alpha}(X_{(k-1)/n}-\rho_0)|}{\beta^2/n}\exp\left( - \frac{(\xi_k-\frac{\beta}{\alpha}(X_{(k-1)/n}-\rho_0))^2}{2\beta^2/n}\right)\1_{\{X_{(k-1)/n}< \rho_0\}} \\
&\hspace{0.2cm} \leq C_{\alpha,\beta} \sqrt{n}\theta \exp\left( - \frac{(\xi_k-\frac{\beta}{\alpha}(X_{(k-1)/n}-\rho_0))^2}{4\beta^2/n}\right)\1_{\{X_{(k-1)/n}< \rho_0\}} \\
&\hspace{0.2cm} \leq C_{\alpha,\beta} \sqrt{n}\theta \exp\left( - \frac{(X_{(k-1)/n}-\rho_0)^2}{4\alpha^2/n}\right)\1_{\{X_{(k-1)/n}< \rho_0\}},
\end{align*}
for a suitable intermediate value $0\leq \xi_k\leq \theta$ for every $k=1,\dots, n$. Then, with Corollary~\ref{cor:bound_exp_X^2},
\begin{align*}
\E_{\rho_0}\left[\sup_{0\leq\theta'\leq\theta}\sup_{s\leq t} \frac{|r_n^{3,1}(s,\theta')|}{\theta'}\right]&\leq C_{\alpha,\beta} n\theta \sum_{k=1}^{\lfloor nt\rfloor} \E_{\rho_0}\left[ \exp\left( - \frac{(X_{(k-1)/n}-\rho_0)^2}{4\alpha^2/n}\right) \right]\\
& \leq C_{\alpha,\beta} n^{3/2}\theta \sqrt{t}.
\end{align*}\smallskip
\item[$- \boldsymbol{r_n^{3,2}}$.] Here, we observe that with Lemma~\ref{lemma:upper_bound_transition_density} and the Gaussian tail inequality,
\begin{align*}
& \1_{\{X_{(k-1)/n}<\rho_0\}}\E_{\rho_0}\left[\1_{\{X_{k/n}> \rho_0+\theta\}}\mid X_{(k-1)/n}\right] \\
&\hspace{0.2cm}\leq C_{\alpha,\beta}\1_{\{X_{(k-1)/n}<\rho_0\}} \int_{\rho_0}^\infty \sqrt{n} \exp\left(-\frac{(y-X_{(k-1)/n})^2}{2\max\{\alpha^2,\beta^2\}/n}\right) dy \\
&\hspace{0.2cm}\leq C_{\alpha,\beta}\1_{\{X_{(k-1)/n}<\rho_0\}} \int_{-(X_{(k-1)/n}-\rho_0)/\sqrt{1/n}}^\infty \exp\left(-\frac{y^2}{2\max\{\alpha^2,\beta^2\}}\right) dy \\
&\hspace{0.2cm}\leq C_{\alpha,\beta} \exp\left(-\frac{(X_{(k-1)/n}-\rho_0)^2}{2\max\{\alpha^2,\beta^2\}/n}\right).
\end{align*}
Then, with Corollary~\ref{cor:bound_exp_X^2},
\begin{align*}
\E_{\rho_0}\left[ \sup_{0\leq\theta'\leq\theta}\sup_{s\leq t} \frac{|r_n^{3,2}(s,\theta')|}{\theta'}\right] &\leq C_{\alpha,\beta} n\theta \sum_{k=1}^{\lfloor nt\rfloor} \E_{\rho_0}\left[ \exp\left(-\frac{(X_{(k-1)/n}-\rho_0)^2}{2\max\{\alpha^2,\beta^2\}/n}\right) \right] \\
& \leq C_{\alpha,\beta} n^{3/2}\theta \sqrt{t}.
\end{align*} 
\end{itemize}\smallskip

\item[$\bullet\ \boldsymbol{j=4}$.] With the expansion~\eqref{eq:expression_case4} and~\eqref{eq:expression_case4_bound} we find using $|\theta|\leq 2K/\sqrt{n}$,
\begin{align*}
&\E_{\rho_0}\left[\left. \left|\log\left( \frac{p_{1/n}^{\rho_0+\theta}(X_{(k-1)/n},X_{k/n})}{p_{1/n}^{\rho_0}(X_{(k-1)/n},X_{k/n})}\right)\right|\1_{I_{4,k}^\theta} \right| X_{(k-1)/n}\right] \\
&\hspace{0.2cm} \leq C_{\alpha,\beta}\theta \E_{\rho_0}\left[\left. \left( \sqrt{n}K + n|X_{k/n}-\rho_0-\theta|\right) \1_{\{X_{k/n}\leq\rho_0 \leq X_{(k-1)/n}<\rho_0+\theta\}}\right| X_{(k-1)/n}\right] \\
&\hspace{0.2cm} \leq C_{\alpha,\beta}\sqrt{n}\theta \1_{\{\rho_0\leq X_{(k-1)/n}<\rho_0+\theta\}} \int_{-\infty}^{\rho_0} \left( K + \sqrt{n}|y-\rho_0-\theta|\right)  \\
&\hspace{7cm} \cdot\sqrt{n}\exp\left(-\frac{(y-X_{(k-1)/n})^2}{2\max\{\alpha^2,\beta^2\}/n}\right) dy \\
&\hspace{0.2cm} \leq C_{\alpha,\beta} \sqrt{n}\theta \1_{\{\rho_0\leq X_{(k-1)/n}<\rho_0+\theta\}} \int_{-\infty}^{\rho_0} \left( 2K + \sqrt{n}|y-\rho_0|\right) \\
&\hspace{7cm} \cdot\sqrt{n}\exp\left(-\frac{(y-\rho_0)^2}{2\max\{\alpha^2,\beta^2\}/n}\right) dy \\
&\hspace{0.2cm} \leq C_{\alpha,\beta}(K)\sqrt{n}\theta \1_{\{\rho_0\leq X_{(k-1)/n}<\rho_0+\theta\}}\int_{-\infty}^0 (1+|y|)\exp\left(-\frac{y^2}{2\max\{\alpha^2,\beta^2\}}\right)dy \\
&\hspace{0.2cm} \leq C_{\alpha,\beta}(K) \sqrt{n}\theta \1_{\{\rho_0\leq X_{(k-1)/n}<\rho_0+\theta\}}.
\end{align*}
As $\1_{\{\rho_0\leq X_{(k-1)/n}<\rho_0+\theta'\}}\leq \1_{\{\rho_0\leq X_{(k-1)/n}<\rho_0+\theta\}}$, Lemma~\ref{lemma:upper_bound_transition_density} reveals
\begin{align*}
&\E_{\rho_0}\left[ \sup_{0\leq\theta'\leq\theta} \sup_{s\leq t} \frac{1}{\theta'}\left| \sum_{k=1}^{\lfloor ns\rfloor}\E_{\rho_0}\left[\left. \log\left( \frac{p_{1/n}^{\rho_0+\theta'}(X_{(k-1)/n},X_{k/n})}{p_{1/n}^{\rho_0}(X_{(k-1)/n},X_{k/n})}\right)\1_{I_{4,k}^\theta} \right| X_{(k-1)/n}\right]\right|\right] \\
&\hspace{1cm} \leq \E_{\rho_0}\left[ \sup_{\theta'\leq\theta}  \frac{1}{\theta'}\sum_{k=1}^{\lfloor nt\rfloor}\E_{\rho_0}\left[\left.\left| \log\left( \frac{p_{1/n}^{\rho_0+\theta'}(X_{(k-1)/n},X_{k/n})}{p_{1/n}^{\rho_0}(X_{(k-1)/n},X_{k/n})}\right)\right|\1_{I_{4,k}^\theta} \right| X_{(k-1)/n}\right]\right] \\
&\hspace{1cm} \leq C_{\alpha,\beta}(K)  \sum_{k=1}^{\lfloor nt\rfloor} \E_{\rho_0}\left[ \sup_{\theta'\leq\theta}  \sqrt{n} \1_{\{\rho_0\leq X_{(k-1)/n}<\rho_0+\theta'\}} \right] \\
&\hspace{1cm} \leq C_{\alpha,\beta}(K) \sqrt{n} \sum_{k=1}^{\lfloor nt\rfloor} \E_{\rho_0}\left[\1_{\{\rho_0\leq X_{(k-1)/n}<\rho_0+\theta\}} \right] \\
&\hspace{1cm} \leq C_{\alpha,\beta}(K) \sqrt{n} \sum_{k=1}^{\lfloor nt\rfloor}  \int_{\rho_0}^{\rho_0+\theta} \frac{1}{\sqrt{(k-1)/n}}\exp\left(-\frac{(y-x_0)^2}{2\max\{\alpha^2,\beta^2\}(k-1)/n}\right) dy \\
&\hspace{1cm} \leq C_{\alpha,\beta}(K) n\theta \sum_{k=1}^{\lfloor nt\rfloor} \frac{1}{\sqrt{k}} = C_{\alpha,\beta}(K) n^{3/2}\theta \sqrt{t}.
\end{align*}\smallskip

\item[$\bullet\ \boldsymbol{j=5}$.] First, we note that for $\rho_0\leq X_{(k-1)/n},X_{k/n}\leq\rho_0+\theta$ and $\theta\leq K/\sqrt{n}$, Lemma~\ref{lemma:upper_bound_transition_density} reveals
\begin{align}\label{eq_bound:P2/P1}
\begin{split}
&\frac{P_2^{\rho_0}(X_{(k-1)/n},X_{k/n};1/n)}{P_1^{\rho_0+\xi_k}(X_{(k-1)/n},X_{k/n};1/n)}\\
&\hspace{1cm} \leq \frac{\max\{\alpha,\beta\}^2}{\min\{\alpha,\beta\}^2} \exp\left(\frac{(X_{k/n}-X_{(k-1)/n})^2}{2\min\{\alpha^2,\beta^2\}/n}-\frac{(X_{k/n}-X_{(k-1)/n})^2}{2\max\{\alpha^2,\beta^2\}/n}\right)\leq C_{\alpha,\beta}(K).
\end{split}
\end{align}
Then, by the expression~\eqref{eq:expression_case5_Taylor} and boundedness of $x\mapsto |x|\exp(-x^2/4)$,
\begin{align*}
&\E_{\rho_0}\left[\left. \left|\log\left( \frac{p_{1/n}^{\rho_0+\theta}(X_{(k-1)/n},X_{k/n})}{p_{1/n}^{\rho_0}(X_{(k-1)/n},X_{k/n})}\right)\right|\1_{I_{5,k}^\theta} \right| X_{(k-1)/n}\right]  \\
&\hspace{0.2cm} \leq C_{\alpha,\beta}(K)\theta \1_{\{\rho_0\leq X_{(k-1)/n}<\rho_0+\theta\}} \int_{\rho_0}^{\rho_0+\theta} \sqrt{n} \frac{2|y-2\rho_0-2\xi+X_{(k-1)/n}|}{\alpha^2/n}\\
&\hspace{6.5cm} \cdot \exp\left(-\frac{(y-2\rho_0-2\xi+X_{(k-1)/n})^2}{2\alpha^2/n}\right) dy \\
&\hspace{0.2cm} \leq C_{\alpha,\beta}(K) n\theta \1_{\{\rho_0\leq X_{(k-1)/n}<\rho_0+\theta\}} \int_{\rho_0}^{\rho_0+\theta}  \exp\left(-\frac{(y-2\rho_0-2\xi+X_{(k-1)/n})^2}{4\alpha^2/n}\right) dy \\
&\hspace{0.2cm} \leq C_{\alpha,\beta}(K) n\theta^2 \1_{\{\rho_0\leq X_{(k-1)/n}<\rho_0+\theta\}}.
\end{align*}
As $\1_{\{\rho_0\leq X_{(k-1)/n}<\rho_0+\theta'\}}\leq \1_{\{\rho_0\leq X_{(k-1)/n}<\rho_0+\theta\}}$ and $\theta\leq K/\sqrt{n}$, Lemma~\ref{lemma:upper_bound_transition_density} reveals
\begin{align*}
&\E_{\rho_0}\left[ \sup_{0\leq\theta'\leq\theta} \sup_{s\leq t}\left| \frac{1}{\theta'}\sum_{k=1}^{\lfloor ns\rfloor}\E_{\rho_0}\left[\left. \log\left( \frac{p_{1/n}^{\rho_0+\theta'}(X_{(k-1)/n},X_{k/n})}{p_{1/n}^{\rho_0}(X_{(k-1)/n},X_{k/n})}\right)\1_{I_{5,k}^\theta} \right| X_{(k-1)/n}\right]\right|\right] \\
&\hspace{0.5cm} \leq \E_{\rho_0}\left[ \sup_{0\leq\theta'\leq\theta}  \frac{1}{\theta'}\sum_{k=1}^{\lfloor nt\rfloor}\E_{\rho_0}\left[\left.\left| \log\left( \frac{p_{1/n}^{\rho_0+\theta'}(X_{(k-1)/n},X_{k/n})}{p_{1/n}^{\rho_0}(X_{(k-1)/n},X_{k/n})}\right)\right|\1_{I_{5,k}^\theta} \right| X_{(k-1)/n}\right]\right] \\
&\hspace{0.5cm} \leq C_{\alpha,\beta}(K)  \sum_{k=1}^{\lfloor nt\rfloor} \E_{\rho_0}\left[ \sup_{\theta'\leq\theta} n\theta' \1_{\{\rho_0\leq X_{(k-1)/n}<\rho_0+\theta'\}} \right] \\
&\hspace{0.5cm} \leq C_{\alpha,\beta}(K) n\theta \sum_{k=1}^{\lfloor nt\rfloor} \E_{\rho_0}\left[\1_{\{\rho_0\leq X_{(k-1)/n}<\rho_0+\theta\}} \right] \\
&\hspace{0.5cm} \leq C_{\alpha,\beta}(K) n\theta \sum_{k=1}^{\lfloor nt\rfloor}  \int_{\rho_0}^{\rho_0+\theta} \frac{1}{\sqrt{(k-1)/n}}\exp\left(-\frac{(y-x_0)^2}{2\max\{\alpha^2,\beta^2\}(k-1)/n}\right) dy \\
&\hspace{0.5cm} \leq C_{\alpha,\beta}(K) n\theta \sum_{k=1}^{\lfloor nt\rfloor} \sqrt{n}\theta\frac{1}{\sqrt{k}} \leq C_{\alpha,\beta}(K) n^{3/2}\theta \sqrt{t}.
\end{align*}\smallskip

\item[$\bullet\ \boldsymbol{j=6}$.] The same argument as for $j=4$ gives with~\eqref{eq:expression_case6_bound}
\begin{align*}
&\E_{\rho_0}\left[ \sup_{0\leq\theta'\leq\theta}\sup_{s\leq t} \frac{1}{\theta'}\left| \sum_{k=1}^{\lfloor ns\rfloor}\E_{\rho_0}\left[\left. \log\left( \frac{p_{1/n}^{\rho_0+\theta}(X_{(k-1)/n},X_{k/n})}{p_{1/n}^{\rho_0}(X_{(k-1)/n},X_{k/n})}\right)\1_{I_{6,k}^\theta} \right| X_{(k-1)/n}\right]\right|\right] \leq C_{\alpha,\beta} n^{3/2} \theta \sqrt{t}.
\end{align*} \smallskip

\item[$\bullet\ \boldsymbol{j=7}$.] Arguing as for $j=3$, we find with~\eqref{eq:expression_case7}
\begin{align*}
&\sum_{k=1}^{\lfloor nt\rfloor} \E_{\rho_0}\left[\left. \log\left( \frac{p_{1/n}^{\rho_0+\theta}(X_{(k-1)/n},X_{k/n})}{p_{1/n}^{\rho_0}(X_{(k-1)/n},X_{k/n})}\right)\1_{I_{7,k}^\theta} \right| X_{(k-1)/n}\right] \\
&\hspace{0.1cm} =- \theta\left(\frac{1}{\alpha}-\frac{1}{\beta}\right)\frac{2\beta}{\alpha+\beta}\frac{1}{\sqrt{2\pi/n}} \sum_{k=1}^{\lfloor nt\rfloor}\1_{\{X_{(k-1)/n}\geq\rho_0\}} \exp\left( - \frac{(X_{(k-1)/n}-\rho_0)^2}{2\beta^2/n}\right) + r_n^{7}(t,\theta).
\end{align*}
where $r_n^7(t,\theta) := r_n^{7,1}(t,\theta)+r_n^{7,2}(t,\theta)$ with
\begin{align*}
r_n^{7,1}(t,\theta) &:= \frac12\left(\frac{1}{\alpha}-\frac{1}{\beta}\right)^2 n\theta^2 \sum_{k=1}^{\lfloor nt\rfloor} \1_{\{X_{(k-1)/n}\geq \rho_0+\theta\}}\E_{\rho_0}\left[ \1_{\{X_{k/n}\leq\rho_0\}}\mid X_{(k-1)/n}\right] \\
r_n^{7,2}(t,\theta) &:= - \theta\left(\frac{1}{\alpha}-\frac{1}{\beta}\right)\frac{2\beta}{\alpha+\beta}\frac{1}{\sqrt{2\pi/n}}\sum_{k=1}^{\lfloor nt\rfloor} \left(\1_{\{X_{(k-1)/n}\geq\rho_0+\theta\}} -\1_{\{X_{(k-1)/n}\geq\rho_0\}}\right)\\
&\hspace{7.5cm} \cdot \exp\left( - \frac{(X_{(k-1)/n}-\rho_0)^2}{2\beta^2/n}\right). 
\end{align*}
We discuss these terms seperately:\smallskip
\begin{itemize}
\item[$\boldsymbol{- r_n^{7,1}}$.] Because $\1_{\{X_{(k-1)/n}\geq \rho_0+\theta'\}}\leq \1_{\{X_{(k-1)/n}\geq \rho_0\}}$ for $\theta'\geq 0$, we get along the lines of the evaluation of $r_n^{3,2}$ that
\begin{align*}
\E_{\rho_0}\left[ \sup_{0\leq\theta'\leq\theta}\sup_{s\leq t} \frac{|r_n^{7,1}(s,\theta')|}{\theta'}\right] &\leq C_{\alpha,\beta} n\theta \sum_{k=1}^{\lfloor nt\rfloor} \E_{\rho_0}\left[ \1_{\{X_{k/n}\leq \rho_0\leq X_{(k-1)/n}\}}\right] \leq C_{\alpha,\beta} n^{3/2}\theta \sqrt{t}.
\end{align*}\smallskip

\item[$\boldsymbol{- r_n^{7,2}}$.]Since $\1_{\{X_{(k-1)/n}\geq\rho_0+\theta\}} -\1_{\{X_{(k-1)/n}\geq\rho_0\}} = -\1_{\{\rho_0\leq X_{(k-1)/n}<\rho_0+\theta\}}$ and $\1_{\{\rho_0\leq X_{(k-1)/n}<\rho_0+\theta'\}}\leq \1_{\{\rho_0\leq X_{(k-1)/n}<\rho_0+\theta\}}$, by Lemma~\ref{lemma:upper_bound_transition_density},
\begin{align*}
&\E_{\rho_0}\left[\sup_{0\leq\theta'\leq\theta}\sup_{s\leq t} \frac{| r_n^{7,2}(s,\theta')|}{\theta'}\right] \\
&\hspace{1cm}\leq C_{\alpha,\beta} \sqrt{n} \E_{\rho_0}\left[\sup_{\theta'\leq\theta} \sum_{k=1}^{\lfloor nt\rfloor} \1_{\{\rho_0\leq X_{(k-1)/n}<\rho_0+\theta'\}}\exp\left( - \frac{(X_{(k-1)/n}-\rho_0)^2}{2\beta^2/n}\right) \right]\\
&\hspace{1cm}\leq C_{\alpha,\beta}\sqrt{n}\sum_{k=1}^{\lfloor nt\rfloor} \frac{1}{\sqrt{(k-1)/n}}\int_{\rho_0}^{\rho_0+\theta}  \exp\left(-\frac{(y-x_0)^2}{2(k-1)\max\{\alpha^2,\beta^2\}/n}\right) dy \\
&\hspace{1cm}\leq C_{\alpha,\beta} n\theta \sum_{k=1}^{\lfloor nt\rfloor} \frac{1}{\sqrt{k}} \leq C_{\alpha,\beta} n^{3/2}\theta \sqrt{t}.
\end{align*}
\end{itemize}
\smallskip

\item[$\bullet\ \boldsymbol{j=8}$.] By the same arguments used for $j=2$ we find with~\eqref{eq:expression_case8}
\begin{align*}
&\sum_{k=1}^{\lfloor nt\rfloor} \E_{\rho_0}\left[\left. \log\left( \frac{p_{1/n}^{\rho_0+\theta}(X_{(k-1)/n},X_{k/n})}{p_{1/n}^{\rho_0}(X_{(k-1)/n},X_{k/n})}\right)\1_{I_{8,k}^\theta} \right| X_{(k-1)/n}\right] \\
&\hspace{0.1cm} =\theta \frac{2}{\alpha+\beta}\frac{\alpha}{\beta}\frac{1}{\sqrt{2\pi/n}}\log\left(\frac{\beta^2}{\alpha^2}\right)\sum_{k=1}^{\lfloor nt\rfloor} \1_{\{X_{(k-1)/n}\geq\rho_0\}}\exp\left(-\frac{(X_{(k-1)/n}-\rho_0)^2}{2\beta^2/n}\right) + r_n^{8}(t,\theta),
\end{align*}
where $r_n^8(t,\theta)=r_n^{8,1}(t,\theta)+ r_n^{8,2}(t,\theta)+r_n^{8,3}(t,\theta)$ with
\begin{align*}
r_n^{8,1}(t,\theta) & := \sum_{k=1}^{\lfloor nt\rfloor} \E_{\rho_0}\left[ R_8(k,0,\theta)\mid X_{(k-1)/n}\right], \\
r_n^{8,2}(t,\theta) & := \sum_{k=1}^{\lfloor nt\rfloor} \1_{\{X_{(k-1)/n}\geq \rho_0+\theta\}}\int_{\rho_0}^{\rho_0+\theta} \left( P_2^{\rho_0}(X_{(k-1)/n},y;1/n) - P_2^{\rho_0}(X_{(k-1)/n},\rho_0;1/n)\right)dy, \\
r_n^{8,3}(t,\theta) & := \theta \frac{2}{\alpha+\beta}\frac{\alpha}{\beta}\frac{1}{\sqrt{2\pi/n}}\log\left(\frac{\beta^2}{\alpha^2}\right)\sum_{k=1}^{\lfloor nt\rfloor} \left(\1_{\{X_{(k-1)/n}\geq\rho_0+\theta\}}-\1_{\{X_{(k-1)/n}\geq\rho_0\}}\right)\\
&\hspace{7.5cm} \cdot\exp\left(-\frac{(X_{(k-1)/n}-\rho_0)^2}{2\beta^2/n}\right).
\end{align*}
We discuss these terms seperately:\smallskip
\begin{itemize}
\item[$\boldsymbol{- r_n^{8,1}}$.] From the bound~\eqref{eq:expression_case8_bound} for $R_8(k,0,\theta)$, Lemma~\ref{lemma:upper_bound_transition_density} and boundedness of the function $x\mapsto (1+x^2)\exp(-x^2/2)$, we find
\begin{align*}
&\E_{\rho_0}\left[\left. |R_8(k,0,\theta)|\right| X_{(k-1)/n}\right] \\
&\hspace{0.2cm} \leq C_{\alpha,\beta}\theta \E_{\rho_0}\left[\left. \left( \sqrt{n}K + n|X_{(k-1)/n}-\rho_0-\theta|\right) \1_{\{\rho_0<X_{k/n}\leq\rho_0+\theta\leq X_{(k-1)/n}\}}\right| X_{(k-1)/n}\right] \\
&\hspace{0.2cm} \leq C_{\alpha,\beta}\theta \1_{\{X_{(k-1)/n}\geq\rho_0+\theta\}} \left( \sqrt{n}K + n|X_{(k-1)/n}-\rho_0-\theta|\right)\\
&\hspace{5cm} \cdot\int_{\rho_0}^{\rho_0+\theta} \sqrt{n}\exp\left(-\frac{(y-X_{(k-1)/n})^2}{2\max\{\alpha^2,\beta^2\}/n}\right) dy \\
&\hspace{0.2cm} \leq C_{\alpha,\beta} n\theta^2 \left( K + \sqrt{n}|X_{(k-1)/n}-\rho_0-\theta|\right) \1_{\{X_{(k-1)/n}\geq\rho_0+\theta\}}\\
&\hspace{5cm} \cdot\exp\left(-\frac{(X_{(k-1)/n}-\rho_0-\theta)^2}{2\max\{\alpha^2,\beta^2\}/n}\right) \\
&\hspace{0.2cm} \leq C_{\alpha,\beta}(K) n\theta^2 \1_{\{X_{(k-1)/n}\geq\rho_0+\theta\}}  \exp\left(-\frac{(X_{(k-1)/n}-\rho_0-\theta)^2}{4\max\{\alpha^2,\beta^2\}/n}\right) \\
&\hspace{0.2cm} \leq C_{\alpha,\beta}(K) n\theta^2 \1_{\{X_{(k-1)/n}\geq\rho_0+\theta\}} \exp\left(-\frac{(X_{(k-1)/n}-\rho_0)^2}{4\max\{\alpha^2,\beta^2\}/n}\right)\exp\left(\frac{(X_{(k-1)/n}-\rho_0)\theta}{2\max\{\alpha^2,\beta^2\}/n}\right).
\end{align*}
Now we have, using $X_{(k-1)/n}\geq\rho_0+\theta$ and $0\leq\theta\leq K/\sqrt{n}$ that
\begin{align}\label{eq_proof:r81}
n(X_{(k-1)/n}-\rho_0)\theta \leq 
\begin{cases}
8K^2, &\quad \textrm{ if } X_{(k-1)/n}-\rho_0\leq 8K/\sqrt{n}, \\
\frac{(X_{(k-1)/n}-\rho_0)^2}{8/n}, &\quad \textrm{ if } X_{(k-1)/n}-\rho_0 >8K/\sqrt{n}.
\end{cases}
\end{align} 
In particular
\[ \E_{\rho_0}\left[\left. |R_8(k,0,\theta)|\right| X_{(k-1)/n}\right] \leq C_{\alpha,\beta}(K) n\theta^2 \exp\left(-\frac{(X_{(k-1)/n}-\rho_0)^2}{8\max\{\alpha^2,\beta^2\}/n}\right).\]
As the random exp-part of this bound is independent of $\theta$, we easily deduce with Corollary~\ref{cor:bound_exp_X^2} that
\[ \E_{\rho_0}\left[\sup_{0\leq\theta'\leq\theta}\sup_{s\leq t}\frac{|r_n^{8,1}(s,\theta')|}{\theta'}\right]\leq C_{\alpha,\beta} n^{3/2}\theta \sqrt{t}.\]\smallskip

\item[$\boldsymbol{- r_n^{8,2}}$.] By a first order Taylor expansion and boundedness of $x\mapsto x\exp(-x^2/4)$, we find for intermediate values $\rho_0\leq \xi_y\leq \rho_0+y$ appearing in the Lagrange version of the remainder term that
\begin{align*}
&\left| P_2^{\rho_0}(X_{(k-1)/n},y;1/n) - P_2^{\rho_0}(X_{(k-1)/n},\rho_0;1/n)\right|\1_{\{X_{(k-1)/n}\geq \rho_0+\theta\}} \\
&\hspace{0.2cm} \leq C_{\alpha,\beta}\sqrt{n}(y-\rho_0) \left[ \frac{X_{(k-1)/n}-\xi_y}{\beta^2/n}\exp\left(-\frac{(X_{(k-1)/n}-\xi_y)^2}{2\beta^2/n}\right) \right. \\
&\hspace{1cm} \left. + \frac{X_{(k-1)/n}-2\rho_0+\xi_y}{\beta^2/n} \exp\left(-\frac{(X_{(k-1)/n}-2\rho_0+\xi_y)^2}{2\beta^2/n}\right) \right]\1_{\{X_{(k-1)/n}\geq \rho_0+\theta\}} \\
&\hspace{0.2cm} \leq C_{\alpha,\beta} n\theta \left[ \exp\left(-\frac{(X_{(k-1)/n}-\xi_y)^2}{4\beta^2/n}\right)+\exp\left(-\frac{(X_{(k-1)/n}-2\rho_0+\xi_y)^2}{4\beta^2/n}\right)\right]\\
&\hspace{1cm} \cdot\1_{\{X_{(k-1)/n}\geq \rho_0+\theta\}}.
\end{align*}
Thus, 
\begin{align*}
&\1_{\{X_{(k-1)/n}\geq \rho_0+\theta\}}\left|\int_{\rho_0}^{\rho_0+\theta} \left( P_2^{\rho_0}(X_{(k-1)/n},y;1/n) - P_2^{\rho_0}(X_{(k-1)/n},\rho_0;1/n)\right)dy \right| \\
&\hspace{0.2cm} \leq C_{\alpha,\beta} n\theta \1_{\{X_{(k-1)/n}\geq \rho_0+\theta\}} \int_{\rho_0}^{\rho_0+\theta}\left(\exp\left(-\frac{(X_{(k-1)/n}-\xi_y)^2}{4\beta^2/n}\right) \right. \\
&\hspace{6cm} \left. +\exp\left(-\frac{(X_{(k-1)/n}-2\rho_0+\xi_y)^2}{4\beta^2/n}\right)\right) dy \\
&\hspace{0.2cm} \leq C_{\alpha,\beta} n\theta^2 \1_{\{X_{(k-1)/n}\geq \rho_0+\theta\}}\left( \exp\left(-\frac{(X_{(k-1)/n}-\rho_0-\theta)^2}{4\beta^2/n}\right)\right. \\
&\hspace{6cm} \left. + \exp\left(-\frac{(X_{(k-1)/n}-\rho_0)^2}{4\beta^2/n}\right)\right) \\
&\hspace{0.2cm} \leq C_{\alpha,\beta}(K) n\theta^2\exp\left(-\frac{(X_{(k-1)/n}-\rho_0)^2}{8\beta^2/n}\right),
\end{align*}
where the second step follows by uniformly bounding the integrand on the integration domain and the last step follows from~\eqref{eq_proof:r81} as in the treatment of $r_n^{8,1}(t,\theta)$. Then, as the random exp-part of this bound is independent of $\theta$, with Corollary~\ref{cor:bound_exp_X^2},
\[ \E_{\rho_0}\left[\sup_{0\leq\theta'\leq\theta}\sup_{s\leq t}\frac{|r_n^{8,2}(s,\theta')|}{\theta'}\right]\leq C_{\alpha,\beta} n^{3/2}\theta \sqrt{t}.\]\smallskip

\item[$\boldsymbol{- r_n^{8,3}}$.] Here, the evaluation is done analogously to that of $r_n^{7,2}(t,\theta)$ as $r_n^{8,3}(t,\theta)$ and $r_n^{7,2}(t,\theta)$ differ only in a factor $C_{\alpha,\beta}$.\smallskip
\end{itemize}

\item[$\bullet\ \boldsymbol{j=9}$.] By the same arguments used for $j=1$, the remainder $r_n^{3,1}(t,\theta)$ and the remainder term $r_n^{6,2}(t,\theta)$, we find
\begin{align*}
&\sum_{k=1}^{\lfloor nt\rfloor} \E_{\rho_0}\left[\left. \log\left( \frac{p_{1/n}^{\rho_0+\theta}(X_{(k-1)/n},X_{k/n})}{p_{1/n}^{\rho_0}(X_{(k-1)/n},X_{k/n})}\right)\1_{I_{9,k}^\theta} \right| X_{(k-1)/n}\right] \\
&\hspace{0.5cm} = \theta \frac{\alpha-\beta}{\alpha+\beta}\frac{2}{\sqrt{2\pi\beta^2/n}}\sum_{k=1}^{\lfloor nt\rfloor} \1_{\{X_{(k-1)/n}\geq \rho_0\}}\exp\left(-\frac{(X_{(k-1)/n}-\rho_0)^2}{2\beta^2/n}\right) +r_n^9(t,\theta),
\end{align*}
where $r_n^9(t,\theta):=r_n^{9,1}(t,\theta)+r_n^{9,2}(t,\theta)+r_n^{9,3}(t,\theta)$ with
\begin{align*}
r_n^{9,1}(t,\theta) &:= \theta \frac{\alpha-\beta}{\alpha+\beta}\frac{2}{\sqrt{2\pi\beta^2/n}}\sum_{k=1}^{\lfloor nt\rfloor} \1_{\{X_{(k-1)/n}\geq \rho_0+\theta\}} \left[\exp\left(-\frac{(X_{(k-1)/n}-\rho_0+\theta)^2}{2\beta^2/n}\right) \right. \\
&\hspace{6cm} \left. -\exp\left(-\frac{(X_{(k-1)/n}-\rho_0)^2}{2\beta^2/n}\right)\right], \\
r_n^{9,2}(t,\theta) &:= \theta \frac{\alpha-\beta}{\alpha+\beta}\frac{2}{\sqrt{2\pi\beta^2/n}}\sum_{k=1}^{\lfloor nt\rfloor} \left(\1_{\{X_{(k-1)/n}\geq \rho_0+\theta\}}- \1_{\{X_{(k-1)/n}\geq \rho_0\}}\right)\\
&\hspace{6cm} \cdot\exp\left(-\frac{(X_{(k-1)/n}-\rho_0)^2}{2\beta^2/n}\right), \\
r_n^{9,3}(t,\theta) &:= \sum_{k=1}^{\lfloor nt\rfloor} \E_{\rho_0}\left[R_9(k,0,\theta)\mid X_{(k-1)/n}\right].
\end{align*}
We discuss these terms seperately:\smallskip
\begin{itemize}
\item[$\boldsymbol{- r_n^{9,1}}$.] By a first order Taylor expansion we find with the intermediate value $0\leq \xi_k\leq \theta$ in the Lagrange remainder,
\begin{align*}
&\1_{\{X_{(k-1)/n}\geq \rho_0+\theta\}}\left|\exp\left(-\frac{(X_{(k-1)/n}-\rho_0+\theta)^2}{2\beta^2/n}\right)  -\exp\left(-\frac{(X_{(k-1)/n}-\rho_0)^2}{2\beta^2/n}\right)\right| \\
&\hspace{0.5cm} = \sqrt{n}\theta \1_{\{X_{(k-1)/n}\geq \rho_0+\theta\}}\frac{|X_{(k-1)/n}-\rho_0-\xi_k|}{\beta^2/n}\exp\left(-\frac{(X_{(k-1)/n}-\rho_0-\xi_k)^2}{2\beta^2/n}\right) \\
&\hspace{0.5cm} \leq C_{\alpha,\beta}\sqrt{n}\theta \1_{\{X_{(k-1)/n}\geq \rho_0+\theta\}}\exp\left(-\frac{(X_{(k-1)/n}-\rho_0-\xi_k)^2}{4\beta^2/n}\right) \\
&\hspace{0.5cm} \leq C_{\alpha,\beta}(K) \sqrt{n}\theta \1_{\{X_{(k-1)/n}\geq \rho_0+\theta\}}\exp\left(-\frac{(X_{(k-1)/n}-\rho_0)^2}{8\beta^2/n}\right),
\end{align*}
where the second step used boundedness of $x\mapsto |x|\exp(-x^2/4)$ and the last one follows by~\eqref{eq_proof:r81} with the same argument used for $r_n^{8,1}(t,\theta)$. Thus, 
\[ \left|r_n^{9,1}(t,\theta)\right|  \leq C_{\alpha,\beta}(K) n\theta^2 \sum_{k=1}^{\lfloor nt\rfloor} \1_{\{X_{(k-1)/n}\geq \rho_0\}}\exp\left(-\frac{(X_{(k-1)/n}-\rho_0)^2}{2\beta^2/n}\right) \]
and by Corollary~\ref{cor:bound_exp_X^2},
\[ \E_{\rho_0}\left[\sup_{0\leq \theta'\leq\theta}\sup_{s\leq t}\frac{|r_n^{9,1}(s,\theta')|}{\theta'}\right]\leq C_{\alpha,\beta} n^{3/2}\theta \sqrt{t}.\]\smallskip

\item[$\boldsymbol{- r_n^{9,2}}$.] Here, the evaluation is done analogously to that of $r_n^{7,2}(t,\theta)$ as $r_n^{9,2}(t,\theta)$ and $r_n^{7,2}(t,\theta)$ differ only in a factor $C_{\alpha,\beta}$.\smallskip

\item[$\boldsymbol{- r_n^{9,3}}$.] We first observe
\begin{align}\label{eq_bound:P2/P2}
\begin{split}
&\frac{P_2^{\rho_0}(X_{(k-1)/n},X_{k/n};1/n)}{P_2^{\rho_0+\xi_k}(X_{(k-1)/n},X_{k/n};1/n)}\\
&\hspace{0.5cm}= \frac{1+\frac{\alpha-\beta}{\alpha+\beta}\exp\left(-\frac{2}{\beta^2/n}(X_{k/n}-\rho_0)(X_{(k-1)/n}-\rho_0)\right)}{1+\frac{\alpha-\beta}{\alpha+\beta}\exp\left(-\frac{2}{\beta^2/n}(X_{k/n}-\rho_0-\xi_k)(X_{(k-1)/n}-\rho_0-\xi_k)\right)} \leq C_{\alpha,\beta} 
\end{split}
\end{align} 
for $X_{(k-1)/n},X_{k/n}\geq \rho_0+\theta$. Then, by an analougous computation as in $j=1$ for $R_1(k,0,\theta)$, we find
\begin{align*}
&\E_{\rho_0}\left[R_9(k,0,\theta)\mid X_{(k-1)/n}\right]\\
&\hspace{1cm}\leq C_{\alpha,\beta} n\theta^2 \1_{\{X_{(k-1)/n}\geq\rho_0+\theta\}} \exp\left(-\frac{(X_{(k-1)/n}+\theta-\rho_0-2\xi)^2}{2\beta^2/n}\right) \\
&\hspace{1cm}\leq C_{\alpha,\beta} n\theta^2 \1_{\{X_{(k-1)/n}\geq\rho_0+\theta\}} \exp\left(-\frac{(X_{(k-1)/n}-\rho_0-\theta)^2}{2\beta^2/n}\right) \\
&\hspace{1cm}\leq C_{\alpha,\beta}(K) n\theta^2 \1_{\{X_{(k-1)/n}\geq\rho_0\}} \exp\left(-\frac{(X_{(k-1)/n}-\rho_0)^2}{4\beta^2/n}\right),
\end{align*} 
where the last step works as with~\eqref{eq_proof:r81} as in the treatment of $r_n^{8,1}(t,\theta)$. Then, by Corollary~\ref{cor:bound_exp_X^2},
\[ \E_{\rho_0}\left[\sup_{0\leq\theta'\leq\theta}\sup_{s\leq t}\frac{|r_n^{9,3}(s,\theta')|}{\theta'}\right]\leq C_{\alpha,\beta} n^{3/2}\theta \sqrt{t}.\]
\end{itemize}\bigskip

\end{itemize}
Finally, we sum up our results: Defining
\[ r_n(t,\theta) := \sum_{j\in\{1,2,3,7,8,9\}} r_n^j(t,\theta) + \sum_{j\in\{4,5,6\}}\sum_{k=1}^{\lfloor nt\rfloor} \E_{\rho_0}\left[\left. \log\left( \frac{p_{1/n}^{\rho_0+\theta}(X_{(k-1)/n},X_{k/n})}{p_{1/n}^{\rho_0}(X_{(k-1)/n},X_{k/n})}\right)\1_{I_{j,k}^\theta} \right| X_{(k-1)/n}\right]\]
we have proven the decomposition
\begin{align*}
&\sum_{j=1}^9 \sum_{k=1}^{\lfloor nt\rfloor} \E_{\rho_0}\left[\left. \log\left( \frac{p_{1/n}^{\rho_0+\theta}(X_{(k-1)/n},X_{k/n})}{p_{1/n}^{\rho_0}(X_{(k-1)/n},X_{k/n})}\right)\1_{I_{j,k}^\theta} \right| X_{(k-1)/n}\right] \\
&\hspace{0.5cm} = \frac{\theta}{\sqrt{2\pi/n}}\left(\frac{\alpha-\beta}{\alpha+\beta}\frac{2}{\alpha}+ \frac{2}{\alpha+\beta}\frac{\alpha}{\beta}\log\left(\frac{\beta^2}{\alpha^2}\right) -\left(\frac{1}{\alpha}-\frac{1}{\beta}\right)\frac{2\alpha}{\alpha+\beta} \right)\\
&\hspace{2cm}\cdot \sum_{k=1}^{\lfloor nt\rfloor} \1_{\{X_{(k-1)/n}< \rho_0\}}\exp\left(-\frac{(X_{(k-1)/n}-\rho_0)^2}{2\alpha^2/n}\right)\\
&\hspace{1cm} +\frac{\theta}{\sqrt{2\pi/n}} \left(-\frac{2\beta}{\alpha+\beta}\left(\frac{1}{\alpha}-\frac{1}{\beta}\right)+ \frac{2}{\alpha+\beta}\frac{\alpha}{\beta}\log\left(\frac{\beta^2}{\alpha^2}\right)+\frac{\alpha-\beta}{\alpha+\beta}\frac{2}{\beta}\right)\\
&\hspace{2cm} \cdot \sum_{k=1}^{\lfloor nt\rfloor} \1_{\{X_{(k-1)/n}\geq \rho_0\}}\exp\left(-\frac{(X_{(k-1)/n}-\rho_0)^2}{2\beta^2/n}\right)\\
&\hspace{1cm} + r_n(t,\theta).
\end{align*}
From this and some simple algebraic manipulations, we get the asserted decomposition in the statement of Proposition~\ref{prop:expansion_of_drift_t}. Moreover, we have shown that
\[ \E_{\rho_0}\left[\sup_{0\leq\theta'\leq\theta}\sup_{s\leq t} \frac{| r_n(s,\theta')|}{\theta'}\right] \leq C_{\alpha,\beta}(K) n^{3/2}\theta \sqrt{t}.\]
The statement for $\theta<0$ can be derived in exactly the same way. Here we use that 
\[ \log\left( \frac{p_{1/n}^{\rho_0+\theta}(X_{(k-1)/n},X_{k/n})}{p_{1/n}^{\rho_0}(X_{(k-1)/n},X_{k/n})}\right) = -\log\left( \frac{p_{1/n}^{\rho_0}(X_{(k-1)/n},X_{k/n})}{p_{1/n}^{\rho_0+\theta}(X_{(k-1)/n},X_{k/n})}\right)\]
and then apply the expansions~\eqref{eq:expression_case1_second_order}-\eqref{eq:expression_case9_second_order_bound} to this right-hand side.
\end{proof}

\begin{proof}[Proof of Proposition~\ref{prop:moment_product}]
The proof is done via induction on $m$. For the induction hypothesis, we investigate $m=1$ and set $d_1=d$, $k_1=k$ and $j_1=j$ for the ease of notation. We are going to show the claim for $Y_k^j=Z_k^j$ for each choice of $d\leq D$, $1\leq k\leq n$ and $j\in\{1,\dots, 9\}$. The claim for $Y_k^j=\overline{Z}_k^j$ is then a consequence of Jensen's inequality for conditional expectations. We will repeatedly use that
\begin{align}\label{eq_proof:inequality_xexp}
\begin{split}
\sup_{x\in\R} \left[(1+|x|)^d\exp(-x^2/4)\right] &\leq 2^d + \sup_{x\in\R} \left[|x|^d\exp(-x^2/4)\right] \\
& = 2^d + (2d)^d\exp(-d^2) \leq 2^D + (2D)^D =: C(D).
\end{split}
\end{align} 
Moreover, we will always assume $\theta'\leq\theta$ and denote by $\xi_k$ an intermediate value of a corresponding (Taylor-) expansion, in particular, we have $\theta'\leq \xi_k\leq\theta$.\smallskip
\begin{itemize}
\item[$\bullet\ \boldsymbol{j=1}.$]  With the expansion~\eqref{eq:expression_case1_first_order} and~\eqref{eq_bound:exp/P1},
\begin{align}\label{eq_proof:lemma_moments_Z_j_1}
\begin{split}
\left|Z_k^1(\theta',\theta)\right|^d &\leq C_{\alpha,\beta}|\theta-\theta'|^d \left|\frac{X_{k/n}-2\rho_0-2\xi_k+X_{(k-1)/n}}{P_1^{\rho_0+\xi_k}(X_{(k-1)/n},X_{k/n};1/n)}\right|^d n^{3d/2}\\
&\hspace{3cm} \cdot \exp\left(-\frac{d(X_{k/n}-2\rho_0-2\xi_k+X_{(k-1)/n})^2}{2\alpha^2/n}\right) \\
&\leq C_{\alpha,\beta}|\theta-\theta'|^d n^{d/2} \frac{|X_{k/n}-2\rho_0-2\xi_k+X_{(k-1)/n}|^d n^{d/2}}{P_1^{\rho_0+\xi_k}(X_{(k-1)/n},X_{k/n};1/n)}\\
&\hspace{3cm}\cdot  \sqrt{n} \exp\left(-\frac{(X_{k/n}-2\rho_0-2\xi_k +X_{(k-1)/n})^2}{2\alpha^2/n}\right).
\end{split}
\end{align}
To evaluate the conditional expectation of this bound under $\Pr_{\rho_0}$, we have to distinguish the cases $\theta'<0$ and $\theta'\geq 0$ in order to perform the calculation with the correct regime of the transition density. \smallskip

\begin{itemize}
\item[$\boldsymbol{- \theta'<0}$.] Here, $\{X_{(k-1)/n}<\rho_0+\theta',X_{k/n}\leq \rho_0+\theta'\}\subset \{X_{(k-1)/n},X_{k/n}\leq \rho_0\}$. Then, using~\eqref{eq_bound:P1/P1}, we find from~\eqref{eq_proof:inequality_xexp}, \eqref{eq_proof:lemma_moments_Z_j_1} and the Gaussian tail inequality,
\begin{align*}
&\E_{\rho_0}\left[\left.\left|Z_k^1(\theta',\theta)\right|^d \right| X_{(k-1)/n}\right] \\
&\hspace{0.5cm} \leq C_{\alpha,\beta}|\theta-\theta'|^d n^{d/2}\int_{-\infty}^{\rho_0+\theta'} \frac{|y-2\rho_0-2\xi_k+X_{(k-1)/n}|^d}{(1/n)^{d/2}}\\
&\hspace{5.5cm} \cdot\sqrt{n}\exp\left(-\frac{(y-2\rho_0-2\xi_k+X_{(k-1)/n})^2}{2\alpha^2/n}\right) dy \\
&\hspace{0.5cm}\leq C_{\alpha,\beta}|\theta-\theta'|^d n^{d/2}\int_{-\infty}^{(X_{(k-1)/n}-\rho_0-2\xi_k+\theta')/\sqrt{\alpha^2/n}} |y|^d \exp\left(-\frac{y^2}{2}\right) dy \\
&\hspace{0.5cm}\leq C_{\alpha,\beta}(D)|\theta-\theta'|^d n^{d/2}\exp\left(-\frac{(X_{(k-1)/n}-\rho_0-2\xi_k+\theta')^2}{4\alpha^2/n}\right).
\end{align*}\smallskip

\item[$\boldsymbol{- \theta'\geq 0}$.] Here, we have to work with different regimes of the transition density, depending in which case on the right-hand side of
\begin{align*}
&\{X_{(k-1)/n}<\rho_0+\theta',X_{k/n}\leq \rho_0+\theta'\}\\
&\hspace{0.5cm} = \{X_{(k-1)/n}<\rho_0,X_{k/n}\leq \rho_0\} \cup \{X_{(k-1)/n}<\rho_0< X_{k/n}\leq \rho_0+\theta'\} \\
&\hspace{1.5cm} \cup \{X_{k/n}\leq \rho_0\leq X_{(k-1)/n}< \rho_0+\theta'\}\\
&\hspace{1.5cm} \cup \{\rho_0\leq X_{(k-1)/n}< \rho_0+\theta', \rho_0<X_{k/n}\leq \rho_0+\theta'\}
\end{align*}
we are in. For the first of them, we get from~\eqref{eq_bound:P1/P1}, \eqref{eq_proof:inequality_xexp}, \eqref{eq_proof:lemma_moments_Z_j_1} and the Gaussian tail inequality,
\begin{align*}
&\E_{\rho_0}\left[\left.\1_{\{X_{(k-1)/n}<\rho_0,X_{k/n}\leq\rho_0\}}\left|Z_k^1(\theta',\theta)\right|^d \right| X_{(k-1)/n}\right] \\
&\hspace{0.5cm} \leq C_{\alpha,\beta}|\theta-\theta'|^d n^{d/2}\int_{-\infty}^{\rho_0} \frac{|y-2\rho_0-2\xi_k+X_{(k-1)/n}|^d}{(1/n)^{d/2}}\\
&\hspace{5.5cm} \cdot \sqrt{n}\exp\left(-\frac{(y-2\rho_0-2\xi_k+X_{(k-1)/n})^2}{2\alpha^2/n}\right) dy \\
&\hspace{0.5cm}\leq C_{\alpha,\beta}|\theta-\theta'|^d n^{d/2}\int_{-\infty}^{(X_{(k-1)/n}-\rho_0-2\xi_k)/\sqrt{\alpha^2/n}} |y|^d \exp\left(-\frac{y^2}{2}\right) dy \\
&\hspace{0.5cm}\leq C_{\alpha,\beta}(D)|\theta-\theta'|^d n^{d/2}\exp\left(-\frac{(X_{(k-1)/n}-\rho_0-2\xi_k)^2}{4\alpha^2/n}\right).
\end{align*}
Using~\eqref{eq_bound:P2/P1} instead of~\eqref{eq_bound:P1/P1}, the same computation reveals
\begin{align*}
&\E_{\rho_0}\left[\left.\1_{\{\rho_0\leq X_{(k-1)/n}< \rho_0+\theta', \rho_0<X_{k/n}\leq \rho_0+\theta'\}}\left|Z_k^1(\theta',\theta)\right|^d \right| X_{(k-1)/n}\right] \\
&\hspace{1cm}\leq C_{\alpha,\beta}|\theta-\theta'|^d n^{d/2}\exp\left(-\frac{(X_{(k-1)/n}-\rho_0-2\xi_k+\theta')^2}{4\alpha^2/n}\right).
\end{align*}
If $X_{(k-1)/n}<\rho_0< X_{k/n}\leq \rho_0+\theta'$, we have $|X_{k/n}-2\rho_0-2\xi_k+X_{(k-1)/n}|\leq 3K/\sqrt{n}+|X_{(k-1)/n}-\rho_0|$. After applying~\eqref{eq_bound:exp/P1} to~\eqref{eq_proof:lemma_moments_Z_j_1} another time, we then find with Lemma~\ref{lemma:upper_bound_transition_density} and~\eqref{eq_proof:inequality_xexp},
\begin{align*}
&\E_{\rho_0}\left[\left.\1_{\{X_{(k-1)/n}<\rho_0< X_{k/n}\leq \rho_0+\theta'\}}\left|Z_k^1(\theta',\theta)\right|^d \right| X_{(k-1)/n}\right] \\
&\hspace{0.2cm} \leq C_{\alpha,\beta} |\theta-\theta'|^d n^{d/2}\left( 3K + \sqrt{n}|X_{(k-1)/n}-\rho_0|\right)^d \\
&\hspace{5cm} \cdot\E_{\rho_0}\left[ \left. \1_{\{X_{(k-1)/n}<\rho_0< X_{k/n}\leq \rho_0+\theta'\}}\right| X_{(k-1)/n}\right] \\
&\hspace{0.2cm} \leq C_{\alpha,\beta} |\theta-\theta'|^d n^{d/2}\left( 3K + \sqrt{n}|X_{(k-1)/n}-\rho_0|\right)^d\1_{\{X_{(k-1)/n}<\rho_0\}} \\
&\hspace{5cm} \cdot\int_{\rho_0}^{\rho_0+\theta'} \sqrt{n}\exp\left(-\frac{(y-X_{(k-1)/n})^2}{2\max\{\alpha^2,\beta^2\}/n}\right) dy \\
&\hspace{0.2cm} \leq C_{\alpha,\beta} |\theta-\theta'|^d n^{d/2}\left( 3K + \sqrt{n}|X_{(k-1)/n}-\rho_0|\right)^d \theta'\sqrt{n} \exp\left(-\frac{(X_{(k-1)/n}-\rho_0)^2}{2\max\{\alpha^2,\beta^2\}/n}\right) \\
&\hspace{0.2cm} \leq C_{\alpha,\beta}(D,K) |\theta-\theta'|^d n^{d/2}\exp\left(-\frac{(X_{(k-1)/n}-\rho_0)^2}{4\max\{\alpha^2,\beta^2\}/n}\right).
\end{align*}
In the remaining case, i.e. for $X_{k/n}\leq \rho_0\leq X_{(k-1)/n}< \rho_0+\theta'$, we first apply~\eqref{eq_bound:exp/P1} to~\eqref{eq_proof:lemma_moments_Z_j_1}. Then, as $|X_{k/n}-2\rho_0-2\xi_k+X_{(k-1)/n}|\leq 3K/\sqrt{n}+|X_{k/n}-\rho_0|$, we bound with Lemma~\ref{lemma:upper_bound_transition_density}, \eqref{eq_bound:exp/P1} and the Gaussian tail inequality,
\begin{align*}
&\E_{\rho_0}\left[\left.\1_{\{X_{k/n}\leq \rho_0\leq X_{(k-1)/n}< \rho_0+\theta'\}}\left|Z_k^1(\theta',\theta)\right|^d \right| X_{(k-1)/n}\right] \\
&\hspace{0.2cm} \leq C_{\alpha,\beta} |\theta-\theta'|^d n^{d/2} \E_{\rho_0}\left[ \left. \1_{\{X_{k/n}\leq \rho_0\leq X_{(k-1)/n}< \rho_0+\theta'\}}\left( 3K + \sqrt{n}|X_{k/n}-\rho_0|\right)^d\right| X_{(k-1)/n}\right] \\
&\hspace{0.2cm} \leq C_{\alpha,\beta} |\theta-\theta'|^d n^{d/2} \1_{\{X_{(k-1)/n}\geq\rho_0\}} \\
&\hspace{3cm} \cdot\int_{-\infty}^{\rho_0} \left( 3K + \sqrt{n}|y-\rho_0|\right)^d \sqrt{n}\exp\left(-\frac{(y-X_{(k-1)/n})^2}{2\max\{\alpha^2,\beta^2\}/n}\right) dy \\
&\hspace{0.2cm} \leq C_{\alpha,\beta} |\theta-\theta'|^d n^{d/2}\int_{-\infty}^{\rho_0} \left( 3K + \sqrt{n}|y-\rho_0|\right)^d \sqrt{n}\exp\left(-\frac{(y-\rho_0)^2}{2\max\{\alpha^2,\beta^2\}/n}\right) dy \\
&\hspace{0.2cm} \leq C_{\alpha,\beta} |\theta-\theta'|^d n^{d/2}\int_{-\infty}^{\sqrt{n}(X_{(k-1)/n}-\rho_0)} \left( 3K + |y|\right)^d \exp\left(-\frac{y^2}{2\max\{\alpha^2,\beta^2\}}\right) dy \\
&\hspace{0.2cm} \leq C_{\alpha,\beta}(D,K) |\theta-\theta'|^d n^{d/2}\exp\left(-\frac{(X_{(k-1)/n}-\rho_0)^2}{4\max\{\alpha^2,\beta^2\}/n}\right).
\end{align*}
\end{itemize}
Now, estimating the expectation of the upper bounds in each of these four cases with Corollary~\ref{cor:bound_exp_X^2} gives
\[ \E_{\rho_0}\left[\left|Z_k^1(\theta',\theta)\right|^d \right] \leq C_{\alpha,\beta}(D,K) |\theta-\theta'|^d n^{d/2}\frac{1}{\sqrt{k}}\]
which is the stated bound of the lemma for $j=1$.\bigskip

\item[$\bullet\ \boldsymbol{j=2}.$] With the expansion~\eqref{eq:expression_case2} and~\eqref{eq:expression_case2_bound} we find with Lemma~\ref{lemma:upper_bound_transition_density}, \eqref{eq_proof:inequality_xexp} and our assumption $|\theta-\theta'|\leq 2K/\sqrt{n}$,
\begin{align*}
&\E_{\rho_0}\left[\left.\left|Z_k^2(\theta',\theta)\right|^d \right| X_{(k-1)/n}\right] \\
&\hspace{0.2cm} \leq C_{\alpha,\beta}|\theta-\theta'|^d \E_{\rho_0}\left[\left. \left( 2K\sqrt{n} + n|X_{(k-1)/n}-\rho_0-\theta|\right)^d \1_{\{X_{(k-1)/n}<\rho_0+\theta'<X_{k/n}\leq\rho_0+\theta\}}\right| X_{(k-1)/n}\right] \\
&\hspace{0.2cm} \leq C_{\alpha,\beta}|\theta-\theta'|^d \1_{\{X_{(k-1)/n}<\rho_0+\theta'\}} \left( 2K\sqrt{n} + n|X_{(k-1)/n}-\rho_0-\theta|\right)^d \\
&\hspace{5cm}\cdot\int_{\rho_0+\theta'}^{\rho_0+\theta} \sqrt{n}\exp\left(-\frac{(y-X_{(k-1)/n})^2}{2\max\{\alpha^2,\beta^2\}/n}\right) dy \\
&\hspace{0.2cm} \leq C_{\alpha,\beta}|\theta-\theta'|^d \sqrt{n}|\theta-\theta'| \left( 2K\sqrt{n} + n|X_{(k-1)/n}-\rho_0-\theta|\right)^d \exp\left(-\frac{(X_{(k-1)/n}-\rho_0-\theta')^2}{2\max\{\alpha^2,\beta^2\}/n}\right) \\
&\hspace{0.2cm} \leq C_{\alpha,\beta}|\theta-\theta'|^d (2K) n^{d/2} \left( 2K + \sqrt{n}|X_{(k-1)/n}-\rho_0-\theta|\right)^d \exp\left(-\frac{(X_{(k-1)/n}-\rho_0-\theta')^2}{2\max\{\alpha^2,\beta^2\}/n}\right) \\
&\hspace{0.2cm} \leq C_{\alpha,\beta}(K) |\theta-\theta'|^d n^{d/2} \left( 4K + \sqrt{n}|X_{(k-1)/n}-\rho_0-\theta'|\right)^d \exp\left(-\frac{(X_{(k-1)/n}-\rho_0-\theta')^2}{2\max\{\alpha^2,\beta^2\}/n}\right) \\
&\hspace{0.2cm} \leq C_{\alpha,\beta}(D,K) |\theta-\theta'|^d n^{d/2}  \exp\left(-\frac{(X_{(k-1)/n}-\rho_0-\theta')^2}{4\max\{\alpha^2,\beta^2\}/n}\right).
\end{align*}
In particular, Corollary~\ref{cor:bound_exp_X^2} then implies
\[ \E_{\rho_0}\left[\left|Z_k^2(\theta',\theta)\right|^d\right] \leq C_{\alpha,\beta}(D,K) |\theta-\theta'|^d n^{d/2} \frac{1}{\sqrt{k}}.\]\smallskip

\item[$\bullet\ \boldsymbol{j=3}.$] By the expansion~\eqref{eq:expression_case3}, inequality~\eqref{eq_proof:inequality_xexp} and the Gaussian tail inequality,
\begin{align*}
&\E_{\rho_0}\left[ \left.\left|Z_k^3(\theta',\theta)\right|^d \right| X_{(k-1)/n}\right] \\
&\hspace{0.1cm} \leq C_{\alpha,\beta} |\theta-\theta'|^d n^{d/2} \E_{\rho_0}\left[ \left( \frac{|\theta-\theta'|}{2/\sqrt{n}} + \frac{X_{k/n}-\rho_0-\theta}{\beta/\sqrt{n}} - \frac{X_{(k-1)/n}-\rho_0-\theta}{\alpha/\sqrt{n}}\right)^d \right.\\
&\hspace{6.5cm} \cdot \1_{\{X_{(k-1)/n}<\rho_0+\theta'\leq \rho_0+\theta< X_{k/n}\}}\Big| X_{(k-1)/n} \Bigg] \\
&\hspace{0.1cm} \leq C_{\alpha,\beta} |\theta-\theta'|^d n^{d/2} \E_{\rho_0}\left[ \left( K + 2\sqrt{n}|X_{k/n}-X_{(k-1)/n}| \right)^d \1_{\{X_{(k-1)/n}<\rho_0+\theta'\leq \rho_0+\theta< X_{k/n}\}}\Big| X_{(k-1)/n} \right] \\
&\hspace{0.1cm} \leq C_{\alpha,\beta}(D) |\theta-\theta'|^d n^{d/2} K^d \int_{\rho_0+\theta}^\infty \sqrt{n}\exp\left(-\frac{(y-X_{(k-1)/n})^2}{2\max\{\alpha^2,\beta^2\}/n}\right) dy \\
&\hspace{0.3cm} + C_{\alpha,\beta}(D) |\theta-\theta'|^d n^{d/2}\1_{\{X_{(k-1)/n}<\rho_0+\theta'\}}\int_{\rho_0+\theta}^\infty n^{d/2}|y-X_{(k-1)/n}|^d \sqrt{n}\exp\left(-\frac{(y-X_{(k-1)/n})^2}{2\max\{\alpha^2,\beta^2\}/n}\right) dy \\
&\hspace{0.1cm} \leq C_{\alpha,\beta}(D,K) |\theta-\theta'|^d n^{d/2}\exp\left(-\frac{(X_{(k-1)/n}-\rho_0-\theta)^2}{2\max\{\alpha^2,\beta^2\}/n}\right) \\
&\hspace{0.3cm} + C_{\alpha,\beta}(D) |\theta-\theta'|^d n^{d/2} \1_{\{X_{(k-1)/n}<\rho_0+\theta'\}} \int_{\sqrt{n}(\rho_0+\theta-X_{(k-1)/n})}^\infty |y|^d \exp\left(-\frac{y^2}{2\max\{\alpha^2,\beta^2\}}\right) dy \\
&\hspace{0.1cm} \leq C_{\alpha,\beta}(D,K) |\theta-\theta'|^d n^{d/2}\exp\left(-\frac{(X_{(k-1)/n}-\rho_0-\theta)^2}{4\max\{\alpha^2,\beta^2\}/n}\right).
\end{align*}
Corollary~\eqref{cor:bound_exp_X^2} then reveals 
\[ \E_{\rho_0}\left[ \left|Z_k^3(\theta',\theta)\right|^d \right] \leq C_{\alpha,\beta}(D,K) |\theta-\theta'|^d n^{d/2}\frac{1}{\sqrt{k}}.\]\smallskip

\item[$\bullet\ \boldsymbol{j=4}.$] With the expansion~\eqref{eq:expression_case4} and~\eqref{eq:expression_case4_bound} we find using $|\theta-\theta'|\leq 2K/\sqrt{n}$, \eqref{eq_proof:inequality_xexp} and the Gaussian tail inequality,
\begin{align*}
&\E_{\rho_0}\left[\left.\left|Z_k^4(\theta',\theta)\right|^d \right| X_{(k-1)/n}\right] \\
&\hspace{0.2cm} \leq C_{\alpha,\beta}|\theta-\theta'|^d \E_{\rho_0}\left[\left. \left( 2K\sqrt{n} + n|X_{k/n}-\rho_0-\theta|\right)^d \1_{\{X_{k/n}\leq\rho_0+\theta'\leq X_{(k-1)/n}<\rho_0+\theta\}}\right| X_{(k-1)/n}\right] \\
&\hspace{0.2cm} \leq C_{\alpha,\beta}|\theta-\theta'|^d n^{d/2} \1_{\{\rho_0+\theta'\leq X_{(k-1)/n}<\rho_0+\theta\}} \int_{-\infty}^{\rho_0+\theta'} \left( 2K + \sqrt{n}|y-\rho_0-\theta|\right)^d \\
&\hspace{8cm} \cdot\sqrt{n}\exp\left(-\frac{(y-X_{(k-1)/n})^2}{2\max\{\alpha^2,\beta^2\}/n}\right) dy \\
&\hspace{0.2cm} \leq C_{\alpha,\beta}|\theta-\theta'|^d n^{d/2} \1_{\{\rho_0+\theta'\leq X_{(k-1)/n}<\rho_0+\theta\}} \\
&\hspace{3cm} \cdot\int_{-\infty}^{\rho_0+\theta'} \left( 4K + \sqrt{n}|y-\rho_0-\theta'|\right)^d \sqrt{n}\exp\left(-\frac{(y-\rho_0-\theta')^2}{2\max\{\alpha^2,\beta^2\}/n}\right) dy \\
&\hspace{0.2cm} \leq C_{\alpha,\beta}(K)|\theta-\theta'|^d n^{d/2} \1_{\{\rho_0+\theta'\leq X_{(k-1)/n}<\rho_0+\theta\}}\int_{-\infty}^0 (1+|y|)^d\exp\left(-\frac{y^2}{2\max\{\alpha^2,\beta^2\}}\right)dy \\
&\hspace{0.2cm} \leq C_{\alpha,\beta}(D,K)|\theta-\theta'|^d n^{d/2} \1_{\{\rho_0+\theta'\leq X_{(k-1)/n}<\rho_0+\theta\}}.
\end{align*}
In particular, we find with Lemma~\ref{lemma:upper_bound_transition_density} and $\sqrt{n}|\theta-\theta'|\leq 2K$,
\begin{align*}
&\E_{\rho_0}\left[\left|Z_k^4(\theta',\theta)\right|^d \right] \\
&\hspace{1cm}\leq C_{\alpha,\beta}(D,K)|\theta-\theta'|^d n^{d/2} \int_{\rho_0+\theta'}^{\rho_0+\theta} \frac{1}{\sqrt{(k-1)/n}}\exp\left(-\frac{(y-x_0)^2}{2\max\{\alpha^2,\beta^2\}(k-1)/n}\right) dy \\
&\hspace{1cm} \leq C_{\alpha,\beta}(D,K) |\theta-\theta'|^d n^{d/2} \sqrt{n}|\theta-\theta'|\frac{1}{\sqrt{k}} \\
&\hspace{1cm}\leq C_{\alpha,\beta}(D,K)|\theta-\theta'|^d n^{d/2}\frac{1}{\sqrt{k}} .
\end{align*} \smallskip

\item[$\bullet\ \boldsymbol{j=5}.$] Using~\eqref{eq_bound:P2/P1} and following the steps from $j=1$, by the expression~\eqref{eq:expression_case5_Taylor},
\begin{align*}
&\E_{\rho_0}\left[\left.\left|Z_k^5(\theta',\theta)\right|^d \right| X_{(k-1)/n}\right] \\
&\hspace{0.5cm}\leq C_{\alpha,\beta}(D,K)|\theta-\theta'|^d n^{d/2}\exp\left(-\frac{(X_{(k-1)/n}-\rho_0-2\xi_k+\theta')^2}{4/n}\right)
\end{align*}
with
\[ \E_{\rho_0}\left[\left|Z_k^5(\theta',\theta)\right|^d \right] \leq C_{\alpha,\beta}(D,K) |\theta-\theta'|^d n^{d/2}\frac{1}{\sqrt{k}}. \]\smallskip

\item[$\bullet\ \boldsymbol{j=6}.$] This works the same as $j=4$ with the bound~\eqref{eq:expression_case6_bound} on the expansion~\eqref{eq:expression_case6}.\smallskip

\item[$\bullet\ \boldsymbol{j=7}.$] This works the same as $j=3$ with the expansion~\eqref{eq:expression_case7}.\smallskip

\item[$\bullet\ \boldsymbol{j=8}.$] This works the same as $j=2$ with the expansion~\eqref{eq:expression_case8} and the bound~\eqref{eq:expression_case8_bound}.\smallskip

\item[$\bullet\ \boldsymbol{j=9}.$] This works analogously to $j=1$ with the expansion~\eqref{eq:expression_case9_first_order} by replacing the bound~\eqref{eq_bound:P1/P1} with~\eqref{eq_bound:P2/P2}. Moreover, we now distinguish the cases $\theta\geq 0$ and $\theta<0$. In the first one, we have the inclusion $\{X_{(k-1)/n}\geq \rho_0+\theta,X_{k/n}>\rho_0+\theta\}\subset \{X_{(k-1)/n},X_{k/n}\geq \rho_0\}$ and we can evaluate the conditional expectation by integrating against $P_2^{\rho_0}$. For $\theta<0$, we have
\begin{align*}
&\{X_{(k-1)/n}\geq \rho_0+\theta, X_{k/n}> \rho_0+\theta\} \\
&\hspace{0.5cm}= \{X_{(k-1)/n}\geq\rho_0, X_{k/n}> \rho_0\} \cup \{\rho_0+\theta\leq X_{(k-1)/n}<\rho_0,\rho_0+\theta< X_{k/n}\leq \rho_0\} \\
&\hspace{1cm} \cup \{\rho_0+\theta \leq X_{(k-1)/n}< \rho_0< X_{k/n}\} \cup \{\rho_0+\theta < X_{k/n}\leq \rho_0\leq X_{(k-1)/n}\},
\end{align*}
and use the corresponding regime of the transition density for the first two cases. For the other two, it is sufficient to work with the bound provided in Lemma~\ref{lemma:upper_bound_transition_density} (see $j=1$).
\end{itemize}
\bigskip

\noindent
This completes the induction hypothesis for $m=1$. The next step is the induction step $m\mapsto m+1$. Here, we have
\[ \E_{\rho_0}\left[ \prod_{i=1}^{m+1} \left|Y_{k_i}^{j_i}(\theta',\theta)\right|^{d_i}\right] = \E_{\rho_0}\left[ \prod_{i=1}^m \left|Y_{k_i}^{j_i}(\theta',\theta)\right|^{d_i} \E_{\rho_0}\left[\left. \left|Y_{k_{m+1}}^{j_{m+1}}(\theta',\theta)\right|^{d_{m+1}} \right| \cF_{k_{m+1}-1}\right]\right]. \]
Within the proof of the induction hypothesis $m=1$, we have shown that for all $j_{m+1}\in \{1,\dots, 9\}$ it holds that
\begin{align*}
&\E_{\rho_0}\left[\left. \left|Z_{k_{m+1}}^{j_{m+1}}(\theta',\theta)\right|^{d_{m+1}} \right| \cF_{k_{m+1}-1}\right] \\
&\hspace{0.5cm} \leq C_{\alpha,\beta}(D,K)|\theta-\theta'|^{d_{m+1}} n^{d_{m+1}/2}\Bigg[ \1_{\{\rho_0+\theta'\leq X_{(k-1)/n}\leq \rho_0+\theta\}}\\
&\hspace{7cm} \left. +\exp\left(-\frac{(X_{k_{m+1}-1}-c(\rho_0,\theta',\theta))^2}{c'(\alpha,\beta)/n}\right) \right].
\end{align*} 
for some constant $c(\rho_0,\theta',\theta)\in\R$ and $c'(\alpha,\beta)>0$. Consequently, 
\begin{align}\label{eq_proof:lemma:moments_Z_j}
\begin{split}
\E_{\rho_0}\left[ \prod_{i=1}^{m+1} \left|Y_{k_i}^{j_i}(\theta',\theta)\right|^{d_i}\right] &\leq C_{\alpha,\beta}(D,K)|\theta-\theta'|^{d_{m+1}} n^{d_{m+1}/2}\E_{\rho_0}\left[ \prod_{i=1}^m \left|Y_{k_i}^{j_i}(\theta',\theta)\right|^{d_i} \right. \\
&\hspace{0.2cm}\left.\cdot \left( \1_{\{\rho_0+\theta'\leq X_{(k-1)/n}\leq \rho_0+\theta\}} +\exp\left(-\frac{(X_{k_{m+1}-1}-c(\rho_0,\theta',\theta))^2}{c'(\alpha,\beta)/n}\right)\right)\right].
\end{split}
\end{align}
We now distinguish two cases that both yield the desired claim:
\begin{itemize}
\item[$k_m=k_{m+1}-1$.] Here, we simply bound
\[ \1_{\{\rho_0+\theta'\leq X_{(k-1)/n}\leq \rho_0+\theta\}}+\exp\left(-\frac{(X_{k_{m+1}-1}-c(\rho_0,\theta',\theta))^2}{c'(\alpha,\beta)/n}\right)\leq 2\]
and then the induction hypothesis for $m$ reveals from~\eqref{eq_proof:lemma:moments_Z_j}
\begin{align*}
&\E_{\rho_0}\left[ \prod_{i=1}^{m+1} \left|Y_{k_i}^{j_i}(\theta',\theta)\right|^{d_i}\right] \\
&\hspace{1cm}\leq C_{\alpha,\beta}(D,K)|\theta-\theta'|^{d_{m+1}} n^{d_{m+1}/2}\E_{\rho_0}\left[ \prod_{i=1}^m \left|Y_{k_i}^{j_i}(\theta',\theta)\right|^{d_i}\right]\\
&\hspace{1cm}\leq C_{\alpha,\beta}(D,K)|\theta-\theta'|^{d_{m+1}}n^{d_{m+1}/2}C n^{\frac12\sum_{i=1}^m d_i} |\theta-\theta'|^{\sum_{i=1}^m d_i} \prod_{i=1}^m \frac{1}{\sqrt{k_i-k_{i-1}}}\\
&\hspace{1cm}\leq C_{\alpha,\beta}(D,K)n^{\frac12\sum_{i=1}^{m+1} d_i} |\theta-\theta'|^{\sum_{i=1}^{m+1} d_i} \prod_{i=1}^{m+1} \frac{1}{\sqrt{k_i-k_{i-1}}},
\end{align*}
where the last step uses that $k_{m+1}-k_m=1$.
\item[$k_m<k_{m+1}-1$.] Here, we use that by Corollary~\ref{cor:bound_exp_X^2}, 
\[ \E_{\rho_0}\left[\left. \exp\left(-\frac{(X_{k_{m+1}-1}-c(\rho_0,\theta',\theta))^2}{c'(\alpha,\beta)/n}\right)\right| X_{k_m} \right] \leq \frac{C_{\alpha,\beta}}{\sqrt{k_{m+1}-k_m}} \]
and by Lemma~\ref{lemma:upper_bound_transition_density} and $|\theta-\theta'|\leq 2K/\sqrt{n}$,
\begin{align*}
&\E_{\rho_0}\left[\left.\1_{\{\rho_0+\theta'\leq X_{(k-1)/n}\leq \rho_0+\theta\}} \right| \cF_{k_m} \right]\\
&\hspace{0.5cm} \leq C_{\alpha,\beta} \int_{\rho_0+\theta'}^{\rho_0+\theta} \frac{1}{\sqrt{(k_{m+1}-k_m)/n}} \exp\left(-\frac{(y-X_{k_m})^2}{2(k_{m+1}-k_m)\max\{\alpha^2,\beta^2\}/n}\right) dy \\
&\hspace{0.5cm} \leq \frac{C_{\alpha,\beta}(K)}{\sqrt{k_{m+1}-k_m}}.
\end{align*} 
Then~\eqref{eq_proof:lemma:moments_Z_j} reveals together with the induction hypothesis for $m$
\begin{align*}
&\E_{\rho_0}\left[ \prod_{i=1}^{m+1} \left|Y_{k_i}^{j_i}(\theta',\theta)\right|^{d_i}\right] \\
&\hspace{0.2cm}\leq C_{\alpha,\beta}(D,K)|\theta-\theta'|^{d_{m+1}} n^{d_{m+1}/2} \E_{\rho_0}\left[ \prod_{i=1}^m \left|Y_{k_i}^{j_i}(\theta',\theta)\right|^{d_i} \right. \\
&\hspace{3cm}\left. \cdot\E_{\rho_0}\left[\left.\exp\left(-\frac{(X_{k_{m+1}-1}-c(\rho_0,\theta',\theta))^2}{c'(\alpha,\beta)/n}\right)\right| \cF_{k_m}\right]\right] \\
&\hspace{0.2cm}\leq C_{\alpha,\beta}(D,K) \frac{1}{\sqrt{k_{m+1}-k_m}}|\theta-\theta'|^{d_{m+1}} n^{d_{m+1}/2} \E_{\rho_0}\left[ \prod_{i=1}^m \left|Y_{k_i}^{j_i}(\theta',\theta)\right|^{d_i}\right] \\
&\hspace{0.2cm}\leq C_{\alpha,\beta}(D,K) \frac{1}{\sqrt{k_{m+1}-k_m}}|\theta-\theta'|^{d_{m+1}} n^{d_{m+1}/2} C n^{\frac12\sum_{i=1}^m d_i} |\theta-\theta'|^{\sum_{i=1}^m d_i} \prod_{i=1}^m \frac{1}{\sqrt{k_i-k_{i-1}}} \\
&\hspace{0.2cm}\leq C_{\alpha,\beta}(D,K) n^{\frac12\sum_{i=1}^{m+1} d_i} |\theta-\theta'|^{\sum_{i=1}^{m+1} d_i} \prod_{i=1}^{m+1} \frac{1}{\sqrt{k_i-k_{i-1}}}.
\end{align*}
\end{itemize}
\end{proof}

\section{Remaining proofs of Section~\ref{Section:Consistency}}

This section is divided into two parts: In Subsection~\ref{App:Section:proofs_sqrt(n)_consistency} the remaining proofs for the $\sqrt{n}$-consistency of Subsection~\ref{subsection:sqrt(n)_consistency} are given, whereas in Subsection~\ref{App:Section:proofs_n_consistency} those for the proof of the $n$-consistency of Subsection~\ref{subsection:n_consistency} may be found.\\
Recall the definition of $I_{j,k}^{\theta',\theta}$ given in~\eqref{eq:cases_I_jk} and remember that $C_{\alpha,\beta}$ denotes some real and positive constant that only depends on $\alpha$ and $\beta$ but may change from line to line.

\subsection{Remaining proofs of Subsection~\ref{subsection:sqrt(n)_consistency}}\label{App:Section:proofs_sqrt(n)_consistency}

\begin{proof}[Proof of Lemma~\ref{lemma:bound_L_sqrt(n)}]
The proof makes use of the inequality
\[ \sup_{\theta\in\Theta_i} \ell_n(\theta)\leq \sum_{j=1}^9 \sup_{\theta\in\Theta_i} \mathcal{I}_j(\theta) =  \sum_{j=1}^9 \sup_{\theta\in\Theta_i} \sum_{k=1}^n \log\left(\frac{p_{1/n}^{\rho_0+\theta}(X_{(k-1)/n},X_{k/n})}{p_{1/n}^{\rho_0}(X_{(k-1)/n},X_{k/n})}\right) \1_{I_{j,k}^\theta}\]
for the nine different cases $\mathcal{I}_1(\theta),\dots, \mathcal{I}_9(\theta)$ given in~\eqref{eq:def_I} and investigate all of those seperately. It will turn out that all $\mathcal{I}_j(\theta)$ for $j\neq 5$ are summarized in $F_n^1(K,L)$, $F_n^2(K,L)$ and the fifth case is split into $N_n^L(\theta)$ and a remainder that also contributes to $F_n^1(K,L)$ and $F_n^2(K,L)$, respectively.\\

\noindent
$\boldsymbol{\mathcal{I}_1(\theta)}$: Recall that $I_{1,k}^\theta=I_{1,k}$ is independ of $\theta$. Using $(1-a)/(1-b) = 1+(b-a)/(1-b)$, $\log(1+x)\leq x$ for $x>-1$ and the fact that for $X_{(k-1)/n},X_{k/n}\leq \rho_0$
\[ 1-\frac{\alpha-\beta}{\alpha+\beta}\exp\left(-\frac{2}{\alpha^2/n}(X_{k/n}-\rho_0)(X_{(k-1)/n}-\rho_0)\right) \geq  1-\frac{|\alpha-\beta|}{\alpha+\beta}>0\]
gives
\begin{align*}
& \1_{I_{1,k}} \log\left(\frac{P_1^{\rho_0+\theta}(X_{(k-1)/n},X_{k/n};1/n)}{P_1^{\rho_0}(X_{(k-1)/n},X_{k/n};1/n)}\right) \\
&\hspace{1cm} = \1_{I_{1,k}} \log\left(\frac{1-\frac{\alpha-\beta}{\alpha+\beta}\exp\left(-\frac{2}{\alpha^2/n}(X_{k/n}-\rho_0-\theta)(X_{(k-1)/n}-\rho_0-\theta)\right)}{1-\frac{\alpha-\beta}{\alpha+\beta}\exp\left(-\frac{2}{\alpha^2/n}(X_{k/n}-\rho_0)(X_{(k-1)/n}-\rho_0)\right)}\right) \\
&\hspace{1cm} \leq C_{\alpha,\beta}\1_{I_{1,k}} \exp\left(-\frac{2}{\alpha^2/n}(X_{k/n}-\rho_0)(X_{(k-1)/n}-\rho_0)\right) \\
&\hspace{3cm} \cdot \left| \exp\left(\frac{2\theta}{\alpha^2/n}(X_{k/n}-\rho_0+X_{(k-1)/n}-\rho_0)- \frac{2\theta^2}{\alpha^2/n}\right) - 1\right| \\
&\hspace{1cm} \leq C_{\alpha,\beta}\1_{I_{1,k}} \exp\left(-\frac{2}{\alpha^2/n}(X_{k/n}-\rho_0)(X_{(k-1)/n}-\rho_0)\right),
\end{align*}
where the last step uses $X_{(k-1)/n},X_{k/n}\leq \rho_0$. Next, we bound its expectation. First, using that $(a+b)^2\geq a^2+b^2$ for $ab\geq 0$, and
\[ \1_{I_{1,k}} P_1^{\rho_0}(X_{(k-1)/n},X_{k/n};1/n) \leq \1_{I_{1,k}} \frac{1}{\sqrt{2\pi\alpha^2/n}} \left(1+\frac{|\alpha-\beta|}{\alpha+\beta}\right)\exp\left(-\frac{(X_{k/n}-X_{(k-1)/n})^2}{2\alpha^2/n}\right),\]
we obtain
\begin{align*}
&\E_{\rho_0}\left[ \left.\exp\left(-\frac{2}{\alpha^2/n}(X_{k/n}-\rho_0)(X_{(k-1)/n}-\rho_0)\right)\1_{I_{1,k}}\right| X_{(k-1)/n}\right] \\
& \hspace{0.2cm} \leq C_{\alpha,\beta}\1_{\{X_{(k-1)/n}<\rho_0\}} \int_{-\infty}^{\rho_0} \exp\left(-\frac{2}{\alpha^2/n}(y-\rho_0)(X_{(k-1)/n}-\rho_0)\right) \\
&\hspace{7cm} \cdot\sqrt{n}\exp\left(-\frac{(y-X_{(k-1)/n})^2}{2\alpha^2/n}\right) dy \\
& \hspace{0.2cm} = C_{\alpha,\beta}\1_{\{X_{(k-1)/n}<\rho_0\}} \int_{-\infty}^{\rho_0} \sqrt{n}\exp\left(-\frac{(y-2\rho_0+X_{(k-1)/n})^2}{2\alpha^2/n}\right) dy \\
&\hspace{0.2cm} \leq C_{\alpha,\beta}\1_{\{X_{(k-1)/n}<\rho_0\}}\exp\left(-\frac{(X_{(k-1)/n}-\rho_0)^2}{2\alpha^2/n}\right) \int_{-\infty}^{\rho_0} \sqrt{n}\exp\left(-\frac{(y-\rho_0)^2}{2\alpha^2/n}\right) dy \\
&\hspace{0.2cm} \leq C_{\alpha,\beta}\1_{\{X_{(k-1)/n}<\rho_0\}}\exp\left(-\frac{(X_{(k-1)/n}-\rho_0)^2}{2\alpha^2/n}\right).
\end{align*}
In consequence, by Corollary~\ref{cor:bound_exp_X^2},
\begin{align*}
&\sum_{k=1}^n \E_{\rho_0}\left[ \exp\left(-\frac{2}{\alpha^2/n}(X_{k/n}-\rho_0)(X_{(k-1)/n}-\rho_0)\right)\1_{I_{1,k}} \right] \leq C_{\alpha,\beta} \sum_{k=1}^n \frac{1}{\sqrt{k}} = C_{\alpha,\beta}\sqrt{n}.
\end{align*}
Summarising,
\begin{align*}
&\sup_{\theta\in\Theta_1\cup\Theta_2} \sum_{k=1}^n \log\left(\frac{p_{1/n}^{\rho_0+\theta}(X_{(k-1)/n},X_{k/n})}{p_{1/n}^{\rho_0}(X_{(k-1)/n},X_{k/n})}\right) \1_{I_{1,k}} \\
&\hspace{2cm} \leq C_{\alpha,\beta}\sum_{k=1}^n \1_{I_{1,k}} \exp\left(-\frac{2}{\alpha^2/n}(X_{k/n}-\rho_0)(X_{(k-1)/n}-\rho_0)\right) =: \Xi_n^{(1)}
\end{align*} 
with $\E_{\rho_0}[|\Xi_n^{(1)}|/\sqrt{n}] \leq C_{\alpha,\beta}L$, as $L\geq 1$.\\

\noindent
$\boldsymbol{\mathcal{I}_2(\theta)}$: With the inequality
\[ \1_{I_{2,k}^\theta}\left(1- \frac{\alpha-\beta}{\alpha+\beta}\exp\left(-\frac{2}{\alpha^2/n}(X_{k/n}-\rho_0-\theta)(X_{(k-1)/n}-\rho_0-\theta)\right)\right) \leq  1+\frac{|\alpha-\beta|}{\alpha+\beta}\]
we deduce
\begin{align*}
& \log\left(\frac{P_1^{\rho_0+\theta}(X_{(k-1)/n},X_{k/n};1/n)}{P_3^{\rho_0}(X_{(k-1)/n},X_{k/n};1/n)}\right)\1_{I_{2,k}^\theta}  \\
&=\1_{I_{2,k}^\theta}  \log\left( \frac{\beta(\alpha+\beta)}{2\alpha^2} \left( 1- \frac{\alpha-\beta}{\alpha+\beta}\exp\left(-\frac{2}{\alpha^2/n}(X_{k/n}-\rho_0-\theta)(X_{(k-1)/n}-\rho_0-\theta)\right)\right) \right. \\
&\hspace{1cm}\left. \cdot\exp\left(-\frac{(X_{k/n}-X_{(k-1)/n})^2}{2\alpha^2/n} + \frac{n}{2}\left(\frac{X_{k/n}-\rho_0}{\beta}-\frac{X_{(k-1)/n}-\rho_0}{\alpha}\right)^2\right)\right)\\
&\leq C_{\alpha,\beta}\1_{I_{2,k}^\theta}\left( 1+ n(X_{k/n}-\rho_0)^2+ n|X_{k/n}-\rho_0||X_{(k-1)/n}-\rho_0|\right) \\
&\leq C_{\alpha,\beta}\1_{\{X_{(k-1)/n}<\rho_0< X_{k/n}\}}\left( 1+ n(X_{k/n}-X_{(k-1)/n})^2\right), 
\end{align*}
where the last step follows as $X_{(k-1)/n}<\rho_0<X_{k/n}$. As $x\mapsto (1+x^2)\exp(-x^2/4)$ is bounded, we obtain with the Gaussian tail inequality,
\begin{align}\label{eq_proof:I2_new}
\begin{split}
&\E_{\rho_0}\left[\left. \left(1+ n(X_{k/n}-X_{(k-1)/n})^2\right) \1_{\{X_{(k-1)/n}<\rho_0<X_{k/n}\}} \right|X_{(k-1)/n}\right] \\
&\hspace{1cm} \leq C_{\alpha,\beta}\int_{\rho_0}^\infty \left(1+ n(y-X_{(k-1)/n})^2\right) \sqrt{n} \exp\left(- \frac{(y-X_{(k-1)/n})^2}{2\max\{\alpha^2,\beta^2\}/n}\right) dy \\
&\hspace{1cm} = C_{\alpha,\beta} \int_{\sqrt{n}(\rho_0-X_{(k-1)/n})}^\infty \left(1+y^2\right) \exp\left(-\frac{y^2}{2\max\{\alpha^2,\beta^2\}}\right) dy \\
&\hspace{1cm} \leq C_{\alpha,\beta} \exp\left(-\frac{(X_{(k-1)/n}-\rho_0)^2}{4 \max\{\alpha^2,\beta^2\}/n}\right).
\end{split}
\end{align}
With Corollary~\ref{cor:bound_exp_X^2}, we then find
\begin{align*}
\sum_{k=1}^n\E_{\rho_0}\left[ \left(1+ n(X_{k/n}-X_{(k-1)/n})^2 \right) \1_{\{X_{(k-1)/n}<\rho_0<X_{k/n}\}} \right]  \leq C_{\alpha,\beta} \sum_{k=1}^n\frac{1}{\sqrt{k}} \leq C_{\alpha,\beta}\sqrt{n}
\end{align*}
and hence,
\begin{align*}
&\sup_{\theta\in\Theta_1\cup\Theta_2} \sum_{k=1}^n \log\left(\frac{P_1^{\rho_0+\theta}(X_{(k-1)/n},X_{k/n};1/n)}{P_3^{\rho_0}(X_{(k-1)/n},X_{k/n};1/n)}\right)\1_{I_{2,k}^\theta}  \\
&\hspace{2cm}\leq C_{\alpha,\beta}\sum_{k=1}^n\1_{\{X_{(k-1)/n}<\rho_0\leq X_{k/n}\}}\left( 1+ n(X_{k/n}-X_{(k-1)/n})^2\right) := \Xi_n^{(2)}
\end{align*}
with $\E_{\rho_0}[|\Xi_n^{(2)}|/\sqrt{n}]\leq C_{\alpha,\beta}$.\\

\noindent
$\boldsymbol{\mathcal{I}_3(\theta)}$: By dropping negative terms and using $X_{(k-1)/n}<\rho_0$ and $X_{k/n}\geq\rho_0+\theta$ on $I_{3,k}^\theta$, we have the bound
\begin{align*}
\log\left(\frac{P_3^{\rho_0+\theta}(X_{(k-1)/n},X_{k/n};1/n)}{P_3^{\rho_0}(X_{(k-1)/n},X_{k/n};1/n)}\right)\1_{I_{3,k}^\theta} & \leq \frac{n}{2} \left(\frac{X_{k/n}-\rho_0}{\beta} - \frac{X_{(k-1)/n}-\rho_0}{\alpha}\right)^2\1_{I_{3,k}^\theta} \\
&\leq n(\alpha^{-1}+\beta^{-1})^2 \left( X_{k/n}-X_{(k-1)/n}\right)^2 \1_{I_{3,k}^{0}},
\end{align*}
where the last step uses that the indicator $\1_{I_{3,k}^\theta}=\1_{\{X_{(k-1)/n}<\rho_0\leq\rho_0+\theta<X_{k/n}\}}$ is larger, the smaller $\theta$. Hence,
\begin{align*}
&\sup_{\theta\in\Theta_1\cup\Theta_2} \sum_{k=1}^n\log\left(\frac{P_3^{\rho_0+\theta}(X_{(k-1)/n},X_{k/n};1/n)}{P_3^{\rho_0}(X_{(k-1)/n},X_{k/n};1/n)}\right)\1_{I_{3,k}^\theta}\\
&\hspace{2cm} \leq (\alpha^{-1}+\beta^{-1})^2\sum_{k=1}^nn\left( X_{k/n}-X_{(k-1)/n}\right)^2 \1_{I_{3,k}^{0}} =:\Xi_n^{(3)} 
\end{align*} 
By~\eqref{eq_proof:I2_new} and Corollary~\ref{cor:bound_exp_X^2} we find $\sup_{K\geq 0}\E_{\rho_0}[|\Xi_n^{(3)}|/\sqrt{n}]\leq C_{\alpha,\beta}$.\\

\noindent
$\boldsymbol{\mathcal{I}_4(\theta)}$: Follows by the same reasoning as $\mathcal{I}_2(\theta)$ by interchanging the roles of $\alpha$ and $\beta$ and gives an upper bound $\Xi_n^{(4)}$ with $\E_{\rho_0}[|\Xi_n^{(4)}|/\sqrt{n}]\leq C_{\alpha,\beta}$.\\

\noindent
$\boldsymbol{\mathcal{I}_5(\theta)}$: First, we split
\begin{align*}
&\frac{P_1^{\rho_0+\theta}(X_{(k-1)/n},X_{k/n};1/n)}{P_2^{\rho_0}(X_{(k-1)/n},X_{k/n};1/n)}\\
&\hspace{0.2cm} = \frac{\beta}{\alpha} \frac{\exp\left(-\frac{(X_{k/n}-X_{(k-1)/n})^2}{2\alpha^2/n}\right)}{\exp\left(-\frac{(X_{k/n}-X_{(k-1)/n})^2}{2\beta^2/n}\right)} \frac{1-\frac{\alpha-\beta}{\alpha+\beta}\exp\left(-\frac{2}{\alpha^2/n}(X_{k/n}-\rho_0-\theta)(X_{(k-1)/n}-\rho_0-\theta)\right)}{1+\frac{\alpha-\beta}{\alpha+\beta}\exp\left(-\frac{2}{\beta^2/n}(X_{k/n}-\rho_0)(X_{(k-1)/n}-\rho_0)\right)}.
\end{align*}
Taking the logarithm on both sides then yields
\begin{align*}
\log\left( \frac{P_1^{\rho_0+\theta}(X_{(k-1)/n},X_{k/n};1/n)}{P_2^{\rho_0}(X_{(k-1)/n},X_{k/n};1/n)}\right) = \log\left(\frac{\beta}{\alpha}\right) -\frac{(X_{k/n}-X_{(k-1)/n})^2}{2/n}\left(\frac{1}{\alpha^2}-\frac{1}{\beta^2}\right) + R_{k,\theta}
\end{align*}
with
\[ R_{k,\theta} = \log\left(  \frac{1-\frac{\alpha-\beta}{\alpha+\beta}\exp\left(-\frac{2}{\alpha^2/n}(X_{k/n}-\rho_0-\theta)(X_{(k-1)/n}-\rho_0-\theta)\right)}{1+\frac{\alpha-\beta}{\alpha+\beta}\exp\left(-\frac{2}{\beta^2/n}(X_{k/n}-\rho_0)(X_{(k-1)/n}-\rho_0)\right)}\right).\]
Denoting
\[ g_k := \log\left(\frac{\beta}{\alpha}\right) - \frac{n}{2}\left(\frac{1}{\alpha^2}-\frac{1}{\beta^2}\right)(X_{k/n}-X_{(k-1)/n})^2 \]
(which is independent of $\theta$) and decomposing
\begin{align*}
\1_{I_{5,k}^\theta} &= \1_{\{\rho_0\leq X_{(k-1)/n}< \rho_0+\theta\}}\1_{I_{5,k}^\theta}  \\
&= \1_{\{ \rho_0+L/\sqrt{n}\leq X_{(k-1)/n}< \rho_0+\theta\}}  - \1_{\{ \rho_0+L/\sqrt{n}\leq X_{(k-1)/n}< \rho_0+\theta\}}\left(1-\1_{\{\rho_0<X_{k/n}\leq \rho_0+\theta\}}\right) \\
&\hspace{1cm} + \1_{\{\rho_0\leq X_{(k-1)/n}< \rho_0+L/\sqrt{n}\}} \1_{\{\rho_0<X_{k/n}\leq \rho_0+\theta\}}
\end{align*} 
then gives the decomposition
\begin{align}\label{eq_proof:bound_L_sqrt(n)_1}
\begin{split}
\log\left( \frac{P_1^{\rho_0+\theta}(X_{(k-1)/n},X_{k/n};1/n)}{P_2^{\rho_0}(X_{(k-1)/n},X_{k/n};1/n)}\right) \1_{I_{5,k}^\theta} &=  S_{1,\theta}(k) + S_{2,\theta}(k) + S_{3,\theta}(k) + S_{4,\theta}(k),
\end{split}
\end{align}
with
\begin{align*}
S_{1,\theta}(k) &:= g_k \1_{\{ \rho_0+L/\sqrt{n}\leq X_{(k-1)/n}< \rho_0+\theta\}},\\
S_{2,\theta}(k) &:= R_{k,\theta}  \1_{\{\rho_0\leq X_{(k-1)/n}< \rho_0+\theta, \rho_0<X_{k/n}\leq \rho_0+\theta\}},\\
S_{3,\theta}(k) &:= -g_k\1_{\{ \rho_0+L/\sqrt{n}\leq X_{(k-1)/n}< \rho_0+\theta\}}\left(1-\1_{\{\rho_0<X_{k/n}\leq \rho_0+\theta\}}\right), \\
S_{4,\theta}(k) &:= g_k \1_{\{\rho_0\leq X_{(k-1)/n}< \rho_0+L/\sqrt{n}\}} \1_{\{\rho_0<X_{k/n}\leq \rho_0+\theta\}}.
\end{align*}
First, we observe that
\[ \sum_{k=1}^n S_{1,\theta}(k) = N_n^L(\theta).\]
In what follows, we will construct upper bounds for $\sum_{k=1}^n S_{j,\theta}(k)$ for $j=2,3,4$ and determine upper bounds of their first moments.\\ 

\noindent
$\bullet\ \boldsymbol{S_{2,\theta}(k)}$. Here, we first observe
\begin{align*}
S_{2,\theta}(k) &= R_{k,\theta} \1_{\{\rho_0\leq X_{(k-1)/n}< \rho_0+\theta, \rho_0<X_{k/n}\leq \rho_0+\theta\}} \\
&\leq |R_{k,\theta}|\1_{\{\rho_0\leq X_{(k-1)/n}< \rho_0+\theta\}}\1_{\{\rho_0< X_{k/n}\leq \rho_0+\theta-1/\sqrt{n}\}}  +|R_{k,\theta}|\1_{\{\rho_0+\theta-1/\sqrt{n}< X_{k/n}\leq \rho_0+\theta\}}.
\end{align*}
For the first term, we use $x/(1+x)\leq \log(1+x)\leq x$ for $x\geq -1$ and the inequalities
\[ 1-\frac{|\alpha-\beta|}{\alpha+\beta} \leq 1-\frac{\alpha-\beta}{\alpha+\beta}\exp\left(-\frac{2}{\alpha^2/n}(X_{k/n}-\rho_0-\theta)(X_{(k-1)/n}-\rho_0-\theta)\right)\leq 1+\frac{|\alpha-\beta|}{\alpha+\beta}\]
and 
\[ 1-\frac{|\alpha-\beta|}{\alpha+\beta} \leq 1+\frac{\alpha-\beta}{\alpha+\beta}\exp\left(-\frac{2}{\beta^2/n}(X_{k/n}-\rho_0)(X_{(k-1)/n}-\rho_0)\right)\leq 1+\frac{|\alpha-\beta|}{\alpha+\beta}\]
to derive
\begin{align*}
&|R_{k,\theta}|\1_{\{\rho_0\leq X_{(k-1)/n}< \rho_0+\theta\}}\1_{\{\rho_0< X_{k/n}\leq \rho_0+\theta-1/\sqrt{n}\}}\\
&\hspace{1cm}\leq C_{\alpha,\beta}\exp\left(-\frac{2}{\alpha^2/n}(X_{k/n}-\rho_0-\theta)(X_{(k-1)/n}-\rho_0-\theta)\right)\\
&\hspace{2.5cm}\cdot\1_{\{\rho_0\leq X_{(k-1)/n}< \rho_0+\theta\}}\1_{\{\rho_0< X_{k/n}\leq \rho_0+\theta-1/\sqrt{n}\}}\\
&\hspace{2cm} + C_{\alpha,\beta}\exp\left(-\frac{2}{\beta^2/n}(X_{k/n}-\rho_0)(X_{(k-1)/n}-\rho_0)\right)\1_{\{X_{(k-1)/n},X_{k/n}\geq \rho_0\}} \\
&\hspace{1cm}\leq C_{\alpha,\beta}\exp\left(\frac{2}{\alpha^2/\sqrt{n}}(X_{(k-1)/n}-\rho_0-\theta)\right)\1_{\{X_{(k-1)/n}\leq \rho_0+\theta\}}\\
&\hspace{2cm} + C_{\alpha,\beta}\exp\left(-\frac{2}{\beta^2/n}(X_{k/n}-\rho_0)(X_{(k-1)/n}-\rho_0)\right)\1_{\{X_{(k-1)/n},X_{k/n}\geq \rho_0\}}.
\end{align*}
For $|R_{k,\theta}|\1_{\{\rho_0+\theta-1/\sqrt{n}< X_{k/n}\leq \rho_0+\theta\}}$ we simply use that $|R_{k,\theta}|$ can be bounded independently of $k$ and $\theta$ and thus have
\[ \sup_{\theta\in\Theta_1} \sum_{k=1}^n S_{2,\theta}(k) \leq C_{\alpha,\beta}\left(\Xi_n^{(5,1)} + \Xi_n^{(5,2)} +\Xi_n^{(5,3)}\right) \]
and
\[ \sup_{\theta\in\Theta_2} \sum_{k=1}^n S_{2,\theta}(k) \leq C_{\alpha,\beta}\left( \tilde{\Xi}_n^{(5,1)} + \Xi_n^{(5,2)} +  \tilde{\Xi}_n^{(5,3)}\right),\]
where
\begin{align*}
\Xi_n^{(5,1)} &:= \sup_{\theta\in\Theta_1} \sum_{k=1}^n \exp\left(\frac{2}{\alpha^2/\sqrt{n}}(X_{(k-1)/n}-\rho_0-\theta)\right)\1_{\{X_{(k-1)/n}\leq \rho_0+\theta\}}, \\
\tilde{\Xi}_n^{(5,1)} &:= \sup_{\theta\in\Theta_2}\sum_{k=1}^n \exp\left(\frac{2}{\alpha^2/\sqrt{n}}(X_{(k-1)/n}-\rho_0-\theta)\right)\1_{\{X_{(k-1)/n}\leq \rho_0+\theta\}}, \\
\Xi_n^{(5,2)} &:=\sum_{k=1}^n\exp\left(-\frac{2}{\beta^2/n}(X_{k/n}-\rho_0)(X_{(k-1)/n}-\rho_0)\right)\1_{\{X_{(k-1)/n},X_{k/n}\geq \rho_0\}}, \\
\Xi_n^{(5,3)} & :=  \sup_{\theta\in\Theta_1} \sum_{k=1}^n \1_{\{\rho_0+\theta-1/\sqrt{n}< X_{k/n}\leq \rho_0+\theta\}}, \\
\tilde{\Xi}_n^{(5,3)} &:= \sup_{\theta\in\Theta_2} \sum_{k=1}^n \1_{\{\rho_0+\theta-1/\sqrt{n}< X_{k/n}\leq \rho_0+\theta\}}.\\
\end{align*}
In what follows, we will derive deterministic bounds on sets with high probability for $\Xi_n^{(5,1)}$, $\tilde{\Xi}_n^{(5,1)}, \Xi_n^{(5,3)},\tilde{\Xi}_n^{(5,3)}$ and an $L^1$-bound on $\Xi_n^{(5,2)}$. To this aim, we consider
\begin{align*}
A_2 &= \{\overline{X}>\rho_0+\gamma, \underline{X}<\rho_0-\gamma\}, \\
A_3' &= \left\{ \inf_{y\in [\rho_0-\gamma/2,\rho_0+\gamma/2]} L_1^y(X) >\xi_l\right\}\cap\left\{\sup_{y\in\R} L_1^y(X)\leq \xi_u\right\}, \\
A_4(n)& = \left\{ \sup_{|t-s|< 1/n} |X_t-X_s|\leq n^{-4/9}\right\}.
\end{align*}
The set $A_2$ was given in~\eqref{eq:set_A2}, $A_4(n)$ was given in~\eqref{eq:set_A4}. For those, we have chosen the parameter $\gamma$ such that $\Pr_{\rho_0}(A_i|L_1^{\rho_0}(X)>0)\geq 1-\epsilon$, $i=2,4$ for $n$ sufficiently large. To choose $\xi_l,\xi_u$ to ensure $\Pr_{\rho_0}(A_3')\geq 1-\epsilon$, we first choose $\xi_l$ sufficiently small such that the first set in the definition of $A_3'$ has probability $\geq 1-\epsilon/2$. This is possible, as this set coincides with $A_3$ given in~\eqref{eq:set_A3}. Now choose a compact interval $[-a,a]$ such that for $A=\{X_s\notin [-a,a] \textrm{ for some } s\in [0,1]\}$, we have $\Pr_{\rho_0}(A)<\epsilon/4$ which is possible by the argument used to bound the probability of $A_1$ in~\eqref{eq:set_A1}. By Corollary~$29.18$ in~\citeSM{App:Kallenberg},
\[ \sup_{y\in\R} L_1^y(X)\1_{A^c} = \sup_{y\in [-a,a]}L_1^y(X)\1_{A^c}.\]
Moreover, as $L_1^\bullet(X)$ is continuous by Corollary~$1.8$ on p.~$226$ in~\citeSM{App:Revuz/Yor}, and attains its maximum and minimum on the compact set $[-a,a]$, $\sup_{y\in [-a,a]}L_1^y(X)$ is a $\R$-valued random variable and consequently, it is tight. Thus, $\xi_u$ can be chosen large enough such that
\[ \Pr_{\rho_0} \left(\sup_{y\in \R} L_1^y > \xi_u\right)\leq \Pr_{\rho_0}\left(A\right)+ \Pr_{\rho_0} \left(\sup_{y\in [-a,a]} L_1^y > \xi_u\right) \leq \frac{\epsilon}{2}.\]

\begin{itemize}
\item[$\boldsymbol{-\ \Xi_n^{(5,1)}}$.] Defining
\[ G_{n,\theta}(x) := \exp\left(\frac{2}{\alpha^2/\sqrt{n}}(x-\rho_0-\theta)\right)\1_{\{x\leq \rho_0+\theta\}},\]
we have by Lemma~\ref{lemma:uniform_stochastic_occupation_approximation} that the probability of the event
\begin{align*}
&E_{5,1} := \left\{ \sup_{K/\sqrt{n}<\theta\leq n^{-1/4}} \left| \frac{1}{\sqrt{n}} \sum_{k=1}^n G_{n,\theta}(X_{(k-1)/n}) - \sqrt{n} \int_0^1 G_{n,\theta}(X_s) ds \right| \leq \sqrt{n} \int_0^1 G_{n,\theta}(X_s) ds \right\} \\
&\hspace{2cm} \cap A_2\cap A_3'  \\
&\hspace{1cm} \supset \left\{ \sup_{K/\sqrt{n}<\theta\leq n^{-1/4}} \left| \frac{1}{\sqrt{n}} \sum_{k=1}^n G_{n,\theta}(X_{(k-1)/n}) - \sqrt{n} \int_0^1 G_{n,\theta}(X_s) ds \right| \leq \frac{\alpha^2\xi_l}{4\max\{\alpha^2,\beta^2\}} \right\}\\
&\hspace{2cm}\cap A_2\cap A_3'
\end{align*}
is larger than $1-3\epsilon$ for $n$ sufficiently large. Here, we used that on $A_3'$ we have $L_1^y(X)>\xi_l$ for $\rho_0-\gamma/2\leq y\leq \rho_0+\gamma/2$ and consequently by the occupation times formula
\begin{align*}
\sqrt{n} \int_0^1 G_{n,\theta}(X_s) ds &\geq \frac{\sqrt{n}}{\max\{\alpha^2,\beta^2\}} \int_0^1 G_{n,\theta}(X_s) \sigma(X_s)^2 ds \\
&= \frac{\sqrt{n}}{\max\{\alpha^2,\beta^2\}} \int_\R G_{n,\theta}(y) L_1^y(X) dy \\
&\geq \frac{\sqrt{n}}{\max\{\alpha^2,\beta^2\}} \int_{-\infty}^0 \exp\left( \frac{2y}{\alpha^2/\sqrt{n}}\right) L_1^{y+\rho_0+\theta} dy \\
&\geq \frac{\sqrt{n} \xi_l}{\max\{\alpha^2,\beta^2\}} \int_{-\gamma/2}^0 \exp\left( \frac{2y}{\alpha^2/\sqrt{n}}\right) dy \\
&=\frac{\alpha^2 \xi_l}{2\max\{\alpha^2,\beta^2\}} \left( 1- \exp\left(-\frac{\gamma}{\alpha^2/\sqrt{n}}\right)\right) \geq \frac{\alpha^2 \xi_l}{4\max\{\alpha^2,\beta^2\}},
\end{align*}
where the last step is true for $n$ large enough to ensure $\exp(-\gamma\sqrt{n}/\alpha^2)\leq 1/2$. Then, on $E_{5,1}$ we have the bound
\begin{align*}
\Xi_n^{(5,1)} &= \sup_{\theta\in\Theta_1} \sum_{k=1}^n \exp\left(\frac{2}{\alpha^2/\sqrt{n}}(X_{(k-1)/n}-\rho_0-\theta)\right)\1_{\{X_{(k-1)/n}\leq \rho_0+\theta\}} \\
& \leq \sup_{\theta\in\Theta_1} 2n \int_0^1 \exp\left(\frac{2}{\alpha^2/\sqrt{n}}(X_{s}-\rho_0-\theta)\right)\1_{\{X_{s}\leq \rho_0+\theta\}} ds \\
&\leq  2n \sup_{\theta\in\Theta_1} \frac{\xi_u}{\min\{\alpha^2,\beta^2\}} \int_{-\infty}^{\rho_0+\theta} \exp\left(\frac{2}{\alpha^2/\sqrt{n}} (y-\rho_0-\theta)\right) dy\\
& \leq \frac{\alpha^2 \xi_u}{\min\{\alpha^2,\beta^2\}} \sqrt{n}.
\end{align*}
\item[$\boldsymbol{-\ \tilde{\Xi}_n^{(5,1)}}$.] On $A_2\cap A_4(n)$, we find with the occupation times formula the upper bound
\begin{align*}
\tilde{\Xi}_n^{(5,1)}&=\sup_{\theta\in\Theta_2}\sum_{k=1}^n \exp\left(\frac{2}{\alpha^2/\sqrt{n}}(X_{(k-1)/n}-\rho_0-\theta)\right)\1_{\{X_{(k-1)/n}\leq \rho_0+\theta\}} \\
& \leq \sup_{\theta\in\Theta_2}\sum_{k=1}^n \exp\left(\frac{2}{\alpha^2/\sqrt{n}}(X_{(k-1)/n}-\rho_0-\theta)\right)\1_{\{X_{(k-1)/n}\leq \rho_0+\theta-n^{-4/9}\}} \\
&\hspace{0.5cm} + \sup_{\theta\in\Theta_2}\sum_{k=1}^n \exp\left(\frac{2}{\alpha^2/\sqrt{n}}(X_{(k-1)/n}-\rho_0-\theta)\right)\1_{\{\rho_0+\theta -n^{-4/9}\leq X_{(k-1)/n}\leq \rho_0+\theta\}} \\
& \leq \sum_{k=1}^n \exp\left(-\frac{2n^{1/18}}{\alpha^2}\right) + \sup_{\theta\in\Theta_2} n\int_0^1 \1_{[\rho_0+\theta-2n^{-4/9}, \rho_0+\theta + n^{-4/9}]}(X_s) ds \\
& \leq n\exp\left(-\frac{2n^{1/18}}{\alpha^2}\right) + \frac{3}{\min\{\alpha^2,\beta^2\}}\xi_u n^{5/9} \leq C_{\alpha,\beta}(\xi_u) n^{5/9}.
\end{align*}

\item[$\boldsymbol{-\ \Xi_n^{(5,2)}}$.] As in the treatment of $\mathcal{I}_1(\theta)$, one can easily see that $\E_{\rho_0}[|\Xi_n^{(5,2)}|/\sqrt{n}] \leq C_{\alpha,\beta}$.

\item[$\boldsymbol{-\ \Xi_n^{(5,3)}}$.] We define
\[ G_{n,\theta}'(x) := \1_{[\rho_0+\theta-1/\sqrt{n},\rho_0+\theta]}(x)\]
and the event
\begin{align*}
E_{5,2} &:= \left\{ \sup_{K/\sqrt{n}<\theta\leq n^{-1/4}} \left| \frac{1}{\sqrt{n}} \sum_{k=1}^n G_{n,\theta}'(X_{k/n}) - \sqrt{n} \int_0^1 G_{n,\theta}'(X_s) ds \right| \leq \sqrt{n} \int_0^1 G_{n,\theta}'(X_s) ds \right\} \\ 
&\hspace{1.5cm} \cap A_2\cap A_3'.
\end{align*} 
Similar to $E_{5,1}$ we find $\Pr_{\rho_0}>1-\epsilon$ for $n$ large enough and the bound
\begin{align*}
\Xi_n^{(5,3)}\1_{E_{5,2}} &\leq \sup_{\theta\in\Theta_1} 2n\int_0^1 \1_{[\rho_0+\theta-1/\sqrt{n},\rho_0+\theta]}(X_s) ds \\
&\leq \frac{2n\xi_u}{\min\{\alpha^2,\beta^2\}} \sup_{\theta\in\Theta_1}\int_{\rho_0+\theta-1/\sqrt{n}}^{\rho_0+\theta} 1dy = \frac{2\xi_u}{\min\{\alpha^2,\beta^2\}} \sqrt{n}.
\end{align*}

\item[$\boldsymbol{-\ \tilde{\Xi}_n^{(5,3)}}$.] Here, we use that $X_{k/n}\in [\rho_0+\theta-1/\sqrt{n},\rho_0+\theta]$ implies on the event $A_4(n)$ that $X_s\in [\rho_0+\theta-1/\sqrt{n}-n^{-4/9}, \rho_0+\theta+n^{-4/9}]$ for all $(k-1)/n\leq s\leq k/n$. In particular, 
\begin{align*}
\tilde{\Xi}_n^{(5,3)}\1_{A_3'\cap A_4(n)} &\leq \1_{A_3'}\sup_{\theta\in\Theta_2} n\int_{0}^{1} \1_{ [\rho_0+\theta-1/\sqrt{n}-n^{-4/9}, \rho_0+\theta+n^{-4/9}]} (X_s) ds \\
&\leq \frac{n\xi_u}{\min\{\alpha^2,\beta^2\}} \sup_{\theta\in\Theta_2}\int_{\rho_0+\theta-1/\sqrt{n}-n^{-4/9}}^{\rho_0+\theta+n^{-4/9}} 1dy \leq \frac{3\xi_u}{\min\{\alpha^2,\beta^2\}}n^{5/9}.
\end{align*}
\end{itemize}\smallskip

$\bullet\ \boldsymbol{S_{3,\theta}(k)}$. Here, we have the bound
\[ S_{3,\theta}(k) \leq |g_k| \1_{\{X_{k/n} <\rho_0\leq \rho_0+L/\sqrt{n}\leq X_{(k-1)/n}\}} + |g_k|\1_{\{X_{(k-1)/n}<\rho_0+\theta< X_{k/n}\}}.\]
The first summand is already independent of $\theta$ and we have in particular
\[ \sup_{\theta\in\Theta_1} \sum_{k=1}^n S_{3,\theta}(k) \leq \Xi_n^{(5,4)} + \Xi_n^{(5,5)} \quad\textrm{ and }\quad \sup_{\theta\in\Theta_2} \sum_{k=1}^n S_{3,\theta}(k) \leq \tilde{\Xi}_n^{(5,4)} + \Xi_n^{(5,5)},\]
where
\begin{align*}
\Xi_n^{(5,4)} &:= \sup_{\theta\in\Theta_1} \sum_{k=1}^n |g_k|\1_{\{X_{(k-1)/n}<\rho_0+\theta< X_{k/n}\}} \\
\tilde{\Xi}_n^{(5,4)} &:= \sup_{\theta\in\Theta_2}\sum_{k=1}^n |g_k|\1_{\{X_{(k-1)/n}<\rho_0+\theta< X_{k/n}\}}\\
\Xi_n^{(5,5)} &:= \sum_{k=1}^n |g_k| \1_{\{X_{k/n} <\rho_0\leq \rho_0+L/\sqrt{n}\leq X_{(k-1)/n}\}}.
\end{align*}
In what follows, we will derive deterministic bounds on sets with high probability for $\Xi_n^{(5,4)}$, $\tilde{\Xi}_n^{(5,4)}$ and an $L^1$-bound on $\Xi_n^{(5,5)}$.
\begin{itemize}
\item[$-\ \boldsymbol{\Xi_n^{(5,4)}}$.] We first note that
\[ \Xi_n^{(5,4)} \leq C_{\alpha,\beta} \sup_{\theta\in\Theta_1}\sum_{k=1}^n \left(1+ \frac{n}{2}(X_{k/n}-X_{(k-1)/n})^2\right)\1_{\{X_{(k-1)/n}<\rho_0+\theta\leq X_{k/n}\}}. \]
Then we decompose
\begin{align*}
\Xi_n^{(5,4)} &\leq C_{\alpha,\beta} \sup_{\theta\in\Theta_1}\sum_{k=1}^n \E_{\rho_0}\left[\left.\left(1+\frac{n}{2}(X_{k/n}-X_{(k-1)/n})^2\right)\1_{\{X_{(k-1)/n}<\rho_0+\theta\leq X_{k/n}\}}\right| X_{(k-1)/n}\right] \\
&\hspace{1cm} + C_{\alpha,\beta} \sup_{\theta\in\Theta_1}\sum_{k=1}^n \left( \left(1+\frac{n}{2}(X_{k/n}-X_{(k-1)/n})^2\right)\1_{\{X_{(k-1)/n}<\rho_0+\theta\leq X_{k/n}\}} \right. \\
&\hspace{1.5cm} \left. - \E_{\rho_0}\left[\left.\left(1+\frac{n}{2}(X_{k/n}-X_{(k-1)/n})^2\right)\1_{\{X_{(k-1)/n}<\rho_0+\theta\leq X_{k/n}\}}\right| X_{(k-1)/n}\right]\right). 
\end{align*} 
By Lemma~\ref{lemma:supremum_at_rho+theta}, the $L^1(\Pr_{\rho_0})$-norm of the second summand is bounded by $C_{\alpha,\beta} n^{3/8}\log(n)$. For the first one, we begin estimating with Lemma~\ref{lemma:upper_bound_transition_density} and using boundedness of $x\mapsto x^2\exp(-x^2/4)$
\begin{align*}
&\E_{\rho_0}\left[\left.\left(1+\frac{n}{2}(X_{k/n}-X_{(k-1)/n})^2\right)\1_{\{X_{(k-1)/n}<\rho_0+\theta\leq X_{k/n}\}}\right| X_{(k-1)/n}\right] \\
&\hspace{0.2cm} \leq C_{\alpha,\beta}\1_{\{X_{(k-1)/n}<\rho_0+\theta\}} \int_{\rho_0+\theta}^\infty \left(1+\frac{n}{2}(y-X_{(k-1)/n})^2\right) \sqrt{n}\exp\left(-\frac{(y-X_{(k-1)/n})^2}{2\max\{\alpha^2,\beta^2\}/n}\right)dy \\
&\hspace{0.2cm} \leq C_{\alpha,\beta}\1_{\{X_{(k-1)/n}<\rho_0+\theta\}} \int_{\sqrt{n}(\rho_0+\theta-X_{(k-1)/n})}^\infty \left(1+\frac{y^2}{2}\right) \exp\left(-\frac{y^2}{2\max\{\alpha^2,\beta^2\}}\right)dy \\
&\hspace{0.2cm} \leq C_{\alpha,\beta}\1_{\{X_{(k-1)/n}<\rho_0+\theta\}}\exp\left(-\frac{(X_{(k-1)/n}-\rho_0-\theta)^2}{4\max\{\alpha^2,\beta^2\}/n}\right) \\
&\hspace{0.2cm} \leq C_{\alpha,\beta} \1_{\{\rho_0+\theta-1/\sqrt{n}\leq X_{(k-1)/n}\leq \rho_0+\theta\}} \\
&\hspace{2cm} + C_{\alpha,\beta}\1_{\{X_{(k-1)/n}\leq\rho_0+\theta-1/\sqrt{n}\}} \exp\left(\frac{(X_{(k-1)/n}-\rho_0-\theta)}{4\max\{\alpha^2,\beta^2\}/\sqrt{n}}\right).
\end{align*}
The expression $\sup_{\theta\in\Theta_1}\sum_{k=1}\1_{\{\rho_0+\theta-1/\sqrt{n}\leq X_{(k-1)/n}\leq \rho_0+\theta\}} $ can be bounded analogously to $\Xi_n^{(5,3)}$. For the other term we define
\begin{align*}
G_{n,\theta}''(x) &:= \exp\left(\frac{(x-\rho_0-\theta)}{4\max\{\alpha^2,\beta^2\}/\sqrt{n}}\right)\1_{\{x\leq \rho_0+\theta\}},
\end{align*} 
and the event
\begin{align*}
E_{5,3} &:= \left\{ \sup_{K/\sqrt{n}<\theta\leq n^{-1/4}} \left| \frac{1}{\sqrt{n}} \sum_{k=1}^n G_{n,\theta}''(X_{(k-1)/n}) - \sqrt{n} \int_0^1 G_{n,\theta}''(X_s) ds \right| \leq \sqrt{n} \int_0^1 G_{n,\theta}''(X_s) ds \right\} \\
&\hspace{1.5cm} \cap A_2\cap A_3'
\end{align*} 
Similar to $E_{5,1}$, we find $\Pr_{\rho_0}(E_{5,3})>1-\epsilon$ for $n$ large enough. 
On these two set, we then obtain the bound
\begin{align*}
&\sup_{\theta\in\Theta_1}\sum_{k=1}^n \1_{\{X_{(k-1)/n}\leq\rho_0+\theta-1/\sqrt{n}\}} \exp\left(\frac{(X_{(k-1)/n}-\rho_0-\theta)}{4\max\{\alpha^2,\beta^2\}/\sqrt{n}}\right)\1_{E_{5,3}} \\
&\hspace{1cm} \leq \sup_{\theta\in\Theta_1} C_{\alpha,\beta}2n\int_0^1 \1_{\{X_{s}<\rho_0+\theta\}}\exp\left(\frac{(X_{s}-\rho_0-\theta)}{4\max\{\alpha^2,\beta^2\}/\sqrt{n}}\right)ds \\
&\hspace{1cm} \leq \frac{C_{\alpha,\beta}\xi_u n}{\min\{\alpha^2,\beta^2\}}\sup_{\theta\in\Theta_1} \int_{-\infty}^{\rho_0+\theta} \exp\left(\frac{(y-\rho_0-\theta)}{4\max\{\alpha^2,\beta^2\}/\sqrt{n}}\right) dy  \\
&\hspace{1cm} = \frac{C_{\alpha,\beta}\xi_u\sqrt{n}}{\min\{\alpha^2,\beta^2\}}\int_{-\infty}^0 \exp\left(\frac{y}{2\max\{\alpha^2,\beta^2\}}\right) dy .
\end{align*}

\item[$\boldsymbol{-\ \tilde{\Xi}_n^{(5,4)}}$.] On $A_4(n)$ we have
\[ |g_k| \leq C_{\alpha,\beta} \left(1 + n n^{-8/9}\right) \leq C_{\alpha,\beta} n^{1/9}.\]
Moreover, again on $A_3'\cap A_4(n)$ we find with the occupation times formula
\begin{align*}
\sum_{k=1}^n \1_{\{X_{(k-1)/n}<\rho_0+\theta < X_{k/n}\}} &= \sum_{k=1}^n \1_{\{\rho_0+\theta - n^{-4/9}\leq X_{(k-1)/n} \leq \rho_0+\theta + n^{-4/9}\}}\\
& \leq n \int_0^1 \1_{[\rho_0+\theta - 2n^{-4/9},\rho_0+\theta + 2n^{-4/9}]}(X_s) ds \\
&\leq n \frac{4\xi_u}{\min\{\alpha^2,\beta^2\}} n^{-4/9} = \frac{4\xi_u}{\min\{\alpha^2,\beta^2\}} n^{5/9}.
\end{align*}
Both bounds do not depend on $\theta$ such that on $A_3'\cap A_4(n)$ we finally have the bound
\[ \tilde{\Xi}_n^{(5,4)} \leq \frac{4C_{\alpha,\beta}\xi_u}{\min\{\alpha^2,\beta^2\}} n^{2/3}.\]

\item[$\boldsymbol{-\ \Xi_n^{(5,5)}}$.] By Lemma~\ref{lemma:upper_bound_transition_density}, Corollary~\ref{cor:bound_exp_X^2} and boundedness of $x\mapsto x\exp(-x^2/4)$,
\begin{align*}
&\E_{\rho_0}\left[ |g_k| \1_{\{X_{k/n} <\rho_0\leq \rho_0+L/\sqrt{n}\leq X_{(k-1)/n}\}}\right] \\
&\hspace{0.2cm} \leq C_{\alpha,\beta} \E_{\rho_0}\left[ \1_{\{X_{(k-1)/n}\geq \rho_0+L/\sqrt{n}\}}\int_{-\infty}^{\rho_0} \left( 1+ n(y-X_{(k-1)/n})^2\right) \right. \\
&\hspace{7cm} \left.\cdot\sqrt{n} \exp\left(-\frac{(y-X_{(k-1)/n})^2}{2\max\{\alpha^2,\beta^2\}/n}\right) dy \right] \\
&\hspace{0.2cm} \leq C_{\alpha,\beta} \E_{\rho_0}\left[ \1_{\{X_{(k-1)/n}\geq \rho_0+L/\sqrt{n}\}}\int_{-\infty}^{-\sqrt{n}(X_{(k-1)/n}-\rho_0)} \left( 1+ y^2\right) \right. \\
&\hspace{7cm} \left. \cdot \exp\left(-\frac{y^2}{2\max\{\alpha^2,\beta^2\}}\right) dy \right] \\
&\hspace{0.2cm} \leq C_{\alpha,\beta} \E_{\rho_0}\left[ \exp\left(-\frac{(X_{(k-1)/n}-\rho_0)^2}{4\max\{\alpha^2,\beta^2\}/n}\right)\right] \\
&\hspace{0.2cm} \leq \frac{C_{\alpha,\beta}}{\sqrt{k-1}}.
\end{align*}
In particular, we have $\E_{\rho_0}[|\Xi_n^{(5,5)}|/\sqrt{n}] \leq C_{\alpha,\beta}$ as desired.
\end{itemize}\smallskip

$\bullet\ \boldsymbol{S_{4,\theta}(k)}$. This term is straigtforward. We estimate
\begin{align*}
\sum_{k=1}^n g_k\1_{\{\rho_0\leq X_{(k-1)/n}< \rho_0+L/\sqrt{n}\}}\1_{\{\rho_0 <X_{k/n}\leq \rho_0+\theta\}} \leq \sum_{k=1}^n g_k\1_{\{\rho_0\leq X_{(k-1)/n}< \rho_0+L/\sqrt{n}\}} =: \Xi_n^{(5,6)},
\end{align*}
which is independent of $\theta$. By Lemma~\ref{lemma:upper_bound_transition_density} and Corollary~\ref{cor:bound_exp_X^2},
\begin{align*}
&\E_{\rho_0}\left[ |g_k|\1_{\{\rho_0\leq X_{(k-1)/n}< \rho_0+L/\sqrt{n}\}} \right] \\
&\hspace{0.2cm} \leq C_{\alpha,\beta} \E_{\rho_0}\left[\1_{\{\rho_0\leq X_{(k-1)/n}< \rho_0+L/\sqrt{n}\}} \int_\R \left(1+ n(y-X_{(k-1)/n})^2\right)^2 \right. \\
&\hspace{7cm} \left.\cdot \sqrt{n}\exp\left(-\frac{(y-X_{(k-1)/n})^2}{2\max\{\alpha^2,\beta^2\}/n}\right) dy \right] \\
&\hspace{0.2cm} \leq C_{\alpha,\beta}   \E_{\rho_0}\left[\1_{\{\rho_0\leq X_{(k-1)/n}< \rho_0+L/\sqrt{n}\}} \int_\R \left(1+y^2\right) \exp\left(-\frac{y^2}{2\max\{\alpha^2,\beta^2\}}\right) dy \right] \\
&\hspace{0.2cm} \leq C_{\alpha,\beta}   \E_{\rho_0}\left[\1_{\{\rho_0\leq X_{(k-1)/n}< \rho_0+L/\sqrt{n}\}} \right] \\
&\hspace{0.2cm} \leq \frac{C_{\alpha,\beta}L}{\sqrt{k}}
\end{align*}
and thus $\E_{\rho_0}[ |\Xi_n^{(5,6)}|/\sqrt{n}] \leq C_{\alpha,\beta}L$.\\

\noindent
$\boldsymbol{\mathcal{I}_6(\theta)}$: With the inequality
\[ \1_{I_{6,k}^\theta}\left(1- \frac{\alpha-\beta}{\alpha+\beta}\exp\left(-\frac{2}{\alpha^2/n}(X_{k/n}-\rho_0)(X_{(k-1)/n}-\rho_0)\right)\right) \geq 1-\frac{|\alpha-\beta|}{\alpha+\beta}\]
we deduce
\begin{align*}
& \log\left(\frac{P_3^{\rho_0+\theta}(X_{(k-1)/n},X_{k/n};1/n)}{P_2^{\rho_0}(X_{(k-1)/n},X_{k/n};1/n)}\right)\1_{I_{6,k}^\theta}  \\
&=\1_{I_{6,k}^\theta}  \log\left( \frac{2\alpha}{\alpha+\beta} \left( 1+ \frac{\alpha-\beta}{\alpha+\beta}\exp\left(-\frac{2}{\alpha^2/n}(X_{k/n}-\rho_0)(X_{(k-1)/n}-\rho_0)\right)\right)^{-1} \right. \\
&\hspace{1cm}\left. \cdot\exp\left(\frac{(X_{k/n}-X_{(k-1)/n})^2}{2\alpha^2/n} - \frac{n}{2}\left(\frac{X_{k/n}-\rho_0-\theta}{\beta}-\frac{X_{(k-1)/n}-\rho_0-\theta}{\alpha}\right)^2\right)\right)\\
&\leq C_{\alpha,\beta}\1_{I_{6,k}^\theta}\left( 1+ n(X_{k/n}-\rho_0-\theta)^2+ n|X_{k/n}-\rho_0-\theta||X_{(k-1)/n}-\rho_0-\theta|\right) \\
&\leq C_{\alpha,\beta}\1_{\{X_{(k-1)/n}<\rho_0+\theta< X_{k/n}\}}\left( 1+ n(X_{k/n}-X_{(k-1)/n})^2\right), 
\end{align*}
where the last step uses that on $I_{6,k}^\theta$ we have $X_{(k-1)/n}< \rho_0+\theta< X_{k/n}$. The claim follows now along the lines of the treatment of the remainder terms $\Xi_n^{(5,4)}$, $\tilde{\Xi}_n^{(5,4)}$ and gives upper bounds $\Xi_n^{(6)}$ and $\tilde{\Xi}_n^{(6)}$ on a set with high probability, where  $\E_{\rho_0}[|\Xi_n^{(6)}|]\leq C_{\alpha,\beta}\sqrt{n}$ and $\E_{\rho_0}[|\tilde{\Xi}_n^{(6)}|]\leq C_{\alpha,\beta}n^{2/3}$, respectively.\\

\noindent
$\boldsymbol{\mathcal{I}_7(\theta)}$: This follows by the same reasoning as $\mathcal{I}_3(\theta)$ and provides an upper bound $\Xi_n^{(7)}$ with $\E_{\rho_0}[ |\Xi_n^{(7)}|/\sqrt{n}] \leq C_{\alpha,\beta}$\\

\noindent
$\boldsymbol{\mathcal{I}_8(\theta)}$: Along the lines of $\mathcal{I}_2(\theta)$, we find
\begin{align*}
& \log\left(\frac{P_4^{\rho_0+\theta}(X_{(k-1)/n},X_{k/n};1/n)}{P_2^{\rho_0}(X_{(k-1)/n},X_{k/n};1/n)}\right)\1_{I_{8,k}^\theta}  \\
&\hspace{1cm}\leq C_{\alpha,\beta}\1_{\{X_{k/n}\leq\rho_0+\theta\leq X_{(k-1)/n}\}}\left( 1+ n(X_{k/n}-X_{(k-1)/n})^2\right). 
\end{align*}
Then the claim follows along the lines of the treatment of the remainder terms $\Xi_n^{(5,4)}$, $\tilde{\Xi}_n^{(5,4)}$ by replacing $\1_{\{X_{(k-1)/n}<\rho_0+\theta< X_{k/n}\}}$ with $\1_{\{X_{k/n}\leq\rho_0+\theta\leq X_{(k-1)/n}\}}$. On a set with high probability, this yields upper bounds $\Xi_n^{(8)}$ and $\tilde{\Xi}_n^{(8)}$ with $\E_{\rho_0}[|\Xi_n^{(8)}|]\leq C_{\alpha,\beta}\sqrt{n}$ and $\E_{\rho_0}[|\tilde{\Xi}_n^{(8)}|] \leq C_{\alpha,\beta}n^{2/3}$, respectively.\\

\noindent
$\boldsymbol{\mathcal{I}_9(\theta)}$: Similar as for $I_1(\theta)$ we obtain
\begin{align*}
& \1_{I_{9,k}^\theta} \log\left(\frac{P_2^{\rho_0+\theta}(X_{(k-1)/n},X_{k/n};1/n)}{P_2^{\rho_0}(X_{(k-1)/n},X_{k/n};1/n)}\right) \\
&\hspace{1cm} = \1_{I_{9,k}^\theta} \log\left(\frac{1+\frac{\alpha-\beta}{\alpha+\beta}\exp\left(-\frac{2}{\beta^2/n}(X_{k/n}-\rho_0-\theta)(X_{(k-1)/n}-\rho_0-\theta)\right)}{1+\frac{\alpha-\beta}{\alpha+\beta}\exp\left(-\frac{2}{\beta^2/n}(X_{k/n}-\rho_0)(X_{(k-1)/n}-\rho_0)\right)}\right) \\
&\hspace{1cm} \leq C_{\alpha,\beta}\1_{I_{9,k}^\theta} \exp\left(-\frac{2}{\beta^2/n}(X_{k/n}-\rho_0-\theta)(X_{(k-1)/n}-\rho_0-\theta)\right) \\
&\hspace{3cm} \cdot \left| \exp\left(-\frac{2\theta}{\beta^2/n}(X_{k/n}-\rho_0+X_{(k-1)/n}-\rho_0-\theta)\right) - 1\right| \\
&\hspace{1cm} \leq C_{\alpha,\beta}\1_{I_{9,k}^\theta} \exp\left(-\frac{2}{\beta^2/n}(X_{k/n}-\rho_0-\theta)(X_{(k-1)/n}-\rho_0-\theta)\right),
\end{align*}
where the last step uses $X_{(k-1)/n}\geq\rho_0+\theta,X_{k/n}> \rho_0+\theta$. From this, it follows that
\begin{align*}
& \sum_{k=1}^n \1_{I_{9,k}^\theta} \log\left(\frac{P_2^{\rho_0+\theta}(X_{(k-1)/n},X_{k/n};1/n)}{P_2^{\rho_0}(X_{(k-1)/n},X_{k/n};1/n)}\right) \\ 
&\hspace{0.5cm} \leq C_{\alpha,\beta} \sum_{k=1}^n \1_{\{\rho_0+\theta \leq X_{(k-1)/n} \leq \rho_0+\theta+1/\sqrt{n}\}} \\
&\hspace{1.5cm}+ C_{\alpha,\beta} \sum_{k=1}^n \1_{\{X_{(k-1)/n}> \rho_0+\theta+1/\sqrt{n}\}}\1_{\{X_{k/n}> \rho_0+\theta\}}\\
&\hspace{3cm} \cdot\exp\left(-\frac{2}{\beta^2/n}(X_{k/n}-\rho_0-\theta)(X_{(k-1)/n}-\rho_0-\theta)\right) \\
&\hspace{0.5cm} \leq C_{\alpha,\beta} \sum_{k=1}^n \1_{\{\rho_0+\theta \leq X_{(k-1)/n} \leq \rho_0+\theta+1/\sqrt{n}\}} \\
&\hspace{2
cm} + C_{\alpha,\beta} \sum_{k=1}^n \1_{\{X_{k/n}> \rho_0+\theta\}}\exp\left(-\frac{2}{\beta^2/\sqrt{n}}(X_{k/n}-\rho_0-\theta)\right).
\end{align*}
The first summand can now be treated as the same part appearing in $\Xi_n^{(5,4)}$ in $\mathcal{I}_{5}(\theta)$ with the slightly different indicator $\1_{\{\rho_0+\theta-1/\sqrt{n}\leq X_{(k-1)/n}\leq \rho_0+\theta\}}$ and the second one as $\Xi_n^{(5,1)}$, $\tilde{\Xi}_n^{(5,1)}$ in the discussion of $\mathcal{I}_5(\theta)$. In particular, on a set with high probability, we get upper bounds $\Xi_n^{(9)}$ and $\tilde{\Xi}_n^{(9)}$ with $\E_{\rho_0}[|\Xi_n^{(9)}|]\leq C_{\alpha,\beta}\sqrt{n}$ and $\E_{\rho_0}[|\tilde{\Xi}_n^{(9)}|] \leq C_{\alpha,\beta}n^{2/3}$, respectively.\\

\noindent 
To sum up, we have shown the statement of the lemma with
\[ F_n^1(K,L) := \sum_{\substack{j=1, \dots, 9 \\ j\neq 5}} \Xi_n^{(j)} + \sum_{l=1}^6  \Xi_n^{(5,l)} \]
and
\[ F_n^2(K,L) := \tilde{\Xi}_n^{(5,1)} +\Xi_n^{(5,2)} + \tilde{\Xi}_n^{(5,3)}+\tilde{\Xi}_n^{(5,4)}+\Xi_n^{(5,5)}+\Xi_n^{(5,6)}+ \sum_{j=1}^4 \Xi_n^{(j)} +\sum_{j=6}^9 \tilde{\Xi}_n^{(j)}. \]
\end{proof}

\begin{proof}[Proof of Lemma~\ref{lemma:E_squared_difference}]
In this proof, we set $B_k := \{\rho_0+L/\sqrt{n} \leq X_{(k-1)/n}<\rho_0+\theta\}$. Decomposing $1=\1_{\{X_{k/n}<\rho_0\}}+\1_{\{X_{k/n}\geq\rho_0\}}$, reveals by the triangle inequality
\begin{align}\label{eq_proof:E_squared_difference_1}
\begin{split}
&\left|\E_{\rho_0}\left[ \left. n(X_{k/n}-X_{(k-1)/n})^2\right| X_{(k-1)/n}\right]-\beta^2\right|  \1_{B_k}  \\
&\hspace{1cm} \leq  \left|\E_{\rho_0}\left[ \left.\left(n(X_{k/n}-X_{(k-1)/n})^2-\beta^2 \right) \1_{B_k} \1_{\{X_{k/n}<\rho_0\}} \right| X_{(k-1)/n}\right] \right| \\
&\hspace{2cm} + \left|\E_{\rho_0}\left[ \left.\left(n(X_{k/n}-X_{(k-1)/n})^2-\beta^2\right) \1_{B_k} \1_{\{X_{k/n}\geq \rho_0\}} \right| X_{(k-1)/n}\right] \right|.
\end{split}
\end{align}
Subsequently, both summmands will be bounded seperately. The first expectation on the right-hand side of~\eqref{eq_proof:E_squared_difference_1} can be rewritten as
\begin{align*}
&\left| \E_{\rho_0}\left[ \left.\left(n(X_{k/n}-X_{(k-1)/n})^2-\beta^2\right) \1_{B_k}\1_{\{X_{k/n}<\rho_0\}} \right| X_{(k-1)/n}\right]\right| \\
&\hspace{0.1cm} = \1_{B_k}\left| \int_{-\infty}^{\rho_0} \left(n(y-X_{(k-1)/n})^2-\beta^2\right) \frac{1}{\sqrt{2\pi/n}} \frac{2}{\alpha+\beta}\frac{\beta}{\alpha}\exp\left(-\frac{n}{2}\left(\frac{y-\rho_0}{\alpha}-\frac{X_{(k-1)/n}-\rho_0}{\beta}\right)^2\right) dy \right| \\
&\hspace{0.1cm} =  \1_{B_k} \frac{2\alpha^2\beta}{\alpha+\beta} \int_{-\infty}^{-(X_{(k-1)/n}-\rho_0)/\sqrt{\beta^2/n}} \left(\left( y - \frac{(\beta-\alpha)(X_{(k-1)/n}-\rho_0)}{\alpha\beta/\sqrt{n}}\right)^2-\frac{\beta^2}{\alpha^2}\right) \frac{1}{\sqrt{2\pi}}\exp\left(-\frac{y^2}{2}\right) dy.
\end{align*}
Now, using $X_{(k-1)/n}-\rho_0\geq L/\sqrt{n}$ on $B_k$, $x\exp(-x/4)\leq 4e^{-1}$ for $x\geq 0$ and the Gaussian tail inequality, we obtain
\begin{align*}
&\1_{B_k}\left| \int_{-\infty}^{-(X_{(k-1)/n}-\rho_0)/\sqrt{\beta^2/n}} \left(\left( y - \frac{(\beta-\alpha)(X_{(k-1)/n}-\rho_0)}{\alpha\beta/\sqrt{n}}\right)^2 - \frac{\beta^2}{\alpha^2}\right) \frac{1}{\sqrt{2\pi}}\exp\left(-\frac{y^2}{2}\right) dy\right| \\
&\hspace{0.2cm} \leq C_{\alpha,\beta}\1_{B_k} \int_{-\infty}^{-L/\beta} \left(y^2 +1\right) \exp\left(-\frac{y^2}{2}\right) dy  \\
&\hspace{1.2cm} +C_{\alpha,\beta}\1_{B_k}n(X_{(k-1)/n}-\rho_0)^2 \int_{-\infty}^{-(X_{(k-1)/n}-\rho_0)/\sqrt{\beta^2/n}} \exp\left(-\frac{y^2}{2}\right) dy \\
&\hspace{0.2cm} \leq C_{\alpha,\beta}  \int_{-\infty}^{-L/\beta} \left(y^2 +1\right) \exp\left(-\frac{y^2}{2}\right) dy  + C_{\alpha,\beta}\1_{B_k}n(X_{(k-1)/n}-\rho_0)^2 \exp\left(-\frac{(X_{(k-1)/n}-\rho_0)^2}{2\beta^2/n}\right)\\
&\hspace{0.2cm} \leq C_{\alpha,\beta} \int_{-\infty}^{-L/\beta} \left(y^2 +1\right) \exp\left(-\frac{y^2}{2}\right) dy  + C_{\alpha,\beta}\1_{B_k}\exp\left(-\frac{(X_{(k-1)/n}-\rho_0)^2}{4\beta^2/n}\right) \\
&\hspace{0.2cm} \leq C_{\alpha,\beta} \int_{-\infty}^{-L/\beta} \left(y^2 +1\right) \exp\left(-\frac{y^2}{2}\right) dy +C_{\alpha,\beta} \exp\left(-\frac{L^2}{4\beta^2}\right).
\end{align*}
Consequently, for every $\epsilon>0$ there exists $L=L(\epsilon)>0$ large enough such that
\begin{align}\label{eq_proof:E_squared_difference_2}
\left| \E_{\rho_0}\left[ \left.\left(n(X_{k/n}-X_{(k-1)/n})^2-\beta^2\right) \1_{B_k} \1_{\{X_{k/n}<\rho_0\}} \right| X_{(k-1)/n}\right]\right| \leq  \frac{\epsilon}{2}.
\end{align} 
For the second expectation on the right-hand side of~\eqref{eq_proof:E_squared_difference_1}, we find
\begin{align*}
& \left| \E_{\rho_0}\left[ \left.\left(n(X_{k/n}-X_{(k-1)/n})^2-\beta^2\right) \1_{B_k} \1_{\{X_{k/n}\geq \rho_0\}} \right| X_{(k-1)/n}\right] \right|\\
&\hspace{0.2cm} \leq  \1_{B_k}\left| \int_{\rho_0}^\infty \left(n(y-X_{(k-1)/n})^2-\beta^2\right) \frac{1}{\sqrt{2\pi\beta^2/n}}\exp\left(-\frac{(y-X_{(k-1)/n})^2}{2\beta^2/n}\right) dy  \right|\\
&\hspace{0.5cm} +  \1_{B_k}\left| \int_{\rho_0}^\infty \left(n(y-X_{(k-1)/n})^2-\beta^2\right) \frac{1}{\sqrt{2\pi\beta^2/n}}\frac{\alpha-\beta}{\alpha+\beta}\exp\left(-\frac{(y-2\rho_0+X_{(k-1)/n})^2}{2\beta^2/n}\right) dy  \right|\\
&\hspace{0.2cm} = \1_{B_k}\beta^2 \left|\int_{(\rho_0-X_{(k-1)/n})/\sqrt{\beta^2/n}}^\infty \left(y^2-1\right) \frac{1}{\sqrt{2\pi}} \exp\left(-\frac{y^2}{2}\right) dy \right| \\
&\hspace{0.5cm} +  \1_{B_k} \beta^2 \left| \frac{\alpha-\beta}{\alpha+\beta} \int_{(X_{(k-1)/n}-\rho_0)/\sqrt{\beta^2/n}}^\infty \left(\left(y-\frac{2(X_{(k-1)/n}-\rho_0)}{\beta/\sqrt{n}}\right)^2-1\right) \frac{1}{\sqrt{2\pi}} \exp\left(-\frac{y^2}{2}\right) dy  \right|.
\end{align*}
Note that we have
\begin{align*}
&\1_{B_k}\beta^2 \left|\int_{(\rho_0-X_{(k-1)/n})/\sqrt{\beta^2/n}}^\infty \left(y^2-1\right) \frac{1}{\sqrt{2\pi}} \exp\left(-\frac{y^2}{2}\right) dy \right| \\
&\hspace{3cm} \leq \beta^2 \left|\int_{-L/\beta}^\infty \left(y^2-1\right) \frac{1}{\sqrt{2\pi}} \exp\left(-\frac{y^2}{2}\right) dy \right|
\end{align*} 
and we see that $L$ can be chosen large enough such that the right-hand side is bounded from above by $\epsilon/4$, because $\int_\R (y^2-1)\exp(-y^2/2)dy =0$. Moreover, using again the inequality $x\exp(-x/4)\leq 4e^{-1}$ for $x\geq 0$ and the Gaussian tail inequality, we obtain
\begin{align*}
& \1_{B_k}\left| \int_{(X_{(k-1)/n}-\rho_0)/\sqrt{\beta^2/n}}^\infty \left(\left(y-\frac{2(X_{(k-1)/n}-\rho_0)}{\beta/\sqrt{n}}\right)^2 -1\right) \frac{1}{\sqrt{2\pi}} \exp\left(-\frac{y^2}{2}\right) dy \right| \\
&\hspace{0.2cm} \leq 2\1_{B_k} \int_{(X_{(k-1)/n}-\rho_0)/\sqrt{\beta^2/n}}^\infty \left(y^2+1\right) \frac{1}{\sqrt{2\pi}} \exp\left(-\frac{y^2}{2}\right) dy  \\
&\hspace{1.2cm} + \frac{8(X_{(k-1)/n}-\rho_0)^2}{\beta^2/n}\int_{(X_{(k-1)/n}-\rho_0)/\sqrt{\beta^2/n}}^\infty  \frac{1}{\sqrt{2\pi}} \exp\left(-\frac{y^2}{2}\right) dy \\
&\hspace{0.2cm} \leq 2 \int_{L/\beta}^\infty \left(y^2+1\right) \frac{1}{\sqrt{2\pi}} \exp\left(-\frac{y^2}{2}\right) dy  + \1_{B_k}\frac{4(X_{(k-1)/n}-\rho_0)^2}{\beta^2/n}\exp\left(-\frac{(X_{(k-1)/n}-\rho_0)^2}{2\beta^2/n}\right)\\
&\hspace{0.2cm} \leq 2\int_{L/\beta}^\infty \left(y^2+1\right) \frac{1}{\sqrt{2\pi}} \exp\left(-\frac{y^2}{2}\right) dy + \1_{B_k}\frac{4}{e}\exp\left(-\frac{(X_{(k-1)/n}-\rho_0)^2}{4\beta^2/n}\right)\\
&\hspace{0.2cm} \leq 2\int_{L/\beta}^\infty \left(y^2+1\right) \frac{1}{\sqrt{2\pi}} \exp\left(-\frac{y^2}{2}\right) dy + \frac{4}{e}\exp\left(-\frac{L^2}{4\beta^2}\right).
\end{align*}
Hence, we can choose $L=L(\epsilon)$ large enough such that this bound is again bounded by $\epsilon/4$. In particular, we have shown that for every $\epsilon>0$ there exists an $L=L(\epsilon)>0$ large enough such that
\begin{align*}
&\left| \E_{\rho_0}\left[\left( n(X_{k/n}-X_{(k-1)/n})^2 -\beta^2\right) \1_{B_k} \1_{\{X_{k/n}\geq \rho_0\}} | X_{(k-1)/n}\right] \right|  \leq \frac{\epsilon}{2}
\end{align*}
and the claim follows by combining this with~\eqref{eq_proof:E_squared_difference_1} and~\eqref{eq_proof:E_squared_difference_2}.
\end{proof}

\begin{proof}[Proof of \eqref{eq_proof:sqrt(n)_chaining_bound_new}]
Exploiting the martingale structure of $N_n-\overline{N}_n$, one can easily see that (details are given direcly after this proof) 
\begin{align}\label{eq_proof:sqrt(n)_consistency_1}
\E_{\rho_0}\left[ \left( N_n^L(\theta)-N_n^L(\theta') + \overline{N}_n^L(\theta') - \overline{N}_n^L(\theta)\right)^2\right] \leq c n |\theta-\theta'|
\end{align} 
for some constant $c=c(\alpha,\beta)>0$. With $\psi(x)=x^2$ and the metric $\rho_n(\theta,\theta') := \sqrt{c n|\theta-\theta'|}$, we obtain
\[ \E_{\rho_0}\left[ \psi\left(\frac{N_n^L(\theta)-N_n^L(\theta') + \overline{N}_n^L(\theta') - \overline{N}_n^L(\theta)}{\rho_n(\theta,\theta')}\right)\right] \leq 1.\]
As the piecewise constant function $\theta\mapsto N_n^L(\theta)-\overline{N}_n^L(\theta)$ is càdlàg, we have
\[  \sup_{\theta\in S_{n,j}} \left(N_n^L(\theta) - \overline{N}_n^L(\theta) \right)=  \sup_{\theta\in S_{n,j} \cap \Q} \left(N_n^L(\theta) - \overline{N}_n^L(\theta) \right)\]
and thus can work on a countable set. For an enumeration $(\theta_m)_{m\in\N}$ of $[0,2^{j+1}/\sqrt{n}]\cap\Q \supset S_{n,j}\cap\Q$, by monotone convergence,
\[  \E_{\rho_0}\left[\sup_{\theta\in [0,2^{j+1}/\sqrt{n}]\cap\Q} \left(N_n^L(\theta) - \overline{N}_n^L(\theta) \right)\right] = \lim_{m\to\infty}\E_{\rho_0} \left[ \sup_{\theta\in \{\theta_1,\dots,\theta_m\}} \left(N_n^L(\theta) - \overline{N}_n^L(\theta) \right)\right]. \]
Hence, it is enough to evaluate the supremum over any finite subset $\mathcal{T}^n$ of $[0,2^{j+1}/\sqrt{n}]\cap\Q$. Without loss of generality, we assume $\theta\in\mathcal{T}^n$ and set $0=\theta_0^n\in \mathcal{T}^n$. Furthermore, we define
\[ \delta_0^n := \sup_{\theta\in \mathcal{T}^n} \rho_n(\theta,\theta_0^n) \leq \sqrt{c n2^{j+1}/\sqrt{n}} = \sqrt{C_2}2^{(j+1)/2}n^{1/4}\]
and $\delta_k^n := 2^{-k}\delta_0^n$ for $k\in\N$. Now we inductively define maximal subsets $\mathcal{T}_k^n\subset \mathcal{T}^n$ such that
\[ \mathcal{T}_{k-1}^n\subset\mathcal{T}_k^n \qquad \textrm{ and }\qquad \rho_n(\theta,\theta') \leq \delta_k^n \textrm{ for different } \theta,\theta'\in\mathcal{T}_k^n.\]
In particular, the cardinality $\#\mathcal{T}_k^n$ is bounded by
\[ \#\mathcal{T}_k^n \leq D\left( \delta_k^n, \mathcal{T}^n,\rho_n\right) = \max\{ \#\mathcal{T}_0: \mathcal{T}_0\subset \mathcal{T}^n, \rho_n(\theta,\theta')\geq \delta_k^n \textrm{ for different } \theta,\theta'\in\mathcal{T}_0\}\]
for all $k\in\N$ and we have
\[ \sup_{\theta\in\mathcal{T}_k^n} \rho_n(\theta,\theta') \leq \delta_k^n \quad \textrm{ for all } \theta'\in\mathcal{T}^n \textrm{ and } k\in\N.\]
By construction, it is now possible to find for each $k\in\N$ and $\theta\in\mathcal{T}^n$ a point $\pi_k^n\theta\in\mathcal{T}_k^n$ such that $\rho_n(\theta,\pi_k^n\theta)\leq\delta_k^n$ and for $Z_n(\theta)=N_n^L(\theta)-\overline{N}_n^L(\theta)$ we have
\begin{align}\label{eq_proof:sqrt(n)_chaining_one_step}
\begin{split}
\sup_{\theta\in \mathcal{T}_{k+1}^n} |Z_n(\theta)| &\leq \max_{\theta\in\mathcal{T}_{k+1}^n}\left( |Z_n(\pi_k^n\theta)| + |Z_n(\theta)-Z_n(\pi_k^n\theta)|\right) \\
&\leq \sup_{\theta\in \mathcal{T}_{k}^n} |Z_n(\theta)| +  \max_{\theta\in\mathcal{T}_{k+1}^n} |Z_n(\theta)-Z_n(\pi_k^n\theta)|.
\end{split}
\end{align}
Using  Pisier's inequality
\begin{align}\label{inequality_Pisier}
\E_{\rho_0}\left[ \max_{j=1,\dots, m} |Y_i|\right] \leq \phi^{-1}\left(\sum_{j=1}^m \E_{\rho_0}[\phi(Y_i)]\right)
\end{align}
for arbitrary real-valued random variables $Y_1,\dots, Y_m$ and even, convex $\phi$ with $\phi(0)=0$, $\phi\neq 0$ and $\phi^{-1}(v):=\max\{u\geq 0: \phi(u)\leq v\}$, we find with $\phi=\psi$ in~\eqref{inequality_Pisier}
\begin{align}\label{eq_proof:sqrt(n)_chaining_integral}
\begin{split}
\E_{\rho_0}\left[ \max_{\theta\in\mathcal{T}_{k+1}^n} |Z_n(\theta)-Z_n(\pi_k^n\theta)|\right] &= \delta_k^n \E_{\rho_0}\left[ \max_{\theta\in\mathcal{T}_{k+1}^n} \frac{|Z_n(\theta)-Z_n(\pi_k^n\theta)|}{\delta_k}\right] \\
&\leq \delta_k \psi^{-1}\left( D(\delta_{k+1}^n, \mathcal{T}^n,\rho_n)\right) \\
&= 4\left(\delta_{k+1}^n - \delta_{k+2}^n \right) \psi^{-1}\left( D(\delta_{k+1}^n, \mathcal{T}^n,\rho_n)\right) \\
&\leq 4\int_{\delta_{k+2}^n}^{\delta_{k+1}^n} \psi^{-1}\left( D(u, \mathcal{T}^n,\rho_n)\right) du.
\end{split}
\end{align}
Combining this inequality for all $k=1,\dots, k_0$, where $k_0$ is the largest value for $k$ such that $\delta_{k-1}^n \geq 1$ and using our particular choice $\psi(x)=x^2$ from above as well as $Z_n(0)=N_n^L(0)-\overline{N}_n^L(0)=0$, we arrive at
\begin{align}\label{eq_proof:sqrt(n)_chaining_bound}
\begin{split}
&\E_{\rho_0}\left[ \sup_{\theta\in\mathcal{T}^n} \left( N_n^L(\theta)-\overline{N}_n^L(\theta)\right)\right]\\
&\hspace{1cm} \leq 4\int_1^{\delta_0^n/2} \sqrt{D(u, \mathcal{T}^n,\rho_n)} du + \E_{\rho_0}\left[\sup_{\theta\in\mathcal{T}_{k_0+1}^n} |Z_n(\theta)-Z_n(\pi_{k_0}^n\theta)|\right].
\end{split}
\end{align}
To evaluate the first integral, we note that
\[D\left(u, \mathcal{T}^n,\rho_n\right) \leq D\left(u, [0,2^{j+1}/\sqrt{n},\rho_n\right) \leq c 2^{j+1}\sqrt{n} u^{-2}. \]
Next, we treat the remainder in~\eqref{eq_proof:sqrt(n)_chaining_bound}. For $\theta,\theta'\in\mathcal{T}_{k_0}^n$ we have $\rho_n(\theta,\theta')\leq 1$ by definition of $\mathcal{T}_k^n$ and $k_0$. In particular we have $|\theta-\theta'|\leq (c n)^{-1}$. Denoting by $U_\epsilon(x)$ an $\epsilon$-environment of $x$, we thus have by using~\eqref{inequality_Pisier}
\begin{align}\label{eq_proof:sqrt(n)_Pisier}
\begin{split}
\E_{\rho_0}\left[\sup_{\theta\in\mathcal{T}_{k_0+1}^n} |Z_n(\theta)-Z_n(\pi_{k_0}^n\theta)|\right] &\leq \E_{\rho_0}\left[\sup_{\theta,\theta'\in\mathcal{T}_{k_0}^n} |Z_n(\theta)-Z_n(\theta')|\right]\\
&\leq \E_{\rho_0}\left[\sup_{\theta\in \mathcal{T}_{k_0}^n} \sup_{\theta'\in U_{1/(cn)}(\theta)} |Z_n(\theta)-Z_n(\theta')|\right] \\
&\leq \left( \sum_{\theta\in \mathcal{T}_{k_0}^n} \E_{\rho_0}\left[ \sup_{\theta'\in U_{1/(cn)}(\theta)} |Z_n(\theta)-Z_n(\theta')|^2\right]\right)^\frac12.
\end{split}
\end{align}
First, note that the number of elements contained in $\mathcal{T}_{k_0}^n$ is bounded from above by $4c\sqrt{n}2^{j+1}$ since $\delta_{k_0}^n\geq 1/2$. Second, we can bound (assuming wlog $\theta<\theta'$)
\begin{align*}
&Z_n(\theta)-Z_n(\theta') \\
&\hspace{0.5cm}= n\left(\frac{1}{\alpha^2}-\frac{1}{\beta^2}\right)\sum_{k=1}^n \left( (X_{k/n}-X_{(k-1)/n})^2 - \E_{\rho_0}[(X_{k/n}-X_{(k-1)/n})^2\mid X_{(k-1)/n}]\right)\\
&\hspace{4cm} \cdot\1_{\{\rho_0+\theta\leq X_{(k-1)/n}<\rho_0+\theta'\}}\1_{\{X_{(k-1)/n}\geq \rho_0+ L/\sqrt{n}\}} \\
&\hspace{0.5cm} \leq  n \left|\frac{1}{\alpha^2}-\frac{1}{\beta^2}\right| \sum_{k=1}^n \left( (X_{k/n}-X_{(k-1)/n})^2 + \E_{\rho_0}[(X_{k/n}-X_{(k-1)/n})^2\mid X_{(k-1)/n}]\right)\\
&\hspace{4cm} \cdot \1_{\{\rho_0+\theta\leq X_{(k-1)/n}<\rho_0+\theta+1/(cn)\}},
\end{align*}
which is independent of $\theta'$. This is due to the special structure of the martingale part $N_n^L(\theta)-\overline{N}_n^L(\theta)$ and the essential observation that makes the bound~\eqref{eq_proof:sqrt(n)_chaining_bound} useful. Using the estimate $(a+b)^2\leq 2a^2 + 2b^2$ and Lemma~\ref{lemma:second_moment_sum_intervals}, we find
\begin{align*}
\begin{split}
&\E_{\rho_0}\left[\left( n\sum_{k=1}^n \left( (X_{k/n}-X_{(k-1)/n})^2 + \E_{\rho_0}[(X_{k/n}-X_{(k-1)/n})^2\mid X_{(k-1)/n}]\right) \right.\right. \\
&\hspace{2.5cm} \cdot \left.\1_{\{\rho_0+\theta\leq X_{(k-1)/n}<\rho_0+\theta+1/(cn)\}}\Big)^2\right] \leq c',
\end{split}
\end{align*}
where the constant $c'=c'(\alpha,\beta)>0$ is independent of $n$ and $\theta$. By using this bound, we get from~\eqref{eq_proof:sqrt(n)_chaining_bound} and~\eqref{eq_proof:sqrt(n)_Pisier} that for some constant $C_2=C_2(\alpha,\beta)>0$
\begin{align*}
&\E_{\rho_0}\left[ \sup_{\theta\in\mathcal{T}^n} \left( N_n^L(\theta)-\overline{N}_n^L(\theta)\right)\right] \\
&\hspace{1cm}\leq 4\int_1^{\sqrt{c}2^{j/2}n^{1/4}} \sqrt{c} 2^{(j+1)/2}n^{1/4} u^{-1} du + 2\sqrt{cc'} 2^{(j+1)/2}n^{1/4} \\
&\hspace{1cm}= C_2 2^{(j+1)/2}n^{1/4} \left( j\log(2) + \log(n)\right).
\end{align*}
\end{proof}

\begin{proof}[Proof of \eqref{eq_proof:sqrt(n)_consistency_1}]
Note that $N_n^L(\theta)-N_n^L(\theta') + \overline{N}_n^L(\theta')-\overline{N}_n^L(\theta) = \sum_{k=1}^n d_k(\theta,\theta')$, where $\E_{\rho_0}[d_k(\theta,\theta')\mid \cF_{(k-1)/n}] =0$. Hence, for $l<k$ we have
\begin{align*}
\E_{\rho_0}\left[ d_k(\theta,\theta') d_l(\theta,\theta')\right] &= \E_{\rho_0}\left[\E_{\rho_0}\left[ d_k(\theta,\theta') d_l(\theta,\theta')\mid \cF_{(k-1)/n}\right]\right] \\
&=  \E_{\rho_0}\left[d_l(\theta,\theta') \E_{\rho_0}\left[ d_k(\theta,\theta')\mid \cF_{(k-1)/n}\right]\right] = 0
\end{align*} 
and consequently,
\begin{align}\label{eq_proof:sqrt(n)_consistency_1_1}
\E_{\rho_0}\left[ \left( \sum_{k=1}^n d_k(\theta,\theta')\right)^2\right] = \sum_{k=1}^n \E_{\rho_0}\left[ d_k(\theta,\theta')^2 \right].
\end{align} 
For $\theta'\leq\theta$,
\begin{align*}
d_k(\theta,\theta') &= n\left(\frac{1}{\alpha^2}-\frac{1}{\beta^2}\right) \left( (X_{k/n}-X_{(k-1)/n})^2 - \E_{\rho_0}[(X_{k/n}-X_{(k-1)/n})^2\mid X_{(k-1)/n}]\right)\\
&\hspace{2cm} \cdot \1_{\{\rho_0+\theta'\leq X_{(k-1)/n}< \rho_0+\theta\}}.
\end{align*}
Using that by Lemma~\ref{lemma:upper_bound_transition_density}
\begin{align}\label{eq_proof:sqrt(n)_consistency_1_2}
\begin{split}
&\E_{\rho_0}\left[\left. n^2(X_{k/n}-X_{(k-1)/n})^4\right| X_{(k-1)/n}\right]\\
&\hspace{1cm}\leq C_{\alpha,\beta}\int_\R n^2(y-X_{(k-1)/n})^4 \sqrt{n} \exp\left(-\frac{(y-X_{(k-1)/n})^2}{2\max\{\alpha^2,\beta^2\}/n}\right) dy \\
&\hspace{1cm}= C_{\alpha,\beta}  \int_\R y^4 \exp\left(-\frac{y^2}{2\max\{\alpha^2,\beta^2\}}\right) dy\leq C_{\alpha,\beta},
\end{split}
\end{align} 
and $(a+b)^2\leq 2a^2 + 2b^2$ yields again with Lemma~\ref{lemma:upper_bound_transition_density}
\begin{align*}
&\E_{\rho_0}\left[ d_k(\theta,\theta')^2 \right]\\
&\hspace{0.2cm}\leq 2 \left(\frac{1}{\alpha^2}-\frac{1}{\beta^2}\right)^2 \E_{\rho_0}\Big[\left(n^2(X_{k/n}-X_{(k-1)/n})^4 + n^2\E_{\rho_0}[(X_{k/n}-X_{(k-1)/n})^2\mid X_{(k-1)/n}]^2\right)\\
&\hspace{1cm} \cdot  \1_{\{\rho_0+\theta'\leq X_{(k-1)/n}< \rho_0+\theta\}} \Big] \\
&\hspace{0.2cm}\leq 4\left(\frac{1}{\alpha^2}-\frac{1}{\beta^2}\right)^2  \E_{\rho_0}\left[\E_{\rho_0}\left[n^2(X_{k/n}-X_{(k-1)/n})^4\mid X_{(k-1)/n}\right] \1_{\{\rho_0+\theta'\leq X_{(k-1)/n}< \rho_0+\theta\}} \right] \\
&\hspace{0.2cm}\leq 4C_{\alpha,\beta}\left(\frac{1}{\alpha^2}-\frac{1}{\beta^2}\right)^2 \E_{\rho_0}\left[\1_{\{\rho_0+\theta'\leq X_{(k-1)/n}< \rho_0+\theta\}} \right]\\
&\hspace{0.2cm}\leq C_{\alpha,\beta} \int_{\rho_0+\theta'}^{\rho_0+\theta} \frac{1}{\sqrt{2\pi(k-1)/n}}\exp\left(-\frac{(y-x_0)^2}{2\max\{\alpha^2,\beta^2\}(k-1)/n}\right) dy \\
&\hspace{0.2cm}\leq C_{\alpha,\beta} \frac{1}{\sqrt{k/n}}|\theta-\theta'|.
\end{align*}
We then arrive at
\begin{align*}
\sum_{k=1}^n \E_{\rho_0}\left[ d_k(\theta,\theta')^2 \right] &\leq C_{\alpha,\beta}\sqrt{n}|\theta-\theta'| \sum_{k=2}^n \frac{1}{\sqrt{k}}\leq C_{\alpha,\beta} n|\theta-\theta'|
\end{align*} 
and~\eqref{eq_proof:sqrt(n)_consistency_1} follows.
\end{proof}

The following result may be well-known, yet we did not find an appropriate reference in the literature.

\begin{lemma}\label{lemma:L2-norm_occupation_approximation}
Let $f_n:\R\longrightarrow\R$, $n\in\N$, be measurable functions such that for $m=1,2$, $\int_\R |x|^m |f_n(x)|dx<\infty$, $\|f_n\|_{L^1}\leq \kappa/\sqrt{n}$, $\|f_n\|_{\textrm{sup}} \leq \kappa$ and $\E_{\rho_0}[|f_n(X_{(k-1)/n})|]\leq \kappa/\sqrt{k}$ for all $1\leq k\leq n+1$ and some constant $\kappa>0$. Then we have
\[ \E_{\rho_0}\left[ \left( \frac{1}{\sqrt{n}}\sum_{k=1}^n f_n(X_{(k-1)/n}) - \sqrt{n} \int_0^1 f_n(X_s) ds \right)^2\right] \leq C_{\alpha,\beta}(\kappa)n^{-1/2}.\]
\end{lemma}
\begin{proof}
The proof makes use of the following inequality several times: For every function $f_n$ that satisfies the assumptions of the lemma, there exists a constant $C_{\alpha,\beta}$ that is independent of $f_n$ such that for all $0<s<t$
\begin{align}\label{eq_proof:inequality_semigroup}
\left| \E_{\rho_0}\left[ f_n(X_{u+t}) - f_n(X_{u+s}) \mid X_u\right]\right| \leq C_{\alpha,\beta} \frac{t-s}{s^{3/2}} \|f_n\|_{L^1}. 
\end{align}
This inequality follows direcly from item (v) in Lemma~$3$ in~\citeSM{App:Mazzonetto}. To start with the proof of the claim, we first decompose
\begin{align}\label{eq_proof:L2-norm_occupation_approximation_1}
\begin{split}
&\E_{\rho_0}\left[ \left( \frac{1}{\sqrt{n}}\sum_{k=1}^n f_n(X_{(k-1)/n}) - \sqrt{n}\int_0^1 f_n(X_s) ds \right)^2\right] \\
&\hspace{0.2cm} = \sum_{k,l=1}^n \E_{\rho_0}\left[ \frac{1}{\sqrt{n}} f_n(X_{(k-1)/n}) \frac{1}{\sqrt{n}} f_n(X_{(l-1)/n}) - \frac{1}{\sqrt{n}}f_n(X_{(k-1)/n}) \int_{(l-1)/n}^{l/n} \sqrt{n} f_n(X_s)ds \right] \\
&\hspace{0.5cm} +  \sum_{k,l=1}^n \E_{\rho_0}\left[ \int_{(k-1)/n}^{k/n} \sqrt{n} f_n(X_s)ds\int_{(l-1)/n}^{l/n} \sqrt{n} f_n(X_s)ds\right. \\
&\hspace{5cm} \left. -\frac{1}{\sqrt{n}} f_n(X_{(l-1)/n}) \int_{(k-1)/n}^{k/n} \sqrt{n} f_n(X_s)ds \right].
\end{split}
\end{align}
For $k=l, l-1$ we directly get the estimate
\begin{align*}
&\left| \E_{\rho_0}\left[ \frac{1}{\sqrt{n}} f_n(X_{(k-1)/n}) \frac{1}{\sqrt{n}} f_n(X_{(l-1)/n}) - \frac{1}{\sqrt{n}} f_n(X_{(k-1)/n}) \int_{(l-1)/n}^{l/n} \sqrt{n} f_n(X_s)ds \right] \right| \\
&\hspace{1cm} \leq \frac1n \E_{\rho_0}\left[ |f_n(X_{(k-1)/n})|\left| f_n(X_{(l-1)/n}) - n \int_{(l-1)/n}^{l/n} f_n(X_s)ds \right| \right] \\
&\hspace{1cm} \leq \frac{2\kappa}{n} \E_{\rho_0}\left[ |f_n(X_{(k-1)/n})|\right] \\
&\hspace{1cm} \leq \frac{C_{\alpha,\beta}\kappa^2}{n\sqrt{k}}.
\end{align*}
For $k<l-1$ we get
\begin{align*}
&\left| \E_{\rho_0}\left[ \frac{1}{\sqrt{n}}f_n(X_{(k-1)/n})\frac{1}{\sqrt{n}} f_n(X_{(l-1)/n}) - \frac{1}{\sqrt{n}}f_n(X_{(k-1)/n}) \int_{(l-1)/n}^{l/n} \sqrt{n} f_n(X_s)ds \right]\right| \\
&\hspace{0.2cm} \leq \E_{\rho_0}\left[ \frac{1}{\sqrt{n}}|f_n(X_{(k-1)/n})| \left|\E_{\rho_0}\left[\left. \frac{1}{\sqrt{n}} f_n(X_{(l-1)/n}) - \int_{(l-1)/n}^{l/n} \sqrt{n} f_n(X_s)ds\ \right| \cF_{(k-1)/n}\right]\right| \right] \\
&\hspace{0.2cm} = \E_{\rho_0}\left[ |f_n(X_{(k-1)/n})| \left|\E_{\rho_0}\left[\left.  \int_{(l-1)/n}^{l/n} \left(f_n(X_{(l-1)/n})- f_n(X_s)\right) ds \right|  \cF_{(k-1)/n}\right]\right| \right] \\
&\hspace{0.2cm} \leq \E_{\rho_0}\left[ |f_n(X_{(k-1)/n})|  \int_{(l-1)/n}^{l/n} \Big|\E_{\rho_0}\left[f_n(X_{(l-1)/n})-f_n(X_s)\mid \cF_{(k-1)/n}\right]  \Big| ds\right] \\
&\hspace{0.2cm} \stackrel{\eqref{eq_proof:inequality_semigroup}}{\leq} \E_{\rho_0}\left[ |f_n(X_{(k-1)/n})|  \int_{(l-1)/n}^{l/n} C_{\alpha,\beta} \frac{(s-(k-1)/n)-(l-k)/n}{((l-k)/n)^{3/2}} \|f_n\|_{L^1} ds\right] \\
&\hspace{0.2cm} \leq C_{\alpha,\beta} \kappa \frac{1}{\sqrt{n}} \E_{\rho_0}\left[ |f_n(X_{(k-1)/n})|  \int_{0}^{1/n}  \frac{s}{((l-k)/n)^{3/2}} ds\right] \\
&\hspace{0.2cm} \leq \frac{C_{\alpha,\beta}\kappa}{(l-k)^{3/2}}n \E_{\rho_0}\left[ |f_n(X_{(k-1)/n})|  \int_{0}^{1/n} s ds\right] \\
&\hspace{0.2cm} \leq \frac{C_{\alpha,\beta}\kappa}{(l-k)^{3/2}}\frac1n \E_{\rho_0}\left[ |f_n(X_{(k-1)/n})| \right] \\
&\hspace{0.2cm} \leq \frac{C_{\alpha,\beta}\kappa^2}{(l-k)^{3/2}}\frac1n \frac{1}{\sqrt{k}}.
\end{align*}
We now find
\begin{align}\label{eq_proof:sum_int_approx}
\begin{split}
&\sum_{k,l=1}^n \E_{\rho_0}\left[ \frac{1}{\sqrt{n}} f_n(X_{(k-1)/n}) \frac{1}{\sqrt{n}} f_n(X_{(l-1)/n}) - \frac{1}{\sqrt{n}} f_n(X_{(k-1)/n}) \int_{(l-1)/n}^{l/n} \sqrt{n} f_n(X_s) \right] \\
&\hspace{1cm} \leq 3C_{\alpha,\beta}\kappa^2 \frac1n \sum_{k=1}^n \frac{1}{\sqrt{k}} + 2C_{\alpha,\beta}\kappa^2 \frac1n \sum_{k=1}^{n-2} \sum_{l=k+2}^n \frac{1}{(l-k)^{3/2}} \frac{1}{\sqrt{k}} \\
&\leq 3C_{\alpha,\beta}\kappa^2 n^{-1/2} + 2C_{\alpha,\beta}\kappa^2 \frac1n \int_1^{n-2} \int_{k+2}^n \frac{1}{(l-k)^{3/2}} \frac{1}{\sqrt{k}} dl dk \\
&= 3C_{\alpha,\beta}\kappa^2 n^{-1/2} + 2C_{\alpha,\beta}\kappa^2 \frac1n \int_1^{n-2} \left(\sqrt{2}-\frac{2}{\sqrt{n-k}}\right) \frac{1}{\sqrt{k}} dk \\
&\leq 3C_{\alpha,\beta}\kappa^2 n^{-1/2} + 4C_{\alpha,\beta}\kappa^2 \frac1n \int_1^{n-2} \frac{1}{\sqrt{k}} dk\\
&\leq C_{\alpha,\beta}\kappa^2 n^{-1/2}.
\end{split}
\end{align}
For the second term in our initial decomposition~\eqref{eq_proof:L2-norm_occupation_approximation_1}, we apply a similar argument. For this, we first deduce from~\eqref{eq_proof:inequality_semigroup} that
\begin{align*}
& \left| \E_{\rho_0}\left[ \frac{1}{\sqrt{n}} |f_n(X_{(k-1)/n})| - \int_{(k-1)/n}^{k/n} \sqrt{n}|f_n(X_s)| ds \right] \right| \\
&\hspace{1cm} \leq \sqrt{n} \int_{(k-1)/n}^{k/n} \Big| \E_{\rho_0}[|f_n(X_{(k-1)/n})|] -\E_{\rho_0}[|f_n(X_s)|]\Big| ds \\
&\hspace{1cm} \leq \sqrt{n} \|f_n\|_{L^1} \int_{(k-1)/n}^{k/n} \frac{s-(k-1)/n}{((k-1)/n)^{3/2}} ds  \leq C_{\alpha,\beta}\kappa\frac{1}{\sqrt{n}}k^{-3/2}.
\end{align*}
From this, we consequently derive
\begin{align}\label{eq_proof:Expectation_increment_occupation}
\begin{split}
\E_{\rho_0}\left[ \int_{(k-1)/n}^{k/n} \sqrt{n} |f_n(X_s)| ds\right] &\leq \frac{1}{\sqrt{n}}\E_{\rho_0}[|f_n(X_{(k-1)/n})|] + C_{\alpha,\beta}\kappa\frac{1}{\sqrt{n}} k^{-3/2} \\
&\leq \frac{1}{\sqrt{n}} \left( \kappa k^{-1/2} + C_{\alpha,\beta}\kappa k^{-3/2}\right)\\
&\leq C_{\alpha,\beta}\kappa \frac{1}{\sqrt{nk}}.
\end{split}
\end{align}
Let $k=l,l-1$. Then we find with~\eqref{eq_proof:Expectation_increment_occupation}
\begin{align*}
&\left|\E_{\rho_0}\left[\int_{(k-1)/n}^{k/n} \sqrt{n} f_n(X_s)ds \int_{(l-1)/n}^{l/n} \sqrt{n} f_n(X_s)ds - \frac{1}{\sqrt{n}} f_n(X_{(l-1)/n}) \int_{(k-1)/n}^{k/n} \sqrt{n} f_n(X_s) ds \right] \right| \\
&\hspace{1cm} \leq \frac{1}{\sqrt{n}} \E_{\rho_0}\left[ \int_{(k-1)/n}^{k/n} \sqrt{n} |f_n(X_s)| ds \left|n\int_{(l-1)/n}^{l/n}  f_n(X_s) -f_n(X_{(l-1)/n}) \right| \right] \\
&\hspace{1cm} \leq \frac{2\kappa}{\sqrt{n}}\E_{\rho_0}\left[ \int_{(k-1)/n}^{k/n} \sqrt{n} |f_n(X_s)| ds \right] \\
&\hspace{1cm} \leq \frac{C_{\alpha,\beta}\kappa^2}{n\sqrt{k}}. 
\end{align*}
Moreover, for $k<l-1$ we can derive with~\eqref{eq_proof:inequality_semigroup} and~\eqref{eq_proof:Expectation_increment_occupation}
\begin{align*}
&\left|\E_{\rho_0}\left[\int_{(k-1)/n}^{k/n} \sqrt{n} f_n(X_s)ds \int_{(l-1)/n}^{l/n} \sqrt{n} f_n(X_s)ds - \frac{1}{\sqrt{n}}f_n(X_{(l-1)/n}) \int_{(k-1)/n}^{k/n} \sqrt{n} f_n(X_s) ds \right] \right| \\
&\hspace{0.2cm} \leq \E_{\rho_0}\left[ \int_{(k-1)/n}^{k/n} \sqrt{n} |f_n(X_s)| ds \left|\E_{\rho_0}\left[ \left. \int_{(l-1)/n}^{l/n} \sqrt{n} f_n(X_s)ds - \frac{1}{\sqrt{n}}f_n(X_{(l-1)/n}) \right| \cF_{k/n}\right] \right| \right] \\
&\hspace{0.2cm} = \sqrt{n}  \E_{\rho_0}\left[ \int_{(k-1)/n}^{k/n} \sqrt{n} |f_n(X_s)| \left|\E_{\rho_0}\left[ \left. \int_{(l-1)/n}^{l/n} f_n(X_s) -f_n(X_{(l-1)/n}) ds \right| \cF_{k/n}\right] \right| \right] \\
&\hspace{0.2cm} \leq  \sqrt{n}  \E_{\rho_0}\left[ \int_{(k-1)/n}^{k/n} \sqrt{n} |f_n(X_s)|ds  \int_{(l-1)/n}^{l/n} \left| \E_{\rho_0}\left[f_n(X_s)-f_n(X_{(l-1)/n})\mid\cF_{k/n}\right] \right|ds  \right] \\
&\hspace{0.2cm} \leq  C_{\alpha,\beta}\sqrt{n} \|f_n\|_{L^1}  \E_{\rho_0}\left[ \int_{(k-1)/n}^{k/n} \sqrt{n} |f_n(X_s)| ds  \int_{(l-1)/n}^{l/n} \frac{s-k/n - (l-1-k)/n}{((l-k-1)/n)^{3/2}} ds  \right] \\
&\hspace{0.2cm} \leq  C_{\alpha,\beta}\kappa  n^{3/2} \frac{1}{(l-k-1)^{3/2}} \frac{1}{n^2} \E_{\rho_0}\left[ \int_{(k-1)/n}^{k/n} \sqrt{n} |f_n(X_s)| ds \right] \\
&\hspace{0.2cm} \leq C_{\alpha,\beta}\kappa^2 \frac1n \frac{1}{(l-k-1)^{3/2}}\frac{1}{\sqrt{k}}.
\end{align*}
Then, we conclude analogously as in~\eqref{eq_proof:sum_int_approx} that
\begin{align*}
&\sum_{k,l=1}^n \E_{\rho_0}\left[ \int_{(k-1)/n}^{k/n} \sqrt{n} f_n(X_s)ds\int_{(l-1)/n}^{l/n} \sqrt{n} f_n(X_s)ds - \frac{1}{\sqrt{n}} f_n(X_{(l-1)/n}) \int_{(k-1)/n}^{k/n} \sqrt{n} f_n(X_s)ds \right] \\
&\hspace{1cm} \leq C_{\alpha,\beta}\kappa^2 n^{-1/2}
\end{align*}
as desired.
\end{proof}

\begin{lemma}\label{lemma:uniform_stochastic_occupation_approximation}
Let $K,\xi_u>0$ and $E:=\{L_1^y(X) \leq \xi_u\textrm{ for all }y\in\R\}$. Then we have
\[\sup_{K/\sqrt{n} \leq \theta \leq n^{-1/4}}\left| \frac{1}{\sqrt{n}}\sum_{k=1}^n f_{n,\theta}^{(i)}(X_{(k-1)/n}) - \sqrt{n} \int_0^1  f_{n,\theta}^{(i)}(X_s)ds\right|\1_E  \longrightarrow_{\Pr_{\rho_0}} 0\]
for
\begin{align*}
f_{n,\theta}^{(1)}(x) &:= \1_{[a+\theta-L/\sqrt{n}, a+\theta]}(x), \\
f_{n,\theta}^{(2)}(x) & := \1_{[a+\theta, a+\theta+ L/\sqrt{n}]}(x), \\
f_{n,\theta}^{(3)}(x) &:=\1_{A_\theta}(x)\exp(-\sqrt{n}C|x-a-\theta|),
\end{align*} 
where $a\in\R$, $C,L>0$ and $A_\theta=(-\infty,a+\theta]$ or $A_\theta =[a+\theta,\infty)$.
\end{lemma}
\begin{proof}
The proof is built on a so-called bracketing argument. We start with $f_{n,\theta}^{(1)}$. Let $\epsilon>0$ and 
\begin{align}\label{eq_proof:uniform_stochastic_occupation_approximation_1}
\theta_j := K/\sqrt{n} + j\epsilon \min\{\alpha^2,\beta^2\}(2\xi_u)^{-1}/\sqrt{n}
\end{align}
for $j=0,1,\dots, \lceil 2\xi_u(\min\{\alpha^2,\beta^2\}\epsilon)^{-1}( n^{1/4}-K)\rceil$. In particular, for every $K/\sqrt{n}<\theta <n^{-1/4}$ there exists $j_0$ such that $\theta_{j_0}\leq \theta\leq \theta_{j_0+1}$ and 
\[ \1_{[a+\theta_{j_0+1} - L/\sqrt{n},a+\theta_{j_0}]}\leq \1_{[a+\theta - L/\sqrt{n},a+\theta]}\leq \1_{[a+\theta_{j_0} - L/\sqrt{n},a+\theta_{j_0+1}]}. \]
From this, we find on $E$
\begin{align*}
&\frac{1}{\sqrt{n}}\sum_{k=1}^n \1_{[a+\theta - L/\sqrt{n},a+\theta]}(X_{(k-1)/n}) - \sqrt{n} \int_0^1  \1_{[a+\theta - L/\sqrt{n},a+\theta]}(X_s)ds \\
&\hspace{1cm} \leq \frac{1}{\sqrt{n}} \sum_{k=1}^n \1_{[a+\theta_{j_0} - L/\sqrt{n},a+\theta_{j_0+1}]}(X_{(k-1)/n})- \sqrt{n}\int_0^1 \1_{[a+\theta_{j_0+1} - L/\sqrt{n},a+\theta_{j_0}]}(X_s)ds \\
&\hspace{1cm} \leq \frac{1}{\sqrt{n}}\sum_{k=1}^n \1_{[a+\theta_{j_0} - L/\sqrt{n},a+\theta_{j_0+1}]}(X_{(k-1)/n})- \sqrt{n}\int_0^1 \1_{[a+\theta_{j_0} - L/\sqrt{n},a+\theta_{j_0+1}]}(X_s)ds \\
&\hspace{1.5cm} + \sqrt{n}\int_0^1  \1_{[a+\theta_{j_0} - L/\sqrt{n},a+\theta_{j_0+1}]}(X_s)-\1_{[a+\theta_{j_0+1} - L/\sqrt{n},a+\theta_{j_0}]}(X_s) ds \\
&\hspace{1cm} \leq \frac{1}{\sqrt{n}}\sum_{k=1}^n \1_{[a+\theta_{j_0} - L/\sqrt{n},a+\theta_{j_0+1}]}(X_{(k-1)/n})- \sqrt{n}\int_0^1  \1_{[a+\theta_{j_0} - L/\sqrt{n},a+\theta_{j_0+1}]}(X_s)ds +\epsilon,
\end{align*}
where we use that by the occupation times formula
\begin{align*}
&\1_E \sqrt{n}\int_0^1  \1_{[a+\theta_{j_0} - L/\sqrt{n},a+\theta_{j_0+1}]}(X_s)-\1_{[a+\theta_{j_0+1} - L/\sqrt{n},a+\theta_{j_0}]}(X_s) ds \\
&\hspace{1cm} \leq \1_E\frac{\sqrt{n}}{\min\{\alpha^2,\beta^2\}} \left(\int_{a+\theta_{j_0} - L/\sqrt{n}}^{a+\theta_{j_0+1} - L/\sqrt{n}} L_1^y(X) dy + \int_{a+\theta_{j_0}}^{a+\theta_{j_0+1}} L_1^y(X) dy\right) \\
&\hspace{1cm} \leq \frac{2\sqrt{n}}{\min\{\alpha^2,\beta^2\}} \xi_u \left( \theta_{j_0+1}-\theta_{j_0}\right)  = \epsilon.
\end{align*}
Similarly, we derive on $E$
\begin{align*}
&\frac{1}{\sqrt{n}}\sum_{k=1}^n \1_{[a+\theta - L/\sqrt{n},a+\theta]}(X_{(k-1)/n}) - \sqrt{n} \int_0^1  \1_{[a+\theta - L/\sqrt{n},a+\theta]}(X_s)ds \\
&\hspace{0.2cm} \geq \frac{1}{\sqrt{n}} \sum_{k=1}^n \1_{[a+\theta_{j_0+1} - L/\sqrt{n},a+\theta_{j_0}]}(X_{(k-1)/n})- \sqrt{n}\int_0^1  \1_{[a+\theta_{j_0+1} - L/\sqrt{n},a+\theta_{j_0}]}(X_s)ds -\epsilon.
\end{align*}
Combining these bounds and abbreviating $c:=2\xi_u\min\{\alpha^2,\beta^2\}^{-1}$ then yields
\begin{align}\label{eq_proof:uniform_stochastic_occupation_approximation_2}
\begin{split}
&\sup_{K/\sqrt{n} \leq \theta \leq n^{-1/4}}\left|\frac{1}{\sqrt{n}} \sum_{k=1}^n \1_{[a+\theta - L/\sqrt{n},a+\theta]}(X_{(k-1)/n}) - \sqrt{n}\int_0^1  \1_{[a+\theta - L/\sqrt{n},a+\theta]}(X_s)ds\right|\1_E \\
&\hspace{0.2cm} \leq \max_{j=1,\dots, \lceil c\epsilon^{-1} n^{1/4}\rceil}\left|\frac{1}{\sqrt{n}}\sum_{k=1}^n \1_{[a+\theta_j - L/\sqrt{n},a+\theta_j]}(X_{(k-1)/n}) - \sqrt{n}\int_0^1  \1_{[a+\theta_j - L/\sqrt{n},a+\theta_j]}(X_s)ds\right| + \epsilon.
\end{split}
\end{align}
In the next step, we want to apply Lemma~\ref{lemma:L2-norm_occupation_approximation} to $f_{n,\theta}^{(1)}$. This is possible, because this function is clearly bounded by $1$, $\int_\R f_n^{(1)}(x) |x|^mdx <\infty$ for $m=1,2$,
\[ \left\| f_{n,\theta}^{(1)} \right\|_{L^1} := \int_{\R} \1_{[a+\theta-L/\sqrt{n}, a+\theta]}(x) dx = \frac{L}{\sqrt{n}},\]
and
\begin{align*}
&\E_{\rho_0}\left[\1_{[a+\theta - L/\sqrt{n},a+\theta]}(X_{(k-1)/n})\right]\\
&\hspace{0.5cm}\leq C_{\alpha,\beta} \int_{a+\theta-L/\sqrt{n}}^{a+\theta} \frac{1}{\sqrt{(k-1)/n}}\exp\left(-\frac{(y-x_0)^2}{2\max\{\alpha^2,\beta^2\}(k-1)/n}\right) dy \leq C_{\alpha,\beta}\frac{L}{\sqrt{k-1}}
\end{align*} 
by Lemma~\ref{lemma:upper_bound_transition_density}, where all bounds are independent of $\theta$. Then, we have with Lemma~\ref{lemma:L2-norm_occupation_approximation}
\begin{align*}
&\Pr_{\rho_0}\left( \sup_{K/\sqrt{n} \leq \theta \leq n^{-1/4}}\left|\frac{1}{\sqrt{n}}\sum_{k=1}^n \1_{[a+\theta - L/\sqrt{n},a+\theta]}(X_{(k-1)/n}) - \sqrt{n} \int_0^1  \1_{[a+\theta - L/\sqrt{n},a+\theta]}(X_s)ds\right| >2\epsilon\right) \\
&\leq \Pr_{\rho_0}\left( \max_{j=1,\dots, \lceil c\epsilon^{-1} n^{1/4}\rceil}\left|\frac{1}{\sqrt{n}}\sum_{k=1}^n \1_{[a+\theta_j - L/\sqrt{n},a+\theta_j]}(X_{(k-1)/n}) - \sqrt{n}\int_0^1  \1_{[a+\theta_j - L/\sqrt{n},a+\theta_j]}(X_s)ds\right| >\epsilon\right) \\
&\leq \frac{1}{\epsilon^2} \sum_{j=1}^{\lceil c\epsilon^{-1}n^{1/4}\rceil} \E_{\rho_0}\left[\left(\frac{1}{\sqrt{n}}\sum_{k=1}^n \1_{[a+\theta_j - L/\sqrt{n},a+\theta_j]}(X_{(k-1)/n}) - \sqrt{n}\int_0^1  \1_{[a+\theta_j - L/\sqrt{n},a+\theta_j]}(X_s)ds\right)^2 \right] \\
&\leq \frac{1}{\epsilon^2} C_{\alpha,\beta}(L) n^{-1/2} n^{1/4} \longrightarrow 0
\end{align*}
which completes the proof for the function $f_{n,\theta}^{(1)}$. The function $f_{n,\theta}^{(2)}$ works exaktly the same way. For $f_{n,\theta}^{(3)}$ with $A_\theta=(-\infty,a+\theta]$, we first recall $\theta_j$ from~\eqref{eq_proof:uniform_stochastic_occupation_approximation_1} and define the brackets
\begin{align*}
\1_{(-\infty,a+\theta_{j}]}(x)\exp(C\sqrt{n}(x-a-\theta_{j+1})) &\leq \1_{(-\infty,a+\theta]}(x)\exp(C\sqrt{n}(x-a-\theta))\\
& \leq \1_{(-\infty,a+\theta_{j+1}]}(x)\exp(C\sqrt{n}(x-a-\theta_j)).
\end{align*} 
for $\theta_j\leq \theta\leq\theta_{j+1}$. For these upper and lower bounds we find with the occupation times formula on $E$
\begin{align*}
&\sqrt{n} \int_0^1 \left( \1_{(-\infty,a+\theta_{j+1}]}(X_s)\exp(C\sqrt{n}(X_s-a-\theta_j)) \right. \\
&\hspace{5cm} \left. - \1_{(-\infty,a+\theta_{j}]}(X_s)\exp(C\sqrt{n}(X_s-a-\theta_{j+1}))\right) ds \\
&\hspace{1cm} \leq \frac{\sqrt{n}}{\min\{\alpha^2,\beta^2\}} \int_\R \left( \1_{(-\infty,a+\theta_{j+1}]}(y)\exp(C\sqrt{n}(y-a-\theta_j)) \right. \\
&\hspace{5cm} \left. - \1_{(-\infty,a+\theta_{j}]}(y)\exp(C\sqrt{n}(y-a-\theta_{j+1}))\right) L_1^y(X) dy \\
&\hspace{1cm} \leq \frac{\sqrt{n}\xi_u}{\min\{\alpha^2,\beta^2\}} \int_{-\infty}^{a+\theta_j} \left(\exp(C\sqrt{n}(y-a-\theta_j))-\exp(C\sqrt{n}(y-a-\theta_{j+1})) \right) dy \\
&\hspace{2cm} + \frac{\sqrt{n}\xi_u}{\min\{\alpha^2,\beta^2\}} \int_{a+\theta_j}^{a+\theta_{j+1}}\exp(C\sqrt{n}(y-a-\theta_j))dy.
\end{align*}
By a Taylor expansion with intermediate point $\theta_j\leq\zeta\leq\theta_{j+1}$ (possibly depending on $C,n,j$) we find
\begin{align*}
&\int_{-\infty}^{a+\theta_j} \left(\exp(C\sqrt{n}(y-a-\theta_j))-\exp(C\sqrt{n}(y-a-\theta_{j+1})) \right) dy \\
&\hspace{1cm} =  \left(\theta_{j+1}-\theta_j\right) \int_{-\infty}^{a+\theta_j}C\sqrt{n}\exp\left( C\sqrt{n}(y-a-\zeta)\right) dy \\
&\hspace{1cm} = \left(\theta_{j+1}-\theta_j\right)\int_{-\infty}^{C\sqrt{n}(\theta_j-\zeta)} e^y dy \leq \theta_{j+1}-\theta_j
\end{align*}
and
\begin{align*}
\int_{a+\theta_j}^{a+\theta_{j+1}}\exp(C\sqrt{n}(y-a-\theta_j))dy &\leq \left(\theta_{j+1}-\theta_j\right) \exp\left(C\sqrt{n}\left(\theta_{j+1}-\theta_j\right)\right) \\
&= \left(\theta_{j+1}-\theta_j\right) \exp\left( \frac{C\epsilon\min\{\alpha^2,\beta^2\}}{2\xi_u}\right),
\end{align*}
which then yields on $E$ the bound
\begin{align*}
&\sqrt{n} \int_0^1 \left( \1_{(-\infty,a+\theta_{j+1}]}(X_s)\exp(C\sqrt{n}(X_s-a-\theta_j)) \right. \\
&\hspace{5cm} \left. - \1_{(-\infty,a+\theta_{j}]}(X_s)\exp(C\sqrt{n}(X_s-a-\theta_{j+1}))\right) ds \\
&\hspace{1cm} \leq \frac{\sqrt{n}\xi_u}{\min\{\alpha^2,\beta^2\}}\left(\theta_{j+1}-\theta_j\right)  \left(1+\exp\left( \frac{C\epsilon\min\{\alpha^2,\beta^2\}}{2\xi_u}\right) \right) \\
&\hspace{1cm} \leq \tilde{C}\epsilon
\end{align*}
for some constant $\tilde{C}>0$. Repeating the steps to derive~\eqref{eq_proof:uniform_stochastic_occupation_approximation_2} then gives
\begin{align*}
&\sup_{K/\sqrt{n} \leq \theta \leq n^{-1/4}}\left|\frac{1}{\sqrt{n}}\sum_{k=1}^n f_{n,\theta}^{(3)}(X_{(k-1)/n}) - \sqrt{n}\int_0^1  f_{n,\theta}^{(3)}(X_s)ds\right|\1_E \\
&\hspace{1cm} \leq \max_{j=1,\dots, \lceil c\epsilon^{-1} n^{1/4}\rceil}\left|\frac{1}{\sqrt{n}}\sum_{k=1}^n f_{n,\theta_j}^{(3)}(X_{(k-1)/n}) - \sqrt{n}\int_0^1  f_{n,\theta_j}^{(3)}(X_s)ds\right| + \tilde{C}\epsilon.
\end{align*}
From this, the claim follows again by applying Lemma~\ref{lemma:L2-norm_occupation_approximation} and Markov's inequality as it was done for $f_{n,\theta}^{(1)}$, provided this Lemma is also applicable for $f_{n,\theta}^{(3)}$. This is indeed the case, since $\int_\R f_n^{(3)}(x) |x|^mdx <\infty$ for $m=1,2$, $|f_n^{(3)}|\leq 1$, 
\[ \|f_n^{(3)}\|_{L^1} = \int_{-\infty}^{a+\theta} \exp\left(C\sqrt{n} (x-a-\theta)\right) dx = \int_{-\infty}^0 \exp\left(C\sqrt{n}x\right) dx = \frac{1}{C\sqrt{n}},\]
and by Lemma~\ref{lemma:upper_bound_transition_density},
\begin{align*}
\E_{\rho_0}\left[f_n^{(3)}(X_{(k-1)/n})\right] &\leq C_{\alpha,\beta} \int_{-\infty}^{a+\theta} \exp\left(C\sqrt{n}(y-a-\theta)\right)\frac{1}{\sqrt{(k-1)/n}}\\
&\hspace{4cm} \cdot\exp\left(-\frac{(y-x_0)^2}{2\max\{\alpha^2,\beta^2\}(k-1)/n} \right) dy \\
&\leq \frac{C_{\alpha,\beta}}{\sqrt{k-1}} \int_{-\infty}^0 \exp\left(Cy\right) \exp\left(-\frac{(y + \sqrt{n}(a+\theta-x_0))^2}{2\max\{\alpha^2,\beta^2\}(k-1)}\right) dy \\
& \leq \frac{C_{\alpha,\beta}}{\sqrt{k-1}} \int_{-\infty}^0 \exp\left(Cy\right) dy \leq \frac{C_{\alpha,\beta}}{\sqrt{k}},
\end{align*} 
with all bounds being independent of $\theta$. Finally, the proof of $f_{n,\theta}^{(3)}$ with $A_\theta = [a+\theta,\infty)$ works exactly the same.
\end{proof}

The next Lemma provides a moment bound for a supremum appearing in the discussion of $\Xi_n^{(5,3)}$ in the proof of Lemma~\ref{lemma:bound_L_sqrt(n)}. Note that the order $n^{3/8}\log(n)$ is suboptimal, but sufficiently good for our purpose. It could be further improved by using the metric $\rho_n(\theta,\theta')^2 = n|\theta-\theta'|\wedge \sqrt{n}$ in the chaining argument together with subexponential tail inequalities for martingales (see~\citeSM{App:Pena}).

\begin{lemma}\label{lemma:supremum_at_rho+theta}
Define
\begin{align*}
M_n(\theta) &= \sum_{k=1}^n \left( 1+ \frac{n}{2}(X_{k/n}-X_{(k-1)/n})^2\right)\1_{\{X_{(k-1)/n}<\rho_0+\theta\leq X_{k/n}\}} \\
&\hspace{2cm} - \E_{\rho_0}\left[ \left. \left( 1+ \frac{n}{2}(X_{k/n}-X_{(k-1)/n})^2\right)\1_{\{X_{(k-1)/n}<\rho_0+\theta\leq X_{k/n}\}}\right| X_{(k-1)/n}\right].
\end{align*} 
Then for some constant $C_{\alpha,\beta}>0$ we have
\[ \E_{\rho_0}\left[ \sup_{0\leq \theta \leq n^{-1/4}} |M_n(\theta)|\right] \leq C_{\alpha,\beta}n^{3/8}\log(n). \]
\end{lemma}
\begin{proof}
The proof relies on a chaining argument similar to that in Subsection~\ref{subsection:sqrt(n)_consistency}. First, for $0\leq\theta'\leq \theta$, we rewrite
\begin{align}\label{eq_proof:lemma_supremum_at_rho+theta}
\begin{split}
&\1_{\{X_{(k-1)/n}<\rho_0+\theta\leq X_{k/n}\}}-\1_{\{X_{(k-1)/n}<\rho_0+\theta'\leq X_{k/n}\}}\\
&\hspace{2cm} = \1_{\{\rho_0+\theta'<X_{(k-1)/n}\leq \rho_0+\theta\}}\1_{\{X_{k/n}\geq \rho_0+\theta\}} -\1_{\{\rho_0+\theta'<X_{k/n}\leq \rho_0+\theta\}}\1_{\{X_{(k-1)/n}<\rho_0+\theta'\}},
\end{split}
\end{align} 
such that all indicators on the right-hand side depend solely on one random variable. Using that $M_n$ is a sum of martingale differences and
\begin{align*}
&\left(\1_{\{\rho_0+\theta'<X_{(k-1)/n}\leq \rho_0+\theta\}}\1_{\{X_{k/n}\geq \rho_0+\theta\}} -\1_{\{\rho_0+\theta'<X_{k/n}\leq \rho_0+\theta\}}\1_{\{X_{(k-1)/n}<\rho_0+\theta'\}}\right)^2 \\
&\hspace{1cm} = \1_{\{\rho_0+\theta'<X_{(k-1)/n}\leq \rho_0+\theta\}}\1_{\{X_{k/n}\geq \rho_0+\theta\}}+\1_{\{\rho_0+\theta'<X_{k/n}\leq \rho_0+\theta\}}\1_{\{X_{(k-1)/n}<\rho_0+\theta'\}}
\end{align*}
we obtain
\begin{align*}
&\E_{\rho_0}\left[ (M_n(\theta)-M_n(\theta'))^2\right]\\
&\hspace{0.5cm}\leq \sum_{k=1}^n \E_{\rho_0}\left[ \left(1+\frac{n}{2}(X_{k/n}-X_{(k-1)/n})^2\right)^2 \1_{\{\rho_0+\theta'<X_{(k-1)/n}\leq \rho_0+\theta\}}\1_{\{X_{k/n}\geq \rho_0+\theta\}}\right] \\
&\hspace{1cm} +\sum_{k=1}^n \E_{\rho_0}\left[ \left(1+\frac{n}{2}(X_{k/n}-X_{(k-1)/n})^2\right)^2 \1_{\{\rho_0+\theta'<X_{k/n}\leq \rho_0+\theta\}}\1_{\{X_{(k-1)/n}<\rho_0+\theta'\}}\right].
\end{align*}
By direct evaluation using Lemma~\ref{lemma:upper_bound_transition_density}
\begin{align*}
&\E_{\rho_0}\left[ \left(1+\frac{n}{2}(X_{k/n}-X_{(k-1)/n})^2\right)^2 \1_{\{\rho_0+\theta'<X_{(k-1)/n}\leq \rho_0+\theta\}}\1_{\{X_{k/n}\geq \rho_0+\theta\}}\right] \\
&\hspace{0.2cm} \leq C_{\alpha,\beta} \E_{\rho_0}\left[\1_{\{\rho_0+\theta'<X_{(k-1)/n}\leq \rho_0+\theta\}} \int_\R \left( 1+\frac{n}{2}(y-X_{(k-1)/n})^2\right)^2 \right.\\
&\hspace{6cm}\left. \cdot\sqrt{n}\exp\left(-\frac{(y-X_{(k-1)/n})^2}{2\max\{\alpha^2,\beta^2\}/n}\right) dy \right] \\
&\hspace{0.2cm} \leq C_{\alpha,\beta}\E_{\rho_0}\left[ \1_{\{\rho_0+\theta'<X_{(k-1)/n}\leq \rho_0+\theta\}} \int_\R \left( 1 + \frac{y^2}{2}\right)^2 \exp\left(-\frac{y^2}{2\max\{\alpha^2,\beta^2\}}\right) dy \right] \\
&\hspace{0.2cm} \leq C_{\alpha,\beta} \int_{\rho_0+\theta'}^{\rho_0+\theta} \frac{1}{\sqrt{(k-1)/n}} \exp\left(-\frac{(y-x_0)^2}{2\max\{\alpha^2,\beta^2\}(k-1)/n}\right) dy \\
&\hspace{0.2cm} \leq C_{\alpha,\beta} \frac{|\theta-\theta'|}{\sqrt{(k-1)/n}}
\end{align*}
and using boundedness of $x\mapsto (1+x^2)^2\exp(-x^2/4)$ together with Corollary~\ref{cor:bound_exp_X^2} in the last step, we also derive
\begin{align*}
& \E_{\rho_0}\left[ \left(1+\frac{n}{2}(X_{k/n}-X_{(k-1)/n})^2\right)^2 \1_{\{\rho_0+\theta'<X_{k/n}\leq \rho_0+\theta\}}\1_{\{X_{(k-1)/n}<\rho_0+\theta'\}}\right] \\
&\hspace{0.2cm} \leq C_{\alpha,\beta} \E_{\rho_0} \left[\1_{\{X_{(k-1)/n}<\rho_0+\theta'\}} \int_{\rho_0+\theta'}^{\rho_0+\theta} \left( 1+\frac{n}{2}(y-X_{(k-1)/n})^2\right)^2 \sqrt{n} \exp\left(-\frac{(y-X_{(k-1)/n})^2}{2\max\{\alpha^2,\beta^2\}/n}\right) dy \right] \\
&\hspace{0.2cm} \leq C_{\alpha,\beta}\E_{\rho_0} \left[\1_{\{X_{(k-1)/n}<\rho_0+\theta'\}}\int_{\sqrt{n}(\rho_0+\theta' -X_{(k-1)/n})}^{\sqrt{n}(\rho_0+\theta -X_{(k-1)/n})} \left(1+\frac{y^2}{2}\right)^2 \exp\left(-\frac{y^2}{2\max\{\alpha^2,\beta^2\}}\right) dy \right] \\
&\hspace{0.2cm} \leq C_{\alpha,\beta} \sqrt{n}|\theta-\theta'|\E_{\rho_0} \left[ \exp\left(-\frac{(X_{(k-1)/n}-\rho_0-\theta')^2}{4\max\{\alpha^2,\beta^2\}/n}\right) \right] \\
&\hspace{0.2cm} \leq C_{\alpha,\beta} \frac{|\theta-\theta'|}{\sqrt{(k-1)/n}}.
\end{align*}
Hence,
\[ \E_{\rho_0}\left[ (M_n(\theta)-M_n(\theta'))^2\right] \leq C_{\alpha,\beta} n|\theta-\theta'|.\]
By the same steps that were used to derive~\eqref{eq_proof:sqrt(n)_chaining_bound} and~\eqref{eq_proof:sqrt(n)_Pisier} applied to the set $\mathcal{T}_n^0 = [0,n^{-1/4}]$, the chaining argument gives
\begin{align}\label{eq_proof:lemma_supremum_at_rho+theta_1}
\begin{split}
\E_{\rho_0}\left[\sup_{0<\theta\leq n^{-1/4}} |M_n(\theta)|\right] &\leq C_{\alpha,\beta}\int_1^{n^{3/8}} n^{3/8} u^{-1} du  \\
&\hspace{0.3cm} + C_{\alpha,\beta}\left( \sum_{\theta\in [0,n^{-1/4}]\cap\mathcal{T}_n}  \E_{\rho_0}\left[ \sup_{\theta'\in U_{C/n}(\theta)} \left( M_n(\theta) - M_n(\theta') \right)^2 \right] \right)^\frac12,
\end{split}
\end{align}
where $\mathcal{T}_n$ is some finite grid of $[0,n^{-1/4}]$ with cardinality $\#\mathcal{T}_n\leq Cn^{3/4}$ for some constant $C>0$. Next, we observe that the supremum within the expectation can be bounded explicitly in terms of $\theta$ only: 
By the decomposition~\eqref{eq_proof:lemma_supremum_at_rho+theta},
\begin{align}\label{eq_proof:lemma_supremum_at_rho+theta_2}
\begin{split}
&\sup_{\theta'\in [\theta-C/n,\theta]} \left| M_n(\theta) - M_n(\theta') \right| \\
&\hspace{0.2cm} \leq \sup_{\theta'\in [\theta-C/n,\theta]} \left( \sum_{k=1}^n \left(1+\frac{n}{2}(X_{k/n}-X_{(k-1)/n})^2\right) \right. \\
&\hspace{3cm} \cdot\left(\1_{\{\rho_0+\theta'<X_{(k-1)/n}\leq \rho_0+\theta\}}+\1_{\{\rho_0+\theta'<X_{k/n}\leq \rho_0+\theta\}} \right)  \\
&\hspace{0.7cm} +  \sum_{k=1}^n \E_{\rho_0}\left[  \left(1+\frac{n}{2}(X_{k/n}-X_{(k-1)/n})^2\right) \right. \\
&\hspace{3cm} \left.\left. \cdot \left(\1_{\{\rho_0+\theta'<X_{(k-1)/n}\leq \rho_0+\theta\}}+\1_{\{\rho_0+\theta'<X_{k/n}\leq \rho_0+\theta\}} \right)\right|X_{(k-1)/n}\right] \Bigg) \\
&\hspace{0.2cm}\leq \sum_{k=1}^n \left(1+\frac{n}{2}(X_{k/n}-X_{(k-1)/n})^2\right)\left( \1_{\{\rho_0+\theta-C/n<X_{(k-1)/n}\leq \rho_0+\theta\}} +\1_{\{\rho_0+\theta-C/n<X_{k/n}\leq \rho_0+\theta\}}\right)\\
&\hspace{0.7cm} +\sum_{k=1}^n \E_{\rho_0}\left[  \left(1+\frac{n}{2}(X_{k/n}-X_{(k-1)/n})^2\right) \right. \\
&\hspace{3cm} \left.\left. \cdot\left( \1_{\{\rho_0+\theta-C/n<X_{(k-1)/n}\leq \rho_0+\theta\}} +\1_{\{\rho_0+\theta-C/n<X_{k/n}\leq \rho_0+\theta\}}\right)\right|X_{(k-1)/n}\right].
\end{split}
\end{align}
The supremum over $\theta'\in [\theta,\theta+C/n]$ can be dealt with in the same way, starting with a similar decomposition as~\eqref{eq_proof:lemma_supremum_at_rho+theta}. By Lemma~\ref{lemma:second_moment_sum_intervals}, the second moment of the bound in~\eqref{eq_proof:lemma_supremum_at_rho+theta_2} is bounded uniformly in $n$ and $\theta$. Thus, the claim of the this lemma is shown due to~\eqref{eq_proof:lemma_supremum_at_rho+theta_1}.
\end{proof}

\subsection{Remaining proofs of Subsection~\ref{subsection:n_consistency}}\label{App:Section:proofs_n_consistency}

\begin{proof}[Proof of Lemma~\ref{lemma:bound_drift}]
We will first specify $\kappa_0$ and $\zeta_1$ such that by the first inequality is valid for $|\theta|\leq \kappa_0/\sqrt{n}$ with $\zeta=\zeta_1$ and then show that the second one is valid for the remaining $\kappa_0/\sqrt{n}\leq |\theta|\leq K/\sqrt{n}$ and some $\zeta_2$. The claim then follows for $\zeta=\min\{\zeta_1,\zeta_2\}$. By Proposition~\ref{prop:expansion_of_drift_t} we have for $|\theta|\leq \kappa/\sqrt{n}$,
\[ B_n(\theta) \leq -n|\theta|\left( \left(\1_{\{\theta\geq 0\}}F_{\alpha,\beta}+\1_{\{\theta< 0\}} \tilde{F}_{\alpha,\beta}\right)\frac1n \Lambda_{\alpha,\beta}^n\left((X_{(k-1)/n})_{1\leq k\leq n}\right) - \frac{r_n(1,\theta)}{n|\theta|} \right),\]
where $\Lambda_{\alpha,\beta}^n\left((X_{(k-1)/n})_{1\leq k\leq n}\right)$ is given in~\eqref{eq:Lambda_nt(X)} and 
\[ \E_{\rho_0}\left[\sup_{|\theta'|\leq |\theta|} \frac{|r_n(1,\theta')|}{|\theta'|}\right] \leq C_{\alpha,\beta} |\theta| n^{3/2}.\]
By Lemma~\ref{lemma:Riemann_approximation},
\begin{align}\label{eq_proof:lemma:bound_drift:main_1}
\frac1n \Lambda_{\alpha,\beta}^n\left((X_{(k-1)/n})_{1\leq k\leq n}\right)\longrightarrow_{\Pr_{\rho_0}} \frac{1}{2}\left(\frac{1}{\alpha} + \frac{1}{\beta}\right) L_1^{\rho_0}(X),
\end{align} 
such that for $n\geq n_0$ and $n_0$ large enough, we have $\Pr_{\rho_0}(E_n)>1-\epsilon/6$ for
\[ E_n := \left\{\frac1n \Lambda_{\alpha,\beta}^n\left((X_{(k-1)/n})_{1\leq k\leq n}\right) \geq \frac{L_1^{\rho_0}(X)(\alpha+\beta)}{4\alpha\beta} \right\}. \]
Recall $A_2\cap A_3$ from~\eqref{eq:set_A2} and~\eqref{eq:set_A3} with $\Pr_{\rho_0}(\{L_1^{\rho_0}(X)>0\}\cap A_2\cap A_3)\geq 1-\epsilon/6$ for suitable constants in the definition of $A_2, A_3$ and define
\[ A_n'(\kappa) := A_2\cap A_3 \cap E_n \cap \left\{ \sup_{\theta'\leq \kappa/\sqrt{n}} \frac{|r_n(1,\theta')|}{|\theta'|} \leq C_r n \kappa \right\}.\]
By Markov's inequality,
\[ \Pr_{\rho_0}\left(\sup_{|\theta'|\leq \kappa/\sqrt{n}} \frac{|r_n(1,\theta')|}{|\theta'|} > C_r n\kappa\right) \leq \frac{C_{\alpha,\beta}\kappa}{C_r \kappa} = \frac{C_{\alpha,\beta}}{C_r}, \]
and we can choose $C_r$ independently of $n$ and $\kappa$ large enough such that $\Pr_{\rho_0}(A_n'(\kappa))\geq 1-\epsilon/2$ for $n\geq n_0$. Then we have for $|\theta|\leq \kappa/\sqrt{n}$,
\[ B_n(\theta)\1_{A_n'(\kappa)} \leq -n|\theta|\left(\left(\1_{\{\theta\geq 0\}}F_{\alpha,\beta}+\1_{\{\theta< 0\}} \tilde{F}_{\alpha,\beta}\right)\frac{\xi(\alpha+\beta)}{4\alpha\beta} - C_r\kappa \right).\]
For
\[ \kappa_0 := \frac{\xi(\alpha+\beta)}{8C_r\alpha\beta}\min\{F_{\alpha,\beta},\tilde{F}_{\alpha,\beta}\}, \]
the first assertion of the lemma then follows with $\zeta=\zeta_1 := C_r\kappa_0$. Next, we consider $\kappa_0/\sqrt{n}\leq |\theta|\leq K/\sqrt{n}$. Denoting by $\Pr_{\rho,1/n}^{X_{(k-1)/n}}(\cdot) = \Pr_{\rho}(X_{k/n}\in\cdot\mid X_{(k-1)/n})$ the distribution of $X_{k/n}$ given $X_{(k-1)/n}$ in our model with parameter $\rho$ and by $KL(\Pr_1,\Pr_2)$ the Kullback--Leibler divergence of two probability measures $\Pr_1$, $\Pr_2$, we observe
\[ \E_{\rho_0}\left[\left. \log\left(\frac{p_{1/n}^{\rho_0+\theta}(X_{(k-1)/n},X_{k/n})}{p_{1/n}^{\rho_0}(X_{(k-1)/n},X_{k/n})}\right)\ \right|\ X_{(k-1)/n}\right] = - KL\left( \Pr_{\rho_0,1/n}^{X_{(k-1)/n}}, \Pr_{\rho_0+\theta,1/n}^{X_{(k-1)/n}}\right). \]
By the first Pinsker inequality (\citeSM{App:Tsybakov}, Lemma~$2.5$) we have
\[ KL\left( \Pr_{\rho_0,1/n}^{X_{(k-1)/n}}, \Pr_{\rho_0+\theta,1/n}^{X_{(k-1)/n}}\right) \geq 2 d_{\textrm{TV}}\left(\Pr_{\rho_0,1/n}^{X_{(k-1)/n}}, \Pr_{\rho_0+\theta,1/n}^{X_{(k-1)/n}}\right)^2,\]
where $d_{\textrm{TV}}$ denots total variation distance. By Scheffé's theorem (\citeSM{App:Tsybakov}, Lemma~$2.1$) we conclude
\begin{align*}
&\E_{\rho_0}\left[\left. \log\left(\frac{p_{1/n}^{\rho_0+\theta}(X_{(k-1)/n},X_{k/n})}{p_{1/n}^{\rho_0}(X_{(k-1)/n},X_{k/n})}\right)\ \right|\ X_{(k-1)/n}\right] \\
&\hspace{2cm} \leq - \frac12 \left(\int_{\R} \left| p_{1/n}^{\rho_0+\theta}(X_{(k-1)/n},y) - p_{1/n}^{\rho_0}(X_{(k-1)/n},y) \right| dy\right)^2.
\end{align*}
Consequently, by~\eqref{eq:bound_TV_distance} (which is proven right after this proof),
\[ B_n(\theta) \leq - \frac12\left(C_{\alpha,\beta}^1\right)^2 n\theta^2 \sum_{k=1}^n \exp\left(-\frac{2(X_{(k-1)/n}-\rho_0)^2}{\min\{\alpha^2,\beta^2\}/n} - 2\left(C_{\alpha,\beta}^2 K\right)^2\right). \]
From Lemma~\ref{lemma:Riemann_approximation}, we find for $f(x):= \exp(-2x^2/\min\{\alpha^2,\beta^2\} - 2(C_{\alpha,\beta}^2 K)^2)$ that
\[ \frac{1}{\sqrt{n}}\sum_{k=1}^n \exp\left(-\frac{2(X_{(k-1)/n}-\rho_0)^2}{\min\{\alpha^2,\beta^2\}/n} - 2\left(C_{\alpha,\beta}^2 K\right)^2\right) \longrightarrow_{\Pr_{\rho_0}} \lambda_{\alpha,\beta}(f) L_1^{\rho_0}(X). \]
Then, for
\begin{align*}
A_n'' &:= A_2\cap A_3\cap\left\{ \frac{1}{\sqrt{n}}\sum_{k=1}^n \exp\left(-\frac{2(X_{(k-1)/n}-\rho_0)^2}{\min\{\alpha^2,\beta^2\}/n} - 2\left(C_{\alpha,\beta}^2 K\right)^2\right) > \frac{\lambda_{\alpha,\beta}(f)}{2} L_1^{\rho_0}(X)\right\}
\end{align*} 
we obtain $\Pr_{\rho_0}(A_n'')\geq 1-\epsilon/2$ for $n$ large enough and 
\[ B_n(\theta)\1_{A_n''} \leq -\frac12 \left(C_{\alpha,\beta}^1 \right)^2 \frac{\lambda_{\alpha,\beta}(f)}{2}\xi n^{3/2}\theta^2. \]
Thus, the second inequality in the statement is true for $\zeta=\zeta_2 := (C_{\alpha,\beta}^1)^2\lambda_{\alpha,\beta}(f)\xi /4$ and the statement follows for $\zeta=\min\{\zeta_1,\zeta_2\}$ and the sets $A_n := A_n'(\kappa_0)\cap A_n''$.
\end{proof}

\begin{proof}[Proof of~\eqref{eq:bound_TV_distance}]
We prove the statement for $\theta\geq 0$. The case $\theta<0$ works the same. For the proof, we distinguish the cases $X_{(k-1)/n}< \rho_0$, $\rho_0\leq X_{(k-1)/n}< \rho_0+\theta$ and $X_{(k-1)/n}\geq\rho_0+\theta$. Moreover, we denote $\varphi(x)= \frac{1}{\sqrt{2\pi}} e^{-x^2/2}$ and by $\Phi(\cdot)$ the cumulative distribution function of the standard Gaussian distribution, i.e. $\Phi(y) = \int_{-\infty}^y \varphi(x) dx$.
\begin{itemize}
\item[$\bullet\ \boldsymbol{X_{(k-1)/n}<\rho_0}$.] Here, using the explicit representation of the transition density and Jensen's inequality in the second step,
\begin{align}\label{eq_proof:lemma_bound_TV_distance_1}
\begin{split}
&\1_{\{X_{(k-1)/n}<\rho_0\}}\int_{\R} \left| p_{1/n}^{\rho_0+\theta}(X_{(k-1)/n},y) - p_{1/n}^{\rho_0}(X_{(k-1)/n},y) \right| dy \\
&\hspace{0.5cm} \geq \1_{\{X_{(k-1)/n}<\rho_0\}}\int_{\rho_0+\theta}^\infty \left| p_{1/n}^{\rho_0+\theta}(X_{(k-1)/n},y) - p_{1/n}^{\rho_0}(X_{(k-1)/n},y) \right| dy \\
&\hspace{0.5cm} \geq \frac{2\alpha}{\alpha+\beta} \frac{1}{\sqrt{2\pi\beta^2/n}}\left| \int_{\rho_0+\theta}^\infty  \exp\left(-\frac{n}{2}\left( \frac{y-\rho_0-\theta}{\beta} - \frac{X_{(k-1)/n}-\rho_0-\theta}{\alpha}\right)^2\right) \right.\\
&\hspace{5.5cm} \left. - \exp\left(-\frac{n}{2}\left( \frac{y-\rho_0}{\beta} - \frac{X_{(k-1)/n}-\rho_0}{\alpha}\right)^2\right)dy \right|.
\end{split}
\end{align}
By substitution we find
\begin{align*}
&\frac{1}{\sqrt{2\pi\beta^2/n}}\int_{\rho_0+\theta}^\infty \exp\left(-\frac{n}{2}\left( \frac{y-\rho_0-\theta}{\beta} - \frac{X_{(k-1)/n}-\rho_0-\theta}{\alpha}\right)^2\right) dy \\
&\hspace{1.5cm} = 1- \Phi\left( -\frac{X_{(k-1)/n}-\rho_0}{\alpha/\sqrt{n}} + \frac{\theta}{\alpha/\sqrt{n}} \right) 
\end{align*} 
and
\begin{align*}
&\frac{1}{\sqrt{2\pi\beta^2/n}}\int_{\rho_0+\theta}^\infty \exp\left(-\frac{n}{2}\left( \frac{y-\rho_0}{\beta} - \frac{X_{(k-1)/n}-\rho_0}{\alpha}\right)^2\right) dy \\
&\hspace{1.5cm}  =1- \Phi\left(-\frac{X_{(k-1)/n}-\rho_0}{\alpha/\sqrt{n}} + \frac{\theta}{\beta/\sqrt{n}}\right). 
\end{align*} 
The mean value theorem gives $\Phi(x+\epsilon)- \Phi(x) = \varphi(\xi_x)\epsilon$ for $x\leq \xi_x\leq x+\epsilon$ and thus by~\eqref{eq_proof:lemma_bound_TV_distance_1},
\begin{align*}
& \1_{\{X_{(k-1)/n}<\rho_0\}}\int_{\R} \left| p_{1/n}^{\rho_0+\theta}(X_{(k-1)/n},y) - p_{1/n}^{\rho_0}(X_{(k-1)/n},y) \right| dy \\
&\hspace{0.5cm}\geq \frac{2\alpha}{\alpha+\beta} \left| \Phi\left( -\frac{X_{(k-1)/n}-\rho_0}{\alpha/\sqrt{n}} + \frac{\theta}{\beta/\sqrt{n}} \right) - \Phi\left(-\frac{X_{(k-1)/n}-\rho_0}{\alpha/\sqrt{n}} + \frac{\theta}{\alpha/\sqrt{n}}\right) \right| \\
&\hspace{0.5cm} = \frac{2\alpha}{\alpha+\beta} \left|\frac{1}{\alpha}-\frac{1}{\beta}\right| \sqrt{n}\theta \varphi(\xi_k)
\end{align*}
for an intermediate value $\xi_k$ satisfying
\[ -\frac{X_{(k-1)/n}-\rho_0}{\alpha/\sqrt{n}} + \frac{\theta}{\max\{\alpha,\beta\}/\sqrt{n}} \leq \xi_k \leq -\frac{X_{(k-1)/n}-\rho_0}{\alpha/\sqrt{n}} + \frac{\theta}{\min\{\alpha,\beta\}/\sqrt{n}}.\]
In particular, using $(a+b)^2\leq 2a^2+2b^2$ and the assumption $|\theta|\leq K/\sqrt{n}$, we have the final bound
\begin{align*}
&\1_{\{X_{(k-1)/n}<\rho_0\}} \int_{\R} \left| p_{1/n}^{\rho_0+\theta}(X_{(k-1)/n},y) - p_{1/n}^{\rho_0}(X_{(k-1)/n},y) \right| dy \\
&\hspace{1cm} \geq \frac{2\alpha}{\alpha+\beta} \left|\frac{1}{\alpha}-\frac{1}{\beta}\right| \theta\sqrt{n} \exp\left(-\frac12\left(-\frac{X_{(k-1)/n}-\rho_0}{\alpha/\sqrt{n}} + \frac{\theta}{\min\{\alpha,\beta\}/\sqrt{n}}\right)^2\right) \\
&\hspace{1cm} \geq \frac{2\alpha}{\alpha+\beta} \left|\frac{1}{\alpha}-\frac{1}{\beta}\right| \theta\sqrt{n} \exp\left(-\frac{(X_{(k-1)/n}-\rho_0)^2}{\alpha^2/n} - \frac{K^2}{\min\{\alpha^2,\beta^2\}}\right).
\end{align*}

\item[$\bullet\ \boldsymbol{\rho_0\leq X_{(k-1)/n}< \rho_0+\theta}$.] Denoting by $\|\cdot\|_{L^1(\lambda)}$ the $L^1$-norm with respect to Lebesgue measure and applying the reverse triangle inequality yields
\begin{align}\label{eq_proof:lemma_bound_TV_distance_2}
\begin{split}
&\1_{\{\rho_0\leq  X_{(k-1)/n}< \rho_0+\theta\}} \int_{\R} \left| p_{1/n}^{\rho_0+\theta}(X_{(k-1)/n},y) - p_{1/n}^{\rho_0}(X_{(k-1)/n},y) \right| dy \\
&\hspace{0.5cm} \geq \1_{\{\rho_0\leq  X_{(k-1)/n}< \rho_0+\theta\}}\int_{-\infty}^{\rho_0} \left| p_{1/n}^{\rho_0+\theta}(X_{(k-1)/n},y) - p_{1/n}^{\rho_0}(X_{(k-1)/n},y) \right| dy \\
&\hspace{0.5cm} = \left\| \left(P_1^{\rho_0+\theta}(X_{(k-1)/n},\cdot\ ;1/n)-P_4^{\rho_0}(X_{(k-1)/n},\cdot\ ;1/n)\right)\1_{(-\infty,\rho_0]}(\cdot)\right\|_{L^1(\lambda)} \\
&\hspace{0.5cm} \geq \left| \left\|P_1^{\rho_0+\theta}(X_{(k-1)/n},\cdot\ ;1/n)\1_{(-\infty,\rho_0]}(\cdot)\right\|_{L^1(\lambda)} - \left\|P_4^{\rho_0}(X_{(k-1)/n},\cdot\ ;1/n)\1_{(-\infty,\rho_0]}(\cdot)\right\|_{L^1(\lambda)}\right|.
\end{split}
\end{align}
By substitution, we find
\begin{align*}
&\left\|P_1^{\rho_0+\theta}(X_{(k-1)/n},\cdot\ ;1/n)\1_{(-\infty,\rho_0]}(\cdot)\right\|_{L^1(\lambda)} \\
&\hspace{1.5cm} = \Phi\left(-\frac{X_{(k-1)/n}-\rho_0}{\alpha/\sqrt{n}}\right) - \frac{\alpha-\beta}{\alpha+\beta}\Phi\left( \frac{X_{(k-1)/n}-\rho_0-2\theta}{\alpha/\sqrt{n}}\right) 
\end{align*} 
and
\[ \left\|P_4^{\rho_0}(X_{(k-1)/n},\cdot\ ;1/n)\1_{(-\infty,\rho_0]}(\cdot)\right\|_{L^1(\lambda)} = \frac{2\beta}{\alpha+\beta}\Phi\left(-\frac{X_{(k-1)/n}-\rho_0}{\beta/\sqrt{n}}\right).\]
Note here that $2\beta/(\alpha+\beta) = 1-(\alpha-\beta)/(\alpha+\beta)$. Using the mean value theorem for $\Phi$, we observe
\[ \Phi\left(-\frac{X_{(k-1)/n}-\rho_0}{\alpha/\sqrt{n}}\right) = \Phi\left(-\frac{X_{(k-1)/n}-\rho_0}{\beta/\sqrt{n}}\right) - \sqrt{n}(X_{(k-1)/n}-\rho_0)\left(\frac{1}{\alpha}-\frac{1}{\beta}\right) \varphi(\xi_{1,k})\]
and 
\[ \Phi\left(\frac{X_{(k-1)/n}-\rho_0-2\theta}{\alpha/\sqrt{n}}\right) = \Phi\left(-\frac{X_{(k-1)/n}-\rho_0}{\alpha/\sqrt{n}}\right) + \frac{2(X_{(k-1)/n}-\rho_0-\theta)}{\alpha/\sqrt{n}} \varphi(\xi_{2,k}),\]
for intermediate values $\xi_1,\xi_2$ with
\[ \min\left\{\frac{\rho_0-X_{(k-1)/n}}{\alpha/\sqrt{n}},\frac{\rho_0-X_{(k-1)/n}}{\beta/\sqrt{n}}  \right\} \leq \xi_{1,k}\leq \max\left\{\frac{\rho_0-X_{(k-1)/n}}{\alpha/\sqrt{n}},\frac{\rho_0-X_{(k-1)/n}}{\beta/\sqrt{n}}  \right\}\]
and
\begin{align*}
&\min\left\{ -\frac{X_{(k-1)/n}-\rho_0}{\alpha/\sqrt{n}}, \frac{X_{(k-1)/n}-\rho_0-2\theta}{\alpha/\sqrt{n}} \right\}\leq \xi_{2,k}\\
&\hspace{4cm} \leq \max\left\{ -\frac{X_{(k-1)/n}-\rho_0}{\alpha/\sqrt{n}}, \frac{X_{(k-1)/n}-\rho_0-2\theta}{\alpha/\sqrt{n}} \right\}.
\end{align*} 
As $\rho_0< X_{(k-1)/n}\leq \rho_0+\theta$, $|X_{(k-1)/n}-\rho_0|\leq |X_{(k-1)/n}-\rho_0-2\theta|$ and consequently for $i=1,2$,
\[ |\xi_{i,k}| \leq \frac{|X_{(k-1)/n}-\rho_0-2\theta|}{2\min\{\alpha^2,\beta^2\}}.\]
In consequence, starting from~\eqref{eq_proof:lemma_bound_TV_distance_2},
\begin{align*}
&\1_{\{\rho_0\leq  X_{(k-1)/n}< \rho_0+\theta\}}\int_{\R} \left| p_{1/n}^{\rho_0+\theta}(X_{(k-1)/n},y) - p_{1/n}^{\rho_0}(X_{(k-1)/n},y) \right| dy \\
&\hspace{0.2cm} \geq \left| \Phi\left(-\frac{X_{(k-1)/n}-\rho_0}{\alpha/\sqrt{n}}\right)-\Phi\left(-\frac{X_{(k-1)/n}-\rho_0}{\beta/\sqrt{n}}\right) \right. \\
&\hspace{3cm} \left. + \frac{\alpha-\beta}{\alpha+\beta} \Phi\left(-\frac{X_{(k-1)/n}-\rho_0}{\beta/\sqrt{n}}\right)-\frac{\alpha-\beta}{\alpha+\beta}\Phi\left( \frac{X_{(k-1)/n}-\rho_0-2\theta}{\alpha/\sqrt{n}}\right) \right| \\
&\hspace{0.2cm} = \left| - \sqrt{n}(X_{(k-1)/n}-\rho_0)\left(\frac{1}{\alpha}-\frac{1}{\beta}\right) \left( 1-\frac{\alpha-\beta}{\alpha+\beta}\right) \varphi(\xi_1) - \frac{\alpha-\beta}{\alpha+\beta}\frac{2(X_{(k-1)/n}-\rho_0-\theta)}{\alpha/\sqrt{n}} \varphi(\xi_2) \right|\\
&\hspace{0.2cm} = \left(\1_{\{\alpha\geq \beta\}}-\1_{\{\alpha<\beta\}}\right) \left[ -\frac{2\beta}{\alpha+\beta} \left(\frac{1}{\alpha}-\frac{1}{\beta}\right) \sqrt{n}(X_{(k-1)/n}-\rho_0)\varphi(\xi_1) \right. \\
&\hspace{5cm} \left. - \frac{\alpha-\beta}{\alpha+\beta}\frac{2(X_{(k-1)/n}-\rho_0-2\theta)}{\alpha/\sqrt{n}}\varphi(\xi_2)\right] \\
&\hspace{0.2cm} \geq \left(\1_{\{\alpha\geq \beta\}}-\1_{\{\alpha<\beta\}}\right) \left[- \frac{2\beta}{\alpha+\beta} \left(\frac{1}{\alpha}-\frac{1}{\beta}\right) \sqrt{n}(X_{(k-1)/n}-\rho_0) \right.  \\
&\hspace{2cm}\left. - \frac{\alpha-\beta}{\alpha+\beta}\frac{2(X_{(k-1)/n}-\rho_0-2\theta)}{\alpha/\sqrt{n}}\right] \frac{1}{\sqrt{2\pi}} \exp\left( - \frac{(X_{(k-1)/n}-\rho_0-2\theta)^2}{2\min\{\alpha^2,\beta^2\}/n}\right) \\
&\hspace{0.2cm} =\left(\1_{\{\alpha\geq \beta\}}-\1_{\{\alpha<\beta\}}\right) \frac{2(\alpha-\beta)}{(\alpha+\beta)\alpha}\sqrt{n}\left(X_{(k-1)/n}-\rho_0 - (X_{(k-1)/n}-\rho_0-2\theta)\right)\\
&\hspace{2cm} \cdot \frac{1}{\sqrt{2\pi}} \exp\left( - \frac{(X_{(k-1)/n}-\rho_0-2\theta)^2}{2\min\{\alpha^2,\beta^2\}/n}\right) \\
&\hspace{0.2cm} = \frac{4|1-\beta/\alpha|}{\alpha+\beta} \frac{\theta}{\sqrt{2\pi\/n}}  \exp\left( - \frac{(X_{(k-1)/n}-\rho_0-2\theta)^2}{2\min\{\alpha^2,\beta^2\}/n}\right),
\end{align*}
where the third step uses that both summands within the absolute value have the same sign for $\rho_0<X_{(k-1)/n}\leq \rho_0+\theta$. Consequently, we have proven that (using again $\theta\leq K/\sqrt{n}$ and $(a+b)^2\leq 2a^2+2b^2$),
\begin{align*}
&\1_{\{\rho_0\leq  X_{(k-1)/n}< \rho_0+\theta\}} \int_{\R} \left| p_{1/n}^{\rho_0+\theta}(X_{(k-1)/n},y) - p_{1/n}^{\rho_0}(X_{(k-1)/n},y) \right| dy \\
&\hspace{1cm} \geq \frac{4|1-\beta/\alpha|}{\alpha+\beta}\frac{1}{\sqrt{2\pi}} \theta \sqrt{n}\exp\left( - \frac{(X_{(k-1)/n}-\rho_0)^2}{\min\{\alpha^2,\beta^2\}/n} - \frac{2K^2}{\min\{\alpha^2,\beta^2\}}\right).
\end{align*}

\item[$\bullet\ \boldsymbol{X_{(k-1)/n}\geq\rho_0+\theta}$.] This case works analogously to the first one with $X_{(k-1)/n}<\rho_0$ and is therefore omitted. It yields
\begin{align*}
&\1_{\{X_{(k-1)/n}\geq\rho_0+\theta\}} \int_{\R} \left| p_{1/n}^{\rho_0+\theta}(X_{(k-1)/n},y) - p_{1/n}^{\rho_0}(X_{(k-1)/n},y) \right| dy  \\
&\hspace{1cm} \geq  \frac{2\alpha}{\alpha+\beta} \left|\frac{1}{\alpha}-\frac{1}{\beta}\right| \theta \sqrt{n} \exp\left(- \frac{(X_{(k-1)/n}-\rho_0)^2}{\beta^2/n} - K^2\left(\frac{1}{\beta}-\frac{1}{\alpha}\right)^2\right).
\end{align*}
\end{itemize}
The inequality in~\eqref{eq:bound_TV_distance} now follows from these three steps by setting
\[ C_{\alpha,\beta}^1 := \min\left\{\frac{2\alpha}{\alpha+\beta} \left|\frac{1}{\alpha}-\frac{1}{\beta}\right|  ,\frac{4|1-\beta/\alpha|}{\alpha+\beta}\frac{1}{\sqrt{2\pi}}  ,\frac{2\alpha}{\alpha+\beta} \left|\frac{1}{\alpha}-\frac{1}{\beta}\right|\right\}\]
and 
\[ C_{\alpha,\beta}^2 := K^2\cdot \max\left\{\frac{2}{\min\{\alpha^2,\beta^2\}}, \left(\frac{1}{\alpha}-\frac{1}{\beta}\right)^2\right\}. \]
\end{proof}

\begin{proof}[Proof of \eqref{eq_proof:n_consistency_5}]
Note that $M_n^1(\theta) - M_n^1(\theta') = \sum_{k=1}^n d_k(\theta,\theta')$, where $d_k(\theta,\theta')$ are martingale increments, i.e. $\E_{\rho_0}[d_k(\theta,\theta')\mid \cF_{(k-1)/n}] =0$ and consequently~\eqref{eq_proof:sqrt(n)_consistency_1_1} holds true. Here, for $0\leq\theta'\leq\theta$, $d_k(\theta,\theta')$ is given as
\begin{align*}
&\log\left(\frac{\beta^2}{\alpha^2}\right)\left(\1_{\{X_{(k-1)/n}<\rho_0\}}\1_{\{\rho_0+\theta'< X_{k/n}\leq\rho_0+\theta\}} \right. \\
&\hspace{4cm} \left. - \E_{\rho_0}\left[\1_{\{X_{(k-1)/n}<\rho_0\}}\1_{\{\rho_0+\theta'< X_{k/n}\leq\rho_0+\theta\}}\mid X_{(k-1)/n}\right] \right) \\
&\hspace{0.2cm}  + \log\left(\frac{\beta^2}{\alpha^2}\right) \left( \1_{\{\rho_0+\theta'<X_{k/n}\leq\rho_0+\theta\}}\1_{\{\rho_0+\theta \leq X_{(k-1)/n}\}}\right. \\
&\hspace{4cm} \left.  - \E_{\rho_0}\left[\1_{\{\rho_0+\theta'<X_{k/n}\leq\rho_0+\theta\}}\1_{\{\rho_0+\theta \leq X_{(k-1)/n}\}} \mid X_{(k-1)/n}\right] \right)\\
&\hspace{0.2cm} - \log\left(\frac{\beta^2}{\alpha^2}\right)\left( \1_{\{\rho_0<X_{k/n}\leq\rho_0+\theta'\}}\1_{\{\rho_0+\theta'\leq X_{(k-1)/n} <\rho_0+\theta\}} \right. \\
&\hspace{4cm} \left. - \E_{\rho_0}\left[\1_{\{\rho_0<X_{k/n}\leq\rho_0+\theta'\}}\1_{\{\rho_0+\theta'\leq X_{(k-1)/n} <\rho_0+\theta\}} \mid X_{(k-1)/n}\right] \right).
\end{align*}
From this expression, we find the upper bound
\begin{align*}
d_k(\theta,\theta')^2 &\leq 2^6 \log\left(\frac{\beta^2}{\alpha^2}\right)^2 \left[\1_{\{\rho_0+\theta'\leq X_{k/n}\leq\rho_0+\theta\}} +  \E_{\rho_0}\left[\1_{\{\rho_0+\theta'\leq X_{k/n}\leq \rho_0+\theta\}}\mid X_{(k-1)/n}\right]^2 \right. \\
&\hspace{0.5cm} + \1_{\{\rho_0+\theta'\leq X_{k/n}\leq \rho_0+\theta\}} +\E_{\rho_0}\left[\1_{\{\rho_0+\theta'\leq X_{k/n}\leq \rho_0+\theta\}}\mid X_{(k-1)/n}\right]^2 \\
&\hspace{0.5cm}\left. +   \1_{\{\rho_0+\theta' \leq X_{(k-1)/n} \leq \rho_0+\theta\}} + \E_{\rho_0}\left[\1_{\{\rho_0+\theta' \leq X_{(k-1)/n} \leq \rho_0+\theta\}}\mid X_{(k-1)/n}\right]^2\right].
\end{align*}
With Lemma~\ref{lemma:upper_bound_transition_density} we then find the moment bounds
\begin{align*}
\E_{\rho_0}\left[\1_{\{\rho_0+\theta'\leq X_{k/n}\leq \rho_0+\theta\}} \right] &\leq C_{\alpha,\beta}\int_{\rho_0+\theta'}^{\rho_0+\theta} \frac{1}{\sqrt{k/n}} \exp\left(-\frac{(y-x_0)^2}{2\max\{\alpha^2,\beta^2\}k/n}\right) dy \\
& \leq C_{\alpha,\beta} \frac{1}{\sqrt{k/n}} |\theta-\theta'|
\end{align*} 
and by Jensen's inequality for conditional expectations,
\begin{align*}
\E_{\rho_0}\left[\E_{\rho_0}\left[\1_{\{\rho_0+\theta'\leq X_{k/n}\leq \rho_0+\theta\}}\mid X_{(k-1)/n}\right]^2 \right] &\leq \E_{\rho_0}\left[\1_{\{\rho_0+\theta'\leq X_{k/n}\leq \rho_0+\theta\}}^2 \right] \\
&  \leq C_{\alpha,\beta} \frac{1}{\sqrt{k/n}}|\theta-\theta'|.
\end{align*} 
The others work the same way and yield the same upper bound. Consequently, we have from~\eqref{eq_proof:sqrt(n)_consistency_1_1} that
\[ \E_{\rho_0}\left[ \left(M_n^1(\theta) - M_n^1(\theta')\right)^2\right] \leq C_{\alpha,\beta}|\theta-\theta'|\sum_{k=1}^n\frac{1}{\sqrt{k/n}} \leq C_{\alpha,\beta} n|\theta-\theta'|. \]
\end{proof}

\begin{proof}[Proof of \eqref{eq_proof:n_consistency_7}]
We use that
\begin{align*}
\sup_{\theta'\in U_{1/(C_1 n)}(\theta)} \left( M_n^1(\theta)-M_n^1(\theta')\right)^2 &= \max\left\{ \sup_{\theta-1/(C_1 n)\leq \theta'\leq \theta} \left( M_n^1(\theta)-M_n^1(\theta')\right)^2, \right. \\
&\hspace{2cm} \left. \sup_{\theta\leq \theta'\leq \theta+1/(C_1 n)} \left( M_n^1(\theta)-M_n^1(\theta')\right)^2 \right\} 
\end{align*} 
and only construct an upper bound for $\sup_{\theta-1/(C_1 n)\leq \theta'\leq \theta} \left( M_n^1(\theta)-M_n^1(\theta')\right)^2$. The other one can be built analogously and the upper bound of the statement is then obtained as the maximum of the two upper bounds. For $\theta-1/(C_1 n)\leq \theta'\leq \theta$, we find
\begin{align*}
&\left| M_n^1(\theta) - M_n^1(\theta')\right|\\
&\hspace{0.2cm} = \left|\log\left(\frac{\beta^2}{\alpha^2}\right)\sum_{k=1}^n \Big( \1_{\{X_{(k-1)/n}<\rho_0\}}\1_{\{\rho_0+\theta'< X_{k/n} \leq \rho_0+\theta\}} +\1_{\{\rho_0+\theta'<X_{k/n}\leq\rho_0+\theta\}}\1_{\{\rho_0+\theta\leq X_{(k-1)/n}\}}\right.\\
&\hspace{3.5cm} - \1_{\{\rho_0<X_{k/n}\leq \rho_0+\theta'\}}\1_{\{\rho_0+\theta' \leq X_{(k-1)/n} <\rho_0+\theta\}}\\
&\hspace{1.5cm}  -\E_{\rho_0}\left[\1_{\{X_{(k-1)/n}<\rho_0\}}\1_{\{\rho_0+\theta'< X_{k/n}\leq \rho_0+\theta\}} + \1_{\{\rho_0+\theta'<X_{k/n}\leq \rho_0+\theta\}}\1_{\{\rho_0+\theta\leq X_{(k-1)/n}\}}\right.\\
&\hspace{3.5cm} \left. - \1_{\{\rho_0<X_{k/n}\leq\rho_0+\theta'\}}\1_{\{\rho_0+\theta'\leq X_{(k-1)/n} <\rho_0+\theta\}} | X_{(k-1)/n}\right] \Big)\Bigg| \\
&\hspace{0.2cm} \leq \left|\log\left(\frac{\beta^2}{\alpha^2}\right)\right|\sum_{k=1}^n \Big(\1_{\{\rho_0+\theta-1/(C_1 n)\leq X_{k/n}\leq \rho_0+\theta\}} + \E_{\rho_0}\left[\1_{\{\rho_0+\theta-1/(C_1 n)\leq X_{k/n}\leq \rho_0+\theta\}}\mid X_{(k-1)/n}\right] \\
&\hspace{1.5cm}  +  \1_{\{\rho_0+\theta-1/(C_1 n)\leq X_{k/n}\leq \rho_0+\theta\}} + \1_{\{\rho_0+\theta-1/(C_1 n) \leq X_{(k-1)/n}\leq \rho_0+\theta\}} \\
&\hspace{1.5cm}  +  \E_{\rho_0}\left[ \1_{\{\rho_0+\theta-1/(C_1 n)\leq X_{k/n}\leq \rho_0+\theta\}} + \1_{\{\rho_0+\theta-1/(C_1 n)\leq X_{(k-1)/n}\leq \rho_0+\theta\}} \mid X_{(k-1)/n}\right] \Big) \\
&\hspace{0.2cm} := \overline{M}_n^1(\theta).
\end{align*}
Using $(a_1+\dots + a_6)^2 \leq 2^6 (a_1^2 + \dots + a_6^2)$, Lemma~\ref{lemma:second_moment_sum_intervals} reveals the second moment of $\overline{M}_n^1(\theta)$ is bounded with a bound independent of $\theta$ and hence~\eqref{eq_proof:n_consistency_7} follows.
\end{proof}

\section{Stable convergence of piecewise constant processes towards discontinuous conditional PIIs}\label{App:Jacod_version}

In this section, we present a modification of the stable limit results given in Theorem~$4.1$ in~\citeSM{App:Jacod_stablePII} and Theorem~$2.1$ in~\citeSM{App:Jacod_stableGaussian}, tailored to our context. The result in~\citeSM{App:Jacod_stableGaussian} is to the best of our knowledge the only one that is applicable for infill asymptotics without a certain nestedness condition on the filtration, but only covers a continuous (in time) limit and thus can not be applied to the process $\ell_{n,t}$ given in~\eqref{eq:def_ell_nt} as the Lindeberg-type condition~$(2.12)$ in Theorem~$2.1$ in this paper is not satisfied (see Section~\ref{Section:Limiting_distribution}). On the other hand, Theorem~$4.1$ in~\citeSM{App:Jacod_stablePII} covers limit processes with jumps but does not allow in its current formulation to treat convergence of processes $X^n$, where each $X^n$ is defined on a different stochastic basis $\mathcal{B}^n$.\\

Starting from a filtered probability space $(\Omega, \cF, (\cF_t)_{t\in [0,1]}, \Pr)$, we work on a very good extension of this space in the sense of Section~$2$ in~\citeSM{App:Jacod_stablePII}. We prove a result on stable convergence of a piecewise constant process towards an $\cF$-conditional process with independent increments (PII) which is allowed to be discontinuous. For more details on the notion of $\cF$-conditional PIIs, the reader is referred to~\citeSM{App:Jacod_stablePII}. Throughout this section, we assume the following:

\begin{itemize}
\item $(\Omega,\cF,\Pr)$ is a probability space supporting a standard Brownian motion $W$ with $\mathbb{F}=(\cF_t)_{t\in [0,1]}$ being the augmented filtration induced by $W$ and restrict attention to the case $\cF=\cF_1$. 
Then, the stochastic basis $\mathcal{B}:=(\Omega,\cF, \mathbb{F}, \Pr)$ has the martingale representation property with respect to $W$ (see Theorem~$19.11$ in~\citeSM{App:Kallenberg}).
\item Corresponding to $\mathbb{F}$, for any $n\in\N$, we introduce the discretized filtration $\mathbb{F}^n=(\cF_t^n)_{t\in [0,1]}$ via $\cF_t^n := \cF_{\lfloor nt\rfloor/n}$.
\item We define the process $W^n = (W_t^n)_{t\in [0,1]}$ via
\[ W_t^n := W_{\lfloor nt\rfloor} = \sum_{k=1}^{\lfloor nt\rfloor} W_{k/n}-W_{(k-1)/n}.\]
Then $W^n$ is a square-integrable $\mathbb{F}^n$-martingale and we have
\[ \langle W^n\rangle_t \longrightarrow_\Pr \langle W\rangle_t = t \quad \textrm{ for all } t\in [0,1].\]
\item For each $n\in\N$, let $X^n$ be a $\mathbb{F}^n$-semimartingale with
\[ X_t^n = \sum_{k=1}^{\lfloor nt\rfloor} \chi_{nk}, \]
where $\chi_{nk}$ is $\cF_{k/n}^n$-measurable and square-integrable.
\item We now consider the $\cF^n$-semimartingale $(X^n, W^n)$. Its first characteristic $B^n$, its second modified characteristic $C^n$ and its third characteristic $\nu^n$ are given as (see Theorem~II.3.11(b) and II.3.18 in~\citeSM{App:Jacod/Shiryaev})
\[  B^n = \begin{pmatrix} \sum_{k=1}^{\lfloor nt\rfloor} \E[\chi_{nk}\mid \cF_{(k-1)/n}] \\ 0 \end{pmatrix}, \qquad C^n =\begin{pmatrix} \langle X^n,X^n\rangle &\langle X^n,W^n\rangle \\ \langle X^n,W^n\rangle & \langle W^n,W^n\rangle\end{pmatrix}, \]
and for any measurable $g:\R^2\longrightarrow [0,\infty)$ 
\[ \left(g\star \nu^n\right)_t = \sum_{k=1}^{\lfloor nt\rfloor} \E\left[\left. g\left( \chi_{nk}, W_{k/n}-W_{(k-1)/n}\right)\right|\cF_{(k-1)/n}\right].  \]
\end{itemize}

\begin{prop}\label{prop:version_Jacod}
Assume there exists a continuous process $B$ with finite variation and a random measure $\nu$ on $[0,1]\times \R^2$ not charging $[0,1]\times \{0\}$ and satisfying $\nu(\{t\}\times\R^2)=0$ identically, such that the following convergences hold for all $t\in [0,1]$:
\begin{align}\label{eq_Jacod:1char}
\sup_{s\leq t} |B_s^n - B_s| \longrightarrow_{\Pr} 0,
\end{align}
\begin{align}\label{eq_Jacod:2char_1}
\sum_{k=1}^{\lfloor nt\rfloor} \E\left[ (\chi_{nk}-\E[\chi_{nk}|\cF_{(k-1)/n}])^2 | \cF_{(k-1)/n}\right]\longrightarrow_{\Pr} (f\star \nu)_t \quad \textrm{ for } f(x,y):=x^2,
\end{align}
\begin{align}\label{eq_Jacod:2char_2}
\sum_{k=1}^{\lfloor nt\rfloor} \E\left[ (\chi_{nk}-\E[\chi_{nk}|\cF_{(k-1)/n}])(W_{k/n}-W_{(k-1)/n}) | \cF_{(k-1)/n}\right]\longrightarrow_{\Pr} 0,
\end{align}
\begin{align}\label{eq_Jacod:3char}
(g\star \nu^n)_t \longrightarrow_{\Pr} (g\star \nu)_t \quad \textrm{ for all } g\in\mathcal{C},
\end{align}
where $\mathcal{C}$ is a countable set of Lipschitz-continuous bounded nonnegative functions on $\R^2$, vanishing in a neighbourhood of $0$ and being a measure-determining class for measures not charging $0$. Moreover, we assume there exists a constant $\kappa>0$ such that
\begin{align}\label{eq_Jacod:large_jumps}
\E\left[\sum_{k=1}^n \chi_{nk}^2\1_{\{|\chi_{nk}|>\kappa\}}\right] \longrightarrow\ 0.
\end{align}
Then there exist a very good extension $\tilde{\mathcal{B}}$ of $\mathcal{B}$ and a quasi-left continuous process $X$ on $\tilde{\mathcal{B}}$ which is an $\cF$-conditional PII and the pair $(X,W)$ admits the characteristics $(B, C, \nu)$, where $C= \begin{pmatrix} 0 & 0 \\ 0 & \mathrm{id}_{[0,1]}\end{pmatrix}$.
Moreover, $X^n$ converges stably in law to $X$.
\end{prop}
\begin{proof}
We define processes $\tilde{W}^n=(\tilde{W}^n_t)_{t\in [0,1]}$ and $\tilde{X}^n=(\tilde{X}^n_t)_{t\in [0,1]}$ via
\[ \tilde{W}^n_t = \sum_{k=1}^{\lfloor nt\rfloor} (W_{k/n}-W_{(k-1)/n})\1_{\{|W_{k/n}-W_{(k-1)/n}|\leq 1\}}\]
and (recall the constant $\kappa$ from~\eqref{eq_Jacod:large_jumps})
\[ \tilde{X}^n_t = \sum_{k=1}^{\lfloor nt\rfloor} \chi_{nk} \1_{\{|\chi_{nk}|\leq \kappa\}}.\]
For the latter we have using~\eqref{eq_Jacod:large_jumps},
\begin{align*}
&\sup_{t\leq 1} |X_t^n - \tilde{X}_t^n| \leq \sum_{k=1}^n |\chi_{nk}|\1_{\{|\chi_{nk}|>\kappa\}}\leq\frac{1}{\kappa} \sum_{k=1}^n |\chi_{nk}|^2\1_{\{|\chi_{nk}|>\kappa\}}\longrightarrow_{\Pr}\ 0.
\end{align*} 
Assuming we have shown $\cF$-stable convergence $\tilde{X}^n\longrightarrow X$, then the claim of $\cF$-stable convergence $X^n\longrightarrow X$ follows by Theorem~$3.18$(a) in~\citeSM{App:Haeusler/Luschgy} together with the previously shown uniform stochastic convergence of $\tilde{X}^n$ towards $X^n$. Thus, in what follows we establish 
\begin{align}\label{eq_proof:Jacod_4}
\tilde{X}^n\ \stackrel{\cF-st}{\longrightarrow}\ X. 
\end{align}
This proof is conducted in eight steps.
\begin{itemize}
\item[(i)] Denote by $\tilde{B}^n$ and $\tilde{\nu}^n$ the first and third characteristic of $(\tilde{X}^n,\tilde{W}^n)$. We now show that~\eqref{eq_Jacod:1char}-\eqref{eq_Jacod:3char} hold true with the same limits for $\tilde{X}^n$ in place of $X^n$ and $\tilde{W}^n$ replacing $W^n$, i.e. we prove for all $t\in [0,1]$, $g\in\mathcal{C}$ and $f(x,y)=x^2$ that:

\smallskip
\begin{itemize}
\item[(a)] $\sup_{s\leq t} |\tilde{B}_s^n - B_s| \longrightarrow_{\Pr} 0$,\smallskip
\item[(b)] $\langle \tilde{X}^n,\tilde{X}^n\rangle_t\longrightarrow_{\Pr} (f\star \nu)_t$,\smallskip
\item[(c)] $\langle \tilde{X}^n,\tilde{W^n}\rangle_t \longrightarrow_{\Pr} 0$,\smallskip
\item[(d)] $(g\star \tilde{\nu}^n)_t \longrightarrow_{\Pr} (g\star \nu)_t$,\smallskip
\end{itemize}
and additionally\smallskip
\begin{itemize}
\item[(e)] $\langle \tilde{W}^n,\tilde{W}^n\rangle_t \longrightarrow_\Pr\ t$.
\end{itemize}\medskip
We start with claim (a). Here, we have
\begin{align*}
\sup_{s\leq t} |\tilde{B}_s^n - B^n_s| \leq \sum_{k=1}^n \E\left[|\chi_{nk}|\1_{\{|\chi_{nk}|>\kappa\}}|\cF_{(k-1)/n}\right] \leq \frac{1}{\kappa} \sum_{k=1}^n \E\left[|\chi_{nk}|^2\1_{\{|\chi_{nk}|>\kappa\}}|\cF_{(k-1)/n}\right].
\end{align*}
The last expression converges to zero in probability by~\eqref{eq_Jacod:large_jumps} and thus (a) follows by~\eqref{eq_Jacod:1char}. For part (b), we note that
\[ \langle \tilde{X}^n,\tilde{X}^n\rangle_t = \sum_{k=1}^{\lfloor nt\rfloor} \E\left[ (\chi_{nk}\1_{\{|\chi_{nk}|\leq \kappa\}}-\E[\chi_{nk}\1_{\{|\chi_{nk}|\leq \kappa\}}|\cF_{(k-1)/n}])^2 | \cF_{(k-1)/n}\right]\]
and 
\[ \langle X^n,X^n\rangle_t = \sum_{k=1}^{\lfloor nt\rfloor} \E\left[ (\chi_{nk}-\E[\chi_{nk}|\cF_{(k-1)/n}])^2 | \cF_{(k-1)/n}\right]. \]
By the Cauchy--Schwarz inequality for symmetric positive semidefinite bilinear forms,
\begin{align}\label{eq_proof:Jacod_6}
\begin{split}
&| \langle \tilde{X}^n,\tilde{X}^n\rangle_t-\langle X^n,X^n\rangle_t| \\
&\hspace{1cm}\leq  \langle \tilde{X}^n-X^n, \tilde{X}^n-X^n\rangle_t + 2|\langle X^n,\tilde{X}^n-X^n\rangle_t| \\
&\hspace{1cm}\leq \langle \tilde{X}^n-X^n, \tilde{X}^n-X^n\rangle_t + \sqrt{|\langle X^n,X^n\rangle_t||\langle \tilde{X}^n-X^n, \tilde{X}^n-X^n\rangle_t|}.
\end{split}
\end{align}
Now,
\begin{align}\label{eq_proof:Jacod_2}
\langle \tilde{X}^n-X^n, \tilde{X}^n-X^n\rangle_t \leq \sum_{k=1}^{\lfloor nt\rfloor} \E\left[\chi_{nk}^2\1_{\{|\chi_{nk}|>\kappa\}}|\cF_{(k-1)/n}\right] \longrightarrow_\Pr 0
\end{align}
by~\eqref{eq_Jacod:large_jumps}. Then, (b) follows from~\eqref{eq_proof:Jacod_6} and $\langle X^n,X^n\rangle_t \rightarrow_{\Pr}(f\star\nu)_t$ given in~\eqref{eq_Jacod:2char_1}. The reasoning for (c) is similar to that of (b). Here, we first observe that with the abbreviation $\tilde{w}_{nk}=(W_{k/n}-W_{(k-1)/n})\1_{\{|W_{k/n}-W_{(k-1)/n}|\leq 1\}}$,
\begin{align*}
\langle \tilde{X}^n, \tilde{W}^n\rangle_t &= \sum_{k=1}^{\lfloor nt\rfloor} \E\Big[ (\chi_{nk}\1_{\{|\chi_{nk}|\leq \kappa\}}-\E[\chi_{nk}\1_{\{|\chi_{nk}|\leq \kappa\}}|\cF_{(k-1)/n}])  \\
&\hspace{4cm} \cdot (\tilde{w}_{nk} - \E[\tilde{w}_{nk}|\cF_{(k-1)/n}]) \big| \cF_{(k-1)/n}\Big]
\end{align*} 
and
\[ \langle X^n, W^n\rangle_t = \sum_{k=1}^{\lfloor nt\rfloor} \E\left[ (\chi_{nk}-\E[\chi_{nk}|\cF_{(k-1)/n}])(W_{k/n}-W_{(k-1)/n}) | \cF_{(k-1)/n}\right]. \]
Then, again by the Cauchy--Schwarz inequality for symmetric and positive semidefinite bilinear forms,
\begin{align*}
\left| \langle \tilde{X}^n, \tilde{W}^n\rangle_t - \langle X^n,W^n\rangle_t\right| &= \left| \langle \tilde{X}^n-X^n,\tilde{W}^n\rangle_t  + \langle X^n, \tilde{W}^n-W^n\rangle_t\right| \\
&\leq \sqrt{\langle \tilde{X}^n-X^n, \tilde{X}^n-X^n\rangle_t\langle \tilde{W}^n,\tilde{W}^n\rangle_t } \\
&\hspace{1cm} + \sqrt{\langle \tilde{W}^n-W^n, \tilde{W}^n-W^n\rangle_t\langle \tilde{X}^n,\tilde{X}^n\rangle_t }.
\end{align*}
The first summand converges to zero in probability by~\eqref{eq_proof:Jacod_2}, for the second one it follows by
\begin{align}\label{eq_proof:Jacod_3}
\begin{split}
&\E\left[\langle \tilde{W}^n-W^n, \tilde{W}^n-W^n\rangle_t \right]\\
&\hspace{1cm}\leq \sum_{k=1}^{\lfloor nt\rfloor}  \E\left[(W_{k/n}-W_{(k-1)/n})^2\1_{\{|W_{k/n}-W_{(k-1)/n}|>1\}} \right] \\
&\hspace{1cm}\leq \sum_{k=1}^{\lfloor nt\rfloor} \sqrt{\E\left[  (W_{k/n}-W_{(k-1)/n})^4\right]} \sqrt{\Pr(|W_{k/n}-W_{(k-1)/n}|>1)}  \\
&\hspace{1cm}\leq \frac{\lfloor nt\rfloor}{n} \sqrt{\Pr\left(|Z|>\sqrt{n}\right)} \longrightarrow 0,
\end{split}
\end{align} 
where $Z\sim\mathcal{N}(0,1)$ denotes a standard Gaussian random variable. Then, condition (c) follows by~\eqref{eq_Jacod:2char_2}. To prove (d), we abbreviate  $\tilde{w}_{nk}=(W_{k/n}-W_{(k-1)/n})\1_{\{|W_{k/n}-W_{(k-1)/n}|\leq 1\}}$ and bound
\begin{align*}
&\left| (g\star \tilde{\nu}^n)_t-(g\star \nu^n)_t\right|\\
&\hspace{1cm}\leq \sum_{k=1}^{\lfloor nt\rfloor} \E\left[\left. | g(\chi_{nk}\1_{\{|\chi_{nk}|\leq \kappa\}}, \tilde{w}_{nk}) - g(\chi_{nk},W_{k/n}-W_{(k-1)/n})|\right|\cF_{(k-1)/n}\right] \\
&\hspace{1cm}\leq 2\|g\|_{\sup} \sum_{k=1}^{\lfloor nt\rfloor} \E\left[ \1_{\{|\chi_{nk}|> \kappa\}}+\1_{\{|W_{k/n}-W_{(k-1)/n}|> 1\}}|\cF_{(k-1)/n}\right] \\
&\hspace{1cm}\leq 2\|g\|_{\sup}\left( \frac{1}{\kappa^2} \sum_{k=1}^{\lfloor nt\rfloor} \E\left[ \chi_{nk}^2\1_{\{|\chi_{nk}|> \kappa\}}|\cF_{(k-1)/n}\right] + \lfloor nt\rfloor \Pr\left( |Z|>\sqrt{n}\right) \right),
\end{align*}
where $Z\sim\mathcal{N}(0,1)$ is a standard Gaussian random variable. The last expression now tends to zero by~\eqref{eq_Jacod:large_jumps} and the Gaussian tail inequality. Then by~\eqref{eq_Jacod:3char}, (d) follows. Following the steps for (b), it is sufficient for (e) to show that
\[ \langle \tilde{W}^n-W^n, \tilde{W}^n-W^n\rangle_t\longrightarrow_{\Pr}\ 0, \]
which is a direct consequence of~\eqref{eq_proof:Jacod_3}. \bigskip

\item[(ii)] We denote the second modified characteristic of $(\tilde{X}^n,\tilde{W}^n)$ by $\tilde{G}^n$, i.e.
\[ \tilde{G}^n := \begin{pmatrix}
\langle \tilde{X}^n,\tilde{X}^n\rangle & \langle \tilde{X}^n,\tilde{W}^n\rangle \\
\langle \tilde{W}^n,\tilde{X}^n\rangle & \langle \tilde{W}^n,\tilde{W}^n\rangle
\end{pmatrix}\]
and define $G=(G_t)_{t\in [0,1]}$ via
\[ G_t = \begin{pmatrix} (f\star \nu)_t & 0 \\ 0 & t \end{pmatrix}. \] 
For any unit vector $u\in\R^2$, by Proposition~II.2.17(b) in~\citeSM{App:Jacod/Shiryaev}, $u^T\tilde{G}_t^n u$ is non-decreasing in $t$; the same is obviously true for $u^TG_t u$. By a standard argument (that allows to deduce uniform convergence from pointwise convergence of non-decreasing functions to a non-decreasing limiting function), we then obtain 
\[ \sup_{t\leq 1} \left|u^T\tilde{G}_t^n u - u^TG_t u\right| = \sup_{t\leq 1} \left| u^T (\tilde{G}_t^n - G_t)u\right| \longrightarrow_{\Pr} 0. \]
By setting $u$ to the standard basis vectors, we conclude uniform stochastic convergence of the diagonal terms. Subsequently choosing $u=(1/\sqrt{2},1/\sqrt{2})$ gives uniform stochastic convergence of the off-diagonal entries and we conclude for any matrix norm $\|\cdot\|$,
\[ \sup_{s\leq t} \| \tilde{G}_s^n - G_s\| \longrightarrow_\Pr 0 \qquad \textrm{ for all } t\in [0,1].\]

\item[(iii)] By the convergences~\eqref{eq_Jacod:1char}-\eqref{eq_Jacod:3char} and that of $\langle \tilde{W}^n\rangle$, the processes $B$, $G$ and $g\star\nu$ for $g\in\mathcal{C}$ are $(\cF_t)_{t\in [0,1]}$-adapted. By assumption, $B$ is continuous in $t$, and for $G$ and $g\star\nu$, continuity in $t$ follows from the assumption $\nu(\{t\}\times \R^2)=0$. Therefore, $B,C$ and $g\star\nu$ are adapted and continuous processes, hence $(\cF_t)_{t\in [0,1]}$-predictable.\smallskip

\item[(iv)] From $(\cF_t)_{t\in [0,1]}$-predictability of $B,G$ and $\nu$ (which was established in (iii)), we can deduce from Theorem~$3.2$ in~\citeSM{App:Jacod_stablePII} that there exists a very good extension $\mathcal{B}'$ of $\mathcal{B}$ and a quasi-left continuous process $X$ on $\mathcal{B}'$ which is an $\cF$-conditional PII and the pair $(X,W)$ admits the characteristics $(B, C, \nu)$. By the same result, $X$ can be realized as follows: Let $(\hat{\Omega}, \hat{\cF}, (\hat{\cF}_t)_{t\in [0,1]})$ be the canonical space of all càdlàg functions and $X$ the canonical process. Then $\Pr'(d\omega,d\hat{\omega}) = \Pr(d\omega)Q_\omega(d\hat{\omega})$ and $Q_\omega$ is entirely determined by the first coordinate of $B$, $(f\star\nu)$ and $\nu(dt, dx,\{0\})$.\smallskip

\item[(v)] As the $\sigma$-field generated by $W$ is countably generated, by Proposition~$3.4.5$ in~\citeSM{App:Cohn} there exists a countable collection $\{V_m: m\in\N\}$ of bounded random variables which is dense in $L^2(\Omega,\cF,\Pr)$. We set $N^m := (N_t^m)_{t\in [0,1]}$ with $N_t^m = \E[V_m|\cF_t]$. According to IX.7.9 and IX.7.10 in~\citeSM{App:Jacod/Shiryaev}, we have\smallskip
\begin{itemize}
\item[(A)] Every bounded martingale on $(\Omega,\cF, (\cF_t)_{t\in [0,1]},\Pr)$ is the limit in $L^2$, locally uniformly in time, of a sequence of sums of stochastic integrals w.r.t. a finite number of different $N^m$.
\item[(B)] $(\cF_t)_{t\in [0,1]}$ is the smallest filtration, up to $\Pr$-null sets, w.r.t. which all $N^m$, $m\in\N$, are adapted.
\end{itemize}\smallskip

\item[(vi)] For each $m\in\N$, we define $N(n)^m=(N(n)_t^m)_{t\in [0,1]}$ via $N(n)_t^m = N_{\lfloor nt\rfloor/n}^m$. As $N^m$ is bounded, we clearly have $\sup_{\omega,t,n}|N(n)_t^m(\omega)| <\infty$ and because $N^m$ is continuous by the martingale representation property, we find
\[ \left( \tilde{W}^n, N(n)^1, \dots, N(n)^m\right) \longrightarrow_\Pr\ \left( W, N^1, \dots, N^m\right)\]
in the Skorohod space $\mathcal{D}([0,1],\R^{m+1})$.\\

We can consider $\mathcal{N}=(N^m)_{m\in\N}$ and $\mathcal{N}(n) = (N^m(n))_{m\in\N}$ as a process with paths in the Skorohod space $\mathcal{D}([0,1],\R^\N)$ and $\tilde{\mathcal{H}}^n =(g\star \tilde{\nu}^n)_{g\in\mathcal{C}}$, $\mathcal{H}=(g\star\nu)_{g\in\mathcal{C}}$ as processes in $\mathcal{D}([0,1],\R^\mathcal{C})$. By our convergence assumption and (ii), we have
\begin{align}\label{eq_proof:Jacod_1}
\left( \tilde{B}^n, \tilde{G}^n, \tilde{\mathcal{H}}^n, \tilde{W}^n, \mathcal{N}(n)\right) \longrightarrow_{\Pr}\ \left( B,G,\mathcal{H},W, \mathcal{N}\right)
\end{align} 
in the Skorohod sense. As the jumps of $(\tilde{X}^n,\tilde{W}^n)$ are uniformly bounded, Theorem~VI.4.18 and Lemma~VI.4.22 in~\citeSM{App:Jacod/Shiryaev} together with (ii) reveal that $(\tilde{X}^n,\tilde{W}^n)$ is tight in $\mathcal{D}([0,1],\R^2)$. The process on the right-hand side of~\eqref{eq_proof:Jacod_1} is continuous and we conclude with Corollary~VI.3.33 in~\citeSM{App:Jacod/Shiryaev} that $(\tilde{X}^n, \tilde{W}^n, \tilde{B}^n, \tilde{G}^n, \tilde{\mathcal{H}}_n, \mathcal{N}(n))$ is tight in the respective Skorohod space. Moreover, for any limiting process $(\hat{X},\hat{W}, \hat{B}, \hat{G},\hat{\mathcal{H}}, \hat{\mathcal{N}})$, we have $\mathcal{L}(\hat{W}, \hat{B}, \hat{G},\hat{\mathcal{H}}, \hat{\mathcal{N}}) = \mathcal{L}(W,G,B,\mathcal{H},\mathcal{N})$.\smallskip

\item[(vii)] We now choose any subsequence, indexed in $n_k$ such that the sequence of distributions $\mathcal{L}(\tilde{X}^{n_k}, \tilde{W}^{n_k}, \tilde{B}^{n_k}, \tilde{G}^{n_k},\tilde{\mathcal{H}}^{n_k}, \mathcal{N}(n_k))$ converges weakly to some measure $\overline{\Q}$ on the corresponding image. From what precedes, one can realize the limit as follows: Consider again the canonical space $(\hat{\Omega},\hat{\cF}, (\hat{\cF}_t)_{t\in [0,1]})$ of real-valued càdlàg functions on $[0,1]$ with the canonical process $X$. Then we set $\tilde{\Omega} = \Omega\times \hat{\Omega}$, $\tilde{\cF}=\cF\otimes \hat{\cF}$ and $\tilde{\cF}_t = \bigcap_{s> t} \cF_s\otimes\hat{\cF}_s$. Since $\cF=\sigma(V_m: m\in\N)$ up to $\Pr$-null sets, the pullback measure of $\overline{\Q}$ is a measure on $(\tilde{\Omega},\tilde{\cF})$, in particular there exists a probability measure $\tilde{\Pr}$ on $(\tilde{\Omega},\tilde{\cF})$ whose $\Omega$-marginal is $\Pr$, and such that $\mathcal{L}(\tilde{X}^{n_k}, \tilde{W}^{n_k}, \tilde{B}^{n_k}, \tilde{G}^{n_k},\tilde{\mathcal{H}}^{n_k}, \mathcal{N}(n_k))$ converges weakly to the law of $(X,W,B,G,\mathcal{H},\mathcal{N})$ under $\tilde{\Pr}$.\\

Therefore, we have an extension $\tilde{\mathcal{B}}=(\tilde{\Omega},\tilde{\cF},(\tilde{\cF}_t)_{t\in [0,1]},\tilde{\Pr})$ of $\mathcal{B}=(\Omega,\cF,(\cF_t)_{t\in [0,1]},\Pr)$ with a disintegration $\tilde{\Pr}(d\omega,d\hat{\omega})=\Pr(d\omega)\tilde{Q}_\omega(d\hat{\omega})$ (since $(\hat{\Omega},\hat{\cF})$ is Polish, see II.1.2 in~\citeSM{App:Jacod/Shiryaev}). Up to $\tilde{\Pr}$-null sets, the filtrations $(\cF)_{t\in [0,1]}$ and $(\tilde{\cF}_t)_{t\in [0,1]}$ are generated by $\mathcal{N}$ and $(\mathcal{N},X)$, respectively (this follows from property (B)).\\

We now show that $(X,W)$ is a semimartingale on the stochastic basis $\tilde{\mathcal{B}}$ with respective characteristics $(B,C,\nu)$. To this aim, denote by $\mu^n$ the jump measure of $(\tilde{X}^n,\tilde{W}^n)$. Then all components of $\mathcal{N}(n)$, $(\tilde{X}^n,\tilde{W}^n) - \tilde{B}^n$ and $(g\star \mu^n - g\star\tilde{\nu}^n)_{g\in\mathcal{C}}$
are $\mathbb{F}^n$-local martingales with uniformly bounded jumps. Additionally, by weak convergence of $(\tilde{X}^n,\tilde{W}^n)$ to $(X,W)$, Corollary~VI.2.8 in~\citeSM{App:Jacod/Shiryaev} provides weak convergence of $g\star\mu^n$ to $g\star \mu$ for all $g\in\mathcal{C}$. Hence, Proposition~IX.1.17 in~\citeSM{App:Jacod/Shiryaev} yields that all components of ($\mu$ denoting the jump measure associated to $(X,W)$)
\begin{align}\label{eq_proof:processes_are_F-martingales}
\mathcal{N},\quad  (X,W) \quad\textrm{ and }\quad (g\star \mu - g\star\nu)_{g\in\mathcal{C}}
\end{align}
are $(\tilde{\cF}_t)_{t\in [0,1]}$-local martingales. Let $\tilde{L}^n = (\tilde{X}^n,\tilde{W}^n)-\tilde{B}^n$ and $L=(X,W)-B$, which should be read as a column vector. Identify $(\tilde{L}^n, \tilde{G}^n), (L,G)$ as elements of $\R^6$ and introduce for the Euclidean norm $\|\cdot\|_2$ the stopping times
\begin{align*}
\tau_R^n &:= \inf\{t\in [0,1]: \|(\tilde{L}_t^n, \tilde{G}_t^n)\|_2\geq R \textrm{ or } \|(\tilde{L}_{t-}^n, \tilde{G}_{t-}^n) \|_2\geq R\}\wedge 1, \\
\tau_R &:= \inf\{t\in [0,1]: \|(L_t,G_t)\|_2\geq R \textrm{ or } \|(L_{t-},G_{t-})\|_2 \geq R\}\wedge 1.
\end{align*} 
Following the construction in the proof of Proposition~IX.1.17 in~\citeSM{App:Jacod/Shiryaev}, it is possible to find a sequence $(R_m)_{m\in\N}$ with $R_m\nearrow\infty$ such that the weak convergence in (v) can be extended to show $\mathcal{L}((\tilde{L}^n)^{\tau_{R_m}^n},(\tilde{G}^n)^{\tau_{R_m}^n}, \tilde{X}^n, \mathcal{N}(n))$ converges weakly (in $\mathcal{D}([0,1],\R^2\times\R^4\times\R\times\R^\N)$) to the law of $(L^{\tau_{R_m}},G^{\tau_{R_m}},X,\mathcal{N})$. Then the $\mathbb{F}^n$-local martingale $(\tilde{L}^n)^{\tau_{R_m}^n}((\tilde{L}^n)^{\tau_{R_m}^n})^T -(\tilde{G}^n)^{\tau_{R_m}^n}$ has bounded jumps and Proposition~IX.1.17 in~\citeSM{App:Jacod/Shiryaev} proves the bounded limit $L^{\tau_{R_m}}(L^{\tau_{R_m}})^T-G^{\tau_{R_m}}$ to be a $(\tilde{\cF}_t)_{t\in [0,1]}$-martingale. As $L$ and $G$ are càdlàg, $\tau_{R_m}\nearrow 1$ for $m\to\infty$ which implies that $LL^T-G$ is a $(\tilde{\cF}_t)_{t\in [0,1]}$-local martingale. This was also shown to be true for the processes in~\eqref{eq_proof:processes_are_F-martingales}, and we conclude that $(X,W)$ is a semimartingale on $\tilde{\mathcal{B}}$ with characteristics $(B,C,\nu)$. Because all elements of $\mathcal{N}$ are $\tilde{\mathcal{B}}$-martingales, property (A) of step (v) gives that every martingale on $\mathcal{B}$ is also a martingale on $\tilde{\mathcal{B}}$ and our extension is very good. Then, Theorem~$3.2$ of~\citeSM{App:Jacod_stablePII} states that the conditional $\tilde{\Pr}$-law of $X$ knowing $\mathcal{F}$ is entirely determined by $W$ and the characteristics of the pair $(X,W)$. In particular, compare (iv), we have $\tilde{Q}_\omega= Q_\omega$ for $\Pr$-almost every $\omega\in\Omega$ and the original sequence $(\tilde{X}^n, \tilde{W}^n, \tilde{B}^n, \tilde{G}^n, \tilde{\mathcal{H}}_n, \mathcal{N}(n))$ converges in distribution to $(X,W,B,G,\mathcal{H},\mathcal{N})$ as defined on the basis $\tilde{\mathcal{B}}$.\smallskip

\item[(viii)] It remains to prove that the convergence is indeed $\cF$-stable. This follows as in the fourth step of the proof of Theorem~$2.1$ in~\citeSM{App:Jacod_stableGaussian}, but is given here for the readers convenience. From (vii) we know that the sequence $(\tilde{X}^n,\mathcal{N}(n))$ converges in law to $(X,\mathcal{N})$. In particular, if $f:\mathcal{D}([0,1],\R)\longrightarrow\R$ is a bounded continuous function we find (denoting with $\tilde{\E}$ the expectation with respect to $\tilde{\Pr}$),
\[ \E\left[ f(\tilde{X}^n) N(n)_1^m\right] \longrightarrow {\E}\left[ f(X) N_1^m\right], \]
since $N(n)^m$ is a component of $\mathcal{N}(n)$ that is uniformly bounded in $n$. By boundedness of $N(n)^m$ and convergence of $N(n)^m$ to $N^m$ in the Skorohod space in probability, we can deduce that $N(n)_1^m \rightarrow N_1^m$ in $L^1$, hence
\[ \E\left[ f(\tilde{X}^n) N_1^m\right] \longrightarrow \tilde{\E}\left[ f(X)N_1^m\right].\]
Since $\tilde{\E}[UN_1^m] = \tilde{\E}[UV_m]$ for any bounded $\tilde{\cF}$-measurable variable $U$, we deduce
\[ \E\left[ f(\tilde{X}^n) V_m\right] \longrightarrow \tilde{\E}\left[ f(X)V_m\right].\]
Finally, any bounded $\cF$-measurable random variable $V$ is the $L^2$-limit of a sequence in $\{V_m: m\in\N\}$ (see (v)), and we find the desired stable convergence
\[ \E\left[ f(\tilde{X}^n) V\right] \longrightarrow \tilde{\E}\left[ f(X)V\right].\]
\end{itemize}
\end{proof}

\section{Remaining proofs of Section~\ref{Section:Limiting_distribution}}\label{App:conditions_prop_stable}

We recall the assumption that $N=m+M$ and $z_m<\dots <z_1<0< z_{m+1}< \dots< z_N$, together with the definition
\[ z_0=0, \qquad z_j^- := z_j,\ j=1,\dots, m\quad\textrm{ and } z_j^+:= z_{m+j},\ j=1,\dots, M.\]
In particular, we have
\[ \sum_{l=1}^N \kappa_l \ell_{n,t}(z_l) = \sum_{l=1}^m \kappa_l^- \ell_{n,t}(z_l^-) + \sum_{l=1}^M \kappa_l^+ \ell_{n,t}(z_l^+).\]
Recalling the definition of $Z_k^j(\theta',\theta)$ for $\theta'\leq\theta$ given in~\eqref{eq:def_Z_j} and~\eqref{eq:def_Z_j28}, we can further write
\[ \sum_{l=1}^m \kappa_l^- \ell_{n,t}(z_l^-) = \sum_{k=1}^{\lfloor nt\rfloor} \left( \sum_{l=1}^m  \kappa_l^- \left(-\sum_{j=1}^9 Z_k^j(z_l^-/n,0) + \log\left(\frac{\alpha^2}{\beta^2}\right)\left( \1_{I_{2,k}^{z_l^-/n,0}} + \1_{I_{8,k}^{z_l^-/n,0}}\right)\right)\right), \]
as well as
\[ \sum_{l=1}^M \kappa_l^+ \ell_{n,t}(z_l^+) = \sum_{k=1}^{\lfloor nt\rfloor} \left( \sum_{l=1}^M  \kappa_l^+ \left(\sum_{j=1}^9Z_k^j(0,z_l^+/n) + \log\left(\frac{\beta^2}{\alpha^2}\right)\left( \1_{I_{2,k}^{0,z_{l}^+/n}} + \1_{I_{8,k}^{0,z_l^+/n}}\right)\right)\right). \]

\begin{proof}[Verification of~\eqref{condition_2char_1}]
We use that for the conditional variance we have
\begin{align*}
\E_{\rho_0}\left[(y_{nk}-\E_{\rho_0}[y_{nk}|\cF_{(k-1)/n}])^2|\cF_{(k-1)/n}\right] =\E_{\rho_0}[y_{nk}^2|\cF_{(k-1)/n}]-\E_{\rho_0}[y_{nk}|\cF_{(k-1)/n}]^2 . 
\end{align*}
Next, we have by Lemma~\ref{lemma:upper_bound_transition_density} for $z\geq 0$,
\begin{align*}
\E_{\rho_0}\left[\left. \1_{I_{2,k}^{0,z/n}}\right|\ \cF_{(k-1)/n} \right] &\leq C_{\alpha,\beta}\1_{\{X_{(k-1)/n}<\rho_0\}} \int_{\rho_0}^{\rho_0+z/n} \sqrt{n} \exp\left(-\frac{(y-X_{(k-1)/n})^2}{2\max\{\alpha^2,\beta^2\}/n}\right) dy \\
&\leq C_{\alpha,\beta}\1_{\{X_{(k-1)/n}<\rho_0\}}\frac{z}{\sqrt{n}} \exp\left(-\frac{(X_{(k-1)/n}-\rho_0)^2}{2\max\{\alpha^2,\beta^2\}/n}\right).
\end{align*}
Analogously, one obtains
\[ \E_{\rho_0}\left[\left. \1_{I_{8,k}^{0,z/n}}\right|\ \cF_{(k-1)/n} \right] \leq C_{\alpha,\beta}\1_{\{X_{(k-1)/n}\geq\rho_0+z/n\}}\frac{z}{\sqrt{n}} \exp\left(-\frac{(X_{(k-1)/n}-\rho_0-z/n)^2}{2\max\{\alpha^2,\beta^2\}/n}\right)\]
and for $z<0$ the estimates
\[ \E_{\rho_0}\left[\left. \1_{I_{2,k}^{z/n,0}}\right|\ \cF_{(k-1)/n} \right] \leq C_{\alpha,\beta}\1_{\{X_{(k-1)/n}<\rho_0+z/n\}}\frac{|z|}{\sqrt{n}} \exp\left(-\frac{(X_{(k-1)/n}-\rho_0-z/n)^2}{2\max\{\alpha^2,\beta^2\}/n}\right)\]
and
\[ \E_{\rho_0}\left[\left. \1_{I_{8,k}^{z/n,0}}\right|\ \cF_{(k-1)/n} \right] \leq C_{\alpha,\beta}\1_{\{X_{(k-1)/n}\geq\rho_0\}}\frac{|z|}{\sqrt{n}} \exp\left(-\frac{(X_{(k-1)/n}-\rho_0)^2}{2\max\{\alpha^2,\beta^2\}/n}\right). \]
In particular, by Corollary~\ref{cor:bound_exp_X^2}, we find for $z\geq 0$,
\begin{align}\label{eq_proof:1_char_condition_moment_Ik}
\E_{\rho_0}\left[ \E_{\rho_0}[\1_{I_k}|\cF_{(k-1)/n}]^2\right] \leq C_{\alpha,\beta}\frac{|z|^2}{n\sqrt{k}}\quad\textrm{ for } I_k\in\left\{I_{2,k}^{0,z/n},I_{8,k}^{0,z/n},I_{2,k}^{-z/n,0},I_{8,k}^{-z/n,0} \right\}
\end{align}
and
\begin{align}\label{eq_proof:1_char_moment_Ik}
\E_{\rho_0}\left[ \1_{I_k} \right] \leq C_{\alpha,\beta} |z| \frac{1}{\sqrt{nk}} \quad\textrm{ for } I_k\in\left\{I_{2,k}^{0,z/n},I_{8,k}^{0,z/n},I_{2,k}^{-z/n,0},I_{8,k}^{-z/n,0} \right\}.
\end{align} 
By Proposition~\ref{prop:moment_product}, for $z\geq 0$,
\begin{align*}
\E_{\rho_0}\left[ \E_{\rho_0}[|Z_k^j(-z/n,0)|\big|\cF_{(k-1)/n}]^2\right] &\leq C_{\alpha,\beta} \frac{|z|^2}{n}\frac{1}{\sqrt{k}}, \\
\E_{\rho_0}\left[ \E_{\rho_0}[|Z_k^j(0,z/n)|\big|\cF_{(k-1)/n}]^2\right] &\leq C_{\alpha,\beta} \frac{|z|^2}{n}\frac{1}{\sqrt{k}},
\end{align*} 
such that together with~\eqref{eq_proof:1_char_condition_moment_Ik},
\begin{align*}
&\E_{\rho_0}\left[ \sum_{k=1}^{\lfloor nt\rfloor} \E_{\rho_0}[y_{nk}|\cF_{(k-1)/n}]^2 \right]\\
&\hspace{1cm}\leq C_{\alpha,\beta}(m,M) \max\{|z_m^-|,\dots, |z_1^-|,|z_1^+|,\dots, |z_M^+|\}^2 \frac{1}{n}\sum_{k=1}^{\lfloor nt\rfloor} \frac{1}{\sqrt{k}} \longrightarrow 0.
\end{align*}
Hence,
\[ \sum_{k=1}^{\lfloor nt\rfloor}\E_{\rho_0}\left[(y_{nk}-\E_{\rho_0}[y_{nk}|\cF_{(k-1)/n}])^2|\cF_{(k-1)/n}\right] = \sum_{k=1}^{\lfloor nt\rfloor}\E_{\rho_0}[y_{nk}^2|\cF_{(k-1)/n}] + o_{\Pr_{\rho_0}}(1)\]
and 
\begin{align*}
&\E_{\rho_0}[y_{nk}^2|\cF_{(k-1)/n}] \\
&\hspace{0.2cm} = \E_{\rho_0}\left[ \left( \sum_{l=1}^m \sum_{j=1}^9 \kappa_l^- \left(-Z_k^j(z_l^-/n,0) + \log\left(\frac{\alpha^2}{\beta^2}\right)\left( \1_{I_{2,k}^{z_l^-/n,0}} + \1_{I_{8,k}^{z_l^-/n,0}}\right)\right) \right.\right. \\
&\hspace{0.7cm} + \left.\left.\left. \sum_{l=1}^M \sum_{j=1}^9 \kappa_l^+ \left(Z_k^j(0,z_l^+/n) + \log\left(\frac{\beta^2}{\alpha^2}\right)\left( \1_{I_{2,k}^{0,z_{l}^+/n}} + \1_{I_{8,k}^{0,z_l^+/n}}\right)\right)\right)^2 \right|\ \cF_{(k-1)/n} \right] \\
&\hspace{0.2cm} = \sum_{j,l}\sum_{Z,Z'\in\{ -Z_k^j(z_l^-/n,0),Z_k^j(0,z_l^+/n)\}} \kappa(Z)\kappa(Z')\E_{\rho_0}\left[ ZZ'|\cF_{(k-1)/n}\right] \\
&\hspace{0.5cm} + \sum_{j,l}\sum_{Z\in\{ -Z_k^j(z_l^-/n,0),Z_k^j(0,z_l^+/n)\}} \sum_{I_k\in\{I_{2,k}^{z_l^-/n,0},I_{8,k}^{z_l^-/n,0},I_{2,k}^{0,z_{l}^+/n},I_{8,k}^{0,z_l^+/n} \}} \hspace{-1cm}\kappa(Z)\kappa(I_k)\E_{\rho_0}\left[ Z \1_{I_k}|\cF_{(k-1)/n}\right]\\
&\hspace{1.5cm} + \E_{\rho_0}\left[ \left( \sum_{l=1}^m  \kappa_l^- \log\left(\frac{\alpha^2}{\beta^2}\right)\left( \1_{I_{2,k}^{z_l^-/n},0} + \1_{I_{8,k}^{z_l^-/n,0}}\right) \right.\right. \\
&\hspace{3.5cm} + \left.\left.\left. \sum_{l=1}^M  \kappa_l^+  \log\left(\frac{\beta^2}{\alpha^2}\right)\left( \1_{I_{2,k}^{0,z_{l}^+/n}} + \1_{I_{8,k}^{0,z_l^+/n}}\right)\right)^2 \right|\ \cF_{(k-1)/n} \right] \\
&\hspace{0.2cm} =: S_1(k) + S_2(k) + S_3(k).
\end{align*}
In what follows, we will work with these summands seperately. For the first two, we bound there $L^1(\Pr_{\rho_0})$-norm as both will turn out to be negligable. The third one gives the main contributing term and is evaluated more explicitly. 
\begin{itemize}
\item[$\bullet\ \boldsymbol{S_1(k)}.$] Using $|ab|\leq (a^2+b^2)/2$, we obtain by Proposition~\ref{prop:moment_product},
\[ \E_{\rho_0}\left[\E_{\rho_0}\left[ ZZ'|\cF_{(k-1)/n}\right]\right] \leq \frac12 \E_{\rho_0}\left[ Z^2 + (Z')^2\right] \leq C_{\alpha,\beta}\max\{|z_m^-|,\dots, |z_M^+|\}^2\frac{1}{n\sqrt{k}}. \]
In particular,
\[ \E_{\rho_0}\left[ |S_1(k)|\right] \leq C_{\alpha,\beta}(m,M)\max\{|z_m^-|,\dots, |z_M^+|\}^2 \max\{|\kappa_m^-|,\dots, |\kappa_M^+|\}^2\frac{1}{n\sqrt{k}}.\]
\smallskip

\item[$\bullet\ \boldsymbol{S_2(k)}.$] By Proposition~\ref{prop:moment_product}, Cauchy-Schwarz' inequality and~\eqref{eq_proof:1_char_moment_Ik},
\begin{align*}
\E_{\rho_0}\left[\E_{\rho_0}\left[ Z\1_{I_k}|\cF_{(k-1)/n}\right]\right] &\leq \sqrt{\E_{\rho_0}[Z^2] \E_{\rho_0}[\1_{I_k}]} \leq C_{\alpha,\beta}\max\{|z_m^-|,\dots, |z_M^+|\}^2\frac{1}{n^{3/4}\sqrt{k}}.
\end{align*}
In particular,
\[ \E_{\rho_0}\left[ |S_2(k)|\right] \leq C_{\alpha,\beta}(m,M)\max\{|z_m^-|,\dots, |z_M^+|\}^2 \max\{|\kappa_m^-|,\dots, |\kappa_M^+|\}^2\frac{1}{n^{3/4}\sqrt{k}}.\] \smallskip

\item[$\bullet\ \boldsymbol{S_3(k)}.$] First, we observe that
\begin{align}\label{eq_proof:indicators_28}
\begin{split}
& \1_{I_{2,k}^{z_l^-/n,0}}\1_{I_{8,k}^{z_{l'}^-/n,0}}=0, \quad \1_{I_{2,k}^{z_l^-/n,0}}\1_{I_{2,k}^{0,z_{l'}^+/n}}=0, \quad \1_{I_{2,k}^{z_l^-/n},0}\1_{I_{8,k}^{0,z_{l'}^+/n}}=0, \\
& \1_{I_{8,k}^{z_l^-/n,0}}\1_{I_{2,k}^{0,z_{l'}^+/n}}=0, \quad \1_{I_{8,k}^{z_l^-/n,0}}\1_{I_{8,k}^{0,z_{l'}^+/n}}=0, \quad \1_{I_{2,k}^{0,z_{l}^+/n}}\1_{I_{8,k}^{0,z_{l'}^+/n}}=0.
\end{split}
\end{align} 
Hence,
\begin{align*}
S_3(k) &= \log\left(\frac{\alpha^2}{\beta^2}\right)^2 \E_{\rho_0}\left[\left. \left( \sum_{l=1}^m  \kappa_l^- \1_{I_{2,k}^{z_l^-/n,0}} \right)^2 + \left(\sum_{l=1}^m  \kappa_l^-  \1_{I_{8,k}^{z_l^-/n,0}} \right)^2 \right|\ \cF_{(k-1)/n}\right] \\
&\hspace{1cm} + \log\left(\frac{\beta^2}{\alpha^2}\right)^2 \E_{\rho_0}\left[ \left. \left(\sum_{l=1}^M  \kappa_l^+  \1_{I_{2,k}^{0,z_{l}^+/n}}\right)^2  + \left( \sum_{l=1}^M  \kappa_l^+\1_{I_{8,k}^{0,z_l^+/n}}\right)^2 \right|\ \cF_{(k-1)/n} \right].
\end{align*}
Both conditional expectations can be further split into two summands, yielding in total $S_3(k) = S_{31}(k)+S_{32}(k)+S_{33}(k)+S_{34}(k)$. We will now evaluate those terms seperately:\smallskip
\begin{itemize}
\item[$\boldsymbol{-\ S_{31}(k)}.$] Here, we first observe
\begin{align*}
\left(\sum_{l=1}^m  \kappa_l^-  \1_{I_{2,k}^{z_l^-/n,0}}\right)^2 &= \left( \sum_{l=1}^m \left( \kappa_l^- +\dots + \kappa_m^-\right) \1_{\{X_{(k-1)/n} < \rho_0+z_l^-/n < X_{k/n} \leq \rho_0+z_{l-1}^-/n \leq \rho_0\}} \right. \\
&\hspace{1cm} \left. -  \sum_{l=1}^m \sum_{j=l}^m \kappa_j^- \1_{\{\rho_0+z_j^-\leq X_{(k-1)/n}< \rho_0, \rho_0 +z_l^- <X_{k/n}\leq \rho_0+z_{l-1}^-/n \}} \right)^2 \\
& \hspace{-1.5cm}= \sum_{l=1}^m \left( \kappa_l^- +\dots + \kappa_m^-\right)^2 \1_{\{X_{(k-1)/n} < \rho_0+z_l^-/n < X_{k/n} \leq \rho_0+z_{l-1}^-/n \leq \rho_0\}} \\
&\hspace{-1cm} - 2\left( \sum_{l=1}^m \left( \kappa_l^- +\dots + \kappa_m^-\right) \1_{\{X_{(k-1)/n} < \rho_0+z_l^-/n < X_{k/n} \leq \rho_0+z_{l-1}^-/n \leq \rho_0\}}\right) \\
&\hspace{0cm} \cdot \left(\sum_{l=1}^m \sum_{j=l}^m \kappa_j^- \1_{\{ \rho_0+z_j^-\leq X_{(k-1)/n}< \rho_0, \rho_0 +z_l^- <X_{k/n}\leq \rho_0+z_{l-1}^-/n \}} \right)\\
&\hspace{-0.5cm} + \left(\sum_{l=1}^m \sum_{j=l}^m \kappa_j^- \1_{\{ \rho_0+z_j^-\leq X_{(k-1)/n}< \rho_0, \rho_0 +z_l^- <X_{k/n}\leq \rho_0+z_{l-1}^-/n \}} \right)^2.
\end{align*}
Next, we see that
\begin{align*}
&\E_{\rho_0}\left[ \1_{\{X_{(k-1)/n} < \rho_0+z_l^-/n \leq \rho_0+z_l^-/n < X_{k/n} \leq \rho_0+z_{l-1}^-/n \leq \rho_0\}}|\cF_{(k-1)/n}\right] \\
&\hspace{0.2cm} = \frac{2}{\alpha+\beta}\frac{\beta}{\alpha} \frac{1}{\sqrt{2\pi/n}} \exp\left(-\frac{(X_{(k-1)/n}-\rho_0)^2}{2\alpha^2/n}\right)\1_{\{X_{(k-1)/n}<\rho_0\}}\frac{\left|z_l^- - z_{l-1}^-\right|}{n} + r_{n,k,l}^{31} + s_{n,k,l}^{31},
\end{align*}
where 
\begin{align*}
r_{n,k,l}^{31} &= \int_{\rho_0+z_l^-/n}^{\rho_0+z_{l-1}^-/n} \left( P_1^{\rho_0}(X_{(k-1)/n},y;1/n) - P_1^{\rho_0}(X_{(k-1)/n},\rho_0;1/n)\right) dy,\\
s_{n,k,l}^{31} &= \frac{2}{\alpha+\beta}\frac{\beta}{\alpha} \frac{1}{\sqrt{2\pi/n}}\left(\1_{\{X_{(k-1)/n}<\rho_0+z_l^-/n\}} - \1_{\{X_{(k-1)/n}<\rho_0\}}\right) \\
&\hspace{5cm}\exp\left(-\frac{(X_{(k-1)/n}-\rho_0)^2}{2\alpha^2/n}\right)\frac{\left|z_l^- - z_{l-1}^-\right|}{n},
\end{align*} 
and by arguments analogous to those used in the evaluation of $r_n^{8,2}$ in the proof of Proposition~\ref{prop:expansion_of_drift_t}, we derive with a Taylor expansion of $P_1^{\rho_0}$ that
\[ \E_{\rho_0}\left[ |r_{n,k,l}^{31}|\right] \leq C_{\alpha,\beta}(z_{l-1}^-,z_l^-)\frac{1}{n\sqrt{k}} \]
and following the reasoning for $r_n^{7,2}$ in the proof of Proposition~\ref{prop:expansion_of_drift_t} gives
\[ \E_{\rho_0}\left[ |s_{n,k,l}^{31}|\right] \leq C_{\alpha,\beta}(z_{l-1},z_l)\frac{1}{n\sqrt{k}}. \]
Next, with Lemma~\ref{lemma:upper_bound_transition_density} we get
\begin{align}\label{eq_proof:1char_1}
\begin{split}
&\E_{\rho_0}\left[ \1_{\{\rho_0+z_m^-/n \leq X_{(k-1)/n},X_{k/n}\leq \rho_0\}}\right] \\
&\hspace{0.2cm} \leq \E_{\rho_0}\left[ C_{\alpha,\beta}\1_{\{\rho_0+z_m^-/n \leq X_{(k-1)/n}\leq \rho_0\}}\int_{\rho_0+z_m^-/n}^{\rho_0} \sqrt{n} \exp\left(-\frac{(y-X_{(k-1)/n})^2}{2\max\{\alpha^2,\beta^2\}/n}\right) dy \right] \\
&\hspace{0.2cm} \leq C_{\alpha,\beta}\frac{|z_m^-|}{\sqrt{n}} \int_{\rho_0+z_m^-/n}^{\rho_0} \frac{1}{\sqrt{(k-1)/n}} \exp\left(-\frac{(y-x_0)^2}{2(k-1)\max\{\alpha^2,\beta^2\}/n}\right) dy \\
&\hspace{0.2cm} \leq C_{\alpha,\beta}\frac{|z_m^-|^2}{n} \frac{1}{\sqrt{k}}.
\end{split}
\end{align}
With this bound~\eqref{eq_proof:1char_1}, we then get
\begin{align*}
&\E_{\rho_0}\left[ 2\left( \sum_{l=1}^m \left( \kappa_l^- +\dots + \kappa_m^-\right) \1_{\{X_{(k-1)/n} < \rho_0+z_l^-/n < X_{k/n} \leq \rho_0+z_{l-1}^-/n \leq \rho_0\}}\right) \right.\\
&\hspace{2cm} \left.\cdot \left(\sum_{l=1}^m \sum_{j=l}^m \kappa_j^- \1_{\{  \rho_0+z_j^-\leq X_{(k-1)/n}<\rho_0, \rho_0 +z_l^- <X_{k/n}\leq \rho_0+z_{l-1}^-/n \}} \right) \right]\\
&\hspace{0.2cm} \leq 2m\left(|\kappa_1^-|+\dots + |\kappa_m^-|\right) \sum_{l=1}^m \sum_{j=l}^m |\kappa_j^-| \E_{\rho_0}\left[\1_{\{\rho_0+z_m^-/n \leq X_{(k-1)/n},X_{k/n}\leq \rho_0 \}}\right] \\
&\hspace{0.2cm} \leq C_{\alpha,\beta} m^3\left(|\kappa_1^-|+\dots + |\kappa_m^-|\right)^2 |z_m^-|^2\frac{1}{n\sqrt{k}}.
\end{align*}
Quite similar, \eqref{eq_proof:1char_1} provides us with the bound
\begin{align*}
&\E_{\rho_0}\left[ \left(\sum_{l=1}^m \sum_{j=l}^m \kappa_j^- \1_{\{ \rho_0+z_j^-\leq X_{(k-1)/n}< \rho_0, \rho_0 +z_l^- <X_{k/n}\leq \rho_0+z_{l-1}^-/n \}} \right)^2 \right] \\
&\hspace{1cm} \leq C(m) \max\{|\kappa_1^-|,\dots, |\kappa_m^-|\}^2 \E_{\rho_0}\left[ \1_{\{\rho_0+z_m^-/n \leq X_{(k-1)/n},X_{k/n}\leq \rho_0 \}} \right]\\
&\hspace{1cm} \leq C_{\alpha,\beta}(m) \max\{|\kappa_1^-|,\dots, |\kappa_m^-|\}^2 |z_m^-|^2 \frac{1}{n\sqrt{k}}.
\end{align*}
Summing up, we have shown that
\begin{align*}
S_{31}(k) &=  \log\left(\frac{\alpha^2}{\beta^2}\right)^2\frac{2}{\alpha+\beta}\frac{\beta}{\alpha} \frac{1}{\sqrt{2\pi}} \frac{1}{\sqrt{n}}\exp\left(-\frac{(X_{(k-1)/n}-\rho_0)^2}{2\alpha^2/n}\right)\1_{\{X_{(k-1)/n}< \rho_0\}} \\
&\hspace{1cm} \cdot \sum_{l=1}^m \left( \kappa_l^- +\dots + \kappa_m^-\right)^2 \left|z_l^- - z_{l-1}^-\right| + r_{nk}^{31}+ s_{nk}^{31},
\end{align*}
where $r_{nk}^{31}+ s_{nk}^{31}$ is a remainder with the $L^1(\Pr_{\rho_0})$-bound
\[ \E_{\rho_0}\left[ |r_{nk}^{31}+ s_{nk}^{31}|\right] \leq C_{\alpha,\beta}(m,z_1^-, \dots, z_m^-) \frac{1}{n\sqrt{k}}.\]

\item[$\boldsymbol{-\ S_{32}(k)}.$] First, we note that
\begin{align*}
\left(\sum_{l=1}^m  \kappa_l^-  \1_{I_{8,k}^{z_l^-/n,0}}\right)^2 &= \left(\sum_{l=1}^m \left( \kappa_l^- +\dots + \kappa_m^-\right)\1_{\{\rho_0+z_l^-/n< X_{k/n} \leq \rho_0+z_{l-1}^-/n \leq \rho_0 \leq X_{(k-1)/n}\}}\right)^2  \\
&= \sum_{l=1}^m \left( \kappa_l^- +\dots + \kappa_m^-\right)^2\1_{\{\rho_0+z_l^-/n< X_{k/n} \leq\rho_0+z_{l-1}^-/n \leq \rho_0 \leq X_{(k-1)/n}\}}.
\end{align*}
Moreover,
\begin{align*}
&\E_{\rho_0}\left[ \1_{\{\rho_0+z_l^-/n< X_{k/n} \leq\rho_0+z_{l-1}^-/n \leq \rho_0 \leq X_{(k-1)/n}\}}|\cF_{(k-1)/n}\right] \\
&\hspace{0.5cm} = \frac{2}{\alpha+\beta}\frac{\beta}{\alpha} \frac{1}{\sqrt{2\pi/n}} \frac{\left|z_l^- - z_{l-1}^-\right|}{n} \exp\left(-\frac{(X_{(k-1)/n}-\rho_0)^2}{2\beta^2/n}\right)\1_{\{X_{(k-1)/n}\geq\rho_0\}} + r_{n,k,l}^{32},
\end{align*}
where
\[ r_{n,k,l}^{32} = \int_{z_{l}^-/n}^{z_{l-1}^-/n} \left( P_4^{\rho_0}(X_{(k-1)/n},y;1/n) - P_4^{\rho_0}(X_{(k-1)/n},\rho_0;1/n)\right) dy\]
and analogously to the treatment of $r_n^{2,2}$ in the proof of Proposition~\ref{prop:expansion_of_drift_t} we see that
\[ \E_{\rho_0}\left[ |r_{n,k,l}^{32}|\right] \leq C_{\alpha,\beta}(z_{l-1}^-,z_l^-) \frac{1}{n\sqrt{k}}.\]
In particular,
\begin{align*}
S_{32}(k) &= \log\left(\frac{\alpha^2}{\beta^2}\right)^2\frac{2}{\alpha+\beta}\frac{\beta}{\alpha} \frac{1}{\sqrt{2\pi}}\frac{1}{\sqrt{n}} \exp\left(-\frac{(X_{(k-1)/n}-\rho_0)^2}{2\beta^2/n}\right)\1_{\{X_{(k-1)/n}\geq\rho_0\}} \\
&\hspace{1cm} \cdot \sum_{l=1}^m \left( \kappa_l^- +\dots + \kappa_m^-\right)^2 \left| z_l^- - z_{l-1}^-\right| +r_{nk}^{32}, 
\end{align*} 
where $r_{nk}^{32}$ is a remainder with the $L^1(\Pr_{\rho_0})$-bound
\[ \E_{\rho_0}\left[ |r_{nk}^{32}|\right] \leq C_{\alpha,\beta}(m,z_1^-, \dots, z_m^-) \frac{1}{n\sqrt{k}}.\]
\smallskip

\item[$\boldsymbol{-\ S_{33}(k)}.$] Analogously to $S_{32}(k)$ we find
\begin{align*}
\left(\sum_{l=1}^M  \kappa_l^+  \1_{I_{2,k}^{0,z_{l}^+/n}}\right)^2 = \sum_{l=1}^M \left( \kappa_l^+ +\dots + \kappa_M^+\right)^2\1_{\{X_{(k-1)/n}< \rho_0 \leq\rho_0+z_{l-1}^+/n < X_{k/n}\leq \rho_0+z_l^+/n\}}
\end{align*} 
and based on that
\begin{align*}
S_{33}(k) &= \log\left(\frac{\beta^2}{\alpha^2}\right)^2 \frac{2}{\alpha+\beta}\frac{\alpha}{\beta} \frac{1}{\sqrt{2\pi}} \frac{1}{\sqrt{n}} \exp\left(-\frac{(X_{(k-1)/n}-\rho_0)^2}{2\alpha^2/n}\right)\1_{\{X_{(k-1)/n}<\rho_0\}}\\
&\hspace{1cm} \cdot \sum_{l=1}^M \left( \kappa_l^+ +\dots + \kappa_M^+\right)^2 \left|z_l^+ - z_{l-1}^+\right|  + r_{nk}^{33}
\end{align*} 
with a remainder $r_{nk}^{33}$ for which we have
\[ \E_{\rho_0}\left[ |r_{nk}^{33}|\right] \leq C_{\alpha,\beta}(M,z_1^+,\dots, z_M^+) \frac{1}{n\sqrt{k}}.\]
\smallskip

\item[$\boldsymbol{-\ S_{34}(k)}.$] By the same arguments employed for $S_{31}(k)$, one obtains
\begin{align*}
S_{34}(k) &=  \log\left(\frac{\beta^2}{\alpha^2}\right)^2\frac{2}{\alpha+\beta}\frac{\alpha}{\beta} \frac{1}{\sqrt{2\pi}} \frac{1}{\sqrt{n}} \exp\left(-\frac{(X_{(k-1)/n}-\rho_0)^2}{2\beta^2/n}\right)\1_{\{X_{(k-1)/n}\geq\rho_0\}}\\
&\hspace{1cm} \cdot \sum_{l=1}^M \left( \kappa_l^+ +\dots + \kappa_M^+\right)^2 \left|z_l^+ - z_{l-1}^+\right|  + r_{nk}^{34},
\end{align*}
where $r_{nk}^{34}$ satisfies
\[ \E_{\rho_0}\left[ |r_{nk}^{34}|\right] \leq C_{\alpha,\beta}(M,z_1^+, \dots, z_M^+) \frac{1}{n\sqrt{k}}.\]
\end{itemize}
\end{itemize}\bigskip
Summing up the results for $S_1(k),S_2(k)$ and $S_3(k)$ gives
\begin{align}\label{eq_proof:1char_2}
\begin{split}
&\E_{\rho_0}[y_{nk}^2|\cF_{(k-1)/n}]\\
&\hspace{0.5cm}= \frac{2}{\alpha+\beta}\frac{\beta}{\alpha} \frac{1}{\sqrt{2\pi}}\frac{1}{\sqrt{n}} \sum_{l=1}^m \left(\log\left(\frac{\alpha^2}{\beta^2}\right) \left(\kappa_l^- +\dots + \kappa_m^-\right)\right)^2 \left|z_l^- - z_{l-1}^-\right| \\
&\hspace{2cm} \cdot  \exp\left(-\frac{(X_{(k-1)/n}-\rho_0)^2}{2(\1_{\{X_{(k-1)/n}<\rho_0\}}\alpha^2+\1_{\{X_{(k-1)/n}\geq\rho_0\}}\beta^2 )/n}\right)  \\
&\hspace{1cm} + \frac{2}{\alpha+\beta}\frac{\alpha}{\beta} \frac{1}{\sqrt{2\pi}}\frac{1}{\sqrt{n}} \sum_{l=1}^M \left(\log\left(\frac{\beta^2}{\alpha^2}\right) \left( \kappa_l^+ +\dots + \kappa_M^+\right)\right)^2 \left|z_l^+ - z_{l-1}^+\right| \\
&\hspace{2cm} \cdot  \exp\left(-\frac{(X_{(k-1)/n}-\rho_0)^2}{2(\1_{\{X_{(k-1)/n}<\rho_0\}}\alpha^2+\1_{\{X_{(k-1)/n}\geq\rho_0\}}\beta^2 )/n}\right) + r_{nk},
\end{split}
\end{align}
where (remember that in this proof $C_{\alpha,\beta}$ may depend on $m,M,z_1^-,\dots, z_m^-, z_1^+,\dots, z_M^+$)
\begin{align}\label{eq_proof:1char_3}
\E_{\rho_0}\left[ |r_{nk}|\right] \leq  C_{\alpha,\beta}(m,M, \kappa_m^-,\dots, \kappa_M^+,z_m^-,\dots, z_M^+) \frac{1}{n^{3/4}\sqrt{k}}.
\end{align}  
Summing this over $k$ then give
\begin{align*}
\sum_{k=1}^{\lfloor nt\rfloor} \E_{\rho_0}[y_{nk}^2|\cF_{(k-1)/n}] &= \sum_{k=1}^{\lfloor nt\rfloor} r_{nk}+\frac{2}{\alpha+\beta} \frac{1}{\sqrt{2\pi}} \Lambda_{\alpha,\beta}^n\left((X_{(k-1)/n})_{1\leq k\leq \lfloor nt\rfloor}\right) \\
&\hspace{0.5cm} \cdot \left(\frac{\beta}{\alpha} \sum_{l=1}^m \left(\log\left(\frac{\alpha^2}{\beta^2}\right) \left(\kappa_l^- +\dots + \kappa_m^-\right)\right)^2 \left|z_l^- - z_{l-1}^-\right| \right. \\
&\hspace{1cm} + \left.\frac{\alpha}{\beta} \sum_{l=1}^M \left(\log\left(\frac{\beta^2}{\alpha^2}\right) \left( \kappa_l^+ +\dots + \kappa_M^+\right)\right)^2 \left|z_l^+ - z_{l-1}^+\right| \right)\\
&\hspace{0.5cm} \longrightarrow_{\Pr_{\rho_0}}\ (f\star\nu)_t,
\end{align*}
where the last step uses Lemma~\ref{lemma:Riemann_approximation} and 
\[ \sum_{k=1}^{\lfloor nt\rfloor} \E_{\rho_0}[|r_{nk}|] \leq C_{\alpha,\beta} n^{-3/4}\sum_{k=1}^n k^{-1/2}\longrightarrow 0. \] 
\end{proof}

\begin{proof}[Verification of~\eqref{condition_2char_2}]
We have to prove 
\begin{align*}
\sum_{k=1}^{\lfloor nt\rfloor} \E_{\rho_0}\left[ \left(y_{nk} - \E_{\rho_0}[y_{nk}|\cF_{(k-1)/n}]\right)\left( W_{k/n}-W_{(k-1)/n}\right) |\cF_{(k-1)/n}\right] \longrightarrow_{\Pr_{\rho_0}} 0.
\end{align*} 
As the conditional expectation given $\cF_{(k-1)/n}$ is $\cF_{(k-1)/n}$-measurable, and the increment $W_{k/n}-W_{(k-1)/n}$ is independent of $\cF_{(k-1)/n}$, we have
\begin{align*}
&\E_{\rho_0}\left[ \E_{\rho_0}[y_{nk}|\cF_{(k-1)/n}]\left( W_{k/n}-W_{(k-1)/n}\right) |\cF_{(k-1)/n}\right] \\
&\hspace{2cm} =\E_{\rho_0}[y_{nk}|\cF_{(k-1)/n}]\E_{\rho_0}\left[W_{k/n}-W_{(k-1)/n} \right] =0 
\end{align*} 
and it suffices to evaluate 
\begin{align*}
&\sum_{k=1}^{\lfloor nt\rfloor} \E_{\rho_0}\left[ y_{nk}\left( W_{k/n}-W_{(k-1)/n}\right) |\cF_{(k-1)/n}\right]\\
&\hspace{0.2cm} \leq \frac{1}{\sqrt{n}}\sum_{k=1}^{\lfloor nt\rfloor}\sqrt{\E_{\rho_0}\left[ y_{nk}^2|\cF_{(k-1)/n}\right]}\\
&\hspace{0.2cm} \leq \frac{1}{\sqrt{n}} C_{\alpha,\beta}(m,M,z_m^-,\dots, z_M^+)\sum_{k=1}^{\lfloor nt\rfloor} \left( \frac{1}{n^{1/4}}\exp\left(-\frac{(X_{(k-1)/n}-\rho_0)^2}{4\max\{\alpha^2,\beta^2\}/n}\right) +  \sqrt{|r_{nk}|}\right) , 
\end{align*} 
where the first inequality is due to the Cauchy-Schwarz inequality for conditional expectation and the second one uses~\eqref{eq_proof:1char_2} together with the inequality $\sqrt{a+b}\leq \sqrt{a}+\sqrt{b}$ for $a,b\geq 0$. By Corollary~\ref{cor:bound_exp_X^2},
\[ \E_{\rho_0}\left[ \frac{1}{\sqrt{n}}\sum_{k=1}^{\lfloor nt\rfloor} \frac{1}{n^{1/4}}\exp\left(-\frac{(X_{(k-1)/n}-\rho_0)^2}{4\max\{\alpha^2,\beta^2\}/n}\right) \right] \leq C_{\alpha,\beta} \frac{1}{n^{3/4}} \sum_{k=1}^{\lfloor nt\rfloor} \frac{1}{\sqrt{k}} \leq C_{\alpha,\beta} n^{-1/4}.\]
Moreover, by Jensen's inequality for sums,
\[ \left( \frac{1}{\sqrt{n}} \sum_{k=1}^n \sqrt{|r_{nk}|}\right)^2 = n \left(\frac1n \sum_{k=1}^n \sqrt{|r_{nk}|}\right)^2 \leq \sum_{k=1}^n |r_{nk}|.\]
By the moment bound~\eqref{eq_proof:1char_3}, we then deduce $n^{-1/2}\sum_{k=1}^{\lfloor nt\rfloor} \sqrt{|r_{nk}|} \longrightarrow_{\Pr_{\rho_0}} 0$ and finally get
\[ \sum_{k=1}^{\lfloor nt\rfloor} \E_{\rho_0}\left[ y_{nk}\left( W_{k/n}-W_{(k-1)/n}\right) |\cF_{(k-1)/n}\right]\longrightarrow_{\Pr_{\rho_0}} 0.\]
\end{proof}

\begin{proof}[Verification of~\eqref{condition_second_moment}]
First, we observe that
\begin{align}\label{eq_proof:Lindeberg_type_cond}
\begin{split}
\left|y_{nk}\right| &= \left| \sum_{l=1}^m  \kappa_l^- \left(-\sum_{j=1}^9 Z_k^j(z_l^-/n,0) + \log\left(\frac{\alpha^2}{\beta^2}\right)\left( \1_{I_{2,k}^{z_l^-/n,0}} + \1_{I_{8,k}^{z_l^-/n,0}}\right)\right) \right. \\
&\hspace{1cm} + \left.\sum_{l=1}^M  \kappa_l^+ \left(\sum_{j=1}^9 Z_k^j(0,z_l^+/n) + \log\left(\frac{\beta^2}{\alpha^2}\right)\left( \1_{I_{2,k}^{0,z_{l}^+/n}} + \1_{I_{8,k}^{0,z_l^+/n}}\right)\right) \right| \\
&\leq C_{\alpha,\beta}(m,M) \max\{|\kappa_m^-|,\dots, |\kappa_1^-|,|\kappa_1^+|,\dots, |\kappa_M^+|\} \\
&\hspace{1cm} \cdot \left( 1+  \sum_{l=1}^m \sum_{j=1}^9 |Z_k^j(z_l^-/n,0)| + \sum_{l=1}^M\sum_{j=1}^9 |Z_k^j(0,z_l^+/n)| \right).
\end{split}
\end{align}
Now let $\epsilon>0$ and choose 
\[ a > 2\left( \sum_{l=1}^m |\kappa_l^-| + \sum_{l=1}^M |\kappa_l^+|\right) \max\left\{ \left|\log\left(\frac{\alpha^2}{\beta^2}\right)\right|, \left|\log\left(\frac{\beta^2}{\alpha^2}\right)\right|\right\} +\epsilon.\]
Then,
\begin{align*}
\left\{ |y_{nk}|>a\right\} &\subset \left\{  \sum_{l=1}^m  |\kappa_l^-| \left(\sum_{j=1}^9 |Z_k^j(z_l^-/n,0)| + \left|\log\left(\frac{\alpha^2}{\beta^2}\right)\right|\left( \1_{I_{2,k}^{z_l^-/n,0}} + \1_{I_{8,k}^{z_l^-/n,0}}\right)\right) \right. \\
&\hspace{0.2cm} + \left.\sum_{l=1}^M  |\kappa_l^+| \left(\sum_{j=1}^9 |Z_k^j(0,z_l^+/n)| + \left|\log\left(\frac{\beta^2}{\alpha^2}\right)\right| \left( \1_{I_{2,k}^{0,z_{l}^+/n}} + \1_{I_{8,k}^{0,z_l^+/n}}\right)\right) > a \right\} \\
&\subset \left\{ \sum_{l=1}^m \sum_{j=1}^9 |\kappa_l^-||Z_k^j(z_l^-/n,0)|+\sum_{l=1}^M \sum_{j=1}^9 |\kappa_l^+| |Z_k^j(0,z_l^+/n)| >\epsilon\right\}
\end{align*}
and in particular
\[ \1_{\{ |y_{nk}|>a\}} \leq \frac{1}{\epsilon^2} \left( \sum_{l=1}^m \sum_{j=1}^9 |\kappa_l^-||Z_k^j(z_l^-/n,0)|+\sum_{l=1}^M \sum_{j=1}^9 |\kappa_l^+| |Z_k^j(0,z_l^+/n)|\right)^2.\]
Together with~\eqref{eq_proof:Lindeberg_type_cond} and Proposition~\ref{prop:moment_product} we then find
\begin{align*}
&\E_{\rho_0}\left[ |y_{nk}|^2 \1_{\{|y_{nk}|>a\}}\right]\\
&\hspace{1cm}\leq C_{\alpha,\beta}(m,M) \max\{|\kappa_m^-|,\dots, |\kappa_1^-|,|\kappa_1^+|,\dots, |\kappa_M^+|\}^4 \epsilon^{-2} \\
&\hspace{2cm} \cdot \E_{\rho_0}\left[ \left(\sum_{l=1}^m \sum_{j=1}^9 |Z_k^j(z_l^-/n,0)| + \sum_{l=1}^M\sum_{j=1}^9 |Z_k^j(0,z_l^+/n)|\right)^2 \right.\\
&\hspace{2.5cm} + \left. \left(\sum_{l=1}^m \sum_{j=1}^9 |Z_k^j(z_l^-/n,0)| + \sum_{l=1}^M\sum_{j=1}^9 |Z_k^j(0,z_l^+/n)|\right)^4 \right] \\
&\hspace{1cm}\leq C_{\alpha,\beta}(m,M) \max\{|\kappa_m^-|,\dots, |\kappa_1^-|,|\kappa_1^+|,\dots, |\kappa_M^+|\}^4 \epsilon^{-2} \\
&\hspace{2cm} \cdot \left( \sum_{l=1}^m \sum_{j=1}^9 \E_{\rho_0}\left[|Z_k^j(z_l^-/n,0)|^2\right] + \sum_{l=1}^M\sum_{j=1}^9 \E_{\rho_0}\left[|Z_k^j(0,z_l^+/n)|^2\right] \right.\\
&\hspace{2.5cm} + \left. \sum_{l=1}^m \sum_{j=1}^9 \E_{\rho_0}\left[|Z_k^j(z_l^-/n,0)|^4\right] + \sum_{l=1}^M\sum_{j=1}^9 \E_{\rho_0}\left[|Z_k^j(0,z_l^+/n)|^4\right] \right) \\
&\hspace{1cm}\leq C_{\alpha,\beta}(m,M) \max\{|\kappa_m^-|,\dots, |\kappa_1^-|,|\kappa_1^+|,\dots, |\kappa_M^+|\}^4 \epsilon^{-2} \\
&\hspace{2cm} \cdot \left( \frac{\max\{|z_m^-|,\dots, |z_M^+|\}^2}{n} +\frac{\max\{|z_m^-|,\dots, |z_M^+|\}^4}{n^2}\right) \frac{1}{\sqrt{k}}.
\end{align*}
Then, summing over $k$ yields for some constant $C$ depending on $\alpha,\beta,m,M$ and all $\kappa_l^\pm$, $z_l^\pm$
\[ \E_{\rho_0}\left[\sum_{k=1}^{\lfloor nt\rfloor} |y_{nk}|^2 \1_{\{|y_{nk}|>a\}}\right] \leq C \frac{1}{n} \sum_{k=1}^{\lfloor nt\rfloor} \frac{1}{\sqrt{k}} \leq \frac{C\sqrt{nt}}{n}\longrightarrow 0.\]
\end{proof}

\begin{lemma}\label{lemma:tightness}
Let $K>0$. Then the sequence $(\ell_n(\cdot/n)_{n\in\N}$ of processes $(\ell_n(z/n))_{z\in [-K,K]}$ is tight in $\mathcal{D}([-K,K])$.
\end{lemma}
\begin{proof}
We prove tightness of $(\ell_n(\cdot/n))_{n\in\N}$ by verifying the moment condition
\begin{align}\label{eq_proof:tightness_moment_condition}
\E_{\rho_0}\left[ |\ell_n(z_1/n) - \ell_n(z_2/n)||\ell_n(z_2/n) - \ell_n(z_3/n)|\right] \leq C |z_1-z_3|^\gamma
\end{align} 
for $z_1\leq z_2\leq z_3$, some constant $C>0$ and $\gamma>1$ (see \citeSM{App:Billingsley_2}, Remark $(13.14)$ after Theorem~$13.5$). Recall the random variables $Z_k^j(\theta',\theta)$ given in Section~\ref{Section:likelihood} in~\eqref{eq:def_Z_j} and~\eqref{eq:def_Z_j28}. With those, we obtain for $a<b$, $a,b\in\{1,2,3\}$,
\begin{align}\label{eq_proof:tightness_1}
\begin{split}
|\ell_n(z_a/n) - \ell_n(z_b/n)| &\leq  \left|\log\left(\frac{\beta^2}{\alpha^2}\right)\right| \sum_{k=1}^n \1_{\{X_{(k-1)/n}\leq \rho_0 + z_a/n\leq X_{k/n} \leq\rho_0+z_b/n\}} \\
&\hspace{1cm} +  \left|\log\left(\frac{\beta^2}{\alpha^2}\right)\right|\sum_{k=1}^n \1_{\{\rho_0 + z_a/n\leq X_{k/n} \leq\rho_0+z_b/n\leq X_{(k-1)/n}\}}  \\
&\hspace{1cm} + \sum_{k=1}^n \sum_{j=1}^9 \left| Z_k^j(z_a/n, z_b/n)\right|.
\end{split}
\end{align}
Next, we will find an upper bound of the $L^2(\Pr_{\rho_0})$-norm of the last of these three summands. By Proposition~\ref{prop:moment_product},
\begin{align*}
&\E_{\rho_0}\left[ \left(\sum_{k=1}^n \sum_{j=1}^9 \left| Z_k^j(z_a/n, z_b/n)\right|\right)^2\right] \\
&\hspace{0.2cm} \leq 2^8 \sum_{j=1}^9 \E_{\rho_0}\left[ \left(\sum_{k=1}^n  \left| Z_k^j(z_a/n, z_b/n)\right|\right)^2\right] \\
&\hspace{0.2cm} \leq 2^8 \sum_{j=1}^9 \sum_{k=1}^n\E_{\rho_0}\left[\left| Z_k^j(z_i/n, z_j/n)\right|^2\right] + 2^{9} \sum_{j=1}^9 \sum_{l=1}^n\sum_{k=1}^l \E_{\rho_0}\left[\left| Z_k^j(z_a/n, z_b/n)Z_l^j(z_a/n, z_b/n)\right|\right] \\
&\hspace{0.2cm} \leq C_{\alpha,\beta} \frac{|z_a-z_b|^2}{n} \left( \sum_{k=1}^n \frac{1}{\sqrt{k}} +\sum_{l=1}^n\sum_{k=1}^{l-1} \frac{1}{\sqrt{k(l-k)}}\right)\\
&\hspace{0.2cm} \leq C_{\alpha,\beta}|z_a-z_b|^2.
\end{align*}
Defining
\[ B_n(z_a,z_b) := |z_a-z_b|^{-1}\sum_{k=1}^n \sum_{j=1}^9 \left| Z_k^j(z_a/n, z_b/n)\right|, \]
the preceding estimate gives
\begin{align}\label{eq_proof:tightness_3}
\E_{\rho_0}[|B_n(z_a,z_b)|^2] \leq C_{\alpha,\beta},
\end{align} 
with a bound $C_{\alpha,\beta}>0$ independent of $n$ and $z_a,z_b$, and we deduce from~\eqref{eq_proof:tightness_1} that
\begin{align}\label{eq_proof:tightness_2}
\begin{split}
&|\ell_n(z_a/n) - \ell_n(z_b/n)| \\
&\hspace{0.5cm} \leq |z_a-z_b| B_n(z_a,z_b) + \left|\log\left(\frac{\beta^2}{\alpha^2}\right)\right| \sum_{k=1}^n \Big(\1_{\{X_{(k-1)/n}\leq \rho_0 + z_a/n\leq X_{k/n}\leq \rho_0+z_b/n\}}\\
&\hspace{6.5cm} + \1_{\{\rho_0 + z_a/n\leq X_{k/n}\leq \rho_0+z_b/n\leq X_{(k-1)/n}\}} \Big).
\end{split}
\end{align}
We now show, how~\eqref{eq_proof:tightness_moment_condition} follows from inequality~\eqref{eq_proof:tightness_2} and the moment bound~\eqref{eq_proof:tightness_3} for $B_n$. By~\eqref{eq_proof:tightness_1},
\begin{align*}
&\E_{\rho_0}\left[ |\ell_n(z_1/n) - \ell_n(z_2/n)||\ell_n(z_2/n) - \ell_n(z_3/n)|\right] \\
&\hspace{0.2cm}\leq C_{\alpha,\beta} \E_{\rho_0}\left[ \left( \sum_{k=1}^n \1_{\{X_{(k-1)/n}\leq \rho_0 + z_1/n\leq X_{k/n}\leq \rho_0+z_2/n\}} + \1_{\{\rho_0 + z_1/n\leq X_{k/n}\leq \rho_0+z_2/n\leq X_{(k-1)/n}\}} \right)\right. \\
&\hspace{2cm} \cdot \left.  \left( \sum_{l=1}^n \1_{\{X_{(l-1)/n}\leq \rho_0 + z_2/n\leq X_{l/n}\leq \rho_0+z_3/n\}} + \1_{\{\rho_0 + z_2/n\leq X_{l/n}\leq \rho_0+z_3/n\leq X_{(l-1)/n}\}} \right)\right]\\
&\hspace{0.7cm} + C_{\alpha,\beta} |z_1-z_2|\E_{\rho_0}\left[ B_n(z_1,z_2)\left(\sum_{l=1}^n \1_{\{X_{(l-1)/n}\leq \rho_0 + z_2/n\leq X_{l/n}\leq \rho_0+z_3/n\}} \right.\right.\\
&\hspace{7cm}  + \1_{\{\rho_0 + z_2/n\leq X_{l/n} \leq \rho_0+z_3/n\leq X_{(l-1)/n}\}} \Bigg)\Bigg] \\
&\hspace{0.7cm} + C_{\alpha,\beta} |z_2-z_3|\E_{\rho_0}\left[ B_n(z_2,z_3)\left(\sum_{k=1}^n \1_{\{X_{(k-1)/n}\leq \rho_0 + z_1/n\leq X_{k/n}\leq \rho_0+z_2/n\}} \right.\right. \\
&\hspace{7cm} + \1_{\{\rho_0 + z_1/n\leq X_{k/n} \leq \rho_0+z_2/n\leq X_{(k-1)/n}\}} \Bigg)\Bigg] \\
&\hspace{0.7cm} + |z_1-z_2||z_2-z_3| \E_{\rho_0}\left[ B_n(z_1,z_2)B_n(z_2,z_3)\right] \\
&\hspace{0.2cm} =: T_1 + T_2 + T_3 + T_4.
\end{align*}
We now show that each summand is bounded by $C_{\alpha,\beta}(K)|z_1-z_3|^\gamma$ with $\gamma=3/2$. Then~\eqref{eq_proof:tightness_moment_condition} is proven and the statement of the lemma then follows.
\begin{itemize}
\item[$\bullet\ \boldsymbol{T_1}$.] Explicitly writing out the expression gives
\begin{align}\label{eq_proof:tightness_5}
\begin{split}
T_1 &= C_{\alpha,\beta} \sum_{k,l=1}^n \E_{\rho_0}\left[\1_{\{X_{(k-1)/n}\leq \rho_0 + z_1/n\leq X_{k/n}\leq \rho_0+z_2/n\}}\1_{\{X_{(l-1)/n}\leq \rho_0 + z_2/n\leq X_{l/n}\leq \rho_0+z_3/n\}}  \right] \\
&\hspace{0.2cm} + C_{\alpha,\beta} \sum_{k,l=1}^n \E_{\rho_0}\left[\1_{\{X_{(k-1)/n}\leq \rho_0 + z_1/n\leq X_{k/n}\leq \rho_0+z_2/n\}} \1_{\{\rho_0 + z_2/n\leq X_{l/n}\leq \rho_0+z_3/n\leq X_{(l-1)/n}\}} \right] \\
&\hspace{0.2cm} + C_{\alpha,\beta} \sum_{k,l=1}^n \E_{\rho_0}\left[ \1_{\{\rho_0 + z_1/n\leq X_{k/n}\leq \rho_0+z_2/n\leq X_{(k-1)/n}\}}\1_{\{X_{(l-1)/n}\leq \rho_0 + z_2/n\leq X_{l/n}\leq \rho_0+z_3/n\}} \right] \\
&\hspace{0.2cm} + C_{\alpha,\beta} \sum_{k,l=1}^n \E_{\rho_0}\left[ \1_{\{\rho_0 + z_1/n\leq X_{k/n}\leq \rho_0+z_2/n\leq X_{(k-1)/n}\}} \1_{\{\rho_0 + z_2/n\leq X_{l/n}\leq \rho_0+z_3/n\leq X_{(l-1)/n}\}}\right].
\end{split}
\end{align}
We construct an upper bound for the first summand over $k,l$ in this expressions, the others can be dealt with analogously. First, notice that the product of the indicators vanishes for $k=l$. This observation is essential in order to derive a bound that is a multiple of $|z_1-z_3|^2$ rather than a multiple of $|z_1-z_3|$. Hence, assume $k<l$. Then by Lemma~\ref{lemma:upper_bound_transition_density} and Corollary~\ref{cor:bound_exp_X^2},
\begin{align*}
&\E_{\rho_0}\left[ \1_{\{X_{(k-1)/n}\leq \rho_0 + z_1/n\leq X_{k/n}\leq \rho_0+z_2/n\}}\1_{\{X_{(l-1)/n}\leq \rho_0 + z_2/n\leq X_{l/n}\leq \rho_0+z_3/n\}}\right] \\
&\hspace{0.5cm} \leq C_{\alpha,\beta}\E_{\rho_0}\left[ \1_{\{X_{(k-1)/n}\leq \rho_0 + z_1/n\leq X_{k/n}\leq \rho_0+z_2/n\}}\1_{\{X_{(l-1)/n}\leq \rho_0 + z_2/n \}} \right.\\
&\hspace{5.5cm} \left. \cdot\int_{\rho_0+z_2/n}^{\rho_0+z_3/n} \sqrt{n}\exp\left(\frac{(y-X_{(l-1)/n})^2}{2\max\{\alpha^2,\beta^2\}/n}\right) dy\right] \\
&\hspace{0.5cm} \leq C_{\alpha,\beta}\frac{|z_2-z_3|}{\sqrt{n}}\E_{\rho_0}\left[ \1_{\{X_{(k-1)/n}\leq \rho_0 + z_1/n\leq X_{k/n}\leq \rho_0+z_2/n\}}\1_{\{X_{(l-1)/n}\leq \rho_0 + z_2/n \}} \right. \\
&\hspace{5.5cm} \left. \cdot\exp\left(-\frac{(X_{(l-1)/n}-\rho_0-z_2/n)^2}{2\max\{\alpha^2,\beta^2\}/n}\right)\right] \\
&\hspace{0.5cm} \leq C_{\alpha,\beta}\frac{|z_2-z_3|}{\sqrt{n}}\E_{\rho_0}\Bigg[ \1_{\{X_{(k-1)/n}\leq \rho_0 + z_1/n\leq X_{k/n}\leq \rho_0+z_2/n\}} \\
&\hspace{4.5cm} \left. \cdot\E_{\rho_0}\left.\left[\exp\left(-\frac{(X_{(l-1)/n}-\rho_0-z_2/n)^2}{2\max\{\alpha^2,\beta^2\}/n}\right) \right|\ \cF_{k/n}\right]\right] \\
&\hspace{0.5cm} \leq C_{\alpha,\beta}\frac{|z_2-z_3|}{\sqrt{n}} \frac{1}{\sqrt{l-k}}\E_{\rho_0}\left[ \1_{\{X_{(k-1)/n}\leq \rho_0 + z_1/n\leq X_{k/n}\leq \rho_0+z_2/n\}} \right] \\
&\hspace{0.5cm} \leq C_{\alpha,\beta}\frac{|z_2-z_3|}{\sqrt{n(l-k)}} \E_{\rho_0}\left[ \1_{\{X_{(k-1)/n}\leq \rho_0 + z_1/n\}} \E_{\rho_0}\left[\left.\1_{\{\rho_0+z_1/n \leq X_{k/n}\leq \rho_0+z_2/n\}}\right|\ \cF_{(k-1)/n}\right]\right] \\
&\hspace{0.5cm} \leq C_{\alpha,\beta}\frac{|z_1-z_2||z_2-z_3|}{n \sqrt{l-k}} \E_{\rho_0}\left[ \1_{\{X_{(k-1)/n}\leq \rho_0 + z_1/n\}} \exp\left(-\frac{(X_{(k-1)/n}-\rho_0-z_1/n)^2}{2 \max\{\alpha^2,\beta^2\}/n}\right)\right] \\
&\hspace{0.5cm} \leq C_{\alpha,\beta}\frac{|z_1-z_3|^2}{n} \frac{1}{\sqrt{l-k}}\frac{1}{\sqrt{k}}.
\end{align*}
This gives
\begin{align*}
&C_{\alpha,\beta} \sum_{k,l=1}^n \E_{\rho_0}\left[\1_{\{X_{(k-1)/n}\leq \rho_0 + z_1/n\leq X_{k/n}\leq \rho_0+z_2/n\}}\1_{\{X_{(l-1)/n}\leq \rho_0 + z_2/n\leq X_{l/n}\leq \rho_0+z_3/n\}}  \right] \\
&\hspace{1cm} \leq C_{\alpha,\beta} \frac{|z_1-z_3|^2}{n}\sum_{l=1}^n\sum_{k=1}^{l-1}  \frac{1}{\sqrt{l-k}}\frac{1}{\sqrt{k}} \leq C_{\alpha,\beta}|z_1-z_3|^2.
\end{align*}
Repeating this step analogously for the other terms in~\eqref{eq_proof:tightness_5} finally gives with $|z_1-z_3|\leq 2K$ that
\[ T_1 \leq C_{\alpha,\beta}(K) |z_1-z_3|^{3/2}.\]\smallskip

\item[$\bullet\ \boldsymbol{T_2}$.] Applying Cauchy--Schwarz' inequality gives (using $|z_2-z_3|\leq |z_1-z_3|$)
\begin{align*}
\begin{split}
T_2 &\leq C_{\alpha,\beta} |z_1-z_3| \sqrt{\E_{\rho_0}\left[ B_n(z_1,z_2)^2\right]} \\
&\hspace{0.2cm} \cdot\E_{\rho_0}\left[ \left(\sum_{k=1}^n \1_{\{X_{(k-1)/n}\leq \rho_0 + z_1/n\leq X_{k/n}\leq \rho_0+z_2/n\}} + \1_{\{\rho_0 + z_1/n\leq X_{k/n}\leq \rho_0+z_2/n\leq X_{(k-1)/n}\}}\right)^2\right]^\frac12 \\
&\hspace{0.2cm} \leq C_{\alpha,\beta}(K) |z_1-z_3|^{3/2},
\end{split}
\end{align*}
where the last inequality uses the $L^2$-bound~\eqref{eq_proof:tightness_3} for $B_n(z_1,z_2)$ and Lemma~\ref{lemma:second_moment_sum_intervals} for the second factor.\smallskip

\item[$\bullet\ \boldsymbol{T_3}$.] This is done along the lines of $T_2$.\smallskip

\item[$\bullet\ \boldsymbol{T_4}$.] Using $z_1\leq z_2\leq z_3$, $|z_1-z_3|\leq 2K$, Cauchy--Schwarz' inequality and~\eqref{eq_proof:tightness_3}, 
\begin{align*}
&|z_1-z_2||z_2-z_3| \E_{\rho_0}\left[ B_n(z_1,z_2)B_n(z_2,z_3)\right]\\
&\hspace{1cm}\leq |z_1-z_3|^2 \sqrt{\E_{\rho_0}[B_n(z_1,z_2)^2] \E_{\rho_0}[B_n(z_2,z_3)^2]}\leq C_{\alpha,\beta}(K) |z_1-z_3|^{3/2}. 
\end{align*}
\end{itemize}
\end{proof}


\bibliographystyleSM{imsart-nameyear} 
\bibliographySM{Bibliography_Appendix}       

\end{document}